\documentclass[oneside,english,reqno]{memo-l}
\usepackage[T1]{fontenc}
\usepackage[latin9]{inputenc}
\usepackage{babel}
\usepackage{varioref}
\usepackage{refstyle}
\usepackage{float}
\usepackage{mathrsfs}
\usepackage{amsthm}
\usepackage{amstext}
\usepackage{amssymb}
\makeindex
\usepackage{graphicx}
\usepackage{esint}
\usepackage[all]{xy}
\PassOptionsToPackage{normalem}{ulem}
\usepackage{ulem}
\usepackage{nomencl}
\providecommand{\printnomenclature}{\printglossary}
\providecommand{\makenomenclature}{\makeglossary}
\makenomenclature
\usepackage[unicode=true,pdfusetitle,
 bookmarks=true,bookmarksnumbered=false,bookmarksopen=false,
 breaklinks=false,pdfborder={0 0 1},backref=false,colorlinks=false]
 {hyperref}
\hypersetup{
 colorlinks=true,citecolor=blue,linkcolor=blue,linktocpage=true}

\makeatletter

\AtBeginDocument{\providecommand\secref[1]{\ref{sec:#1}}}
\AtBeginDocument{\providecommand\subref[1]{\ref{sub:#1}}}
\AtBeginDocument{\providecommand\chapref[1]{\ref{chap:#1}}}
\AtBeginDocument{\providecommand\lemref[1]{\ref{lem:#1}}}
\AtBeginDocument{\providecommand\figref[1]{\ref{fig:#1}}}
\AtBeginDocument{\providecommand\eqref[1]{\ref{eq:#1}}}
\AtBeginDocument{\providecommand\tabref[1]{\ref{tab:#1}}}
\newcommand{\lyxmathsym}[1]{\ifmmode\begingroup\def\b@ld{bold}
  \text{\ifx\math@version\b@ld\bfseries\fi#1}\endgroup\else#1\fi}

\providecommand{\tabularnewline}{\\}
\RS@ifundefined{subref}
  {\def\RSsubtxt{section~}\newref{sub}{name = \RSsubtxt}}
  {}
\RS@ifundefined{thmref}
  {\def\RSthmtxt{theorem~}\newref{thm}{name = \RSthmtxt}}
  {}
\RS@ifundefined{lemref}
  {\def\RSlemtxt{lemma~}\newref{lem}{name = \RSlemtxt}}
  {}

\numberwithin{section}{chapter}
\numberwithin{equation}{section}
\numberwithin{figure}{section}
\numberwithin{table}{section}
\usepackage{enumitem}		% customizable list environments
\theoremstyle{plain}
\newtheorem{thm}{\protect\theoremname}[section]
  \theoremstyle{definition}
  \newtheorem{defn}[thm]{\protect\definitionname}
  \theoremstyle{plain}
  \newtheorem{lem}[thm]{\protect\lemmaname}
  \theoremstyle{plain}
  \newtheorem{cor}[thm]{\protect\corollaryname}
  \theoremstyle{remark}
  \newtheorem{rem}[thm]{\protect\remarkname}
  \theoremstyle{definition}
  \newtheorem{example}[thm]{\protect\examplename}
  \theoremstyle{plain}
  \newtheorem{fact}[thm]{\protect\factname}
  \theoremstyle{plain}
  \newtheorem*{question*}{\protect\questionname}
  \theoremstyle{plain}
  \newtheorem{prop}[thm]{\protect\propositionname}
  \theoremstyle{remark}
  \newtheorem*{rem*}{\protect\remarkname}
  \theoremstyle{remark}
  \newtheorem{summary}[thm]{\protect\summaryname}
  \theoremstyle{plain}
  \newtheorem{question}[thm]{\protect\questionname}

\def\l@figure#1#2{\@tocline{0}{3pt plus2pt}{0pt}{2.5pc}{}{#1}{#2}}
\def\l@table#1#2{\@tocline{0}{3pt plus2pt}{0pt}{2.5pc}{}{#1}{#2}}

\providecommand{\MR}[1]{}

\usepackage{blkarray}

\allowdisplaybreaks

\theoremstyle{definition}

\renewcommand{\tabref}{\Tabref}
\renewcommand{\figref}{\Figref}

\newref{sec}{%
      name      = \RSsectxt,
      names     = \RSsecstxt,
      Name      = \RSSectxt,
      Names     = \RSSecstxt,
      refcmd    = {\ifRSstar\S\fi\ref{#1}},
      rngtxt    = \RSrngtxt,
      lsttwotxt = \RSlsttwotxt,
      lsttxt    = \RSlsttxt}

\newcommand\restr[2]{{% we make the whole thing an ordinary symbol
  \left.\kern-\nulldelimiterspace % automatically resize the bar with \right
  #1 % the function
  \vphantom{\big|} % pretend it's a little taller at normal size
  \right|_{#2} % this is the delimiter
  }}

\makeatother

  \providecommand{\corollaryname}{Corollary}
  \providecommand{\definitionname}{Definition}
  \providecommand{\examplename}{Example}
  \providecommand{\factname}{Fact}
  \providecommand{\lemmaname}{Lemma}
  \providecommand{\propositionname}{Proposition}
  \providecommand{\questionname}{Question}
  \providecommand{\remarkname}{Remark}
  \providecommand{\summaryname}{Summary}
\providecommand{\theoremname}{Theorem}

\begin{document}

\title{Harmonic analysis of a class of reproducing kernel Hilbert spaces
arising from groups}

\author{Palle Jorgensen, Steen Pedersen, and Feng Tian}

\date{01/19/14}

\address{(Palle E.T. Jorgensen) Department of Mathematics, The University
of Iowa, Iowa City, IA 52242-1419, U.S.A. }

\email{palle-jorgensen@uiowa.edu}

\urladdr{http://www.math.uiowa.edu/\textasciitilde{}jorgen/}

\address{(Steen Pedersen) Department of Mathematics, Wright State University,
Dayton, OH 45435, U.S.A. }

\email{steen@math.wright.edu }

\urladdr{http://www.wright.edu/\textasciitilde{}steen.pedersen/}

\address{(Feng Tian) Department of Mathematics, Wright State University, Dayton,
OH 45435, U.S.A.}

\email{feng.tian@wright.edu}

\urladdr{http://www.wright.edu/\textasciitilde{}feng.tian/}

\thanks{The co-authors thank the following for enlightening discussions:
Professors Sergii Bezuglyi, Dorin Dutkay, Paul Muhly, Myung-Sin Song,
Wayne Polyzou, Gestur Olafsson, Robert Niedzialomski, and members
in the Math Physics seminar at the University of Iowa.}



\maketitle
\pagestyle{myheadings}
\markright{REPRODUCING KERNEL HILBERT SPACES AND THEIR HARMONIC ANALYSIS}
\begin{abstract}
We study two extension problems, and their interconnections: (i) extension
of positive definite (p.d.) continuous functions defined on subsets
in locally compact groups $G$; and (ii) (in case of Lie groups $G$)
representations of the associated Lie algebras $La\left(G\right)$,
i.e., representations of $La\left(G\right)$ by unbounded skew-Hermitian
operators acting in a reproducing kernel Hilbert space $\mathscr{H}_{F}$
(RKHS). Our analysis is non-trivial even if $G=\mathbb{R}^{n}$, and
even if $n=1$. If $G=\mathbb{R}^{n}$, (ii), we are concerned with
finding systems of strongly commuting selfadjoint operators $\left\{ T_{i}\right\} $
extending a system of commuting Hermitian operators with common dense
domain in $\mathscr{H}_{F}$. 

Specifically, we consider partially defined positive definite (p.d.)
continuous functions $F$ on a fixed group. From $F$ we then build
a reproducing kernel Hilbert space $\mathscr{H}_{F}$, and the operator
extension problem is concerned with operators acting in $\mathscr{H}_{F}$,
and with unitary representations of $G$ acting on $\mathscr{H}_{F}$.
Our emphasis is on the interplay between the two problems, and on
the harmonic analysis of our RKHSs $\mathscr{H}_{F}$.

In the cases of $G=\mathbb{R}^{n}$, and $G=\mathbb{T}^{n}=\mathbb{R}^{n}/\mathbb{Z}^{n}$,
and generally for locally compact Abelian groups, we establish a new
Fourier duality theory; including for $G=\mathbb{R}^{n}$ a time/frequency
duality, where the extension questions (i) are in time domain, and
extensions from (ii) in frequency domain. Specializing to $n=1$,
we arrive of a spectral theoretic characterization of all skew-Hermitian
operators with dense domain in a separable Hilbert space, having deficiency-indices
$\left(1,1\right)$.

Our general results include non-compact and non-Abelian Lie groups,
where the study of unitary representations in $\mathscr{H}_{F}$ is
subtle.
\end{abstract}
\tableofcontents{}

\chapter{\label{chap:intro}Introduction}

In this Memoir, we study two extension problems, and their interconnections.
The first class of extension problems concerns (i) positive definite
(p.d.) continuous functions on Lie groups $G$, and the second deals
with (ii) Lie algebras of unbounded skew-Hermitian operators in a
certain family of reproducing kernel Hilbert spaces (RKHS). Our analysis
is non-trivial even if $G=\mathbb{R}^{n}$, and even if $n=1$. If
$G=\mathbb{R}^{n}$, we are concerned in (ii) with the study of systems
of $n$ skew-Hermitian operators $\left\{ S_{i}\right\} $ on a common
dense domain in Hilbert space, and in deciding whether it is possible
to find a corresponding system of strongly commuting selfadjoint operators
$\left\{ T_{i}\right\} $ such that, for each value of $i$, the operator
$T_{i}$ extends $S_{i}$.

\index{extension problem}

\index{operator!selfadjoint}

\index{group!Lie}

\index{operator!skew-Hermitian}

\index{RKHS}

\index{positive definite}

The version of this for non-commutative Lie groups $G$ will be stated
in the language of unitary representations of $G$, and corresponding
representations of the Lie algebra $La\left(G\right)$ by skew-Hermitian
unbounded operators.

In summary, for (i) we are concerned with partially defined positive
definite (p.d.) continuous functions $F$ on a Lie group; i.e., at
the outset, such a function $F$ will only be defined on a connected
proper subset in $G$. From this partially defined p.d. function $F$
we then build a reproducing kernel Hilbert space $\mathscr{H}_{F}$,
and the operator extension problem (ii) is concerned with operators
acting on $\mathscr{H}_{F}$, as well as with unitary representations
of $G$ acting on $\mathscr{H}_{F}$. If the Lie group $G$ is not
simply connected, this adds a complication, and we are then making
use of the associated simply connected covering group. For an overview
of highpoints in our Memoir, see sections \ref{sec:Organization}
and \ref{sec:introRKHS} below.

By a theorem of Kolmogorov, every Hilbert space may be realized as
a (Gaussian) reproducing kernel Hilbert space (RKHS), see e.g., \cite{PaSc75,IM65,NF10}.
Since this setting is too general for many applications, it is useful
to restrict the very general framework for RKHSs to concrete cases
in the study of particular spectral theoretic problems; positive definite
functions on groups is a case in point. Such specific issues arise
in physics (see e.g., \cite{Fal74,Jor07}) where one is faced with
extending positive definite functions $F$ which are only defined
on a subset of a given group, say $G$. For a given such $F$, we
introduce an associated RKHS $\mathscr{H}_{F}$ constructed directly
from the given, partially defined, p.d. function $F$.

The axioms of quantum physics (see e.g., \cite{BoMc13,OdHo13,KS02,CKS79,AAR13,Fan10,Maa10,Par09}
for relevant recent papers), are based on Hilbert space, and selfadjoint
operators. 

A brief sketch: A quantum mechanical observable is a Hermitian (selfadjoint)
linear operator mapping a Hilbert space, the space of states, into
itself. The values obtained in a physical measurement are in general
described by a probability distribution; and the distribution represents
a suitable ``average'' (or ``expectation'') in a measurement of
values of some quantum observable in a state of some prepared system.
The states are (up to phase) unit vectors in the Hilbert space, and
a measurement corresponds to a probability distribution (derived from
a projection-valued spectral measure). The particular probability
distribution used depends on both the state and the selfadjoint operator.
The associated spectral type may be continuous (such as position and
momentum; both unbounded) or discrete (such as spin); this depends
on the physical quantity being measured.

Since the spectral theorem serves as the central tool in our study
of measurements, we must be precise about the distinction between
linear operators with dense domain which are only Hermitian as opposed
to selfadjoint. This distinction is accounted for by von Neumann\textquoteright s
theory of deficiency indices \cite{AG93,DS88b}.

(Starting with \cite{vN32a,vN32b,vN32c}, J. von Neumann and M. Stone
did pioneering work in the 1930s on spectral theory for unbounded
operators in Hilbert space; much of it in private correspondence.
The first named author has from conversations with M. Stone, that
the notions \textquotedblleft deficiency-index,\textquotedblright{}
and \textquotedblleft deficiency space\textquotedblright{} are due
to them; suggested by MS to vN as means of translating more classical
notions of \textquotedblleft boundary values\textquotedblright{} into
rigorous tools in abstract Hilbert space: closed subspaces, projections,
and dimension count.)

\index{deficiency indices}

\index{von Neumann}

\index{operator!unbounded}

\index{unitary representation}

\index{group!non-commutative}

\index{group!simply connected}

\index{Kolmogorov}

\index{group!covering}

\section{\label{sec:2ext}Two Extension Problems}

While each of the two extension problems\index{extension problem}
has received a considerable amount of attention in the literature,
our emphasis here will be the interplay between the two problems:
Our aim is a duality theory; and, in the case $G=\mathbb{R}^{n}$,
and $G=\mathbb{T}^{n}=\mathbb{R}^{n}/\mathbb{Z}^{n}$, we will state
our theorems in the language of Fourier duality \index{Fourier duality}
of abelian groups: With the time frequency duality formulation of
Fourier duality for $G=\mathbb{R}^{n}$ we have that both the time
domain and the frequency domain constitute a copy of $\mathbb{R}^{n}$.
We then arrive at a setup such that our extension questions (i) are
in time domain, and extensions from (ii) are in frequency domain.
Moreover we show that each of the extensions from (i) has a variant
in (ii). Specializing to $n=1$, we arrive of a spectral theoretic
characterization of all skew-Hermitian operators with dense domain
in a separable Hilbert space, having deficiency-indices $\left(1,1\right)$. 

A systematic study of densely defined Hermitian operators with deficiency
indices $\left(1,1\right)$, and later $\left(d,d\right)$, was initiated
by M. Krein\index{Krein} \cite{Kre46}, and is also part of de Branges\textquoteright{}
model theory; see \cite{deB68,BrRo66}. The direct connection between
this theme and the problem of extending continuous positive definite
(p.d.) functions $F$ when they are only defined on a fixed open subset
to $\mathbb{R}^{n}$ was one of our motivations. One desires continuous
p.d. extensions to $\mathbb{R}^{n}$. 

If $F$ is given, we denote the set of such extensions $Ext\left(F\right)$.
If $n=1$, $Ext\left(F\right)$ is always non-empty, but for $n=2$,
Rudin gave examples in \cite{Ru70,Ru63} when $Ext\left(F\right)$
may be empty. Here we extend these results, and we also cover a number
of classes of positive definite functions on locally compact groups
in general; so cases when $\mathbb{R}^{n}$ is replaced with other
groups, both Abelian and non-abelian.

Our results in the framework of locally compact Abelian groups are
more complete than their counterparts for non-Abelian\index{group!non-Abelian}
Lie groups, one reason is the availability of Bochner\textquoteright s
duality theorem for locally compact Abelian groups; -- not available
for non-Abelian Lie groups.

\section{Stochastic Processes}

The interest in positive definite functions has at least three roots:
(i) Fourier analysis, and harmonic analysis more generally, including
the non-commutative variant where we study unitary representations
of groups; (ii) optimization and approximation problems, involving
for example spline\index{spline} approximations as envisioned by
I. Schöenberg\index{Schöenberg}; and (iii) the study of stochastic
(random) processes. 

Below, we sketch a few details regarding (iii). A stochastic process
is an indexed family of random variables based on a fixed probability
space; in our present analysis, our processes will be indexed by some
group $G$; for example $G=\mathbb{R}$, or $G=\mathbb{Z}$ correspond
to processes indexed by real time, respectively discrete time. A main
tool in the analysis of stochastic processes is an associated covariance
function, see (\ref{eq:stat1}).

A process $\left\{ X_{g}\:\big|\: g\in G\right\} $ is called Gaussian
if each random variable $X_{g}$ is Gaussian, i.e., its distribution
is Gaussian. For Gaussian processes\index{Gaussian processes} we
only need two moments. So if we normalize, setting the mean equal
to $0$, then the process is determined by the covariance function.
In general the covariance function is a function on $G\times G$,
or on a subset, but if the process is stationary, the covariance function
will in fact be a positive definite\index{positive definite} function
defined on $G$, or a subset of $G$. We will be using three stochastic
processes in the Memoir, Brownian motion, Brownian Bridge, and the
Ornstein-Uhlenbeck\index{Ornstein-Uhlenbeck} process, all Gaussian,
or Ito integrals.

We outline a brief sketch of these facts below.

Let $G$ be a locally compact group, and let $\left(\Omega,\mathscr{F},\mathbb{P}\right)$
be a probability space, $\mathscr{F}$ a sigma-algebra, and $\mathbb{P}$
a probability measure \index{measure!probability} defined on $\mathscr{F}$.
A stochastic $L^{2}$-process is a system of random variables $\left\{ X_{g}\right\} _{g\in G}$,
$X_{g}\in L^{2}\left(\Omega,\mathscr{F},\mathbb{P}\right)$. The covariance\index{covariance}
function $c_{X}$ of the process is the function $G\times G\rightarrow\mathbb{C}$
given by 
\begin{equation}
c_{X}\left(g_{1},g_{2}\right)=\mathbb{E}\left(\overline{X}_{g_{1}}X_{g_{2}}\right),\;\forall\left(g_{1},g_{2}\right)\in G\times G.\label{eq:stat1}
\end{equation}
To simplify will assume that the mean $\mathbb{E}\left(X_{g}\right)=\int_{\Omega}X_{g}d\mathbb{P}\left(\omega\right)=0$
for all $g\in G$. 

We say that $\left(X_{g}\right)$ is stationary iff 
\begin{equation}
c_{X}\left(hg_{1},hg_{2}\right)=c_{X}\left(g_{1},g_{2}\right),\;\forall h\in G.\label{eq:stat2}
\end{equation}
In this case $c_{X}$ is a function of $g_{1}^{-1}g_{2}$, i.e., 
\begin{equation}
\mathbb{E}\left(X_{g_{1},}X_{g_{2}}\right)=c_{X}\left(g_{1}^{-1}g_{2}\right),\;\forall g_{1},g_{2}\in G.\label{eq:stat3}
\end{equation}
(Just take $h=g_{1}^{-1}$ in (\ref{eq:stat2}).)

We now recall the following theorem of Kolmogorov\index{Kolmogorov}
(see \cite{PaSc75}). One direction is easy, and the other is the
deep part:
\begin{defn}
A function $c$ defined on a subset of $G$ is said to be \emph{\uline{positive
definite}} iff
\[
\sum_{i}\sum_{j}\overline{\lambda_{i}}\lambda_{j}c\left(g_{i}^{-1}g_{j}\right)\geq0
\]
for all finite summation, where $x_{i}\in\mathbb{C}$ and $g_{i}^{-1}g_{j}$
in the domain of $c$. \end{defn}
\begin{thm}[Kolmogorov]
 A function $c:G\rightarrow\mathbb{C}$ is positive definite if and
only if there is a stationary Gaussian process $\left(\Omega,\mathscr{F},\mathbb{P},X\right)$
with mean zero, such that $c=c_{X}$. \end{thm}
\begin{proof}
To stress the idea, we include the easy part of the theorem, and we
refer to \cite{PaSc75} for the non-trivial direction: 

Let $\lambda_{1},\lambda_{2},\ldots,\lambda_{n}\in\mathbb{C}$, and
$\left\{ g_{i}\right\} _{i=1}^{N}\subset G$, then for all finite
summations, we have:
\[
\sum_{i}\sum_{j}\overline{\lambda_{i}}\lambda_{j}c_{X}\left(g_{i}^{-1}g_{j}\right)=\mathbb{E}\left(\left|\sum_{i=1}^{N}\lambda_{i}X_{g_{i}}\right|^{2}\right)\geq0.
\]

\end{proof}

\section{Earlier Papers}

Below we mention some earlier papers dealing with one or the other
of the two extension problems (i) or (ii) in \secref{2ext}. To begin
with, there is a rich literature on (i), a little on (ii), but comparatively
much less is known about their interconnections.

As for positive definite functions, their use and applications are
extensive and includes such areas as stochastic processes, \index{stochastic processes}
see e.g., \cite{JorPea13,AJSV13,JP12,AJ12}; harmonic analysis (see
\cite{BCR84,JO00,JO98}) , and the references there); potential theory
\cite{Fu74b,KL14}; operators in Hilbert space \cite{ADLR10,Al92,AD86};
and spectral theory \cite{AH13,Nus75,Dev72,Dev59}. We stress that
the literature is vast, and the above list is only a small sample.

Extensions of positive definite\index{positive definite} (p.d.) continuous
functions defined on subsets of Lie groups $G$ was studied in \cite{Jor91}.
In our present analysis of connections between extensions of positive
definite (p.d.) continuous functions and extension questions for associated
operators in Hilbert space, we will be making use of tools from spectral
theory, and from the theory of reproducing kernel-Hilbert spaces,
such as can be found in e.g., \cite{Ne69,Jor81,ABDS93,Aro50}.

There is a different kind of notion of positivity involving reflections,
restrictions, and extensions. It comes up in physics and in stochastic
processes, and is somewhat related to our present theme. While they
have several names, \textquotedblleft refection positivity\textquotedblright{}
is a popular term.

In broad terms, the issue is about realizing geometric reflections
as \textquotedblleft conjugations\textquotedblright{} in Hilbert space.
When the program is successful, for a given unitary representation
$U$ of a Lie group $G$, for example $G=\mathbb{R}$, it is possible
to renormalize the Hilbert space on which $U$ is acting.

Now the Bochner\index{Bochner} transform $F$ of a probability measure
(e.g., the distribution of a stochastic process) which further satisfies
refection positivity, has two positivity properties: one (i) because
$F$ is the transform of a positive measure, so $F$ is positive definite;
and in addition the other, (ii) because of refection symmetry. We
have not followed up below with structural characterizations of this
family of positive definite functions, but readers interested in the
theme, will find details in \cite{JO00,JO98,Ar86,OS73}, and in the
references given there.

\section{\label{sec:Organization}Organization}

The Memoir is organized around the following themes, some involving
dichotomies; e.g., 
\begin{enumerate}
\item \label{enu:o-1}abelian vs non-abelian; 
\item \label{enu:o-2}simply connected vs quotients;
\item \label{enu:o-3}spectral theoretic vs geometric; 
\item \label{enu:o-4}extending of positive definite functions vs extensions
of systems of operators; and
\item \label{enu:o-5}existence vs computation and classification.
\end{enumerate}
Item (\ref{enu:o-1}) refers to the group $G$ under consideration.
In order to get started, we will need $G$ to be locally compact so
it comes with Haar measure, but it may be non-abelian. It may be a
Lie group, or it may be non-locally Euclidean. In the other end of
this dichotomy, we look at $G=\mathbb{R}$, the real line. In all
cases, in the study of the themes from (\ref{enu:o-1}) it is important
whether the group is simply connected or not.

In order to quickly get to concrete examples, we begin the real line
$G=\mathbb{R}$, and $G=\mathbb{R}^{n}$, $n>1$ (\subref{euclid});
and the circle group, $G=\mathbb{T}=\mathbb{R}/\mathbb{Z}$ (\subref{G=00003DT}).

Of the other groups, we offer a systematic treatment of the classes
when $G$ is locally compact Abelian (\subref{lcg}), and the case
of Lie groups (\subref{lie}).

We note that the subdivision into classes of groups is necessary as
the theorems we prove in the case of $G=\mathbb{R}$ have a lot more
specificity than their counterparts do, for the more general classes
of groups. One reason for this is that our harmonic analysis relies
on unitary representations, and the non-commutative theory for unitary
representations is much more subtle than is the Abelian counterpart.

Taking a choice of group $G$ as our starting point, we then study
continuous positive definite\index{positive definite} functions $F$
defined on certain subsets in $G$. In the case of $G=\mathbb{R}$,
our choice of subset will be a finite open interval centered at $x=0$.

Our next step is to introduce a reproducing kernel Hilbert space (RKHS)\index{RKHS}
$\mathscr{H}_{F}$ that captures the properties of the given p.d.
function $F$. The nature and the harmonic analysis of this particular
RKHS are of independent interest; see sections \ref{sec:introRKHS},
\ref{sec:embedding}, \ref{sub:euclid}, and \ref{sec:mercer}.

In \secref{mercer}, we study a certain trace class integral operator
(the Mercer operator). A mercer operator $T_{F}$ is naturally associated
to a given a continuous and positive definite function $F$ defined
on the open interval $\left(-1,1\right)$. We use $T_{F}$ in order
to identify natural Bessel frame in the RKHS $\mathscr{H}_{F}$ corresponding
to $F$. We then introduce a notion of Shannon sampling of finite
Borel measures on $\mathbb{R}$, sampling from integer points in $\mathbb{R}$.
In Corollary \ref{cor:shan} we then use this to give a necessary
and sufficient condition for a given finite Borel measure $\mu$ to
fall in the convex set $Ext\left(F\right)$: The measures in $Ext\left(F\right)$
are precisely those whose Shannon sampling recover the given p.d.
function $F$ on the interval $\left(-1,1\right)$.

The questions we address (in the general case for $G$) are as follows:

(a) What (if any) are the continuous positive definite functions on
$G$ which extend $F$? (See \secref{embedding}, \subref{exp(-|x|)},
\chapref{types}, and \chapref{Ext1}.) Denoting the set of these
extensions $Ext\left(F\right)$, then $Ext\left(F\right)$ is a compact
convex set. Our next questions are:

(b) What are the parameters for $Ext\left(F\right)$? (See \subref{exp(-|x|)},
\chapref{types}, and \chapref{question}.) 

And (c) How can we understand $Ext\left(F\right)$ from a (generally
non-commutative) extension problem\index{extension problem} for operators
in $\mathscr{H}_{F}$? (See especially \secref{mercer}.)

We are further concerned with (d) applications to scattering theory
(e.g., Theorem \ref{thm:R^n-spect}), and to commutative and non-commutative
harmonic analysis.

The unbounded operators we consider are defined naturally from given
p.d. function $F$, and they have a common dense domain in the RKHS
$\mathscr{H}_{F}$. In studying possible selfadjoint operator extensions
in $\mathscr{H}_{F}$ we make use of von Neumann\textquoteright s
theory of deficiency indices\index{deficiency indices}.

For concrete cases, (e) we must then find the deficiency indices;
they must be equal, but whether they are $\left(0,0\right)$, $\left(1,1\right)$,
$\left(d,d\right)$, $d>1$, is of great significance to the answers
to the questions from (a)\textendash (d).

Finally, (f): What is the relevance of the solutions in (a) and (b)
for the theory of operators in Hilbert space and their harmonic (and
spectral) analysis? (Sections \ref{sub:euclid}, \ref{sec:index 11},
\ref{sec:index (d,d)}, and \chapref{question}.)

\section{Notation and Preliminaries}

In our theorems and proofs, we shall make use of reproducing kernel
Hilbert spaces (RKHS), but the particular RKHSs we need here will
have additional properties (as compared to a general framework); which
allow us to give explicit formulas for our solutions. In a general
setup, reproducing kernel Hilbert spaces were pioneered by Aronszajn
in the 1950s \cite{Aro50}; and subsequently they have been used in
a host of applications; e.g., \cite{SZ09,SZ07}. \index{RKHS}

The key idea of Aronszajn is that a RKHS is a Hilbert space $\mathscr{H}_{K}$
of functions $f$ on a set such that the values $f(x)$ are \textquotedblleft reproduced\textquotedblright{}
from $f$ in $\mathscr{H}_{K}$ and a vector $K_{x}$ in RKHS, in
such a way that the inner product $\left\langle K_{x},K_{y}\right\rangle =:K\left(x,y\right)$
is a positive definite kernel. Our present setting is more restrictive
in two ways: (\emph{i}) we study groups $G,$ and translation-invariant
kernels, and (\emph{ii}) we further impose continuity. By \textquotedblleft translation\textquotedblright{}
we mean relative to the operation in the particular group under discussion.
Our presentation below begins with the special case when $G$ is the
circle group $\mathbb{T}(:=\mathbb{R}/\mathbb{Z}),$ or the real line
$\mathbb{R}.$

\subsection{Reproducing Kernel Hilbert spaces}

For simplicity we focus on the case $G=\mathbb{R},$ indicating the
changes needed for $G=\mathbb{T}.$ Modifications, if any, necessitated
by considering other groups $G$ will be described in the body of
the Memoir. 
\begin{defn}
\label{def:pdf}Fix $0<a,$ let $\Omega$ be an open interval of length
$a,$ then $\Omega-\Omega=(-a,a).$ A function 
\begin{equation}
F:\Omega-\Omega\rightarrow\mathbb{C}\label{eq:F}
\end{equation}
be continuous, bounded, and defined on $\Omega-\Omega.$ $F$ is \emph{positive
definite}\index{positive definite}\emph{ (p.d.)} if 
\begin{equation}
\sum_{i}\sum_{j}\overline{c_{i}}c_{j}F(x_{i}-x_{j})\geq0,\text{ for all finite sums with }c_{i}\in\mathbb{C},\mbox{ and all }x_{i}\in\Omega.\label{eq:def-pd}
\end{equation}
Hence, $F$ is positive definite iff the $N\times N$ matrices $\left(F(x_{i}-x_{j})\right)_{i,j=1}^{N}$
are positive definite for all $x_{1},\ldots,x_{N}$ in $\Omega$ and
all $N.$ \end{defn}
\begin{lem}
\label{lem:Fdef}$F$ is p.d. iff\textup{ 
\[
\int_{\Omega}\int_{\Omega}\overline{\varphi(x)}\varphi(y)F(x-y)dxdy\geq0\text{ for all }\varphi\in C_{c}^{\infty}(\Omega).
\]
}\end{lem}
\begin{proof}
Standard.
\end{proof}
\noindent Consider a continuous positive definite function so $F$
is defined on $\Omega-\Omega$. Let $\mathscr{H}_{F}$ be the \emph{reproducing
kernel Hilbert space (RKHS)}\index{RKHS}, which is the completion
of 
\begin{equation}
\sum_{\mbox{finite}}c_{j}F\left(\cdot-x_{j}\right):c_{j}\in\mathbb{C}\label{eq:H1}
\end{equation}
with respect to the inner product
\[
\left\langle F\left(\cdot-x\right),F\left(\cdot-y\right)\right\rangle _{\mathscr{H}_{F}}=F\left(x-y\right),\;\forall x,y\in\Omega
\]
and
\begin{equation}
\left\langle \sum_{i}c_{i}F\left(\cdot-x_{i}\right),\sum_{j}d_{j}F\left(\cdot-y_{j}\right)\right\rangle _{\mathscr{H}_{F}}=\sum_{i}\sum_{j}\overline{c_{i}}d_{j}F\left(x_{i}-y_{j}\right),\label{eq:ip-discrete}
\end{equation}
Throughout, we use the convention that the inner product is conjugate
linear in the first variable, and linear in the second variable. When
more than one inner product is used, subscripts will make reference
to the Hilbert space. 
\begin{lem}
\label{lem:RKHS-def-by-integral}The RKHS, $\mathscr{H}_{F}$, is
the Hilbert completion of the functions 
\begin{equation}
F_{\varphi}\left(x\right)=\int_{\Omega}\varphi\left(y\right)F\left(x-y\right)dy,\;\forall\varphi\in C_{c}^{\infty}\left(\Omega\right),x\in\Omega\label{eq:H2}
\end{equation}
with respect to the inner product
\begin{equation}
\left\langle F_{\varphi},F_{\psi}\right\rangle _{\mathscr{H}_{F}}=\int_{\Omega}\int_{\Omega}\overline{\varphi\left(x\right)}\psi\left(y\right)F\left(x-y\right)dxdy,\;\forall\varphi,\psi\in C_{c}^{\infty}\left(\Omega\right).\label{eq:hi2}
\end{equation}
In particular, 
\begin{equation}
\left\Vert F_{\varphi}\right\Vert _{\mathscr{H}_{F}}^{2}=\int_{\Omega}\int_{\Omega}\overline{\varphi\left(x\right)}\varphi\left(y\right)F\left(x-y\right)dxdy,\;\forall\varphi\in C_{c}^{\infty}\left(\Omega\right)\label{eq:hn2}
\end{equation}
and 
\begin{equation}
\left\langle F_{\varphi},F_{\psi}\right\rangle _{\mathscr{H}_{F}}=\int_{\Omega}\overline{\varphi\left(x\right)}F_{\psi}\left(x\right)dx,\;\forall\phi,\psi\in C_{c}^{\infty}(\Omega).
\end{equation}

\end{lem}

\section{\label{sec:introRKHS}Overview of Applications of RKHSs}

Subsequently, we shall be revisiting a number of specific instances
of these Reproducing kernel Hilbert spaces (RKHSs). \index{RKHS}Our
use of them ranges from the most general case, when a continuous positive
definite\index{positive definite} (p.d.) function $F$ is defined
on an open subset of a locally compact group; the RKHS will be denoted
$\mathscr{H}_{F}$. However, we stress that the associated RKHS will
depend on both the function $F$, and on the subset of $G$ where
$F$ is defined; hence on occasion, to be specific about the subset,
we shall index the RKHS by the pair $\left(\Omega,F\right)$. If the
choice of subset is implicit in the context, we shall write simply
$\mathscr{H}_{F}$. Depending on the context, a particular RKHS typically
will have any number of concrete, hands-on realizations, allowing
us thereby to remove the otherwise obtuse abstraction entailed in
its initial definition.

A glance at the Table of Contents indicates a large variety of the
classes of groups, and locally defined p.d. functions we consider,
the subsets. In each case, both the specific continuous, locally defined
p.d. function considered, and its domain are important. Each of the
separate cases has definite applications. The most explicit computations
work best in the case when $G=\mathbb{R}$; and we offer a number
of applications in three areas: applications to stochastic processes
(sect \ref{sec:Polya}-\ref{sec:hdim}), to harmonic analysis (sections
\ref{sec:R^1}-\ref{sec:index11}), and to operator/spectral theory
(\secref{mercer}.).

\chapter{Extensions of Continuous Positive Definite Functions}

Our main theme is the interconnection between (i) the study of extensions
of locally defined continuous and positive definite (p.d.) functions
$F$ on groups on the one hand, and, on the other, (ii) the question
of extensions for an associated system of unbounded Hermitian operators
with dense domain in a reproducing kernel Hilbert space (RKHS) $\mathscr{H}_{F}$
associated to $F$.

Because of the role of positive definite functions in harmonic analysis,
in statistics, and in physics, the connections in both directions
is of interest, i.e., from (i) to (ii), and vice versa. This means
that the notion of \textquotedblleft extension\textquotedblright{}
for question (ii) must be inclusive enough in order to encompass all
the extensions encountered in (i). For this reason enlargement of
the initial Hilbert space $\mathscr{H}_{F}$ are needed. In other
words, it is necessary to consider also operator extensions which
are realized in a dilation-Hilbert space; a new Hilbert space containing
$\mathscr{H}_{F}$ isometrically, and with the isometry intertwining
the respective operators.

\index{intertwining} \index{RKHS}\index{type 1}\index{type 2}

\section{\label{sec:embedding}Enlarging the Hilbert Space}

The purpose of this section is to describe this in detail, and to
prove some lemmas which will then be used in \chapref{types}, below.
In \chapref{types}, we identify extensions of the initial p.d. function
$F$ which are associated with operator extensions in $\mathscr{H}_{F}$
(type 1), and those which require an enlargement of $\mathscr{H}_{F}$,
type 2.

Let $F:\Omega-\Omega\to\mathbb{C}$ be a continuous p.d. function.
Let $\mathscr{H}_{F}$ be the corresponding RKHS and $\xi_{x}:=F(x-\cdot)\in\mathscr{H}_{F}.$
\begin{equation}
\left\langle \xi_{x},\xi_{y}\right\rangle _{\mathscr{H}_{F}}=F(x-y),\forall x,y\in\Omega.\label{eq:e.1}
\end{equation}
As usual $F_{\varphi}=\varphi*F,$ $\varphi\in C_{c}^{\infty}(\Omega).$
Then 
\begin{align*}
\left\langle F_{\varphi},F_{\psi}\right\rangle _{\mathscr{H}_{F}} & =\left\langle \xi_{0},\pi\left(\varphi^{\#}*\psi\right)\xi_{0}\right\rangle \\
 & =\left\langle \pi\left(\phi\right)\xi_{0},\pi\left(\psi\right)\xi_{0}\right\rangle 
\end{align*}
where $\pi\left(\varphi\right)\xi_{0}=F_{\varphi}.$ The following
lemma also holds in $\mathbb{R}^{n}$ with $n>1,$ but we state it
for $n=1$ to illustrate the ``enlargement'' of $\mathscr{H}_{F}$
question. 
\begin{thm}
\label{thm:pd-extension-bigger-H-space}The following two conditions
are equivalent:

(i) F is extendable to a continuous p.d. function $\widetilde{F}$
defined on $\mathbb{R},$ i.e., $\widetilde{F}$ is a continuous p.d.
function defined on $\mathbb{R}$ and $F(x)=\widetilde{F}(x)$ for
all $x$ in $\Omega-\Omega.$ 

(ii) There is a Hilbert space $\mathscr{K},$ an isometry\index{isometry}
$W:\mathscr{H}_{F}\to\mathscr{K},$ and a strongly continuous unitary
group $U_{t}:\mathscr{K}\to\mathscr{K},$ $t\in\mathbb{R}$ such that,
if $A$ is the skew-adjoint generator of $U_{t},$ i.e., 
\begin{equation}
\tfrac{1}{t}\left(U_{t}k-k\right)\to Ak,\forall k\in\mathrm{dom}(A),\label{eq:e.3a}
\end{equation}
then 
\begin{equation}
WF_{\varphi}\in\mathrm{domain}\left(A\right),\forall\varphi\in C_{c}^{\infty}(\Omega)\label{eq:e.3b}
\end{equation}
and 
\begin{equation}
AWF_{\varphi}=WF_{\varphi'},\forall\varphi\in C_{c}^{\infty}\left(\Omega\right).\label{eq:e.4}
\end{equation}

\end{thm}
The rest of this section is devoted to the proof of Theorem \ref{thm:pd-extension-bigger-H-space}.
\begin{proof}
$\underline{\Uparrow}:$ First, assume there exists $\mathscr{K},$
$W,$ $U_{t},$ and $A$ as in (\emph{ii}). Set 
\begin{equation}
\widetilde{F}(t)=\left\langle W\xi_{0},U_{t}W\xi_{0}\right\rangle ,t\in\mathbb{R}.\label{eq:e.5}
\end{equation}
Then, if $U_{t}=\int_{\mathbb{R}}e_{t}(\lambda)P(d\lambda),$ set
\[
d\mu(\lambda)=\left\Vert P(d\lambda)W\xi_{0}\right\Vert _{\mathscr{K}}^{2}=\left\langle W\xi_{0},P(d\lambda)W\xi_{0}\right\rangle _{\mathscr{K}}
\]
and $\widetilde{F}=\widehat{d\mu}$ is the Bochner transform. \index{Bochner}
\begin{lem}
\label{lem:abc}If $s,t\in\Omega,$ and $\varphi\in C_{c}^{\infty}(\Omega),$
then 
\begin{equation}
\left\langle WF_{\varphi},U_{t}WF_{\varphi}\right\rangle =\left\langle \xi_{0},\pi\left(\varphi^{\#}*\varphi_{t}\right)\xi_{0}\right\rangle _{\mathscr{H}_{F}},\label{eq:e.6}
\end{equation}
where $\varphi_{t}(\cdot)=\varphi(t-\cdot).$\index{approximate identity}
\end{lem}

Suppose first (\ref{eq:e.6}) has been checked. Let $\phi_{\epsilon}$
be an approximate identity at $x=0.$ Then 
\begin{eqnarray}
\widetilde{F}(t) & = & \left\langle W\xi_{0},U_{t}W\xi_{0}\right\rangle \nonumber \\
 & = & \lim_{\epsilon\to0}\left\langle \xi_{0},\pi\left(\left(\phi_{\epsilon}\right)_{t}\right)\xi_{0}\right\rangle _{\mathscr{H}_{F}}\label{eq:e.8}\\
 & = & \left\langle \xi_{0},\xi_{t}\right\rangle =F(t)\nonumber 
\end{eqnarray}
by (\ref{eq:e.5}), (\ref{eq:e.6}), and (\ref{eq:e.1}).
\begin{proof}[Proof of Lemma \ref{lem:abc}]
 Now (\ref{eq:e.6}) follows from 
\begin{equation}
U_{t}WF_{\varphi}=WF_{\varphi_{t}},\label{eq:e.9}
\end{equation}
$\varphi\in C_{c}^{\infty}(\Omega),t\in\Omega.$ Consider now
\begin{equation}
U_{t-s}WF_{\varphi_{s}}=\begin{cases}
U_{t}WF_{\varphi} & \text{at }s=0\\
WF_{\varphi_{t}} & \text{at }s=t
\end{cases}\label{eq:e.10}
\end{equation}
and
\begin{equation}
\int_{0}^{t}\tfrac{d}{ds}U_{t-s}WF_{\varphi_{s}}ds=WF_{\varphi_{t}}-U_{t}WF_{\varphi}.\label{eq:e.11}
\end{equation}
We claim that the left hand side of (\ref{eq:e.11}) equals zero.
By (\ref{eq:e.3a}) and (\ref{eq:e.3b}) 
\[
\tfrac{d}{ds}\left[U_{t-s}WF_{\varphi_{s}}\right]=-U_{t-s}AWF_{\varphi_{s}}+U_{t-s}WF_{\varphi_{s}'}.
\]
But, by (\ref{eq:e.4}) applied to $\varphi_{s},$ we get 
\begin{equation}
AWF_{\varphi_{s}}=WF_{\varphi_{s}'}\label{eq:e.12}
\end{equation}
and the desired conclusion (\ref{eq:e.9}) follows. 
\end{proof}

$\underline{\Downarrow}:$ Assume (\emph{i}), let $\widetilde{F}=\widehat{d\mu}$
be a p.d. extension and Bochner transform. Then $\mathscr{H}_{\widetilde{F}}\backsimeq L^{2}(\mu)$;
and for $\varphi\in C_{c}^{\infty}(\Omega),$ set 
\begin{equation}
WF_{\varphi}=\widetilde{F}_{\varphi},\label{eq:e.13}
\end{equation}
then $W:\mathscr{H}_{F}\to\mathscr{H}_{\widetilde{F}}$ is an isometry. 
\begin{proof}[Proof that (\ref{eq:e.13}) is an isometry ]
Let $\varphi\in C_{c}^{\infty}(\Omega).$ Then
\begin{align*}
\left\Vert \widetilde{F}_{\varphi}\right\Vert {}_{\mathscr{H}_{\widetilde{F}}}^{2} & =\int\int\overline{\varphi(s)}\varphi(t)\widetilde{F}(t-s)dsdt\\
 & =\left\Vert F_{\varphi}\right\Vert _{\mathscr{H}_{F}}^{2}=\int_{\mathbb{R}}\left|\widehat{\varphi}\left(\lambda\right)\right|^{2}d\mu\left(\lambda\right)
\end{align*}
since $\widetilde{F}$ is an extension of $F.$ 
\end{proof}

Now set $U_{t}:L^{2}(\mu)\to L^{2}(\mu)$ 
\[
\left(U_{t}f\right)\left(\lambda\right)=e_{t}\left(\lambda\right)f\left(\lambda\right),
\]
a unitary group acting in $\mathscr{H}_{\widetilde{F}}\backsimeq L^{2}(\mu).$
Using (\ref{eq:e.1}), we get 
\begin{equation}
\left(WF_{\varphi}\right)\left(x\right)=\int e_{x}\left(\lambda\right)\widehat{\varphi}\left(\lambda\right)d\mu\left(\lambda\right),\forall x\in\Omega,\forall\varphi\in C_{c}^{\infty}\left(\Omega\right).\label{eq:e.14}
\end{equation}
And therefore (\emph{ii}) follows. By (\ref{eq:e.14})
\begin{align*}
\left(WF_{\varphi'}\right)\left(x\right) & =\int e_{x}\left(\lambda\right)i\lambda\widehat{\varphi}\left(\lambda\right)d\mu\left(\lambda\right)\\
 & =\left.\tfrac{d}{dt}\right|_{t=0}U_{t}WF_{\varphi}\\
 & =AWF_{\varphi}
\end{align*}
as claimed. 

\end{proof}
An early instance of dilations (i.e., enlarging the Hilbert space)
is the theorem by Sz.-Nagy \cite{RSN56,Muh74} on unitary dilations\index{unitary dilations}
of strongly continuous semigroups.
\begin{thm}[Sz.-Nagy]
 Let $\left\{ S_{t},t\in\mathbb{R}_{+}\right\} $ be a strongly continuous
semigroup of contractive operator in a Hilbert space $\mathscr{H}$;
then there is
\begin{enumerate}
\item a Hilbert space $\mathscr{K}$, 
\item an isometry $V:\mathscr{H}\rightarrow\mathscr{K}$, 
\item a strongly continuous one-parameter unitary group $\left\{ U\left(t\right)\:|\: t\in\mathbb{R}\right\} $
acting on $\mathscr{K}$ such that
\begin{equation}
VS_{t}=U\left(t\right)V,\;\forall t\in\mathbb{R}_{+}\label{eq:nagy-1}
\end{equation}

\end{enumerate}
\end{thm}
We mention this result here to stress that positive definite functions
on $\mathbb{R}$ (and subsets of $\mathbb{R}$) often arise from contraction
semigroups: Sz.-Nagy proved the following:
\begin{thm}[Sz.-Nagy]
Let $\left(S_{t},\mathscr{H}\right)$ be a contraction semigroup,
$t\geq0$, (such that $S_{0}=I_{\mathscr{H}}$;) and let $f_{0}\in\mathscr{H}\backslash\left\{ 0\right\} $;
then the following function $F$ on $\mathbb{R}$ is positive definite\index{positive definite}:
\begin{equation}
F\left(t\right)=\begin{cases}
\left\langle f_{0},S_{t}f_{0}\right\rangle _{\mathscr{H}} & \mbox{if }t\geq0,\\
\left\langle f_{0},S_{-t}^{*}f_{0}\right\rangle _{\mathscr{H}} & \mbox{if }t<0.
\end{cases}\label{eq:nagy-2}
\end{equation}
\end{thm}
\begin{cor}
Every p.d. function as in (\ref{eq:nagy-2}) has the form:
\begin{equation}
F\left(t\right):=\left\langle k_{0},U\left(t\right)k_{0}\right\rangle _{\mathscr{K}},\; t\in\mathbb{R}\label{eq:nagy-3}
\end{equation}
where $\left(U\left(t\right),\mathscr{K}\right)$ is a unitary representation
of $\mathbb{R}$.\index{unitary representation}
\end{cor}

\section{\label{sub:ExtSpace}$Ext_{1}(F)$ and $Ext_{2}(F)$}
\begin{defn}
Let $G$ be a locally compact group, and let $\Omega$ be an open
connected subset of $G$. Let $F:\Omega^{-1}\cdot\Omega\rightarrow\mathbb{C}$
be a continuous positive definite\index{positive definite} function. 

Consider a strongly continuous unitary representation $U$ of $G$
acting in some Hilbert space $\mathscr{K}$, containing the RKHS $\mathscr{H}_{F}$.
We say that $\left(U,\mathscr{K}\right)\in Ext\left(F\right)$ iff
there is a vector $k_{0}\in\mathscr{K}$ such that\index{unitary representation}
\begin{equation}
F\left(g\right)=\left\langle k_{0},U\left(g\right)k_{0}\right\rangle _{\mathscr{K}},\;\forall g\in\Omega^{-1}\cdot\Omega.\label{eq:ext-1-1}
\end{equation}

\begin{enumerate}[leftmargin=*,label=\Roman{enumi}.]
\item The subset of $Ext\left(F\right)$ consisting of $\left(U,\mathscr{H}_{F},k_{0}=F_{e}\right)$
with 
\begin{equation}
F\left(g\right)=\left\langle F_{e},U\left(g\right)F_{e}\right\rangle _{\mathscr{H}_{F}},\;\forall g\in\Omega^{-1}\cdot\Omega\label{eq:ext-1-2}
\end{equation}
is denoted $Ext_{1}\left(F\right)$; and we set 
\[
Ext_{2}\left(F\right):=Ext\left(F\right)\backslash Ext_{1}\left(F\right);
\]
i.e., $Ext_{2}\left(F\right)$ consists of the solutions to problem
(\ref{eq:ext-1-1}) for which $\mathscr{K}\supsetneqq\mathscr{H}_{F}$,
i.e., unitary representations realized in an enlargement Hilbert space.
\\
(We write $F_{e}\in\mathscr{H}_{F}$ for the vector satisfying $\left\langle F_{e},\xi\right\rangle _{\mathscr{H}_{F}}=\xi\left(e\right)$,
$\forall\xi\in\mathscr{H}_{F}$, where $e$ is the neutral (unit)
element in $G$, i.e., $e\, g=g$, $\forall g\in G$.)
\item In the special case, where $G=\mathbb{R}^{n}$, and $\Omega\subset\mathbb{R}^{n}$
is open and connected, we consider 
\[
F:\Omega-\Omega\rightarrow\mathbb{C}
\]
continuous and positive definite. In this case,
\begin{align}
Ext\left(F\right)= & \Bigl\{\mu\in\mathscr{M}_{+}\left(\mathbb{R}^{n}\right)\:\big|\:\widehat{\mu}\left(x\right)=\int_{\mathbb{R}^{n}}e^{i\lambda\cdot x}d\mu\left(\lambda\right)\label{eq:ext-1-4}\\
 & \mbox{ is a p.d. extensiont of \ensuremath{F}}\Bigr\}.\nonumber 
\end{align}

\end{enumerate}
\end{defn}
\begin{rem}
Note that (\ref{eq:ext-1-4}) is consistent with (\ref{eq:ext-1-1}):
For if $\left(U,\mathscr{K},k_{0}\right)$ is a unitary representation
of $G=\mathbb{R}^{n}$, such that (\ref{eq:ext-1-1}) holds; then,
by a theorem of Stone, there is a projection-valued measure (PVM)
$P_{U}\left(\cdot\right)$, defined on the Borel subsets of $\mathbb{R}^{n}$
s.t. 
\begin{equation}
U\left(x\right)=\int_{\mathbb{R}^{n}}e^{i\lambda\cdot x}P_{U}\left(d\lambda\right),\; x\in\mathbb{R}^{n}.\label{eq:ex-1-5}
\end{equation}
Setting 
\begin{equation}
d\mu\left(\lambda\right):=\left\Vert P_{U}\left(d\lambda\right)k_{0}\right\Vert _{\mathscr{K}}^{2},\label{eq:ext-1-6-7}
\end{equation}
it is then immediate that we have: $\mu\in\mathscr{M}_{+}\left(\mathbb{R}^{n}\right)$,
and that the finite measure $\mu$ satisfies 
\begin{equation}
\widehat{\mu}\left(x\right)=F\left(x\right),\;\forall x\in\Omega-\Omega.\label{eq:ext-1-6}
\end{equation}

\end{rem}
Set $n=1$: Start with a local p.d. continuous function $F$, and
let $\mathscr{H}_{F}$ be the corresponding RKHS.\index{RKHS} Let
$Ext(F)$ be the compact convex set of probability measures on $\mathbb{R}$
defining extensions of $F$.
\begin{defn}
For $\varphi\in C_{c}^{\infty}\left(\Omega\right)$, set 
\begin{eqnarray}
F_{\varphi}\left(x\right) & := & \int_{\Omega}\varphi\left(y\right)F\left(x-y\right)dy,\mbox{ and}\label{eq:D-3-1}\\
D^{\left(F\right)}\left(F_{\varphi}\right) & := & F_{\frac{d\varphi}{dx};}\label{eq:D-3-2}
\end{eqnarray}
then $D^{\left(F\right)}$ defines a skew-Hermitian\index{operator!skew-Hermitian}
operator with dense domain in $\mathscr{H}_{F}$.
\end{defn}
The role of deficiency indices\index{deficiency indices} (computed
in $\mathscr{H}_{F}$) for the canonical skew-Hermitian operator in
the RKHS $\mathscr{H}_{F}$ is as follows: the deficiency indices
can be only $\left(0,0\right)$ or $\left(1,1\right)$.

We now divide $Ext(F)$ into two parts, say $Ext_{1}(F)$ and $Ext_{2}(F)$. 

Recall that all continuous p.d. extensions of $F$ come from strongly
continuous unitary representations. So in the case of 1D, from unitary
one-parameter groups of course, say $U(t)$.\index{operator!unitary one-parameter group}

Further recall from \secref{embedding}, that some of the p.d. extensions
of $F$ may entail a bigger Hilbert space, say $\mathscr{K}$. By
this we mean that $\mathscr{K}$ creates a dilation (enlargement)
of $\mathscr{H}_{F}$ in the sense that $\mathscr{H}_{F}$ is isometrically
embedded in $\mathscr{K}$. Via the embedding we may therefore view
$\mathscr{H}_{F}$ as a closed subspace in $\mathscr{K}$.

Now let $Ext_{1}(F)$ be the subset of $Ext(F)$ corresponding to
extensions when the unitary representation $U(t)$ acts in $\mathscr{H}_{F}$
(internal extensions), and $Ext_{2}(F)$ denote the part of $Ext(F)$
associated to unitary representations $U(t)$ acting in a proper enlargement
Hilbert space $\mathscr{K}$ (if any), i.e., acting in a Hilbert space
$\mathscr{K}$ corresponding to a proper dilation. The Polya extensions,
see \chapref{types}, account for a part of $Ext_{2}(F)$. We have
the following:\index{Polya}\index{deficiency indices}
\begin{thm}
The deficiency indices computed in $\mathscr{H}_{F}$ are $\left(0,0\right)$
if and only if $Ext_{1}(F)$ is a singleton. \end{thm}
\begin{rem}
Even if $Ext_{1}(F)$ is a singleton, we can still have non-empty
$Ext_{2}(F)$. 
\end{rem}
In \chapref{types}, we include a host of examples, including one
with a Polya extension where $\mathscr{K}$ is infinite dimensional,
while $\mathscr{H}_{F}$ is 2 dimensional. (If $\mathscr{H}_{F}$
is 2 dimensional, then obviously we must have deficiency indices $\left(0,0\right)$.)
In other examples we have $\mathscr{H}_{F}$ infinite dimensional,
non-trivial Polya extensions and deficiency indices $\left(0,0\right)$.
\begin{defn}
Let $K_{i}$, $i=1,2$, be two positive definite kernels defined on
some product $S\times S$ where $S$ is a set. We say that $K_{1}\ll K_{2}$
iff there is a finite constant $A$ such that 
\begin{equation}
\sum_{i}\sum_{j}\overline{c_{i}}c_{j}K_{1}\left(s_{i},s_{j}\right)\leq A\sum_{i}\sum_{j}\overline{c_{i}}c_{j}K_{2}\left(s_{i},s_{j}\right)\label{eq:o-2-1}
\end{equation}
for all finite systems $\left\{ c_{i}\right\} $ of complex numbers.

If $F_{i}$, $i=1,2$, are positive definite functions defined on
a subset of a group then we say that $F_{1}\ll F_{2}$ iff the two
kernels 
\[
K_{i}\left(x,y\right):=K_{F_{i}}\left(x,y\right)=F_{i}\left(x^{-1}y\right),\; i=1,2
\]
satisfies the condition in (\ref{eq:o-2-1}).\end{defn}
\begin{lem}
\label{lem:li-meas}Let $\mu_{i}\in\mathscr{M}_{+}\left(\mathbb{R}^{n}\right)$,
$i=1,2$, i.e., two finite positive Borel measures \index{measure!Borel}on
$\mathbb{R}^{n}$, and let $F_{i}:=\widehat{d\mu_{i}}$ be the corresponding
Bochner transforms. Then the following two conditions are equivalent:
\begin{enumerate}
\item \label{enu:order1}$\mu_{1}\ll\mu_{2}$ (meaning absolutely continuous\index{absolutely continuous})
with $\frac{d\mu_{1}}{d\mu_{2}}\in L^{1}\left(\mu_{2}\right)\cap L^{\infty}\left(\mu_{2}\right)$. 
\item \label{enu:order2}$F_{1}\ll F_{2}$, referring to the order of positive
definite functions on $\mathbb{R}^{n}$.
\end{enumerate}
\end{lem}
\begin{proof}
$\Downarrow$ If $ $(\ref{enu:o-1}) holds, then there is a Radon-Nikodym\index{Radon-Nikodym}
derivative $g\in L_{+}^{2}\left(\mathbb{R}^{n},\mu_{2}\right)$ s.t.
\begin{equation}
d\mu_{1}=gd\mu_{2}.\label{eq:o-1-1}
\end{equation}
Let $\left\{ c_{i}\right\} _{1}^{N}$, $\left\{ x_{i}\right\} _{1}^{N}$
be given: $c_{i}\in\mathbb{C}$, $x_{i}\in\mathbb{R}^{n}$; then 
\begin{eqnarray*}
\sum_{j}\sum_{k}\overline{c_{j}}c_{k}F_{1}\left(x_{j}-x_{k}\right) & = & \int_{\mathbb{R}^{n}}\Bigl|\sum_{j}c_{j}e^{ix_{j}\lambda}\Bigr|^{2}d\mu_{1}\left(\lambda\right)\\
 & = & \int_{\mathbb{R}^{n}}\Bigl|\sum_{j}c_{j}e^{ix_{j}\lambda}\Bigr|^{2}g\left(\lambda\right)d\mu_{2}\left(\lambda\right)\quad\mbox{\ensuremath{\left(\text{by }\left(\ref{eq:o-1-1}\right)\&\left(\ref{enu:order1}\right)\right)} }\\
 & \overset{\text{(by \ensuremath{\left(\ref{enu:order1}\right)})}}{\leq} & \left\Vert g\right\Vert _{L^{\infty}\left(\mu_{2}\right)}\int_{\mathbb{R}^{n}}\Bigl|\sum_{j}c_{j}e^{ix_{j}\lambda}\Bigr|^{2}d\mu_{2}\left(\lambda\right)\\
 & = & \left\Vert g\right\Vert _{L^{\infty}\left(\mu_{2}\right)}\sum_{j}\sum_{k}\overline{c_{j}}c_{k}F_{2}\left(x_{j}-x_{k}\right).
\end{eqnarray*}

$\Uparrow$ If (\ref{enu:order2}) holds, $\exists A<\infty$, s.t.
\begin{equation}
\iint\overline{\varphi\left(x\right)}\varphi\left(y\right)F_{1}\left(x-y\right)dxdy\leq A\iint\overline{\varphi\left(x\right)}\varphi\left(y\right)F_{2}\left(x-y\right),\;\forall\varphi\in C_{c}\left(\mathbb{R}\right).\label{eq:o-1-2}
\end{equation}
Using that $F_{i}=\widehat{d\mu_{i}}$, eq. (\ref{eq:o-1-2}) is equivalent
to 
\begin{equation}
\int_{\mathbb{R}^{n}}\left|\widehat{\varphi}\left(\lambda\right)\right|^{2}d\mu_{1}\left(\lambda\right)\leq A\int_{\mathbb{R}^{n}}\left|\widehat{\varphi}\left(\lambda\right)\right|^{2}d\mu_{2}\left(\lambda\right).\label{eq:o-1-3}
\end{equation}

Now for the functions 
\[
\left|\widehat{\varphi}\left(\lambda\right)\right|^{2}=\widehat{\varphi\ast\varphi^{\#}}\left(\lambda\right),\;\lambda\in\mathbb{R}^{n},
\]
we have that $\left\{ \widehat{\psi}\left(\cdot\right)\:\big|\:\psi\in C_{c}\left(\mathbb{R}^{n}\right)\right\} \cap L^{1}\left(\mathbb{R}^{n},\mu\right)$
is dense in $L^{1}\left(\mathbb{R}^{n},\mu\right)$ for all $\mu\in\mathscr{M}_{+}\left(\mathbb{R}^{n}\right)$. 

It follows that $\mu_{1}\ll\mu_{2}$, and $g=\frac{d\mu_{1}}{d\mu_{2}}\in L_{+}^{1}\left(\mathbb{R}^{n},\mu_{2}\right)\cap L^{\infty}\left(\mathbb{R}^{n},\mu_{2}\right)$
by the argument in the first half of the proof.
\end{proof}

\section{Preliminaries}

In the preliminary discussion below, we begin with the special case
when $G=\mathbb{R}$, and when $\Omega$  is a bounded open interval.
\begin{lem}
\label{lem:dense}Fix $\Omega=(\alpha,\beta),$ let $a=\beta-\alpha$.
Let $\alpha<x<\beta$ and let $\varphi_{n,x}\left(t\right)=n\varphi\left(n\left(t-x\right)\right)$,
where $\varphi$ satisfies 
\begin{enumerate}
\item $\mathrm{supp}\left(\varphi\right)\subset\left(-a,a\right)$;
\item $\varphi\in C_{c}^{\infty}$, $\varphi\geq0$;
\item $\int\varphi\left(t\right)dt=1$. Note that $\lim_{n\rightarrow\infty}\varphi_{n,x}=\delta_{x}$,
the Dirac measure\index{measure!Dirac} at $x$. 
\end{enumerate}
\begin{flushleft}
Then 
\begin{equation}
\left\Vert F_{\varphi_{n,x}}-F\left(\cdot-x\right)\right\Vert _{\mathscr{H}_{F}}\rightarrow0,\;\mbox{as }n\rightarrow\infty.\label{eq:approx}
\end{equation}
Hence $\left\{ F_{\varphi}\right\} _{\varphi\in C_{c}^{\infty}\left(\Omega\right)}$
spans a dense subspace in $\mathscr{H}_{F}$. 
\par\end{flushleft}
\end{lem}
\begin{figure}
\includegraphics[scale=0.5]{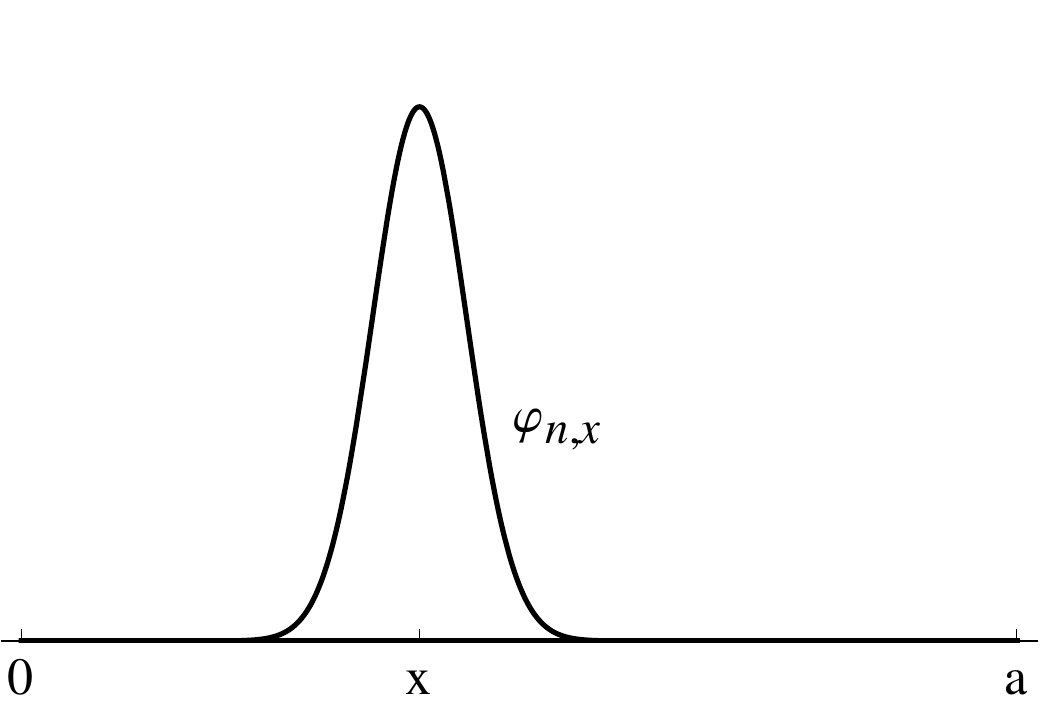}

\protect\caption{The approximate identity $\varphi_{n,x}\left(\cdot\right)$}
\end{figure}

Recall, the following facts about $\mathscr{H}_{F},$ which follow
from the general theory \cite{Aro50} of RKHS:\index{RKHS}
\begin{itemize}
\item $F(0)>0,$ so we can always arrange $F(0)=1.$
\item $F(-x)=\overline{F(x)}$
\item $\mathscr{H}_{F}$ consists of continuous functions $\xi:\Omega-\Omega\rightarrow\mathbb{C}.$
\item The reproducing property:
\[
\left\langle F\left(\cdot-x\right),\xi\right\rangle _{\mathscr{H}_{F}}=\xi\left(x\right),\;\forall\xi\in\mathscr{H}_{F},\forall x\in\Omega,
\]
is a direct consequence of (\ref{eq:ip-discrete}).\end{itemize}
\begin{rem}
It follows from the reproducing property that if $F_{\phi_{n}}\to\xi$
in $\mathscr{H}_{F},$ then $F_{\phi_{n}}$ converges uniformly to
$\xi$ in $\Omega.$ In fact 
\begin{align*}
\left|F_{\phi_{n}}\left(x\right)-\xi\left(x\right)\right| & =\left|\left\langle F\left(\cdot-x\right),F_{\phi_{n}}-\xi\right\rangle _{\mathscr{H}_{F}}\right|\\
 & \leq\left\Vert F\left(\cdot-x\right)\right\Vert _{\mathscr{H}_{F}}\left\Vert F_{\phi_{n}}-\xi\right\Vert _{\mathscr{H}_{F}}\\
 & =F\left(0\right)\left\Vert F_{\phi_{n}}-\xi\right\Vert _{\mathscr{H}_{F}}.
\end{align*}
\end{rem}
\begin{thm}
\label{thm:HF}A continuous function $\xi:\Omega\rightarrow\mathbb{C}$
is in $\mathscr{H}_{F}$ if and only if there exists $A_{0}>0$, such
that
\begin{equation}
\sum_{i}\sum_{j}\overline{c_{i}}c_{j}\overline{\xi\left(x_{i}\right)}\xi\left(x_{j}\right)\leq A_{0}\sum_{i}\sum_{j}\overline{c_{i}}c_{j}F\left(x_{i}-x_{j}\right)\label{eq:bdd}
\end{equation}
for all finite system $\left\{ c_{i}\right\} \subset\mathbb{C}$ and
$\left\{ x_{i}\right\} \subset\Omega$.

Equivalently, for all $\psi\in C_{c}^{\infty}\left(\Omega\right)$,
\begin{eqnarray}
\left|\int_{\Omega}\psi\left(y\right)\xi\left(y\right)dy\right|^{2} & \leq & A_{0}\int_{\Omega}\int_{\Omega}\overline{\psi\left(x\right)}\psi\left(y\right)F\left(x-y\right)dxdy\label{eq:bdd2}
\end{eqnarray}
Note that, if $\xi\in\mathscr{H}_{F}$, then the LHS of (\ref{eq:bdd2})
is $\left|\left\langle F_{\psi},\xi\right\rangle _{\mathscr{H}_{F}}\right|^{2}$. 
\end{thm}
These two conditions (\ref{eq:bdd})($\Leftrightarrow$(\ref{eq:bdd2}))
are the best way to characterize elements in the Hilbert space $\mathscr{H}_{F}$,
where $F$ is a continuous positive definite function on $\Omega-\Omega$.

We will be using this when considering for example the deficiency-subspaces
for skew-symmetric operators with dense domain in $\mathscr{H}_{F}$.
\begin{example}
Let $G=\mathbb{T}=\mathbb{R}/\mathbb{Z}$, e.g., represented as $\left(-\frac{1}{2},\frac{1}{2}\right]$.
Fix $0<a<\frac{1}{2}$, then $\Omega-\Omega=\left(-a,a\right)$ mod
$\mathbb{Z}$. So, for example, (\ref{eq:F}) takes the form $F:(-a,a)\text{ mod }\mathbb{Z}\to\mathbb{C}.$
See Figure \ref{fig:Two-versions-of-T}. 
\end{example}
\begin{figure}[H]
\begin{tabular}{cc}
\includegraphics[scale=0.4]{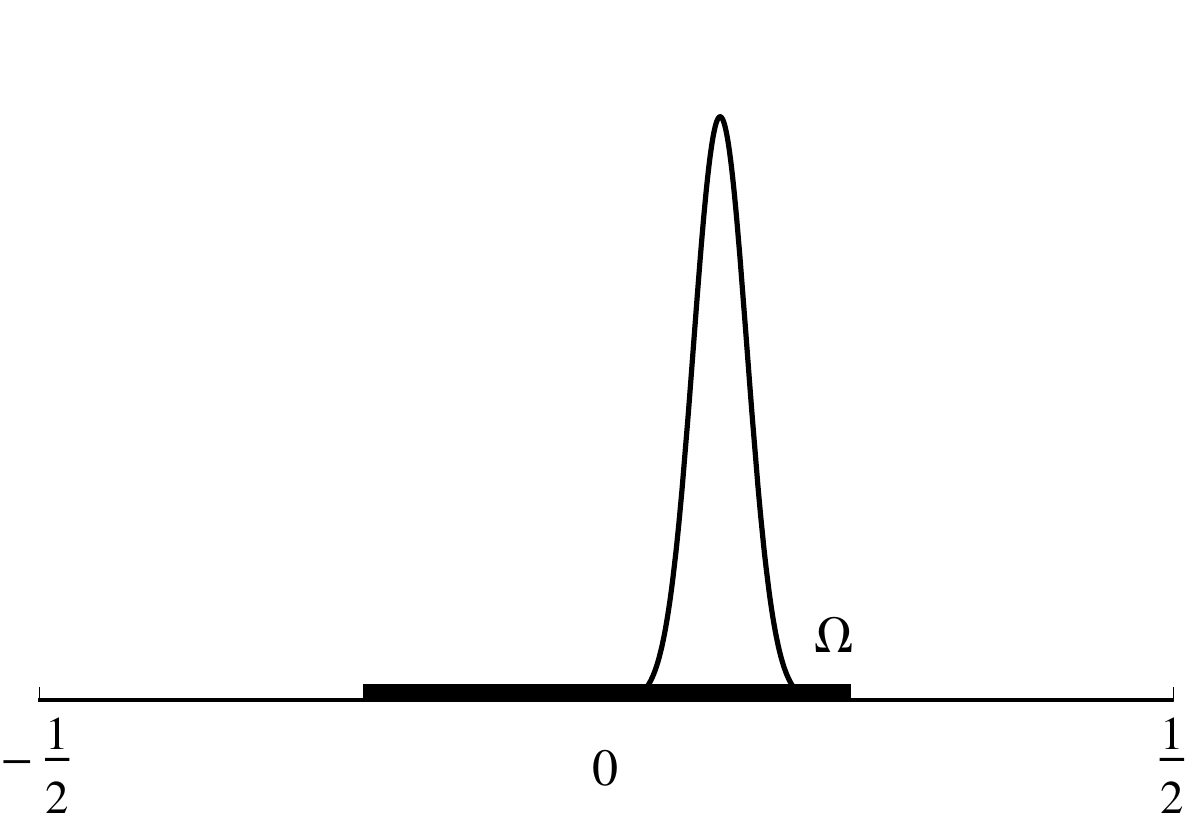}\hspace{1cm} & \includegraphics[scale=0.4]{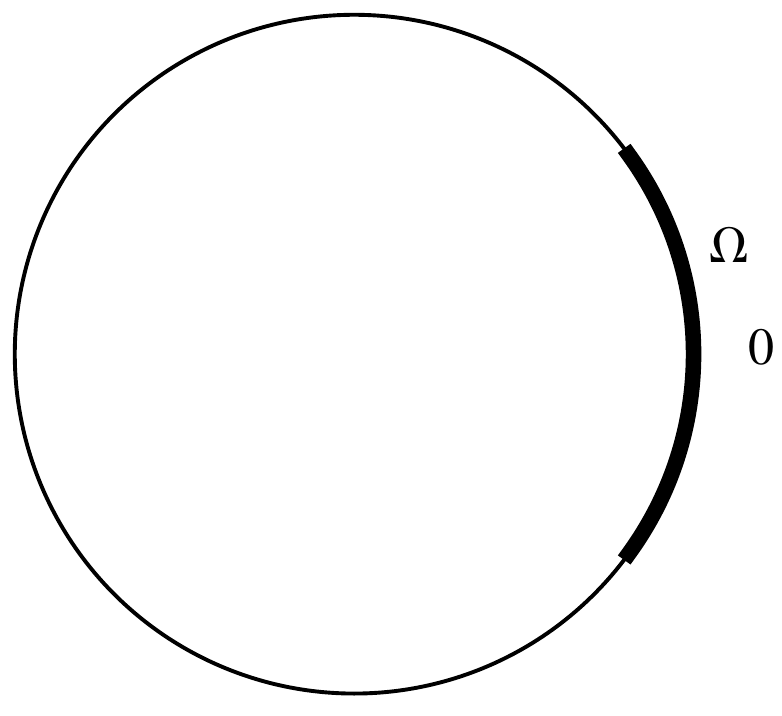}\tabularnewline
 & \tabularnewline
\end{tabular}

\protect\caption{\label{fig:Omega}\label{fig:Two-versions-of-T}Two versions of $\Omega=\left(0,a\right)\subset\mathbb{T}^{1}$ }
\end{figure}

\subsection{The Operator $D^{(F)}$}

Fix $0<a$ and a continuous positive definite\index{positive definite}
function $F$ defined on $\Omega-\Omega,$ where $\Omega=(0,a)$ as
above. Let $\mathscr{H}_{F}$ be the corresponding RKHS as in (\ref{eq:H1}). 
\begin{defn}
\label{def:D}Let $D^{\left(F\right)}F_{\psi}=F_{\psi'}$, for all
$\psi\in C_{c}^{\infty}\left(\Omega\right)$, where $\psi'=\frac{d\psi}{dt}$
and $F_{\psi}$ is as in (\ref{eq:H2}). In particular, the domain
$\mathrm{dom}\left(D^{\left(F\right)}\right)$ of $D^{\left(F\right)}$
is the set of all $F_{\psi},$ $\psi\in C_{c}^{\infty}\left(\Omega\right).$
\end{defn}
Note that the recipe for $D^{\left(F\right)}$ yields a well-defined
operator with dense domain in $\mathscr{H}_{F}$. To see this, use
Schwarz\textquoteright{} lemma to show that if $F_{\psi}=0$ in $\mathscr{H}_{F}$,
then it follows that the vector $F_{\psi'}\in\mathscr{H}_{F}$ is
0 as well.
\begin{lem}
\label{lem:DF}The operator $D^{\left(F\right)}$ is skew-symmetric
and densely defined in $\mathscr{H}_{F}$. \end{lem}
\begin{proof}
By Lemma \ref{lem:RKHS-def-by-integral} $\mathrm{dom}\left(D^{(F)}\right)$
is dense in $\mathscr{H}_{F}.$ If $\psi\in C_{c}^{\infty}\left(0,a\right)$
and $\left|t\right|<\mathrm{dist}\left(\mathrm{supp}\left(\psi\right),\mbox{endpoints}\right)$,
then
\begin{equation}
\left\Vert F_{\psi\left(\cdot+t\right)}\right\Vert _{\mathscr{H}_{F}}^{2}=\left\Vert F_{\psi}\right\Vert _{\mathscr{H}_{F}}^{2}=\int_{0}^{a}\int_{0}^{a}\overline{\psi\left(x\right)}\psi\left(y\right)F\left(x-y\right)dxdy
\end{equation}
see (\ref{eq:hn2}), so 
\[
\frac{d}{dt}\left\Vert F_{\psi\left(\cdot+t\right)}\right\Vert _{\mathscr{H}_{F}}^{2}=0
\]
which is equivalent to 
\begin{equation}
\left\langle D^{\left(F\right)}F_{\psi},F_{\psi}\right\rangle _{\mathscr{H}_{F}}+\left\langle F_{\psi},D^{\left(F\right)}F_{\psi}\right\rangle _{\mathscr{H}_{F}}=0.
\end{equation}
It follows that $D^{\left(F\right)}$ is skew-symmetric. 

To show that $D^{\left(F\right)}F_{\psi}=F_{\psi'}$ is a well-defined
operator on is dense domain in $\mathscr{H}_{F}$, we proceed as follows:
\begin{lem}
The following implication holds:
\begin{eqnarray}
 & \left[\psi\in C_{c}^{\infty}\left(\Omega\right),\: F_{\psi}=0\mbox{ in }\mathscr{H}_{F}\right]\label{eq:H-1-1}\\
 & \Downarrow\nonumber \\
 & \left[F_{\psi'}=0\mbox{ in }\mathscr{H}_{F}\right]\label{eq:H-1-2}
\end{eqnarray}
\end{lem}
\begin{proof}
Substituting (\ref{eq:H-1-1}) into 
\[
\left\langle F_{\varphi},F_{\psi'}\right\rangle _{\mathscr{H}_{F}}+\left\langle F_{\varphi'},F_{\psi}\right\rangle _{\mathscr{H}_{F}}=0
\]
we get 
\[
\left\langle F_{\varphi},F_{\psi'}\right\rangle _{\mathscr{H}_{F}}=0,\;\forall\varphi\in C_{c}^{\infty}\left(\Omega\right).
\]
Taking $\varphi=\psi'$, yields
\[
\left\langle F_{\psi'},F_{\psi'}\right\rangle =\bigl\Vert F_{\psi'}\bigr\Vert_{\mathscr{H}_{F}}^{2}=0
\]
 which is the desired conclusion (\ref{eq:H-1-2}).
\end{proof}

This finishes the proof of Lemma \ref{lem:DF}.

\end{proof}
\begin{lem}
\label{lem:Conjugation-Operator} Let $\Omega=(\alpha,\beta).$ Suppose
$F$ is a real valued positive definite function defined on $\Omega-\Omega.$
The operator $J$ on $\mathscr{H}_{F}$ determined by
\[
JF_{\varphi}=\overline{F_{\varphi(\alpha+\beta-x)}},\varphi\in C_{c}^{\infty}(\Omega)
\]
is a conjugation, i.e., $J$ is conjugate-linear, $J^{2}$ is the
identity operator, and 
\begin{equation}
\left\langle JF_{\phi},JF_{\psi}\right\rangle _{\mathscr{H}_{F}}=\left\langle F_{\psi},F_{\phi,}\right\rangle _{\mathscr{H}_{F}}.\label{eq:ConjugationIdentity}
\end{equation}
Moreover, 
\begin{equation}
D^{\left(F\right)}J=-JD^{\left(F\right)}.\label{eq:Conjugation}
\end{equation}
\end{lem}
\begin{proof}
Let $a:=\alpha+\beta$ and $\phi\in C_{c}^{\infty}(\Omega).$ Since
$F$ is real valued 
\begin{align*}
JF_{\phi}(x) & =\int_{\alpha}^{\beta}\overline{\phi(a-y)}\,\overline{F(x-y)}dy\\
 & =\int_{\alpha}^{\beta}\psi(y)\, F(x-y)dy
\end{align*}
where $\psi(y):=\overline{\phi(a-y)}$ is in $C_{c}^{\infty}(\Omega).$
It follows that $J$ maps the domain $\mathrm{dom}\left(D^{(F)}\right)$
of $D^{(F)}$ onto itself. For $\phi,\psi\in C_{c}^{\infty}(\Omega),$
\begin{align*}
\left\langle JF_{\phi},F_{\psi}\right\rangle _{\mathscr{H}_{F}} & =\int_{\alpha}^{\beta}F_{\phi(a-\cdot)}(x)\psi(x)dx\\
 & =\int_{\alpha}^{\beta}\int_{\alpha}^{\beta}\phi(a-y)F(x-y)\psi(x)dydx.
\end{align*}
Making the change of variables $(x,y)\to(a-x,a-y)$ and interchanging
the order of integration we see that
\begin{align*}
\left\langle JF_{\phi},F_{\psi}\right\rangle _{\mathscr{H}_{F}} & =\int_{\alpha}^{\beta}\int_{\alpha}^{\beta}\phi(y)F(y-x)\psi(a-x)dydx\\
 & =\int_{\alpha}^{\beta}\phi(y)F_{\psi(a-\cdot)}(y)dy\\
 & =\left\langle JF_{\psi},F_{\phi}\right\rangle _{\mathscr{H}_{F}},
\end{align*}
establishing (\ref{eq:ConjugationIdentity}). For $\phi\in C_{c}^{\infty}(\Omega),$
\[
JD^{(F)}F_{\phi}=\overline{F_{\phi'(a-\cdot)}}=-\overline{F_{\tfrac{d}{dx}\left(\phi\left(a-\cdot\right)\right)}}=-D^{(F)}JF_{\phi},
\]
hence (\ref{eq:Conjugation}) holds. \end{proof}
\begin{defn}
Let $\left(D^{\left(F\right)}\right)^{*}$ be the adjoint of $D^{\left(F\right)}$.
The deficiency spaces $DEF^{\pm}$ consists of $\xi_{\pm}\in\mathscr{H}_{F}$,
such that $\left(D^{\left(F\right)}\right)^{*}\xi_{\pm}=\pm\xi_{\pm}$.
That is,
\[
DEF^{\pm}=\left\{ \xi_{\pm}\in\mathscr{H}_{F}:\left\langle F_{\psi'},\xi_{\pm}\right\rangle _{\mathscr{H}_{F}}=\left\langle F_{\psi},\pm\xi_{\pm}\right\rangle _{\mathscr{H}_{F}},\forall\psi\in C_{c}^{\infty}\left(\Omega\right)\right\} .
\]
\end{defn}
\begin{cor}
If $F$ is real valued, then $DEF^{+}$ and $DEF^{-}$ have the same
dimension.\end{cor}
\begin{proof}
This follows from Lemma \ref{lem:Conjugation-Operator}, see e.g,
\cite{AG93} or \cite{DS88b}. \end{proof}
\begin{lem}
If $\xi\in DEF^{\pm}$ then $\xi(y)=\mathrm{constant}\, e^{\mp y}.$\end{lem}
\begin{proof}
Specifically, $\xi\in DEF^{+}$ if and only if 
\[
\int_{0}^{a}\psi'\left(y\right)\xi\left(y\right)dy=\int_{0}^{a}\psi\left(y\right)\xi\left(y\right)dy,\:\forall\psi\in C_{c}^{\infty}\left(0,a\right).
\]
Equivalently, $y\mapsto\xi\left(y\right)$ is a weak solution to 
\[
-\xi'=\xi.
\]
i.e., a strong solution in $C^{1}$. Thus, $\xi\left(y\right)=\mbox{constant }e^{-y}$.
The $DEF^{-}$ case is similar. \end{proof}
\begin{cor}
Suppose $F$ is real valued. Let $\xi_{\pm}(y):=e^{\mp y},$ for $y\in\Omega.$
Then $\xi_{+}\in\mathscr{H}_{F}$ iff $\xi_{-}\in\mathscr{H}_{F}.$
In the affirmative case $\left\Vert \xi_{-}\right\Vert _{\mathscr{H}_{F}}=e^{a}\left\Vert \xi_{+}\right\Vert _{\mathscr{H}_{F}}.$\end{cor}
\begin{proof}
Let $J$ be the conjugation from Lemma \ref{lem:Conjugation-Operator}.
A short calculation:
\begin{align*}
\left\langle J\xi,F_{\phi}\right\rangle _{\mathscr{H_{F}}} & =\left\langle F_{\overline{\phi(a-\cdot)}},\xi\right\rangle _{\mathscr{H_{F}}}=\int\phi(a-x)\xi(x)dx\\
 & =\int\phi(x)\xi(a-x)dx=\left\langle \overline{\xi(a-\cdot)},F_{\phi}\right\rangle _{\mathscr{H_{F}}}
\end{align*}
shows that $\left(J\xi\right)(x)=\overline{\xi(a-x)},$ for $\xi\in\mathscr{H}_{F}.$
In particular, $J\xi_{-}=e^{a}\xi_{+}.$ Since, $\left\Vert J\xi_{-}\right\Vert _{\mathscr{H}_{F}}=\left\Vert \xi_{-}\right\Vert _{\mathscr{H}_{F}},$
the proof is easily completed.\end{proof}
\begin{cor}
\label{cor:DF}The deficiency indices of $D^{\left(F\right)}$, with
its dense domain in $\mathscr{H}_{F}$ are $\left(0,0\right)$, $(0,1),$
$(1,0),$ or $\left(1,1\right)$. 
\end{cor}
The second case in the above corollary happens precisely when $y\mapsto e^{-y}\in\mathscr{H}_{F}$.
We can decide this with the use of (\ref{eq:bdd})($\Leftrightarrow$(\ref{eq:bdd2})). 

In \chapref{CompareFK} we will give some \emph{a priori} estimates,
which enable us to strengthen Corollary \ref{cor:DF} above. For this,
see Corollary \ref{cor:RI}.
\begin{rem}
Note that deficiency indices $\left(1,1\right)$ is equivalent to
\begin{eqnarray}
\sum_{i}\sum_{j}\overline{c_{i}}c_{j}e^{-\left(x_{i}+x_{j}\right)} & \leq & A_{0}\sum_{i}\sum_{j}\overline{c_{i}}c_{j}F\left(x_{i}-x_{j}\right)\nonumber \\
 & \Updownarrow\label{eq:HF}\\
\left|\int_{0}^{a}\psi\left(y\right)e^{-y}dy\right|^{2} & \leq & A_{0}\int_{0}^{a}\int_{0}^{a}\overline{\psi\left(x\right)}\psi\left(y\right)F\left(x-y\right)dxdy\nonumber 
\end{eqnarray}
But it depends on $F$ (given on $\left(-a,a\right)$).\end{rem}
\begin{lem}
On $\mathbb{R}\times\mathbb{R}$, define the following kernel $K_{+}\left(x,y\right)=e^{-\left|x+y\right|}$,
$\left(x,y\right)\in\mathbb{R}\times\mathbb{R}$; then this is a positive
definite kernel on $\mathbb{R}_{+}\times\mathbb{R}_{+}$; (see \cite{Aro50}
for details on positive definite kernels.)\end{lem}
\begin{proof}
Let $\left\{ c_{j}\right\} \subset\mathbb{C}^{N}$ be a finite system
of numbers, and let $\left\{ x_{j}\right\} \subset\mathbb{R}_{+}^{N}$.
Then 
\[
\sum_{j}\sum_{k}\overline{c_{j}}c_{k}e^{-\left(x_{j}+x_{k}\right)}=\left|\sum_{j}c_{j}e^{-x_{j}}\right|^{2}\geq0.
\]
\end{proof}
\begin{cor}
Let $F$, $\mathscr{H}_{F}$, and $D^{\left(F\right)}$ be as in Corollary
\ref{cor:DF}; then $D^{\left(F\right)}$ has deficiency indices $\left(1,1\right)$
if and only if the kernel $K_{+}\left(x,y\right)=e^{-\left|x+y\right|}$
is dominated by $K_{F}\left(x,y\right):=F\left(x-y\right)$ on $\left(0,a\right)\times\left(0,a\right)$,
i.e., there is a finite positive constant $A_{0}$ such that $A_{0}K_{F}-K_{+}$
is positive definite on $\left(0,a\right)\times\left(0,a\right)$. \end{cor}
\begin{proof}
This is immediate from the lemma and (\ref{eq:HF}) above.
\end{proof}
By Corollary \ref{cor:DF}, we conclude that there exists skew-adjoint
extension $A^{\left(F\right)}\supset D^{\left(F\right)}$ in $\mathscr{H}_{F}$.
That is, $\left(A^{\left(F\right)}\right)^{*}=-A^{\left(F\right)}$,
and $\left\{ F_{\psi}\right\} _{\psi\in C_{c}^{\infty}\left(0,a\right)}\subset\mathrm{dom}\left(A^{\left(F\right)}\right)\subset\mathscr{H}_{F}$. 

Hence, set $U\left(t\right)=e^{tA^{\left(F\right)}}:\mathscr{H}_{F}\rightarrow\mathscr{H}_{F}$,
and get the unitary one-parameter group 
\[
\left\{ U\left(t\right):t\in\mathbb{R}\right\} ,\; U\left(s+t\right)=U\left(s\right)U\left(t\right),\:\forall s,t\in\mathbb{R};
\]
and if 
\[
\xi\in dom\left(A^{\left(F\right)}\right)=\left\{ \xi\in\mathscr{H}_{F}:\: s.t.\lim_{t\rightarrow0}\frac{U\left(t\right)\xi-\xi}{t}\:\mbox{exists}\right\} 
\]
then 
\begin{equation}
A^{\left(F\right)}\xi=\lim_{t\rightarrow0}\frac{U\left(t\right)\xi-\xi}{t}.
\end{equation}

Now use $F_{x}(\cdot)=F\left(x-\cdot\right)$ defined in $\left(0,a\right)$;
and set 
\begin{equation}
F_{A}\left(t\right):=\left\langle F_{0},U\left(t\right)F_{0}\right\rangle _{\mathscr{H}_{F}},\;\forall t\in\mathbb{R}\label{eq:Fext}
\end{equation}
then using (\ref{eq:approx}), we see that $F_{A}$ is a continuous
positive definite extension of $F$ on $\left(-a,a\right)$, i.e.,
a continuous positive definite function on $\mathbb{R}$, and if $x\in\left(0,a\right)$,
then we get the following conclusion:\index{operator!unitary one-parameter group}
\begin{lem}
$F_{A}$ is a continuous bounded positive definite function of $\mathbb{R}$
and
\begin{equation}
F_{A}\left(t\right)=F\left(t\right).
\end{equation}
for $t\in\left(-a,a\right)$.\end{lem}
\begin{proof}
But $\mathbb{R}\ni t\mapsto F_{A}\left(t\right)$ is bounded and continuous,
since $\left\{ U\left(t\right)\right\} $ is a strongly continuous
unitary group acting on $\mathscr{H}_{F}$, and 
\[
\left|F_{A}\left(t\right)\right|=\left|\left\langle F_{0},U\left(t\right)F_{0}\right\rangle \right|\leq\left\Vert F_{0}\right\Vert _{\mathscr{H}_{F}}\left\Vert U\left(t\right)F_{0}\right\Vert _{\mathscr{H}_{F}}=\left\Vert F_{0}\right\Vert _{\mathscr{H}_{F}}^{2}
\]
where $\left|\left\langle F_{0},U\left(t\right)F_{0}\right\rangle \right|\leq\left\Vert F_{0}\right\Vert _{\mathscr{H}_{F}}^{2}=F\left(0\right)$,
see (\ref{eq:exp-4}). See the proof of Theorem \ref{thm:pd-extension-bigger-H-space}
and \cite{MR1004167,MR1069255,MR1104120} for the remaining details.\end{proof}
\begin{rem}
In the circle case $\mathbb{T}$ the function $F_{A}$ in (\ref{eq:Fext})
needs not be $\mathbb{Z}$-periodic, and we want a positive definite
continuous function on $\mathbb{T}$; hence in this case we are not
yet done. 

$F$ can be normalized by $F\left(0\right)=1$. Recall that $F$ is
defined on $\left(-a,a\right)=\Omega-\Omega$ if $\Omega=\left(0,a\right)$. 
\end{rem}
Consider the spectral representation:
\begin{equation}
U\left(t\right)=\int_{-\infty}^{\infty}e_{t}\left(\lambda\right)P\left(d\lambda\right)\label{eq:Ut}
\end{equation}
where $e_{t}\left(\lambda\right)=e^{i2\pi\lambda t}$; and $P\left(\cdot\right)$
is a projection-valued measure\index{measure!projection-valued} on
$\mathbb{R}$, $P\left(B\right):\mathscr{H}_{F}\rightarrow\mathscr{H}_{F}$,
$\forall B\in Borel\left(\mathbb{R}\right)$. Then 
\[
d\mu\left(\lambda\right)=\left\Vert P\left(d\lambda\right)F_{0}\right\Vert _{\mathscr{H}_{F}}^{2}
\]
satisfies 
\begin{equation}
F_{A}\left(t\right)=\int_{-\infty}^{\infty}e_{t}\left(\lambda\right)d\mu\left(\lambda\right),\:\forall t\in\mathbb{R}.\label{eq:Fext1}
\end{equation}

\textbf{Conclusion.} The extension $F_{A}$ from (\ref{eq:Fext})
has nice transform properties, and via (\ref{eq:Fext1}) we get 
\[
\mathscr{H}_{F_{A}}\simeq L^{2}\left(\mathbb{R},\mu\right)
\]
where $\mathscr{H}_{F_{A}}$ is the RKHS of $F_{A}$.

\section{\label{sub:lcg}The Case of Locally Compact Abelian Groups}

We are concerned with extensions of locally defined continuous and
positive definite (p.d.) functions $F$ on Lie groups, say $G$, but
some results apply to locally compact groups \index{group!locally compact abelian}as
well. However in the case of locally compact Abelian groups, we have
stronger theorems, due to the powerful Fourier analysis theory for
locally compact Abelian groups.

We must fix notations:
\begin{itemize}
\item $G$: a given locally compact abelian group, write the operation in
$G$ additively; 
\item $dx$: denotes the Haar measure \index{measure!Haar}of $G$ (unique
up to a scalar multiple.)
\item $\widehat{G}$: the dual group\index{group!dual}, i.e., $\widehat{G}$
consists of all continuous homomorphisms: $\lambda:G\rightarrow\mathbb{T}$,
$\lambda\left(x+y\right)=\lambda\left(x\right)\lambda\left(y\right)$,
$\forall x,y\in G$; $\lambda\left(-x\right)=\overline{\lambda\left(x\right)}$,
$\forall x\in G$. Occasionally, we shall write $\left\langle \lambda,x\right\rangle $
for $\lambda\left(x\right)$. Note that $\widehat{G}$ also has its
Haar measure. \end{itemize}
\begin{thm}[Pontryagin \cite{Ru90}]
\label{thm:lcg-dualG}$\widehat{\widehat{G}}\simeq G$, and we have
the following:
\[
\left[G\mbox{ is compact}\right]\Longleftrightarrow\left[\widehat{G}\mbox{ is discrete}\right]
\]

\end{thm}
Let $\phi\neq\Omega\subset G$ be an open connected subset, and let
$F:\Omega-\Omega\rightarrow\mathbb{C}$ be a fixed continuous positive
definite\index{positive definite} (p.d.) function. We choose the
normalization $F\left(0\right)=1$. 

Given $F$, we now introduce the corresponding reproducing kernel
Hilbert space, \index{RKHS}(RKHS) for short:
\begin{lem}
\label{lem:lcg-F_varphi}For $\varphi\in C_{c}\left(\Omega\right)$,
set 
\begin{equation}
F_{\varphi}\left(\cdot\right)=\int_{\Omega}\varphi\left(y\right)F\left(\cdot-y\right)dy,\label{eq:lcg-1}
\end{equation}
then $\mathscr{H}_{F}$ is the Hilbert completion of $\left\{ F_{\varphi}\:\big|\:\varphi\in C_{c}\left(\Omega\right)\right\} $
in the inner product:
\begin{equation}
\left\langle F_{\varphi},F_{\psi}\right\rangle _{\mathscr{H}_{F}}=\int_{\Omega}\int_{\Omega}\overline{\varphi\left(x\right)}\psi\left(y\right)F\left(x-y\right)dxdy.\label{eq:lcg-2}
\end{equation}
Here $C_{c}\left(\Omega\right):=$ all continuous compactly supported
functions in $\Omega$.
\end{lem}

\begin{lem}
\label{lem:lcg-bdd}The Hilbert space $\mathscr{H}_{F}$ is also a
Hilbert space of continuous functions on $\Omega$ as follows:

If $\xi:\Omega\rightarrow\mathbb{C}$ is a fixed continuous function,
then $\xi\in\mathscr{H}_{F}$ if and only if $\exists$ $K=K_{\xi}<\infty$
such that
\begin{equation}
\left|\int_{\Omega}\overline{\xi\left(x\right)}\varphi\left(x\right)dx\right|^{2}\leq K\int_{\Omega}\int_{\Omega}\overline{\varphi\left(y_{1}\right)}\varphi\left(y_{2}\right)F\left(y_{1}-y_{2}\right)dy_{1}dy_{2}.\label{eq:lcg-3}
\end{equation}
When (\ref{eq:lcg-3}) holds, then 
\[
\left\langle \xi,F_{\varphi}\right\rangle _{\mathscr{H}_{F}}=\int_{\Omega}\overline{\xi\left(x\right)}\varphi\left(x\right)dx,\;\mbox{for all }\varphi\in C_{c}\left(\Omega\right).
\]
\end{lem}
\begin{proof}
We refer to the basics on the theory of RKHSs; e.g., \cite{Aro50}.\end{proof}
\begin{lem}
\label{lem:lcg-Bochner}There is a bijective correspondence between
all continuous p.d. extensions $\tilde{F}$ to $G$ of the given p.d.
function $F$ on $\Omega-\Omega$, on the one hand; and all Borel
probability measures \index{measure!probability} $\mu$ on $\widehat{G}$,
on the other, i.e., all $\mu\in\mathscr{M}(\widehat{G})$ s.t.
\begin{equation}
F\left(x\right)=\widehat{\mu}\left(x\right),\:\forall x\in\Omega-\Omega\label{eq:lcg-bochner}
\end{equation}
where
\[
\widehat{\mu}\left(x\right)=\int_{\widehat{G}}\lambda\left(x\right)d\mu\left(\lambda\right)=\int_{\widehat{G}}\left\langle \lambda,x\right\rangle d\mu\left(\lambda\right),\:\forall x\in G.
\]
\end{lem}
\begin{proof}
This is an immediate application of Bochner's characterization of
the continuous positive definite functions on locally compact abelian
groups.\end{proof}
\begin{defn}
Set 
\[
Ext\left(F\right)=\left\{ \mu\in\mathscr{M}(\widehat{G})\:\Big|\: s.t.\:(\ref{eq:lcg-bochner})\mbox{ holds}\right\} .
\]
\end{defn}
\begin{rem}
There are examples where $Ext\left(F\right)=\phi$. See \subref{G=00003DT}
where $G=\mathbb{T}$.\end{rem}
\begin{lem}
$Ext\left(F\right)$ is weak $\ast$-compact and convex.\end{lem}
\begin{proof}
Left to the reader; see e.g., \cite{Rud73}.\end{proof}
\begin{thm}
\label{thm:lcg-1}~
\begin{enumerate}
\item Let $F$ and $\mathscr{H}_{F}$ be as above; and let $\mu\in\mathscr{M}(\widehat{G})$;
then there is a positive Borel function $h$ on $\widehat{G}$ s.t.
$h^{-1}\in L^{\infty}(\widehat{G})$, and $hd\mu\in Ext\left(F\right)$,
if and only if $\exists K_{\mu}<\infty$ such that
\begin{equation}
\int_{\widehat{G}}\left|\widehat{\varphi}\left(\lambda\right)\right|^{2}d\mu\left(\lambda\right)\leq K_{\mu}\int_{\Omega}\int_{\Omega}\overline{\varphi\left(y_{1}\right)}\varphi\left(y_{2}\right)F\left(y_{1}-y_{2}\right)dy_{1}dy_{2}.\label{eq:lcg-4}
\end{equation}

\item Assume $\mu\in Ext\left(F\right)$, then 
\begin{equation}
\left(fd\mu\right)^{\vee}\in\mathscr{H}_{F},\:\forall f\in L^{2}(\widehat{G},\mu).\label{eq:lcg-5}
\end{equation}

\end{enumerate}
\end{thm}
\begin{proof}
The assertion in (\ref{eq:lcg-4}) is immediate from Lemma \ref{lem:lcg-bdd}.

Our conventions for the two transforms used in (\ref{eq:lcg-4}) and
(\ref{eq:lcg-5}) are as follows:
\begin{equation}
\widehat{\varphi}\left(\lambda\right)=\int_{G}\overline{\left\langle \lambda,x\right\rangle }\varphi\left(x\right)dx;\label{eq:lcg-6}
\end{equation}
if $\varphi\in C_{c}\left(\Omega\right)$, of course $\widehat{\varphi}\left(\lambda\right)=\int_{G}\overline{\left\langle \lambda,x\right\rangle }\varphi\left(x\right)dx$.

The transform in (\ref{eq:lcg-5}) is:
\begin{equation}
\left(fd\mu\right)^{\vee}=\int_{\widehat{G}}\left\langle \lambda,x\right\rangle f\left(\lambda\right)d\mu\left(\lambda\right).\label{eq:lcg-7}
\end{equation}
The notation $\left(fd\mu\right)^{\vee}$ may be more logical. 

The remaining computations are left to the reader.\end{proof}
\begin{cor}
\label{cor:lcg-isom}~
\begin{enumerate}
\item \label{enu:1}Let $F$ be as above; then $\mu\in Ext\left(F\right)$
iff the following operator
\[
T\left(F_{\varphi}\right)=\widehat{\varphi},\:\varphi\in C_{c}\left(\Omega\right)
\]
is well-defined on $\mathscr{H}_{F}$, and bounded\index{operator!bounded}
as follows: $T:\mathscr{H}_{F}\rightarrow L^{2}(\widehat{G},\mu)$.
\item \label{enu:2}In this case, the adjoint operator $T^{*}:L^{2}(\widehat{G},\mu)\rightarrow\mathscr{H}_{F}$
is given by
\begin{equation}
T^{*}\left(f\right)=\left(fd\mu\right)^{\vee},\:\forall f\in L^{2}(\widehat{G},\mu).\label{eq:lcg-8}
\end{equation}

\end{enumerate}
\end{cor}
\begin{proof}
If $\mu\in Ext\left(F\right)$, then for all $\varphi\in C_{c}\left(\Omega\right)$,
and $x\in\Omega$, we have (see (\ref{eq:lcg-1}))
\begin{eqnarray*}
F_{\varphi}\left(x\right) & = & \int_{\Omega}\varphi\left(y\right)F\left(x-y\right)dy\\
 & = & \int_{\Omega}\varphi\left(y\right)\widehat{\mu}\left(x-y\right)dy\\
 & = & \int_{\Omega}\varphi\left(y\right)\left\langle \lambda,x-y\right\rangle d\mu\left(\lambda\right)dy\\
 & \overset{\small\mbox{Fubini}}{=} & \int_{\widehat{G}}\left\langle \lambda,x\right\rangle \widehat{\varphi}\left(\lambda\right)d\mu\left(\lambda\right).
\end{eqnarray*}

By Lemma \ref{lem:lcg-bdd}, we note that $\left(\widehat{\varphi}d\mu\right)^{\vee}\in\mathscr{H}_{F}$,
see (\ref{eq:lcg-7}). Hence $\exists K<\infty$ such that the estimate
(\ref{eq:lcg-4}) holds. To see that $T\left(F_{\varphi}\right)=\widehat{\varphi}$
is well-defined on $\mathscr{H}_{F}$, we must check the implication:
\[
\Bigl(F_{\varphi}=0\mbox{ in }\mathscr{H}_{F}\Bigr)\Longrightarrow\Bigl(\widehat{\varphi}=0\mbox{ in }L^{2}(\widehat{G},\mu)\Bigr)
\]
but this now follows from estimate (\ref{eq:lcg-4}).

Using the definition of the respective inner products in $\mathscr{H}_{F}$
and in $L^{2}(\widehat{G},\mu)$, we check directly that, if $\varphi\in C_{c}\left(\Omega\right)$,
and $f\in L^{2}(\widehat{G},\mu)$ then we have:
\begin{equation}
\left\langle \widehat{\varphi},f\right\rangle _{L^{2}\left(\mu\right)}=\left\langle F_{\varphi},\left(fd\mu\right)^{\vee}\right\rangle _{\mathscr{H}_{F}}.\label{eq:lcg-9}
\end{equation}

On the RHS in (\ref{eq:lcg-9}), we note that, when $\mu\in Ext\left(F\right)$,
then $\widehat{fd\mu}\in\mathscr{H}_{F}$. This last conclusion is
a consequence of Lemma \ref{lem:lcg-bdd}. Indeed, since $\mu$ is
finite, $L^{2}(\widehat{G},\mu)\subset L^{1}(\widehat{G},\mu)$, so
$\widehat{fd\mu}$ in (\ref{eq:lcg-7}) is continuous on $G$ by Riemann-Lebesgue;
and so is its restriction to $\Omega$. If $\mu$ is further assumed
absolutely continuous\index{absolutely continuous}, then $\widehat{fd\mu}\rightarrow0$
at $\infty$.

With a direct calculation, using the reproducing property in $\mathscr{H}_{F}$,
and Fubini's theorem, we check directly that the following estimate
holds:
\[
\left|\int_{\Omega}\overline{\varphi\left(x\right)}\left(fd\mu\right)^{\vee}\left(x\right)dx\right|^{2}\leq\left(\int_{\Omega}\int_{\Omega}\overline{\varphi\left(y_{1}\right)}\varphi\left(y_{2}\right)F\left(y_{1}-y_{2}\right)dy_{1}dy_{2}\right)\left\Vert f\right\Vert _{L^{2}\left(\mu\right)}^{2}
\]
and so Lemma \ref{lem:lcg-bdd} applies; we get $\left(fd\mu\right)^{\vee}\in\mathscr{H}_{F}$.

It remains to verify the formula (\ref{eq:lcg-9}) for all $\varphi\in C_{c}\left(\Omega\right)$
and all $f\in L^{2}(\widehat{G},\mu)$; but this now follows from
the reproducing property in $\mathscr{H}_{F}$, and Fubini. 

Once we have this, both assertions in (\ref{enu:1}) and (\ref{enu:2})
in the Corollary follow directly from the definition of the adjoint
operator $T^{*}$ with respect to the two Hilbert spaces in $\mathscr{H}_{F}\overset{T}{\longrightarrow}L^{2}(\widehat{G},\mu)$.
Indeed then (\ref{eq:lcg-8}) follows. 
\end{proof}
We recall a general result on continuity of positive definite (p.d.)
functions on any locally compact Lie group:
\begin{thm}
If $F$ is p.d. function on a locally compact group $G$, assumed
continuous only in a neighborhood of $e\in G$; then it is automatically
continuous everywhere on $G$.\end{thm}
\begin{proof}
Since $F$ is positive definite, we may apply the Gelfand-Naimark-Segal
(GNS) theorem to get a cyclic unitary representation $\left(U,\mathscr{H},v\right)$,
$v$ denoting the cyclic vector, such that $F\left(g\right)=\left\langle v,U\left(g\right)v\right\rangle $,
$g\in G$. The stated assertion about continuity for unitary representations
is easy to verify; and so it follows for $F$.
\end{proof}
\index{GNS}

\index{unitary representation}

\textbf{Question. }Suppose $Ext\left(F\right)\neq\phi$, then what
are its extreme points? Equivalently, characterize $ext\left(Ext\left(F\right)\right)$. 

Let $\Omega\subset G$, $\Omega\neq\phi$, $\Omega$ open and connected,
and let ``$\widehat{\cdot}$'' denote the Fourier transform in the
given locally compact Abelian group $G$. Set 
\[
K_{\Omega}\left(\lambda\right)=\widehat{\chi_{\Omega}\left(\lambda\right)},\;\forall\lambda\in\widehat{G},
\]
where $\widehat{G}$ is the dual character group\index{group!character}
(see Theorem \ref{thm:lcg-dualG}).
\begin{thm}
Let $F:\Omega-\Omega\rightarrow\mathbb{C}$ be continuous, and positive
definite on $\Omega-\Omega$; and assume $Ext\left(F\right)\neq\phi$.
Let $\mu\in Ext\left(F\right)$, and let $T_{\mu}\left(F_{\phi}\right):=\widehat{\varphi}$,
defined initially only for $\varphi\in C_{c}\left(\Omega\right)$,
be the isometry $T_{\mu}:\mathscr{H}_{F}\rightarrow L^{2}\left(\mu\right)=L^{2}(\widehat{G},\mu)$.
Then $Q_{\mu}:=T_{\mu}T_{\mu}^{*}$ is a projection in \textup{$L^{2}\left(\mu\right)$
with $K_{\Omega}\left(\cdot\right)$ as kernel:
\begin{equation}
\left(Q_{\mu}f\right)\left(\lambda\right)=\int_{\widehat{G}}K_{\Omega}\left(\lambda-\xi\right)f\left(\xi\right)d\mu\left(\xi\right),\;\forall f\in L^{2}(\widehat{G},\mu),\forall\lambda\in\widehat{G}.\label{eq:lcg-10}
\end{equation}
}\end{thm}
\begin{proof}
We showed in Theorem \ref{thm:lcg-1} that $T_{\mu}:\mathscr{H}_{F}\rightarrow L^{2}\left(\mu\right)$
is isometric, and so $Q_{\mu}:=T_{\mu}T_{\mu}^{*}$ is the projection
in $L^{2}\left(\mu\right)$. For $f\in L^{2}\left(\mu\right)$, $\lambda\in\widehat{G}$,
we have the following computation, where the interchanging of integrals
is justified by Fubini's theorem:
\begin{eqnarray*}
\left(Q_{\mu}f\right)\left(\lambda\right) & = & \int_{\Omega}\left(fd\mu\right)^{\vee}\left(x\right)\left\langle \lambda,x\right\rangle dx\;(\mbox{where }dx\mbox{ is Haar measure on G})\\
 & = & \int_{\Omega}\left\langle \lambda,x\right\rangle \left(\int_{\widehat{G}}f\left(\xi\right)\overline{\left\langle \xi,x\right\rangle }d\mu\left(\xi\right)\right)dx\\
 & \overset{\mbox{\small Fubini}}{=} & \int_{\widehat{G}}K_{\Omega}\left(\lambda-\xi\right)f\left(\xi\right)d\mu\left(\xi\right)
\end{eqnarray*}
which is the desired conclusion (\ref{eq:lcg-10}).
\end{proof}

\section{\label{sub:euclid}The Case of $G=\mathbb{R}^{n}$}

Of course the case of $G=\mathbb{R}^{n}$ is a special case of the
setting of locally compact Abelian groups from above, and the results
available for $\mathbb{R}^{n}$ are more refined. We focus on this
in the present section. This is also the setting of the more classical
studies of extension questions.

Setting. Let $\Omega\subset\mathbb{R}^{n}$ be a fixed open and connected
subset; and let $F:\Omega-\Omega\rightarrow\mathbb{C}$ be a given
continuous and positive definite function defined on 
\begin{equation}
\Omega-\Omega:=\left\{ x-y\in\mathbb{R}^{n}\:\Big|\: x,y\in\Omega\right\} .\label{eq:rn1}
\end{equation}
Let $\mathscr{H}_{F}$ be the corresponding reproducing kernel Hilbert
space\index{RKHS} (RKHS). We showed that $Ext\left(F\right)\not\neq\phi$
if and only if there is a strongly continuous unitary representation
$\left\{ U\left(t\right)\right\} _{t\in\mathbb{R}^{n}}$ acting on
$\mathscr{H}_{F}$ such that 
\begin{equation}
\mathbb{R}^{n}\ni t\mapsto\left\langle F_{0},U\left(t\right)F_{0}\right\rangle _{\mathscr{H}_{F}}\label{eq:rn2}
\end{equation}
is a p.d. extension of $F$, extending from (\ref{eq:rn1}) to $\mathbb{R}^{n}$.
Finally, if $U$ is a unitary representation of $G=\mathbb{R}^{n}$
we denote by $P_{U}\left(\cdot\right)$ the associated projection
valued measure (PVM) \index{measure!PVM} on $\mathscr{B}\left(\mathbb{R}^{n}\right)$
(= the sigma--algebra of all Borel subsets in $\mathbb{R}^{n}$).

We have 
\begin{equation}
U\left(t\right)=\int_{\mathbb{R}^{n}}e^{it\cdot\lambda}P_{U}\left(d\lambda\right),\:\forall t\in\mathbb{R}^{n};\label{eq:rn3}
\end{equation}
where $t=\left(t_{1},\ldots,t_{n}\right)$, $\lambda=\left(\lambda_{1},\ldots,\lambda_{n}\right)$,
and $t\cdot\lambda=\sum_{j=1}^{n}t_{j}\lambda_{j}$. Recall, setting
\begin{equation}
d\mu\left(\cdot\right)=\left\Vert P_{U}\left(\cdot\right)F_{0}\right\Vert _{\mathscr{H}_{F}}^{2},\label{eq:rn4}
\end{equation}
then the p.d. function on RHS in (\ref{eq:rn2}) satisfies 
\begin{equation}
\mbox{RHS}_{\left(\ref{eq:rn2}\right)}=\int_{\mathbb{R}^{n}}e^{it\cdot\lambda}d\mu\left(\lambda\right),\;\forall t\in\mathbb{R}^{n}.\label{eq:rn5}
\end{equation}

The purpose of the next theorem is to give an orthogonal\index{orthogonal}
splitting of the RKHS $\mathscr{H}_{F}$ associated to a fixed $\left(\Omega,F\right)$
when it is assumed that $Ext\left(F\right)$ is non-empty. This orthogonal
splitting of $\mathscr{H}_{F}$ depends on a choice of $\mu\in Ext\left(F\right)$,
and the splitting is into three orthogonal subspaces of $\mathscr{H}_{F}$,
correspond a splitting of spectral types into atomic, \index{atom}completely
continuous (with respect to Lebesgue measure)\index{measure!Lebesgue},
and singular\index{measure!singular}.
\begin{thm}
\label{thm:R^n-spect}Let $\Omega\subset\mathbb{R}^{n}$ be given,
$\Omega\neq\phi$, open and connected. Suppose $F$ is given p.d.
and continuous on $\Omega-\Omega$, and assume $Ext\left(F\right)\neq\phi$.
Let $U$ be the corresponding unitary representations of $G=\mathbb{R}^{n}$,
and let $P_{U}\left(\cdot\right)$ be its associated PVM acting on
$\mathscr{H}_{F}$ ($=$ the RKHS of $F$.)
\begin{enumerate}
\item Then $\mathscr{H}_{F}$ splits up as an orthogonal sum of three closed
and $U\left(\cdot\right)$ invariant subspaces
\begin{equation}
\mathscr{H}_{F}=\mathscr{H}_{F}^{\left(atom\right)}\oplus\mathscr{H}_{F}^{\left(ac\right)}\oplus\mathscr{H}_{F}^{\left(sing\right)}\label{eq:rn6}
\end{equation}
with these subspaces characterized as follows:\\
The PVM $P_{U}\left(\cdot\right)$ restricted to $\mathscr{H}_{F}^{\left(atom\right)}$,
($\mathscr{H}_{F}^{\left(ac\right)}$, resp., $\mathscr{H}_{F}^{\left(sing\right)}$)
is purely atomic, is absolutely continuous\index{absolutely continuous}
with respect to Lebesgue measure $d\lambda=d\lambda_{1}\cdots d\lambda_{n}$
on $\mathbb{R}^{n}$, respectively, $P_{U}\left(\cdot\right)$ is
continuous, purely singular, when restricted to $\mathscr{H}_{F}^{\left(sing\right)}$.
\item \textbf{\label{enu:rn2}Case $\mathscr{H}_{F}^{\left(atom\right)}$.}
If $\lambda\in\mathbb{R}^{n}$ is an atom in $P_{U}\left(\cdot\right)$,
i.e., $P_{U}\left(\left\{ \lambda\right\} \right)\neq0$, where $\left\{ \lambda\right\} $
denotes the singleton with $\lambda$ fixed; then $P_{U}$$\left(\left\{ \lambda\right\} \right)\mathscr{H}_{F}$
is one-dimensional, and the function $e_{\lambda}\left(x\right):=e^{i\lambda\cdot x}$,
(complex exponential) restricted to $\Omega$, is in $\mathscr{H}_{F}$.
We have:
\begin{equation}
P_{U}\left(\left\{ \lambda\right\} \right)\mathscr{H}_{F}=\mathbb{C}e_{\lambda}\Big|_{\Omega}.\label{eq:rn7}
\end{equation}
\textbf{}\\
\textbf{Case $\mathscr{H}_{F}^{\left(ac\right)}$.} If $\xi\in\mathscr{H}_{F}^{\left(ac\right)}$,
then it is represented as a continuous function on $\Omega$, and
\begin{equation}
\left\langle \xi,F_{\varphi}\right\rangle _{\mathscr{H}_{F}}=\int_{\Omega}\overline{\xi\left(x\right)}\varphi\left(x\right)dx_{\left(\mbox{Lebesgue meas.}\right)},\;\forall\varphi\in C_{c}\left(\Omega\right).\label{eq:rn8}
\end{equation}
Moreover, there is a $f\in L^{2}\left(\mathbb{R}^{n},\mu\right)$
(where $\mu$ is given in (\ref{eq:rn4}) such that
\begin{equation}
\int_{\Omega}\overline{\xi\left(x\right)}\varphi\left(x\right)dx=\int_{\mathbb{R}^{n}}\overline{f\left(\lambda\right)}\widehat{\varphi}\left(\lambda\right)d\mu\left(\lambda\right),\;\forall\varphi\in C_{c}\left(\Omega\right);\label{eq:rn9}
\end{equation}
and 
\begin{equation}
\xi=\left(fd\mu\right)^{\vee}\Big|_{\Omega}.\label{eq:rn10}
\end{equation}
(We say that $\left(fd\mu\right)^{\vee}$ is the $\mu$-extension
of $\xi$.)\textbf{}\\
\textbf{}\\
\uline{Conclusion}.\textbf{ }Every $\mu$-extension of $\xi$ is
continuos on $\mathbb{R}^{n}$, and goes to $0$ at infinity (in $\mathbb{R}^{n}$,);
so the $\mu$-extension $\tilde{\xi}$ satisfies $\lim_{\left|x\right|\rightarrow\infty}\tilde{\xi}\left(x\right)=0$.\\
\\
\textbf{Case $\mathscr{H}_{F}^{\left(sing\right)}$.} Vectors $\xi\in\mathscr{H}_{F}^{\left(sing\right)}$
are characterized by the following property:\\
The measure
\begin{equation}
d\mu_{\xi}\left(\cdot\right):=\left\Vert P_{U}\left(\cdot\right)\xi\right\Vert _{\mathscr{H}_{F}}^{2}\label{eq:rn11}
\end{equation}
is continuous and purely singular.
\end{enumerate}
\end{thm}
\begin{proof}
Most of the proof details are contained in the previous discussion.

For (\ref{enu:rn2}),\textbf{ }Case $\mathscr{H}_{F}^{\left(atom\right)}$;
suppose $\lambda\in(\mathbb{R}^{n})$ is an atom, and that $\xi\in\mathscr{H}_{F}\backslash\left\{ 0\right\} $
satisfies 
\begin{equation}
P_{U}\left(\left\{ \lambda\right\} \right)\xi=\xi;\label{eq:rn12}
\end{equation}
then 
\begin{equation}
U\left(t\right)\xi=e^{it\cdot\lambda}\xi,\;\forall t\in\mathbb{R}^{n}.\label{eq:rn13}
\end{equation}
Using now (\ref{eq:rn2})-(\ref{eq:rn3}), we conclude that $\xi$
(as a continuous function on $\mathbb{R}^{n}$) is a weak solution
to the following elliptic system
\begin{equation}
\frac{\partial}{\partial x_{j}}\xi=\sqrt{-1}\lambda_{j}\xi\:\left(\mbox{on }\Omega\right),\;1\leq j\leq n.\label{eq:rn14}
\end{equation}
Hence $\xi=\mbox{const}\cdot e_{\lambda}\Big|_{\Omega}$ as asserted
in (\ref{enu:rn2}).

Case (\ref{enu:rn2}), $\mathscr{H}_{F}^{\left(ac\right)}$ follows
from (\ref{eq:rn10}) and the Riemann-Lebesgue theorem applied to
$\mathbb{R}^{n}$; and case $\mathscr{H}_{F}^{\left(sing\right)}$
is immediate.\end{proof}
\begin{example}
\label{ex:splitting}Consider the following continuous positive definite
function $F$ on $\mathbb{R}$, or on some bounded interval $\left(-a,a\right)$,
$a>0$. 
\begin{equation}
F\left(x\right)=\frac{1}{3}\left(e^{-ix}+\prod_{n=1}^{\infty}\cos\left(\frac{2\pi x}{3^{n}}\right)+e^{i3x/2}\frac{\sin\left(x/2\right)}{\left(x/2\right)}\right).\label{eq:r-2-1}
\end{equation}

\begin{enumerate}
\item This is the decomposition (\ref{eq:rn6}) of the corresponding RKHSs
$\mathscr{H}_{F}$, all three subspaces $\mathscr{H}_{F}^{\left(atom\right)}$,
$\mathscr{H}_{F}^{\left(ac\right)}$, and $\mathscr{H}_{F}^{\left(sing\right)}$
are non-zero; the first one is one-dimensional, and the other two
are infinite-dimensional.
\item The operator 
\begin{gather}
D^{\left(F\right)}\left(F_{\varphi}\right):=F_{\varphi'}\mbox{ on}\label{eq:r-2-2}\\
dom\left(D^{\left(F\right)}\right)=\left\{ F_{\varphi}\:\big|\:\varphi\in C_{c}^{\infty}\left(0,a\right)\right\} \nonumber 
\end{gather}
\[
\]
is bounded, and so extends by \uline{closure} to a skew-adjoint
operator $\overline{D^{\left(F\right)}}=-\left(D^{\left(F\right)}\right)^{*}$.\end{enumerate}
\begin{proof}
Using infinite convolutions of operators (see \chapref{conv}), and
results from \cite{DJ12}, we conclude that $F$ defined in (\ref{eq:r-2-1})
is entire analytic, and $F=\widehat{d\mu}$ (Bochner-transform) where
\index{Bochner}
\begin{equation}
d\mu\left(\lambda\right)=\frac{1}{3}\left(\delta_{-1}+\mu_{Cantor}+\chi_{\left[1,2\right]}\left(\lambda\right)d\lambda\right).\label{eq:r-2-3}
\end{equation}

The measures on the RHS in (\ref{eq:r-2-3}) are as follows:\end{proof}
\begin{itemize}
\item $\delta_{-1}$ is the Dirac mass at $-1$, i.e., $\delta\left(\lambda+1\right)$.
\index{measure!Dirac}
\item $\mu_{Cantor}=$ the middle-third Cantor measure $\mu_{c}$ determined
as the unique solution in $\mathscr{M}_{+}^{prob}\left(\mathbb{R}\right)$
to \index{measure!Cantor} 
\[
\int f\left(\lambda\right)d\mu_{c}\left(\lambda\right)=\frac{1}{2}\left(\int f\left(\frac{\lambda+1}{3}\right)d\mu_{c}\left(\lambda\right)+\int f\left(\frac{\lambda-1}{3}\right)d\mu_{c}\right)
\]
for all $f\in C_{c}\left(\mathbb{R}\right)$; and the last term
\item $\chi_{\left[1,2\right]}\left(\lambda\right)d\lambda$ is restriction
to the closed interval $\left[1,2\right]$ of Lebesgue measure.
\end{itemize}

See Figure \ref{fig:cantor}-\ref{fig:cantor2}.

It follows from the literature (e.g. \cite{DJ12}) that $\mu_{c}$
is supported in $\left[-\frac{1}{2},\frac{1}{2}\right]$; and so the
three measures on the RHS in (\ref{eq:r-2-3}) have disjoint compact
support, with the three supports positively separately.

The conclusions asserted in the example follow from this, in particular
the properties for $D^{\left(F\right)}$, in fact 
\begin{equation}
spectrum\left(D^{\left(F\right)}\right)\subseteq\left\{ -1\right\} \cup\left[-\tfrac{1}{2},\tfrac{1}{2}\right]\cup\left[1,2\right]\label{eq:r-2-5}
\end{equation}

\end{example}
\begin{figure}
\includegraphics[scale=0.4]{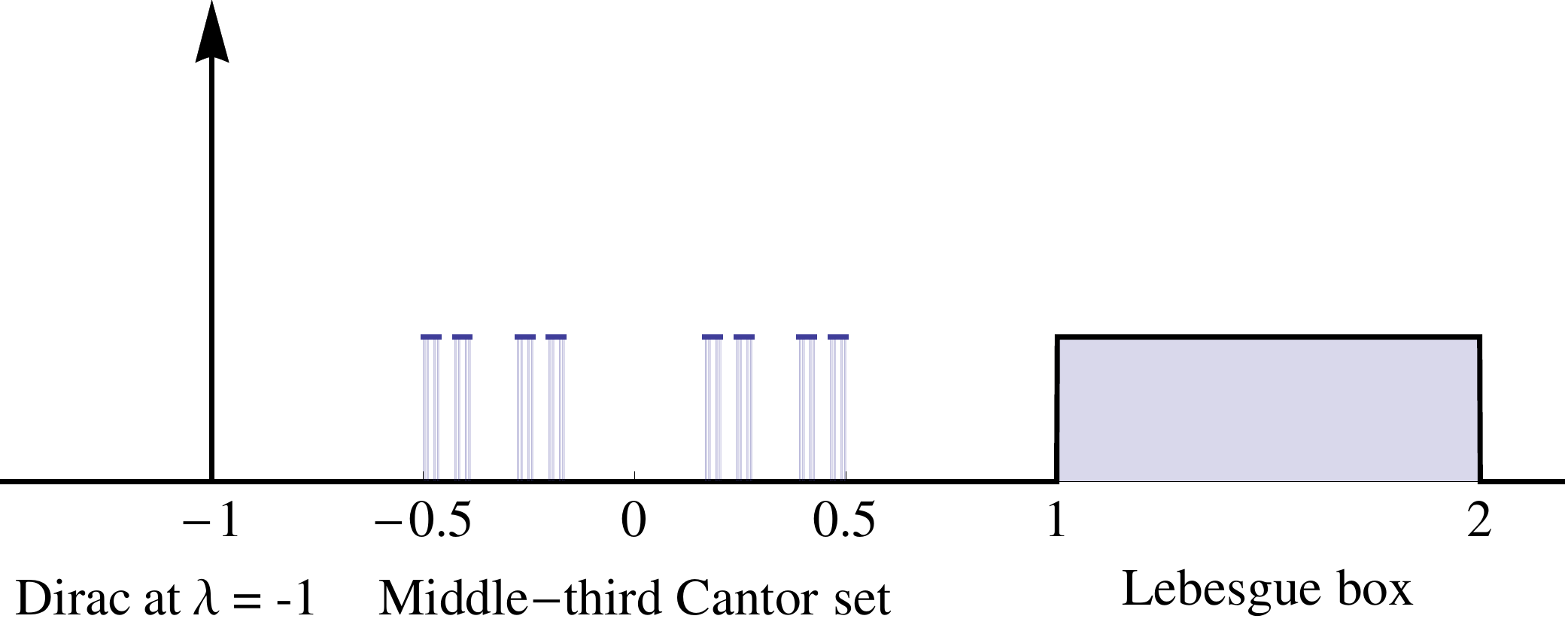}

\protect\caption{\label{fig:cantor}The measure $d\mu\left(\lambda\right)$ in example
\ref{ex:splitting}. }
\end{figure}

\begin{figure}
\includegraphics[scale=0.4]{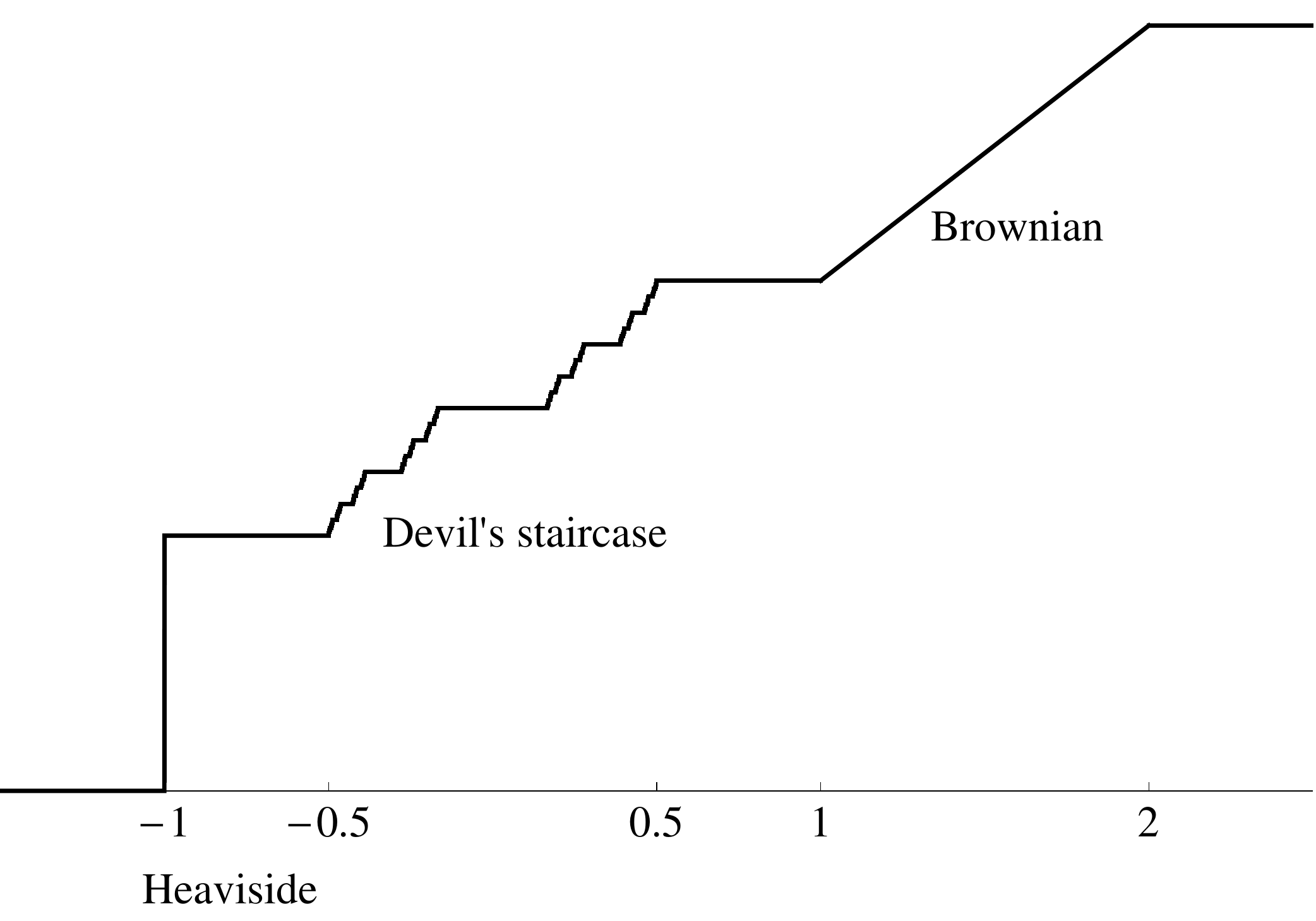}

\protect\caption{\label{fig:cantor2}Cumulative distribution $=\int_{-\infty}^{\lambda}d\mu\left(\lambda\right)$,
as in example \ref{ex:splitting}.}

\end{figure}

\begin{cor}
Let $\Omega\subset\mathbb{R}^{n}$ be non-empty, open and connected
subset. Let $F:\Omega-\Omega\rightarrow\mathbb{C}$ be a fixed continuous
and positive definite function; and let $\mathscr{H}_{F}$ be the
corresponding RKHS (of functions on $\Omega$.)
\begin{enumerate}
\item \label{enu:rn-1}If there is a compactly supported measure $\mu\in Ext\left(F\right)$,
then every function $\xi$ on $\Omega$, which is in $\mathscr{H}_{F}$,
has an entire analytic extension to $\mathbb{C}^{n}$, i.e., extension
from $\Omega\subset\mathbb{R}^{n}$ to $\mathbb{C}^{n}$. 
\item \label{enu:rn-2}If, in addition, it is assume that $\mu\ll d\lambda\left(=d\lambda_{1}\cdots d\lambda_{n}\right)=$
the Lebesgue measure on $\mathbb{R}^{n}$, where ``$\ll$'' means
``absolutely continuous;'' then the constant non-zero functions
on $\Omega$ are \uline{NOT} in $\mathscr{H}_{F}$.
\end{enumerate}
\end{cor}
\begin{proof}
Part (\ref{enu:rn-1}). Let $\Omega$, $F$, $\mathscr{H}_{F}$ and
$\mu$ be as in the statement of the corollary. Let 
\begin{equation}
T\left(F_{\varphi}\right):=\widehat{\varphi},\;\varphi\in C_{c}\left(\Omega\right)\label{eq:rn15}
\end{equation}
be the isometry $\mathscr{H}_{F}\overset{T}{\longrightarrow}L^{2}\left(\mathbb{R}^{n},\mu\right)$
from Corollary \ref{cor:lcg-isom}. By Coroll. \ref{cor:lcg-isom}
(\ref{enu:2}); for all $f\in L^{2}\left(\mathbb{R}^{n},\mu\right)$,
we have
\begin{equation}
\left(T^{*}f\right)\left(x\right)=\int_{\mathbb{R}^{n}}e^{ix\cdot\lambda}f\left(\lambda\right)d\mu\left(\lambda\right);\; x\in\Omega;\label{eq:rn16}
\end{equation}
and further that
\begin{equation}
T^{*}\left(L^{2}\left(\mathbb{R}^{n},\mu\right)\right)=\mathscr{H}_{F}.\label{eq:rn17}
\end{equation}
Now, if $\mu$ is of compact support, then so is the complex measure
$fd\mu$. This measure is finite since $L^{2}\left(\mu\right)\subset L^{1}\left(\mu\right)$.
Hence the desired conclusion follows from (\ref{eq:rn16}), (\ref{eq:rn17}),
and the Paley-Wiener theorem; see e.g., \cite{Rud73}.
\end{proof}
We now prove the assertion in part (\ref{enu:rn-2}) of the corollary.
But this follows from Theorem \ref{thm:R^n-spect}, case $\mathscr{H}_{F}^{\left(cc\right)}$.
See also Proposition \ref{prop:exp} below. 

\index{isometry}\index{Paley-Wiener}

\section{\label{sub:lie}Lie Groups}
\begin{defn}
\label{def:li-1}Let $G$ be a Lie group. We consider the extension
problem for continuous positive definite functions \index{positive definite}\index{extension problem}
\begin{equation}
F:\Omega^{-1}\Omega\rightarrow\mathbb{C}\label{eq:li-1}
\end{equation}
where $\Omega\neq\phi$, is a connected and open subset in $G$, i.e.,
it is assumed that
\begin{equation}
\sum_{i}\sum_{j}\overline{c_{i}}c_{j}F\left(x_{i}^{-1}x_{j}\right)\geq0,\label{eq:li-2}
\end{equation}
for all finite systems $\left\{ c_{i}\right\} \subset\mathbb{C}$,
and points $\left\{ x_{i}\right\} \subset\Omega$. 

Equivalent,
\begin{equation}
\int_{\Omega}\overline{\varphi\left(x\right)}\varphi\left(y\right)F\left(x^{-1}y\right)dxdy\geq0,\label{eq:li-3}
\end{equation}
for all $\varphi\in C_{c}\left(\Omega\right)$; where $dx$ denotes
a choice of \uline{left}-invariant Haar measure \index{measure!Haar}on
$G$.\end{defn}
\begin{lem}
Let $F$ be defined as in (\ref{eq:li-1})-(\ref{eq:li-2}); and for
all $X\in La$$\left(G\right)=$ the Lie algebra of $G$, set 
\begin{equation}
\left(\tilde{X}\varphi\right)\left(g\right):=\frac{d}{dt}\varphi\left(\exp_{G}\left(-tX\right)g\right)\label{eq:li-4}
\end{equation}
for all $\varphi\in C_{c}^{\infty}\left(\Omega\right)$. Set 
\begin{equation}
F_{\varphi}\left(x\right):=\int_{\Omega}\varphi\left(y\right)F\left(x^{-1}y\right)dy;\label{eq:li-5}
\end{equation}
then 
\begin{equation}
S_{X}^{\left(F\right)}\left(F_{\varphi}\right):=F_{\tilde{X}\varphi},\;\varphi\in C_{c}^{\infty}\left(\Omega\right)\label{eq:li-6}
\end{equation}
defines a representation of the Lie algebra $La\left(G\right)$ by
skew-Hermitian\index{operator!skew-Hermitian} operators in the RKHS
$\mathscr{H}_{F}$, with the operator in (\ref{eq:li-6}) defined
on the common dense domain $\left\{ F_{\varphi}\:\big|\:\varphi\in C_{c}^{\infty}\left(\Omega\right)\right\} \subset\mathscr{H}_{F}$. \end{lem}
\begin{proof}
The arguments here follow those of the proof of \lemref{DF} \emph{mutatis
mutandis}.
\end{proof}
\index{unitary representation}
\begin{defn}
\label{def:li-2}~
\begin{enumerate}
\item \label{enu:li-1}We say that a continuous p.d. function $F:\Omega^{-1}\Omega\rightarrow\mathbb{C}$
is extendable iff there is a continuous p.d. function $F_{ex}:G\rightarrow G$
such that
\begin{equation}
F_{ex}\Big|_{\Omega^{-1}\Omega}=F.\label{eq:li-7}
\end{equation}

\item \label{enu:li-2}Let $U\in Rep\left(G,\mathscr{K}\right)$ be a strongly
continuous unitary representation of $G$ acting in some Hilbert space
$\mathscr{K}$. We say that $U\in Ext\left(F\right)$ iff (Def.) there
is an isometry $\mathscr{H}_{F}\hookrightarrow\mathscr{K}$ such that
the function
\begin{equation}
G\ni g\mapsto\left\langle JF_{e},U\left(g\right)JF_{e}\right\rangle _{\mathscr{K}}\label{eq:li-8}
\end{equation}
satisfies the condition in (\ref{enu:li-1}).
\end{enumerate}
\end{defn}
\begin{thm}
Every extension of some continuous p.d. function $F$ on $\Omega^{-1}\Omega$
as in (\ref{enu:li-1}) arises from a unitary representation of $G$
as specified in (\ref{enu:li-2}).\end{thm}
\begin{proof}
First assume some unitary representation $U$ of $G$ satisfies (\ref{enu:li-2}),
then (\ref{eq:li-8}) is an extension of $F$. This follows from the
argument in our proof of \lemref{DF}. 

For the converse; assume some continuous p.d. function $F_{ex}$ on
$G$ satisfies (\ref{eq:li-7}). Now apply the GNS-theorem to $F_{ex}$;
and, as a result, we get a cyclic representation $\left(U,\mathscr{K},v_{0}\right)$
where \index{GNS}
\begin{itemize}
\item $\mathscr{K}$ is a Hilbert space;
\item $U$ is a strongly continuous unitary representation of $G$ acting
on $\mathscr{K}$; and
\item $v_{0}\in\mathscr{K}$ is a cyclic vector, $\left\Vert v_{0}\right\Vert =1$;
and
\begin{equation}
F_{ex}\left(g\right)=\left\langle v_{0},U\left(g\right)v_{0}\right\rangle ,\; g\in G.\label{eq:li-9}
\end{equation}

\end{itemize}

Defining now $J:\mathscr{H}_{F}\rightarrow\mathscr{K}$ as follows,
\[
J\left(F\left(\cdot g\right)\right):=U\bigl(g^{-1}\bigr)v_{0},\;\forall g\in\Omega;
\]
and extension by limit, we check that $J$ is isometric and satisfies
the condition from (\ref{enu:li-2}) in Definition \ref{def:li-2}.
We omit details as they parallel arguments already contained in \chapref{intro}.

\end{proof}
\begin{thm}
\label{thm:gEn}Let $\Omega$, $G$, $La\left(G\right)$, and $F:\Omega^{-1}\Omega\rightarrow\mathbb{C}$
be as in Definition \ref{def:li-1}. Let $\tilde{G}$ be the simply
connected universal covering group for $G$. Then $F$ has an extension
to a p.d. continuous function on $\tilde{G}$ iff there is a unitary
representation $U$ of $\tilde{G}$ and an isometry $\mathscr{H}_{F}\overset{J}{\longrightarrow}\mathscr{K}$
such that
\begin{equation}
JS_{X}^{\left(F\right)}=dU\left(X\right)J\label{eq:li-11}
\end{equation}
holds on $\left\{ F_{\varphi}\:\big|\:\varphi\in C_{c}^{\infty}\left(\Omega\right)\right\} $,
for all $X\in La\left(G\right)$; where 
\[
dU\left(X\right)U\left(\varphi\right)v_{0}=U\bigl(\tilde{X}\varphi\bigr)v_{0}.
\]

\end{thm}
\index{group!covering}\index{group!simply connected}
\begin{proof}
Details are contained in sections \ref{sec:embedding}, \ref{sub:G=00003DT},
and \chapref{types}.
\end{proof}
Assume $G$ is connected. Note that on $C_{c}^{\infty}\left(\Omega\right)$,
the Lie group $G$ acts locally, i.e., by $\varphi\mapsto\varphi_{g}$
where $\varphi_{g}$ denotes translation of $\varphi$ by some element
$g\in G$, such that $\varphi_{g}$ is also supported in $\Omega$.
Then 
\begin{equation}
\left\Vert F_{\varphi}\right\Vert _{\mathscr{H}_{F}}=\left\Vert F_{\varphi_{g}}\right\Vert _{\mathscr{H}_{F}};\label{eq:li-1-1}
\end{equation}
but only for elements $g\in G$ in a neighborhood of $e\in G$ , and
with the neighborhood depending on $\varphi$. 
\begin{cor}
Given 
\begin{equation}
F:\Omega^{-1}\cdot\Omega\rightarrow\mathbb{C}\label{eq:li-1-2}
\end{equation}
continuous and positive definite, then set 
\begin{equation}
L_{g}\left(F_{\varphi}\right):=F_{\varphi_{g}},\;\varphi\in C_{c}^{\infty}\left(\Omega\right),\label{eq:li-1-3}
\end{equation}
defining a local representation of $G$ in $\mathscr{H}_{E}$, see
\cite{Jor87,Jor86}.
\end{cor}

\begin{cor}
Given $F$, positive definite and continuous, as in (\ref{eq:li-1-2}),
and let $L$ be the corresponding local representation of $G$ acting
on $\mathscr{H}_{F}$. Then $Ext\left(F\right)\neq\phi$ if and only
if the local representation (\ref{eq:li-1-3}) extends to a global
unitary representation acting in some Hilbert space $\mathscr{K}$,
containing $\mathscr{H}_{F}$ isometrically.\end{cor}
\begin{proof}
We refer to \cite{Jor87,Jor86} for details, as well as \chapref{Ext1}
below.
\end{proof}

\section{\label{sub:G=00003DT}The Circle Group $\mathbb{T}$}

While we consider extensions of locally defined continuous and positive
definite (p.d.) functions $F$ on groups, say $G$, the question of
whether $G$ is simply connected or not plays an essential role in
the analysis, and in the possible extensions. It turns out that the
geometric issues arising for general Lie groups can be illustrated
well in a simple case: To illustrate this point, we isolate below
the two groups, the circle group, and its universal cover, the real
line. We study extensions defined on a small arc in the circle group
$G=\mathbb{T}=\mathbb{R}/\mathbb{Z}$, versus extensions to the universal
covering group $\mathbb{R}$.

Let $G=\mathbb{T}=\mathbb{R}/\mathbb{Z}$ represented as $\left(-\tfrac{1}{2},\tfrac{1}{2}\right].$
Pick $0<\varepsilon<\tfrac{1}{2},$ set $\Omega=(0,\varepsilon),$
then $\Omega-\Omega=(-\varepsilon,\varepsilon)$ mod $\mathbb{Z}.$ 

\index{group!circle}
\begin{lem}
If $L$ has deficiency indices $(1,1),$ there is a selfadjoint extension
$A$ of $L$ acting in $\mathscr{H}_{F}$, such that the corresponding
p.d. extension $\widetilde{F}$ of $F$ has period one, then $\varepsilon$
is rational. \index{selfadjoint extension}\end{lem}
\begin{proof}
In the notation of Theorem \ref{thm:pd-extension-bigger-H-space}
our assumptions imply that $\mathscr{K}=\mathscr{H}_{F}.$ Suppose
the deficiency indices of $D^{(F)}=L$ are $(1,1)$, the generator
of $U_{A}(t),$ $A$ is a selfadjoint extension of $D^{(F)},$ and
\[
F_{A}(t)=\left\langle F_{0},U_{A}(t)F_{0}\right\rangle _{F}=\int_{\mathbb{R}}e_{t}(\lambda)d\mu(\lambda)
\]
has period one, where 
\[
U_{A}(t)=\int_{\mathbb{R}}e_{t}(\lambda)P_{A}(d\lambda)
\]
and 
\[
d\mu_{A}(\lambda)=\left\Vert P_{A}(d\lambda)F_{0}\right\Vert _{F}^{2}.
\]
Since the deficiency indices are $(1,1),$ the selfadjoint extensions
of $D^{(F)}$ are determined by boundary conditions of the form 
\[
\xi(\varepsilon)=e(\theta)\xi(0),
\]
where $0\leq\theta<1$ is fixed. Let $A_{\theta}$ be the s.a. corresponding
to $\theta.$ Repeating the calculation of the defect spaces and using
that eigenfunctions must satisfy the boundary condition, it follows
that the spectrum\index{spectrum} of $A_{\theta}$ the
\[
\lambda_{n}=\tfrac{\theta+n}{\varepsilon},
\]
$n\in\mathbb{Z}$ for which $e_{\lambda_{n}}$ is in $\mathscr{H}_{F}.$ 

The support of $\mu_{\theta}$ is a subset of the spectrum of $A_{\theta}.$
(The support is equal to the spectrum, if $F$ is a cyclic vector
for $U_{t}.)$ Since $F_{A}$ is has period one, the support of $\mu_{\theta}$
to consist of integers. If $\lambda_{n}$ and $\lambda_{m}$ are in
the support of $\mu_{\theta},$ then 

\[
\varepsilon=\frac{n-m}{\lambda_{n}-\lambda_{m}}.
\]
It follows that $\varepsilon$ is rational. 
\end{proof}

\section{\label{sec:expT}Example: $e^{i2\pi x}$}

This is a trivial example, but it is helpful to understand why the
extensions of positive definite\index{positive definite} functions
is a quite different question from the other extensions we looked
at before.

Here we refer to classes of unbounded Hermitian operators which arise
in scattering theory for wave equations in physics; see e.g., \cite{PeTi13,JPT12,LP89,LP85}.
However the context of these questions is different. For a discussion
of connections, see e.g., \subref{momentum}, and Lemma \ref{lem:(01)}
below.

Take our momentum operator on $L^{2}\left(-\frac{1}{2},\frac{1}{2}\right)$,
where $\left(-\frac{1}{2},\frac{1}{2}\right]\simeq\mathbb{R}/\mathbb{Z}\simeq\mathbb{T}$. 

One of the extensions is characterized by 
\begin{enumerate}
\item $\left(U\left(t\right)f\right)\left(x\right)=f\left(x+t\;\mbox{mod}\ 1\right)$,
$f\in L^{2}\left(\mathbb{R}/\mathbb{Z}\right)$; and if $\xi_{n}\left(x\right):=e^{i2\pi nx}$,
then $U\left(t\right)\xi_{n}=e_{n}\left(t\right)\xi_{n}$. 
\item Take 
\begin{equation}
F\left(\cdot\right):=e_{1}\left(\cdot\right)\Big|_{\left(-\epsilon,\epsilon\right)}\label{eq:e1}
\end{equation}
 where $0<\epsilon<\frac{1}{2}$. \end{enumerate}
\begin{lem}
The RKHS $\mathscr{H}_{F}$ of $F$ in (\ref{eq:e1}) is one-dimensional.\index{RKHS}\end{lem}
\begin{proof}
Recall that $\mathscr{H}_{F}$ is the completion of 
\[
\left\{ F_{\varphi}:\varphi\in C_{c}^{\infty}\left(0,\epsilon\right)\right\} 
\]
with 
\begin{eqnarray*}
F_{\varphi}\left(x\right) & = & \int_{0}^{\epsilon}\varphi\left(y\right)F\left(x-y\right)dy\\
 & = & e_{1}\left(x\right)\int_{0}^{\epsilon}\varphi\left(y\right)\overline{e_{1}\left(y\right)}dy\\
 & = & e_{1}\left(x\right)\widehat{\varphi}\left(1\right);
\end{eqnarray*}
and note that $span\left\{ F_{\varphi}:\varphi\in C_{c}^{\infty}\left(0,\epsilon\right)\right\} $
is one-dimensional.
\end{proof}
Recall that 
\begin{eqnarray*}
\left\langle F_{\varphi},F_{\psi}\right\rangle _{\mathscr{H}_{F}} & = & \int_{0}^{\epsilon}\int_{0}^{\epsilon}\overline{\varphi\left(x\right)}\psi\left(y\right)F\left(y-x\right)dxdy\\
 & = & \overline{\widehat{\varphi}\left(1\right)}\widehat{\psi}\left(1\right),\:\forall\varphi,\psi\in C_{c}^{\infty}\left(0,\epsilon\right);
\end{eqnarray*}
and 
\[
\left\Vert F_{\varphi}\right\Vert _{\mathscr{H}_{F}}^{2}=\left|\widehat{\varphi}\left(1\right)\right|^{2}.
\]
Now this $F$ is defined as a restriction (of $e_{1}\left(\cdot\right)$)
and so it has at least this extension $e_{1}\left(\cdot\right)$.
It is also the unique continuous positive definite extension to $\mathbb{T}$.
\begin{proof}[Proof of the uniqueness assertion:]
\textbf{Case $\mathbb{T}$.} If $F$ is a continuous p.d. extension
of $e_{1}$ to $\mathbb{T}$, then by Bochner's theorem, 
\begin{equation}
e_{1}\left(x\right)=\sum_{n\in\mathbb{Z}}e_{n}\left(x\right)\mu_{n},\;\forall x\in\left(-\epsilon,\epsilon\right)\label{eq:e-2-1}
\end{equation}
where $\mu_{n}\geq0$, and $\sum_{n\in\mathbb{Z}}\mu_{n}=1$. Now
each side of the above equation extends continuously in $x$. By uniqueness
of the Fourier expansion in (\ref{eq:e-2-1}), we get $\mu_{1}=1$,
and $\mu_{n}=0$, $n\in\mathbb{Z}\backslash\left\{ 1\right\} $. 

We argue as follows: By change of summation index in (\ref{eq:e-2-1}),
we get 
\[
1\equiv\sum_{n\in\mathbb{Z}}e_{n}\left(x\right)\mu_{n+1},\;\left|x\right|<\epsilon.
\]
Writing $e_{n}\left(x\right)=\cos\left(2\pi nx\right)+i\sin\left(2\pi nx\right)$,
we note that (\ref{eq:e-2-1}) is equivalent to a cosine, and a sine
expansion. By a change of variable, we are left with proving the following
implication:
\begin{equation}
\left[\sum_{n=1}^{\infty}k_{n}\cos\left(2\pi nx\right)\equiv0,\; k_{n}\geq0,\left|x\right|<\epsilon\right]\Longrightarrow\left[k_{1}=k_{2}=\cdots=0\right];\label{eq:e-2-2}
\end{equation}
and a similar implication for the sine expansion. Now for all $N\in\mathbb{N}$,
we then write the infinite sum in (\ref{eq:e-2-2}) as 
\[
\underset{A_{N}}{\underbrace{\sum_{n>N}k_{n}\cos\left(2\pi nx\right)}}+\underset{B_{N}}{\underbrace{\sum_{n=1}^{N}k_{n}\cos\left(2\pi nx\right)}}\equiv0\;\mbox{in}\left|x\right|<\epsilon.
\]
But if $\left|x\right|<\frac{\epsilon}{4N}$, then $\cos\left(2\pi nx\right)>0$,
$1\leq n\leq N$. Since $A_{N}\rightarrow0$ uniformly, we conclude
that $k_{1}=k_{2}=\cdots=0$. 
\end{proof}

\section{Example: $e^{-a\left|x\right|}$}

Consider $\Omega=\left(0,\epsilon\right)\subset\mathbb{R}$ v.s. $\Omega=\left(0,\epsilon\right)\subset\mathbb{T}$.
In both cases, we assume $F$ is positive definite\index{positive definite}
and continuous on $\left(-\epsilon,\epsilon\right)=\Omega-\Omega$,
but there are many more solutions $\tilde{F}$, continuous, positive
definite on $\mathbb{R}$ than there are periodic solutions $\tilde{F}_{per}$
on $\mathbb{T}=\mathbb{R}/\mathbb{Z}$. 
\begin{fact}
A solution $\tilde{F}$ to the $\mathbb{R}$-problem is 
\[
\tilde{F}\left(t\right)=\left\langle F_{0},U\left(t\right)F_{0}\right\rangle _{\mathscr{H}_{F}},\: t\in\mathbb{R}.
\]
The $\mathbb{T}$-problem is equivalent to 
\[
\tilde{F}_{per}\left(t+n\right)=\tilde{F}_{per}\left(t\right),\:\forall t\in\mathbb{R},\forall n\in\mathbb{Z}.
\]
Note that 
\begin{equation}
\tilde{F}_{per}\left(t\right)=\sum_{k\in\mathbb{Z}}e^{i2\pi kt}w_{k},\; w_{k}\geq0,\sum_{k}w_{k}<\infty
\end{equation}
i.e., a Fourier transform of a finite measure on $\mathbb{Z}$,
\begin{equation}
\mu_{\mathbb{Z}}\left(A\right)=\sum_{k\in\mathbb{Z}\cap A}w_{k},\:\sum_{k\in\mathbb{Z}}w_{k}<\infty.
\end{equation}
\end{fact}
\begin{example}
Fix $a>0$, $w_{k}=\frac{1}{a^{2}+k^{2}}$, then 
\begin{equation}
\tilde{F}_{per}\left(t\right)=\frac{1}{\pi}\sum_{k\in\mathbb{Z}}\frac{e^{i2\pi kt}}{a^{2}+k^{2}}=\sum_{n\in\mathbb{Z}}e^{-a\left|t-n\right|}\label{eq:poisson}
\end{equation}
by the Poisson summation formula. This is a function on $\mathbb{T}=\mathbb{R}/\mathbb{Z}$,
and the RHS of (\ref{eq:poisson}) is a positive definite continuous
on $\mathbb{R}/\mathbb{Z}$ and an extension of its restriction to
$\left(0,\epsilon\right)\subset\mathbb{T}$. \end{example}
\begin{rem}
There is a bijection between

(i) functions on $\mathbb{T}=\mathbb{R}/\mathbb{Z}$, and

(ii) $1$-periodic functions on $\mathbb{R}$. 

Note that (ii) include positive definite functions $F$ defined only
on a subset $\left(-\epsilon,\epsilon\right)\subset\mathbb{T}$. Hence,
it is tricky to get extension $\tilde{F}$ to $\mathbb{T}=\mathbb{R}/\mathbb{Z}$.

\begin{figure}
\begin{tabular}{c}
\includegraphics[scale=0.4]{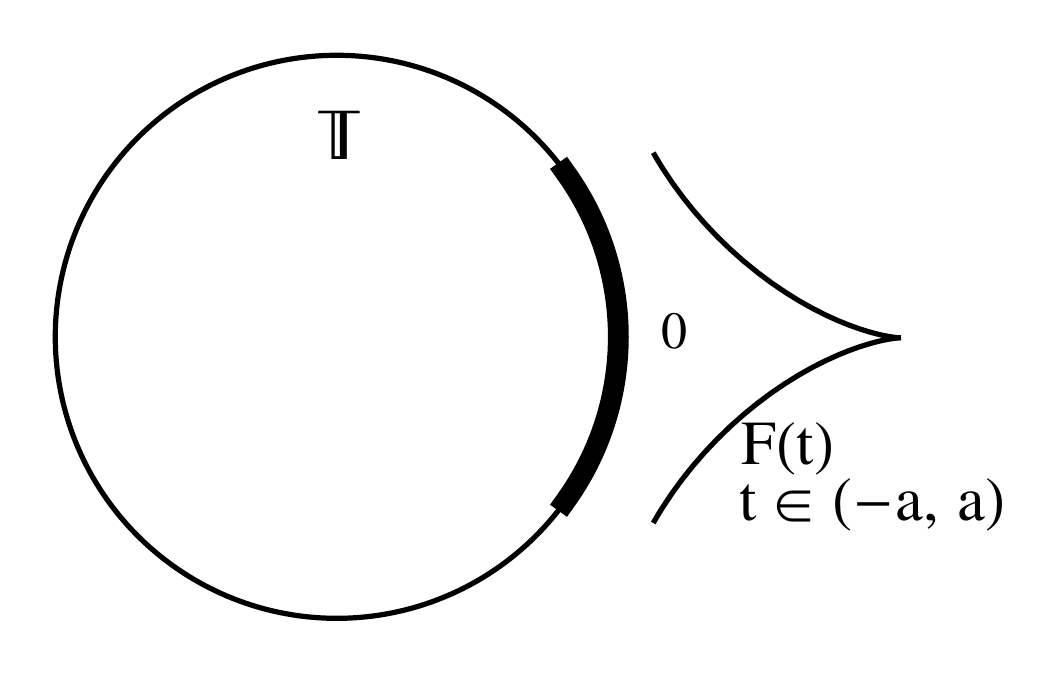}\tabularnewline
\includegraphics[scale=0.5]{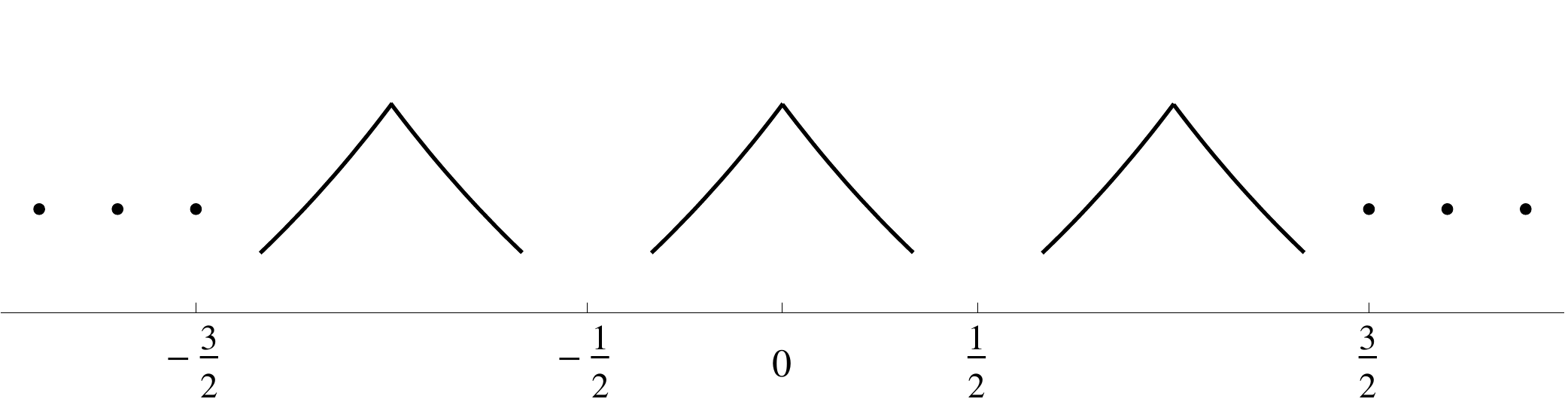}\tabularnewline
\end{tabular}

\begin{minipage}[t]{1\columnwidth}%
Note this bijection also applies if a continuous p.d. function $F$
is only defined on a subset $\left(-a,a\right)\subset\mathbb{T}$,
with $0<a<\frac{1}{2}$.%
\end{minipage}

\protect\caption{\label{fig:ext}Functions on $\mathbb{T}=\mathbb{R}/\mathbb{Z}\longleftrightarrow$
$\bigl(\mbox{1-periodic functions on \ensuremath{\mathbb{R}} }\bigr)$. }
\end{figure}

\end{rem}
\begin{figure}
\includegraphics[scale=0.7]{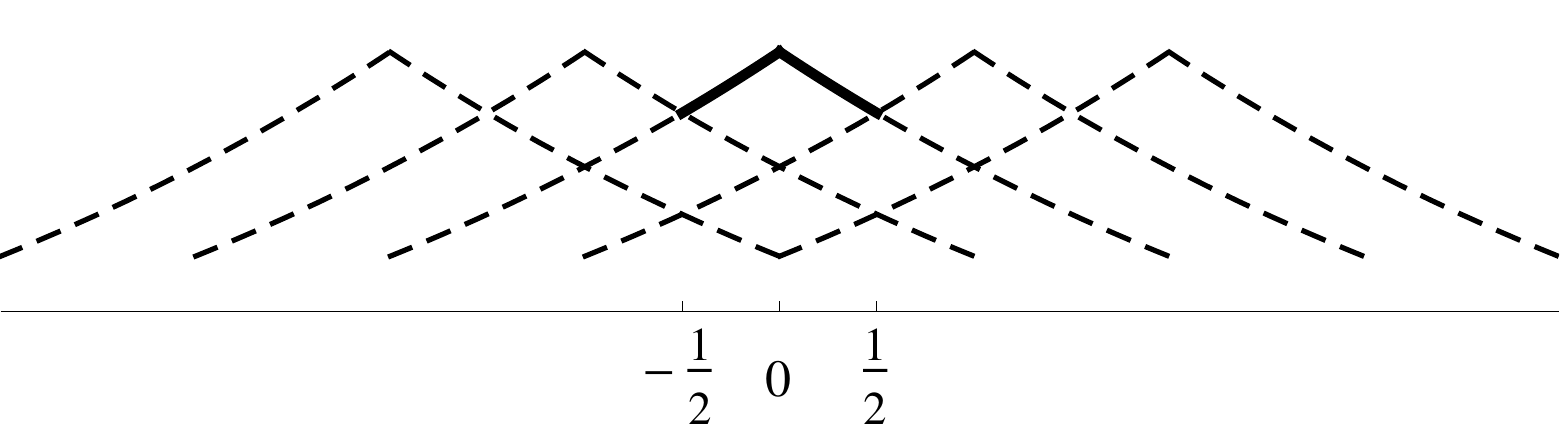}

\protect\caption{\label{fig:pext}$\tilde{F}_{per}\left(t\right)>F\left(t\right)$
on $\left[-\frac{1}{2},\frac{1}{2}\right]$}
\end{figure}

\begin{example}
Fix $a>0$. Let $F\left(t\right)=e^{-a\left|t\right|}$, for all $t\in\left(-\epsilon,\epsilon\right)$.
See Figure \ref{fig:ext}. Then $\tilde{F}\left(t\right)=e^{-a\left|t\right|}$,
$t\in\mathbb{R}$, is a continuous positive definite extension to
$\mathbb{R}$. The periodic version 
\[
\tilde{F}_{per}\left(t\right)=\sum_{n\in\mathbb{Z}}e^{-a\left|t-n\right|}
\]
is continuous, positive definite on $\mathbb{T}$; but $\tilde{F}_{per}\left(t\right)$
is not an extension of $F$. See Figure \ref{fig:pext}. Note $\tilde{F}_{per}\left(t\right)>\tilde{F}\left(t\right)=e^{-a\left|t\right|}$
on $\mathbb{R}$, since 
\[
\tilde{F}_{per}\left(t\right)=\underset{F\left(t\right)}{\underbrace{e^{-a\left|t\right|}}}+\underset{>0}{\underbrace{\sum_{n\in\mathbb{Z}\backslash\left\{ 0\right\} }e^{-a\left|t-n\right|}}}
\]
\end{example}
\begin{rem}
If $0<\epsilon<\frac{1}{2}$ is given, and if $F$ is continuous and
positive definite on $\left(-\epsilon,\epsilon\right)\subset\mathbb{T}=\mathbb{R}/\mathbb{Z}$,
then the analysis of $\mathscr{H}_{F}$ is totally independent of
considerations of periods. $\mathscr{H}_{F}$ does not see global
properties. To understand $F$ on $\left(-\epsilon,\epsilon\right)$
we only need one period interval; see Fig (ii), but when we pass to
$\mathbb{R}$, things can be complicated. See Fig (ii)-(iii).\end{rem}
\begin{question*}
In the example $F\left(t\right)=e^{-a\left|t\right|}$, $t\in\left(-\epsilon,\epsilon\right)$,
what are the deficiency indices of $D^{\left(F\right)}$, i.e., the
densely defined skew-Hermitian operator on $\mathscr{H}_{F}$, $F_{\psi}\mapsto F_{\psi'}$,
$\psi\in C_{c}^{\infty}\left(0,\epsilon\right)$, understood as an
operator in $\mathscr{H}_{F}$? 
\end{question*}
It may be $\left(0,0\right)$. We must decide if the two functions
$\xi_{\pm}\left(t\right)=e^{\pm t}$, $\left|t\right|<\epsilon$,
are in the RKHS $\mathscr{H}_{F}$. To do this we use the estimates
(\ref{eq:bdd})($\Leftrightarrow$(\ref{eq:bdd2})) above. Or we can
solve with Laplace transform. 

But we do know the RHS of (\ref{eq:poisson}) is a periodization of
a positive definite function on $\mathbb{R}$, i.e., 
\begin{equation}
t\mapsto\frac{1}{\pi}e^{-a\left|t\right|}=\int_{-\infty}^{\infty}e^{i2\pi\lambda t}\frac{d\lambda}{a^{2}+\lambda^{2}}\label{eq:Fper1}
\end{equation}
and hence (\ref{eq:poisson}) is an application of the Poisson-summation
formula.

\section{\label{sub:exp(-|x|)}The Example $e^{-\left|x\right|}$ in $\left(-a,a\right)$}

While \subref{euclid} deals with the case of $G=\mathbb{R}^{n}$,
here we set $n=1$. With this restriction we are able to answer classical
questions about positive definite functions, their harmonic analysis,
and their operator theory. As an application of this we obtain a spectral
classification of all Hermitian operators with dense domain in a separable
Hilbert space, having deficiency indices\index{deficiency indices}
$\left(1,1\right)$. Further spectral theoretic results are given
in \chapref{question}.

Here we study the following example $F\left(x\right)=e^{-\left|x\right|}$
on any fixed and finite interval, $\left|x\right|<a$ . So $a$ is
fixed; it can be any positive real number, and we study the case when
the group $G$ is $\mathbb{R}$.

When we compute the associated RKHS $\mathscr{H}_{F}$ from this $F$,
and its skew Hermitian operator in $\mathscr{H}_{F}$, we get deficiency
indices\index{deficiency indices} $\left(1,1\right)$. This is a
special example, but still enlightening in connection with deficiency
index considerations. But, for contrast, we note that, in general,
there will be a host of other interesting 1D examples of partially
defined continuous positive definite\index{positive definite} functions.
Some will have indices $\left(1,1\right)$ and others $\left(0,0\right)$.

In case $\left(1,1\right)$, the convex set $Ext\left(F\right)$ is
parameterized by $\mathbb{T}$; while, in case $\left(0,0\right)$,
$Ext(F)$ is a singleton.

Setting: Fix $a$, $0<a<\infty$. Set $\Omega=\left(0,a\right)$,
so $\Omega-\Omega=\left(-a,a\right)$; and let 
\begin{equation}
F\left(x\right):=e^{-\left|x\right|},\;\mbox{for }\left|x\right|<a.\label{eq:exp-1}
\end{equation}
Note, we do not define $F$ outside $\left(-a,a\right)$. 

We shall need the following probability measure $\mu$ \index{measure!probability}on
$\mathbb{R}$:
\begin{equation}
d\mu\left(\lambda\right)=\frac{d\lambda}{\pi\left(1+\lambda^{2}\right)},\:\lambda\in\mathbb{R}.\label{eq:exp-2}
\end{equation}
The Fourier transform on $\mathbb{R}$ will be denoted $\widehat{\cdot}$. 
\begin{lem}
~
\begin{equation}
\left|\left(\chi_{\left(0,a\right)}\left(x\right)e^{x}\right)^{\wedge}\left(\lambda\right)\right|^{2}=\frac{e^{2a}+1-2e^{a}\cos\left(\lambda a\right)}{1+\lambda^{2}}\;\left(\in L^{2}\left(\mathbb{R},d\lambda\right)\right)\label{eq:exp-3}
\end{equation}
where $d\lambda$ here means Lebesgue measure on $\mathbb{R}$.\end{lem}
\begin{proof}
We have
\[
\left(\chi_{\left(0,a\right)}\left(x\right)e^{x}\right)^{\wedge}\left(\lambda\right)=\int_{0}^{a}e^{-i\lambda x}e^{x}dx=\frac{e^{\left(1-i\lambda\right)a}-1}{1-i\lambda}
\]
and (\ref{eq:exp-3}) follows.
\end{proof}
Now return to the RKHS \index{RKHS}$\mathscr{H}_{F}$, defined from
the function $F$ in (\ref{eq:exp-1}). The set 
\[
\left\{ F_{\varphi}\:\big|\:\varphi\in C_{c}^{\infty}\left(0,a\right)\right\} 
\]
spans a dense subspace in $\mathscr{H}_{F}$ where
\begin{equation}
F_{\varphi}\left(x\right)=\int_{0}^{a}\varphi\left(y\right)e^{-\left|x-y\right|}dy\label{eq:exp-4}
\end{equation}

\begin{lem}
We have 
\begin{eqnarray}
\left\Vert F_{\varphi}\right\Vert _{\mathscr{H}_{F}}^{2} & = & \int_{0}^{a}\int_{0}^{a}\overline{\varphi\left(y\right)}\varphi\left(x\right)e^{-\left|x-y\right|}dxdy\nonumber \\
 & = & \int_{-\infty}^{\infty}\left|\widehat{\varphi}\left(\lambda\right)\right|^{2}\frac{d\lambda}{\pi\left(1+\lambda^{2}\right)}\label{eq:exp-5}
\end{eqnarray}
\end{lem}
\begin{proof}
This follows from (\ref{eq:exp-2}), and the fact that $F$ in (\ref{eq:exp-1})
is the restriction of $\left(e^{-\left|x\right|},x\in\mathbb{R}\right)$.
Hence (\ref{eq:exp-5}) follows from Parseval's identity. 
\end{proof}

We note that, if $\xi:\left(0,a\right)\rightarrow\mathbb{C}$ is any
given continuous function, then $\xi$ is in the RKHS $\mathscr{H}_{F}$
if and only if 
\begin{equation}
\left(\lambda\mapsto\left(\chi_{\left(0,a\right)}\left(x\right)\xi\left(x\right)\right)^{\wedge}\left(\lambda\right)\right)\in L^{2}\left(\mathbb{R},d\mu\left(\lambda\right)\right)\label{eq:exp-6}
\end{equation}
where $d\mu\left(\lambda\right)$ is the measure in (\ref{eq:exp-2}).

\begin{lem}
Let $\Omega=\left(0,a\right)$, $0<a<\infty$, and let $F$ be a continuous
positive definite function on $\Omega-\Omega=\left(-a,a\right)$ such
that $F\left(0\right)=1$. We know that $F$ always has an extension,
i.e., $Ext\left(F\right)\neq\phi$, so that there is a probability
measure \index{measure!probability}$\mu$ on $\mathbb{R}$ such that
\begin{equation}
F\left(x\right)=\widehat{d\mu}\left(x\right),\:\forall x\in\left(-a,a\right).\label{eq:exp-7}
\end{equation}
Suppose further that $\mu$ is absolutely continuous\index{absolutely continuous},
i.e., $\exists\Phi\in L^{1}\left(\mathbb{R},d\lambda\right)$ such
that 
\begin{equation}
d\mu\left(\lambda\right)=\Phi\left(\lambda\right)d\lambda.\label{eq:exp-8}
\end{equation}

Now consider the operator
\begin{equation}
D^{\left(F\right)}\left(F_{\varphi}\right)=F_{\varphi'}\label{eq:exp-9}
\end{equation}
on the dense domain $\left\{ F_{\varphi}\:\big|\:\varphi\in C_{c}^{\infty}\left(0,a\right)\right\} \subseteq\mathscr{H}_{F}$.
Then this operator $D^{\left(F\right)}$ in (\ref{eq:exp-9}) has
indices $\left(0,0\right)$, or $\left(1,1\right)$. 

Its indices are $\left(1,1\right)$ if and only if 
\begin{equation}
\int_{-\infty}^{\infty}\frac{e^{2a}+1-2e^{a}\cos\left(\lambda a\right)}{1+\lambda^{2}}\Phi\left(\lambda\right)d\lambda<\infty.\label{eq:exp-10}
\end{equation}
\end{lem}
\begin{proof}
This follows from the considerations above, using (\ref{eq:exp-3})
and (\ref{eq:exp-5}), but applied to (\ref{eq:exp-8}).\end{proof}
\begin{lem}
\label{lem:exp-1}Fix $a$, $0<a<\infty$, and let $F\left(\cdot\right):=e^{-\left|\cdot\right|}\big|_{\left(-a,a\right)}$
as in (\ref{eq:exp-1}). For all $0\leq x_{0}\leq a$, let 
\begin{equation}
F_{x_{0}}\left(x\right):=F\left(x-x_{0}\right)\Big|_{\left(0,a\right)}\left(\in C\left(0,a\right)\right).\label{eq:exp-11}
\end{equation}
With 
\begin{align}
DEF^{+} & =\left\{ \xi:\bigl(D^{\left(F\right)}\bigr)^{*}\xi=\xi\right\} =span\left\{ \xi_{+}\left(x\right):=e^{-x}\Big|_{\left(0,a\right)}\right\} \label{eq:exp-12}\\
DEF^{-} & =\left\{ \xi:\bigl(D^{\left(F\right)}\bigr)^{*}\xi=-\xi\right\} =span\left\{ \xi_{-}\left(x\right):=e^{-a}e^{+x}\Big|_{\left(0,a\right)}\right\} \label{eq:exp-13}
\end{align}
we get 
\begin{equation}
\left\Vert \xi_{+}\right\Vert _{\mathscr{H}_{F}}^{2}=\left\Vert \xi_{-}\right\Vert _{\mathscr{H}_{F}}^{2}=1.\label{eq:exp-14}
\end{equation}
\end{lem}
\begin{proof}
Note that $x=0$ and $x=a$ are the endpoints in the open interval
$\left(0,a\right)$: 
\begin{align*}
\xi_{+}\left(x\right) & =e^{-x}\Big|_{\left(0,a\right)}=F_{0}\left(x\right)\\
\xi_{-}\left(x\right) & =e^{-a}e^{x}\Big|_{\left(0,a\right)}=F_{a}\left(x\right).
\end{align*}
Let $\psi_{n}\in C_{c}^{\infty}\left(0,a\right)$ be an approximate
identity\index{approximate identity}, such that 
\begin{enumerate}
\item $\psi_{n}\geq0$, $\int\psi_{n}=1$; 
\item $\psi_{n}\rightarrow\delta_{a}$, as $n\rightarrow\infty$. 
\end{enumerate}

\begin{flushleft}
Then 
\[
\xi_{-}\left(x\right)=F_{a}\left(x\right)=\lim_{n\rightarrow\infty}\int_{0}^{a}\psi_{n}\left(y\right)F_{y}\left(x\right)dy.
\]
This shows that $\xi_{-}\in\mathscr{H}_{F}$. Also, 
\begin{eqnarray*}
\left\Vert \xi_{-}\right\Vert _{\mathscr{H}_{F}}^{2} & = & \left\Vert F_{a}\right\Vert _{\mathscr{H}_{F}}^{2}\\
 & = & \lim_{n\rightarrow\infty}\frac{1}{2\pi}\int_{-\infty}^{\infty}\left|\widehat{\psi_{n}}\left(y\right)\right|^{2}\widehat{F}\left(y\right)dy\\
 & = & \frac{1}{2\pi}\int_{-\infty}^{\infty}\widehat{F}\left(y\right)dy=1.
\end{eqnarray*}
Similarly, if instead, $\psi_{n}\rightarrow\delta_{0}$, then 
\[
\xi_{+}\left(x\right)=F_{0}\left(x\right)=\lim_{n\rightarrow\infty}\int_{0}^{a}\psi_{n}\left(y\right)F_{y}\left(x\right)dy
\]
and $\left\Vert \xi_{+}\right\Vert _{\mathscr{H}_{F}}^{2}=1$. 
\par\end{flushleft}

\end{proof}
\begin{lem}
Let $F$ be as in (\ref{eq:exp-1}). We have the following for its
Fourier transform:
\[
\widehat{F}\left(y\right)=\frac{2-2e^{-a}\left(\cos\left(ay\right)-y\sin\left(ay\right)\right)}{1+y^{2}}.
\]
\end{lem}
\begin{proof}
Let $y\in\mathbb{R}$, then 
\begin{eqnarray*}
\widehat{F}\left(y\right) & = & \int_{-a}^{a}e^{iyx}e^{-\left|x\right|}dx\\
 & = & \int_{-a}^{0}e^{iyx}e^{x}dx+\int_{0}^{a}e^{iyx}e^{-x}dx\\
 & = & \frac{1-e^{-a(1+iy)}}{1+iy}+\frac{e^{ia(y+i)}-1}{-1+iy}\\
 & = & \frac{2-2e^{-a}\left(\cos\left(ay\right)-y\sin\left(ay\right)\right)}{1+y^{2}}
\end{eqnarray*}
which is the assertion.\end{proof}
\begin{rem}
If $\left(F,\Omega\right)$ is such that $D^{\left(F\right)}$ has
deficiency indices $\left(1,1\right)$, then by von Neumann's theory,
the family of selfadjoint extensions is characterized by \index{von Neumann}\index{selfadjoint extension}
\begin{align*}
dom\left(A_{\theta}^{\left(F\right)}\right) & =\left\{ F_{\psi}+c\left(\xi_{+}+e^{i\theta}\xi_{-}\right):\psi\in C_{c}^{\infty}\left(0,a\right),c\in\mathbb{C}\right\} \\
A_{\theta}^{\left(F\right)}: & F_{\psi}+c\left(\xi_{+}+e^{i\theta}\xi_{-}\right)\mapsto F_{i\psi'}+c\, i\left(\xi_{+}-e^{i\theta}\xi_{-}\right),\mbox{ where }i=\sqrt{-1}.
\end{align*}
\end{rem}
\begin{prop}
Fix $a>0$, and set $\Omega=\left(0,a\right)$, so $\Omega-\Omega=\left(-a,a\right)$.
Let $F:\left(-a,a\right)\rightarrow\mathbb{C}$ be positive definite
and continuous, $F\left(0\right)=1$; let $D^{\left(F\right)}$ be
the corresponding skew-Hermitian\index{operator!skew-Hermitian},
i.e., $D^{\left(F\right)}\left(F_{\varphi}\right)=F_{\varphi'}$,
for all $\varphi\in C_{c}^{\infty}\left(0,a\right)$. We now assume
that $D^{\left(F\right)}$ has a skew-adjoint extension (in the RKHS
$\mathscr{H}_{F}$) which has simple and purely atomic\index{atom}
spectrum\index{spectrum}, say $\left\{ i\lambda_{n}\right\} _{n\in\mathbb{N}}$,
$\lambda_{n}\in\mathbb{R}$. Then the complex exponentials\index{complex exponential}
\begin{equation}
e_{\lambda_{n}}\left(x\right)=e^{i\lambda_{n}x}\label{eq:exp-15}
\end{equation}
are in $\mathscr{H}_{F}$, and they are orthogonal in $\mathscr{H}_{F}$;
and total.\end{prop}
\begin{proof}
By the assumptions, we may pick some skew-adjoint extension (in $\mathscr{H}_{F}$),
say $A$, 
\begin{equation}
D^{\left(F\right)}\subseteq A\subseteq-\bigl(D^{\left(F\right)}\bigr)^{*}.\label{eq:exp-16}
\end{equation}
Set $U_{A}\left(t\right):=e^{tA}$, $t\in\mathbb{R}$, and recall
that $U_{A}\left(t\right):\mathscr{H}_{F}\rightarrow\mathscr{H}_{F}$
is a unitary representation of $\mathbb{R}$, acting on $\mathscr{H}_{F}$.
By the assumption on its spectrum, we may find an orthonormal basis
(ONB) $\left\{ \xi_{n}\right\} $ in $\mathscr{H}_{F}$, such that
\begin{equation}
U_{A}\left(t\right)=\sum_{n\in\mathbb{N}}e^{it\lambda_{n}}\left|\xi_{n}\left\rangle \right\langle \xi_{n}\right|,\; t\in\mathbb{R},\label{eq:exp-17}
\end{equation}
where $\left|\xi_{n}\left\rangle \right\langle \xi_{n}\right|$ is
Dirac's term for the rank-1 projection onto $\mathbb{C}\xi_{n}$ in
$\mathscr{H}_{F}$. 

From (\ref{eq:exp-16}), we get that if $0<t<a$, then 
\begin{equation}
U_{A}\left(t\right)F_{0}=F_{t};\label{eq:exp-18}
\end{equation}
and by (\ref{eq:exp-17}), therefore (with $0<t<a$):
\begin{equation}
F_{t}\left(\cdot\right)=\sum_{n\in\mathbb{N}}e^{it\lambda_{n}}\overline{\xi_{n}\left(0\right)}\xi_{n}\left(\cdot\right)\mbox{ on }\Omega.\label{eq:exp-19}
\end{equation}
We have used: $\left|\xi_{n}\left\rangle \right\langle \xi_{n}\right|\left(F_{0}\right)=\overline{\xi_{n}\left(0\right)}\xi_{n}\in\mathscr{H}_{F}$.
We are also using the reproducing property in $\mathscr{H}_{F}$. 

Now fix $n\in\mathbb{N}$, and take the inner-product $\left\langle \xi_{n},\cdot\right\rangle _{\mathscr{H}_{F}}$
on both sides in (\ref{eq:exp-19}). Using again the reproducing property,
we get 
\begin{equation}
\xi_{n}\left(t\right)=e^{it\lambda_{n}}\overline{\xi_{n}\left(0\right)},\; t\in\Omega;\label{eq:exp-20}
\end{equation}
which is the desired conclusion.

Note that this makes the functions $\left\{ e_{\lambda_{n}}\right\} _{n\in\mathbb{N}}$
in (\ref{eq:exp-15}) orthogonal, and total in $\mathscr{H}_{F}$;
but they are not normalized; in fact, one checks from (\ref{eq:exp-20})
that 
\[
\left\Vert e_{\lambda_{n}}\left(\cdot\right)\right\Vert _{\mathscr{H}_{F}}=\frac{1}{\left|\xi_{n}\left(0\right)\right|}.
\]
\end{proof}
\begin{prop}
\label{prop:exp}For $F\left(x\right)=e^{-\left|x\right|}\big|_{\left(-a,a\right)}$,
the complex exponentials $e_{\lambda}$ are NOT in the corresponding
RKHS $\mathscr{H}_{F}$. \index{complex exponential}\end{prop}
\begin{proof}
We check that $e_{0}\equiv\mathbf{1}$, i.e., the constant function,
is not in $\mathscr{H}_{F}$. The general case is similar. 

Suppose $\mathbf{1}\in\mathscr{H}_{F}$. Recall the isometric isomorphic
$T:\mathscr{H}_{F}\rightarrow L^{2}\left(\mathbb{R},\mu\right)$,
with 
\begin{equation}
d\mu\left(\lambda\right)=\frac{d\lambda}{\pi\left(1+\lambda^{2}\right)}\label{eq:exp-21}
\end{equation}
and $T\left(F_{\varphi}\right)=\widehat{\varphi}$, for all $\varphi\in C_{c}\left(0,a\right)$;
$T^{*}\left(f\right)=\left(fd\mu\right)^{\vee}$, for all $f\in L^{2}\left(\mu\right)$.
Also note that $R_{T^{*}}=\left(N_{T}\right)^{\perp}=\mathscr{H}_{F}$,
since $T$ is isometric. So, if $\mathbf{1}=e_{0}\in\mathscr{H}_{F}$,
then $\exists f\in L^{2}\left(\mathbb{R},\mu\right)$ s.t.
\begin{equation}
T^{*}f=\mathbf{1}\Longleftrightarrow\int_{\mathbb{R}}e^{-i\lambda x}\frac{f\left(\lambda\right)}{\pi\left(1+\lambda^{2}\right)}d\lambda\equiv1,\;\forall x\in\left(0,a\right).\label{eq:exp-22}
\end{equation}

But there is only a distribution solution $f$ to (\ref{eq:exp-22});
it is $f\left(\lambda\right):=\delta\left(\lambda-0\right)=$ Dirac
mass at $\lambda=0$, so that 
\[
\int_{\mathbb{R}}e^{i\lambda x}\frac{\delta\left(\lambda-0\right)}{\pi\left(1+\lambda^{2}\right)}d\lambda\equiv1,\ \forall x;
\]
and since $f\left(\lambda\right)=\delta\left(\lambda-0\right)$ is
not in $L^{2}\left(\mathbb{R},\mu\right)$ it follows that $\mathbf{1}\notin\mathscr{H}_{F}$.\end{proof}
\begin{thm}
\label{thm:Eigenspaces-for-the-adjoint}Fix an open interval $\Omega=(0,a)$
and a continuous positive definite\index{positive definite} function
$F$ on $\Omega-\Omega=(-a,a).$ Let $z$ be a complex number, then
$y\to e^{-zy}$ is in $\mathscr{H}_{F}$ iff $z$ is an eigenvalue
for the adjoint $\left(D^{(F)}\right)^{*}$ of $D^{(F)}.$ In the
affirmative case the corresponding eigenspace is $\mathbb{C}e_{z},$
in particular, it has dimension one. \end{thm}
\begin{proof}
If $z$ is an eigenvalue for $\left(D^{(F)}\right)^{*}$ and $\xi$
in $\mathscr{H}_{F}$ is a corresponding eigenvector, then 
\[
\left\langle D^{(F)}F_{\psi},\xi\right\rangle _{F}=\left\langle F_{\psi},z\xi\right\rangle _{F}
\]
for all $\psi$ in $C_{c}^{\infty}(\Omega).$ Hence 
\[
\int_{\Omega}\psi'(y)\xi(y)dy=\int_{\Omega}z\psi(y)\xi(y)dy
\]
for all $\psi$ in $C_{c}^{\infty}(\Omega).$ Equivalently, $\xi$
is a weak solution to 
\[
-\xi'(y)=z\xi(y)
\]
in $\Omega.$ Thus $\xi(y)=\mathrm{constant}\, e^{-zy}.$

Conversely, suppose $\xi_{z}:y\to e^{-zy}$ is in $\mathscr{H}_{F}.$
It is sufficient to show $\xi_{z}$ is in the domain of $\left(D^{(F)}\right)^{*}.$
Equivalently, we must show there is a constant $C$ such that 
\begin{equation}
\left|\left\langle D^{(F)}F_{\psi},\xi_{z}\right\rangle _{F}\right|\leq C\left\Vert F_{\psi}\right\Vert _{F}\label{eq:te-23}
\end{equation}
for all $\psi$ in $C_{c}^{\infty}(\Omega).$ But 
\begin{align*}
\left|\left\langle D^{(F)}F_{\psi},\xi_{z}\right\rangle _{F}\right| & =\left|\int_{\Omega}\overline{\psi'(y)}\xi_{z}(y)dy\right|\\
 & =\left|z\int_{\Omega}\overline{\psi(y)}\xi_{z}(y)dy\right|\\
 & =\left|z\right|\left|\left\langle F_{\psi},\xi_{z}\right\rangle _{F}\right|.
\end{align*}
Consequently, (\ref{eq:te-23}) follows from the Cauchy-Schwarz inequality. \end{proof}
\begin{cor}
\label{cor:eigen}If $\xi_{r+is}(y)=e^{(r+is)y}$ is in $\mathscr{H}_{F}$
for some real numbers $r\neq0$ and $s,$ then $\xi_{r+is}$ is in
$\mathscr{H}_{F}$ for all real numbers $r\neq0$ and $s.$ \end{cor}
\begin{proof}
Let $r\neq0$ and $s$ be real numbers. Consider the deficiency spaces
\[
DEF^{r+is}=\left\{ \xi\in\mathscr{H}_{F}:\left\langle D^{(F)}F_{\psi},\xi\right\rangle _{F}=\left\langle F_{\psi},z\xi\right\rangle _{F}\right\} .
\]
For $r>0$ the spaces $DEF^{r+is}$ all have the same dimension as
$DEF^{+}=DEF^{1}$ and the spaces $DEF^{-r+is}$ all have the same
dimension as $DEF^{-}=DEF^{-1}.$ See, e.g., \cite{AG93} or \cite{DS88b}.
But the spaces $DEF^{\pm}$ have the same dimension. By assumption
one of the spaces $D^{r+is}$ has dimension $\geq1.$ Since the deficiency
spaces are eigenspaces for $\left(D^{(F)}\right)^{*}$ the rest follows
from Theorem \ref{thm:Eigenspaces-for-the-adjoint}.\end{proof}
\begin{thm}
Let $a\in\mathbb{R}_{+}$, and $\lambda_{1}\in\mathbb{R}$ be given.
Let $F:\left(-a,a\right)\rightarrow\mathbb{C}$ be a fixed continuous
positive function. Then there following two conditions (\ref{enu:en2-1})
and (\ref{enu:en-2-2}) are equivalent:
\begin{enumerate}
\item \label{enu:en2-1}$\exists\mu_{1}\in Ext\left(F\right)$ such that
$\mu_{1}\left(\left\{ \lambda_{1}\right\} \right)>0$; and
\item \label{enu:en-2-2}$e_{\lambda_{1}}\left(x\right):=e^{i\lambda_{1}x}\Big|_{\left[-a,a\right]}\in\mathscr{H}_{F}$. 
\end{enumerate}
\end{thm}
\begin{rem}
Assertion (\ref{enu:en2-1}) is stating that $\lambda_{1}$ is an
atom for some $\mu_{1}\in Ext\left(F\right)$.\end{rem}
\begin{proof}
The implication (\ref{enu:en2-1})$\Longrightarrow$(\ref{enu:en-2-2})
is already contained in the previous discussion. 

Proof of (\ref{enu:en-2-2})$\Longrightarrow$(\ref{enu:en2-1}).
We first consider the skew-Hermitian operator $D^{\left(F\right)}\left(F_{\varphi}\right):=F_{\varphi'}$,
$\varphi\in C_{c}^{\infty}\left(0,a\right)$. Using an idea of M.
Krein\index{Krein}, (see \cite{Kre46,KL14}), we may always find
a Hilbert space $\mathscr{K}$, and an isometry $J:\mathscr{H}\rightarrow\mathscr{K}$,
and a strongly continuous unitary one-parameter group $U_{A}\left(t\right)=e^{tA}$,
$t\in\mathbb{R}$, with $A^{*}=-A$; $U_{A}\left(t\right)$ acting
in $\mathscr{K}$, such that 
\begin{equation}
JD^{\left(F\right)}=AJ\mbox{ on }\label{eq:en-4-1}
\end{equation}
\begin{equation}
dom\left(D^{\left(F\right)}\right)=\left\{ F_{\varphi}\:\big|\:\varphi\in C_{c}^{\infty}\left(0,a\right)\right\} ;\label{eq:en-4-2}
\end{equation}
see also Theorem \ref{thm:gEn}. Since 
\begin{equation}
e_{1}\left(x\right)=e^{i\lambda_{1}x}\Big|_{\left[-a,a\right]}\label{eq:en-4-3}
\end{equation}
is in $\mathscr{H}_{F}$, we can form the following measure $\mu_{1}\in\mathscr{M}_{+}\left(\mathbb{R}\right)$,
now given by 
\begin{equation}
d\mu_{1}\left(\lambda\right):=\left\Vert P_{A}\left(d\lambda\right)Je_{1}\right\Vert _{\mathscr{K}}^{2},\;\lambda\in\mathbb{R},\label{eq:en-4-4}
\end{equation}
where $P_{A}\left(\cdot\right)$ is the PVM of $U_{A}\left(t\right)$,
i.e., 
\begin{equation}
U_{A}\left(t\right)=\int_{\mathbb{R}}e^{it\lambda}P_{A}\left(d\lambda\right),\; t\in\mathbb{R}.\label{eq:en-4-5}
\end{equation}

We claim the following two assertions: 
\begin{enumerate}[label=(\roman{enumi})]
\item \label{enu:en-2-3}$\mu_{1}\in Ext\left(F\right)$; and 
\item \label{enu:en-2-4}$\lambda_{1}$ is an atom in $\mu_{1}$, i.e.,
$\mu_{1}\left(\left\{ \lambda_{1}\right\} \right)>0$. 
\end{enumerate}

This is the remaining conclusion in the theorem.

The proof of of \ref{enu:en-2-3} is immediate from the construction
above; using the intertwining\index{intertwining} isometry\index{isometry}
$J$ from (\ref{eq:en-4-1}), and formulas (\ref{eq:en-4-4})-(\ref{eq:en-4-5}).

To prove \ref{enu:en-2-4}, we need the following:
\begin{lem}
Let $F$, $\lambda_{1}$, $e_{1}$, $\mathscr{K}$, $J$ and $\left\{ U_{A}\left(t\right)\right\} _{t\in\mathbb{R}}$
be as above; then we have the identity:
\begin{equation}
\left\langle Je_{1},U_{A}\left(t\right)Je_{1}\right\rangle _{\mathscr{K}}=e^{it\lambda_{1}}\left\Vert e_{1}\right\Vert _{\mathscr{H}_{F}}^{2},\;\forall t\in\mathbb{R}.\label{eq:en-4-6}
\end{equation}
\end{lem}
\begin{proof}
It is immediate from (\ref{eq:en-4-1})-(\ref{eq:en-4-4}) that (\ref{eq:en-4-6})
holds for $t=0$. To get it for all $t$, fix $t$, say $t>0$ (the
argument is the same if $t<0$); and we check that 
\begin{equation}
\frac{d}{ds}\left(\left\langle Je_{1},U_{A}\left(t-s\right)Je_{1}\right\rangle _{\mathscr{K}}-e^{i\left(t-s\right)\lambda_{1}}\left\Vert e_{1}\right\Vert _{\mathscr{H}_{F}}^{2}\right)\equiv0,\;\forall s\in\left(0,t\right).\label{eq:en-4-7}
\end{equation}
But this, in turn, follow from the assertions above: First 
\[
D_{\left(F\right)}^{*}e_{1}=D_{\left(F\right)}^{*}J^{*}Je_{1}=J^{*}AJe_{1}
\]
holds on account of (\ref{eq:en-4-1}). We get: $e_{1}\in dom\left(D_{\left(F\right)}^{*}\right)$,
and $D_{\left(F\right)}^{*}e_{1}=-i\lambda_{1}e_{1}$. 

Using this, the verification of is (\ref{eq:en-4-6}) now immediate.
\end{proof}

As a result, we get:
\[
U_{A}\left(t\right)Je_{1}=e^{it\lambda_{1}}Je_{1},\;\forall t\in\mathbb{R},
\]
and by (\ref{eq:en-4-5}):
\[
P_{A}\left(\left\{ \lambda_{1}\right\} \right)Je_{1}=Je_{1}
\]
where $\left\{ \lambda_{1}\right\} $ denotes the desired $\lambda_{1}$-atom.
Hence, by (\ref{eq:en-4-4}), $\mu_{1}\left(\left\{ \lambda_{1}\right\} \right)=\left\Vert Je_{1}\right\Vert _{\mathscr{K}}^{2}=\left\Vert e_{1}\right\Vert _{\mathscr{H}_{F}}^{2}$,
which is the desired conclusion in (\ref{enu:en-2-2}).

\end{proof}

\section{Discussion}

Now in the $\mathbb{R}$-case, the positive definite continuous extension
$\tilde{F}$ are as in (\ref{eq:lcg-bochner}), and 
\[
\tilde{F}\left(t\right)=\int_{\mathbb{R}}e_{t}\left(\lambda\right)d\mu\left(\lambda\right)
\]
where $\mu$ is a finite measure on $\mathbb{R}$. In this case, it
is much easier to specify the RKHS.
\begin{lem}
$\mathscr{H}_{\tilde{F}}\simeq L^{2}\left(\mathbb{R},\mu\right)$. \end{lem}
\begin{proof}
Denote $\widehat{\varphi}$ the Fourier transform of $\varphi$, and
$\check{\varphi}$ be the inverse transform.

For all $\varphi\in C_{c}^{\infty}\left(\mathbb{R}\right)$, 
\begin{eqnarray*}
\tilde{F}_{\varphi}\left(x\right) & = & \int_{\mathbb{R}}\varphi\left(y\right)\tilde{F}\left(y-x\right)dy\\
 & = & \int_{\mathbb{R}}\varphi\left(y\right)\left(\int e_{y-x}\left(\lambda\right)d\mu\left(\lambda\right)\right)dy\\
 & = & \int\widehat{\varphi}\left(\lambda\right)\overline{e_{x}\left(\lambda\right)}d\mu\left(\lambda\right),\mbox{ and }
\end{eqnarray*}
\begin{equation}
\left\Vert \tilde{F}_{\varphi}\right\Vert _{\mathscr{H}_{\tilde{F}}}^{2}=\int\left|\widehat{\varphi}\left(\lambda\right)\right|^{2}d\mu\left(\lambda\right)\label{eq:norm}
\end{equation}
The lemma follows from the fact the $\mathscr{H}_{\tilde{F}}$ is
the completion under (\ref{eq:norm}) of the set of functions $\tilde{F}_{\varphi}$. 
\end{proof}
\textbf{Caution. }
\begin{enumerate}
\item There is \uline{NOT} a similar formula for the RKHS $\mathscr{H}_{F}$
of $F$ is positive definite, continuous, but only defined on $\left(0,\epsilon\right)$. 
\item There is natural isometry\index{isometry}
\begin{eqnarray}
\mathscr{H}_{F} & \hookrightarrow & \mathscr{H}_{\tilde{F}}\nonumber \\
\left(F_{\varphi}\right) & \overset{W}{\rightarrow} & \left(\tilde{F}_{\varphi}\right),\:\varphi\in C_{c}^{\infty}\left(0,\epsilon\right)\label{eq:iso}
\end{eqnarray}
but the isometry $W$ in (\ref{eq:iso}) may not be onto. In some
cases yes, and in other cases no. 
\end{enumerate}
We discussed the Lie group case $G$, $\Omega\subset G$ open connected,
$F$ continuous and positive definite on $\Omega\cdot\Omega^{-1}=\left\{ xy^{-1}:x,y\in\Omega\right\} $.
In this case the most easy extensions to positive definite continuous
functions $\tilde{F}$ will be to $\tilde{G}$, being the simply connected
covering group of $G$. For example, $G=\mathbb{T}$, $\tilde{G}=\mathbb{R}$. 

If $G=SL_{2}\left(\mathbb{R}\right)$, then $\tilde{G}$ is a 2-fold
cover. But $\mathbb{R}\rightarrow\mathbb{R}/\mathbb{Z}$ is an $\infty$
covering of $\mathbb{T}$.

\chapter{\label{chap:types}Type I v.s. Type II Extensions}

In this chapter, we identify extensions of the initially give p.d.
function $F$ which are associated with operator extensions in the
RKHS $\mathscr{H}_{F}$ \index{RKHS}itself (Type I), and those which
require an enlargement of $\mathscr{H}_{F}$, Type II. In the case
of $G=\mathbb{R}$ (the real line) some of these continuous p.d. extensions
arising from the second construction involve a spline\index{spline}-procedure,
and a theorem of G. Polya, which leads to p.d. extensions of $F$
that are symmetric around $x=0$, and convex on the left and right
half-lines. Further these extensions are supported in a compact interval,
symmetric around $x=0$.

\section{\label{sec:Polya}Polya Extensions}

Part of this is the construction of Polya extensions as follow: Starting
with $F$ on $\left(-a,a\right)$; we create a new $F_{ex}$, p.d.
on $\mathbb{R}$, such that $F_{ex}\big|_{\mathbb{R}_{+}}$ is convex.

\index{positive definite}

\index{Polya}

In Figure \ref{fig:spline0}, the slope of $L_{+}$ is chosen to be
$F'\left(a\right)$; and we take the slope of $L_{-}$ to be $F'\left(-a\right)=-F'\left(a\right)$.
Recall that $F$ is defined initially only on some fixed interval
$\left(-a,a\right)$. It then follows by Polya's theorem that each
of these spline extensions is continuous and positive definite. 

\begin{figure}[H]
\includegraphics[scale=0.5]{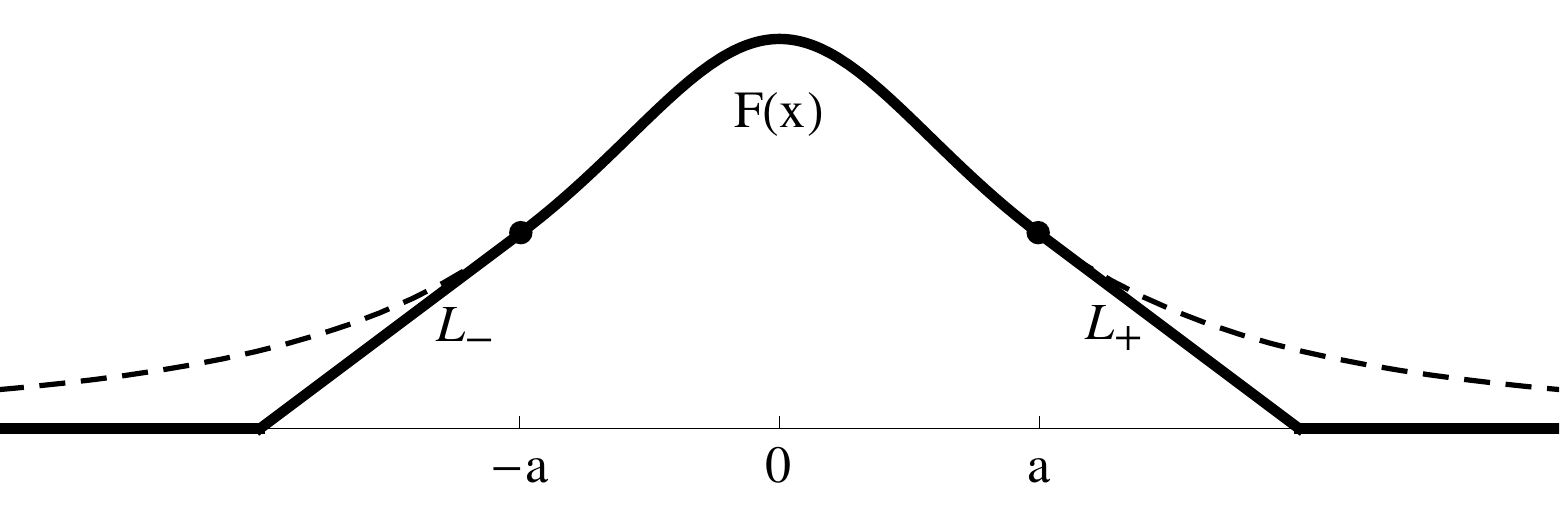}

\protect\caption{\label{fig:spline0}Spline extension of $F:\left(-a,a\right)\rightarrow\mathbb{R}$}
\end{figure}

After extending $F$ from $\left(-a,a\right)$ by adding one or more
line-segments over $\mathbb{R}_{+}$, and using symmetry by $x=0$;
the lines in the extensions will be such that there is a $c$, $0<a<c$,
and the extension $F_{Polya}\left(\cdot\right)$ satisfies $F_{Polya}\left(x\right)=0$
for all $\left|x\right|\geq c$. See Figure \ref{fig:polya} below.

\begin{figure}[H]
\includegraphics[scale=0.5]{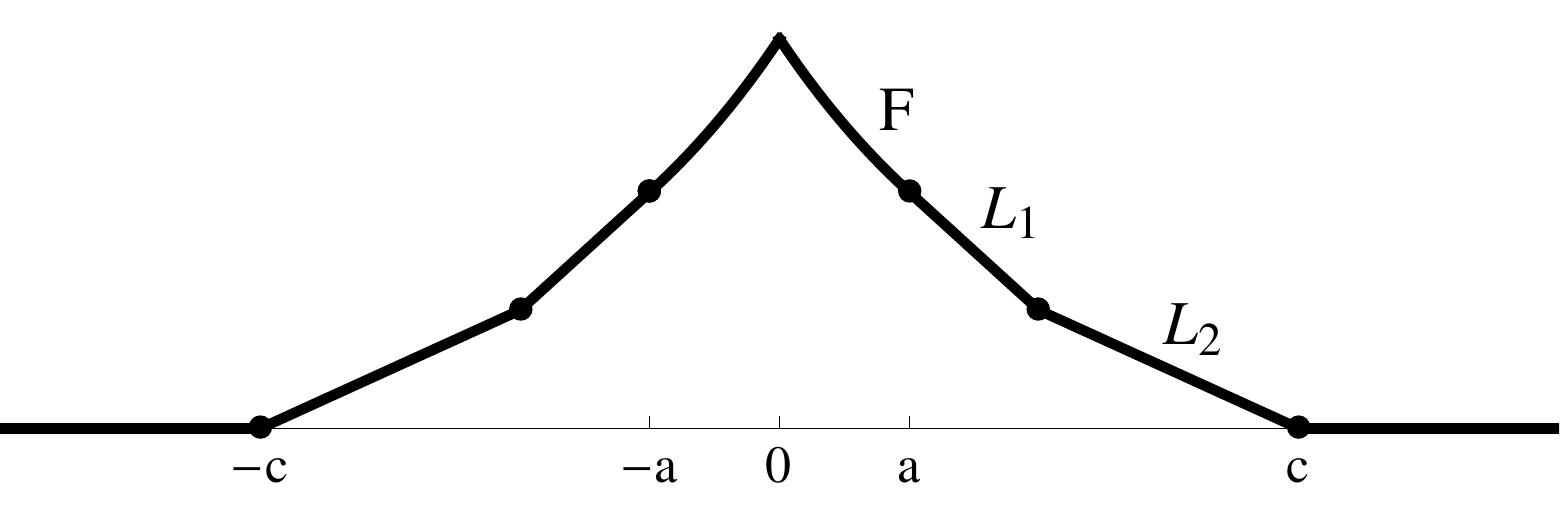}

\protect\caption{\label{fig:polya}An example of Polya extension of $F$ on $\left(-a,a\right)$.}
\end{figure}

\begin{lem}
\label{lem:distd}Consider the following two functions $F_{2}\left(x\right)=1-\left|x\right|$,
and $F_{3}\left(x\right)=e^{-\left|x\right|}$, each defined on a
finite interval $\left(-a,a\right)$, possibly a different value of
$a$ from one to the next. The distributional double derivatives are
as follows:
\begin{eqnarray*}
F_{2}'' & = & -2\delta_{0}\\
F_{3}'' & = & F_{3}-2\delta_{0}
\end{eqnarray*}
where $\delta_{0}$ is Dirac's delta function (i.e., point mass at
$x=0$, $\delta_{0}=\delta\left(x-0\right)$.)\end{lem}
\begin{proof}
The conclusion follows from a computation, making use of L. Schwartz'
theory of distributions; see \cite{Tre06}.\index{Schwartz!distribution}\end{proof}
\begin{prop}
Given $F:\left(-a,a\right)\rightarrow\mathbb{C}$, and assume $F$
has a Polya extension $F_{ex}$, then the corresponding measure $\mu_{ex}\in Ext\left(F\right)$
has the following form 
\[
d\mu_{ex}\left(\lambda\right)=\Phi_{ex}\left(\lambda\right)d\lambda
\]
where
\[
\Phi_{ex}\left(\lambda\right)=\frac{1}{2\pi}\int_{-c}^{c}e^{-i\lambda y}F_{ex}\left(y\right)dy
\]
is entire analytic in $\lambda$.\end{prop}
\begin{proof}
An application of Fourier inversion, and the Paley-Wiener theorem.
\end{proof}
For splines and positive definite functions, we refer to \cite{Sch83,GSS83}.
\begin{example}[Cauchy distribution]
\label{eg:F1} $F_{1}\left(x\right)={\displaystyle \frac{1}{1+x^{2}}}$;
$\left|x\right|<1$. $F_{1}$ is concave near $x=0$.

\begin{figure}[H]
\includegraphics[scale=0.5]{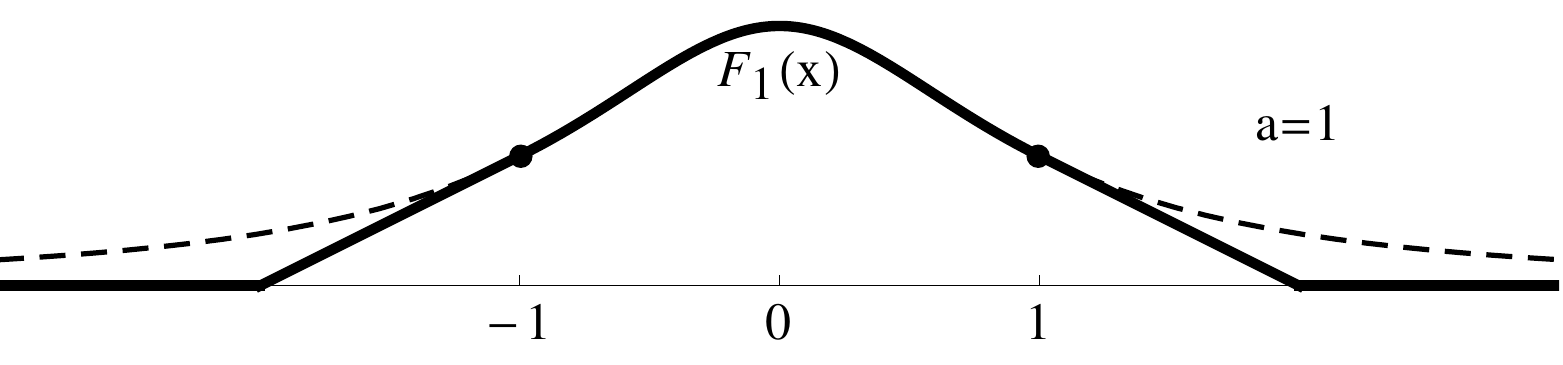}

\protect\caption{\label{fig:spline1}Extension of $F_{1}\left(x\right)={\displaystyle \tfrac{1}{1+x^{2}}}$;
$\Omega=\left(0,1\right)$}
\end{figure}

\end{example}

\begin{example}
\label{eg:F2}$F_{2}\left(x\right)=1-\left|x\right|$; $\left|x\right|<\frac{1}{2}$.
Consider the following Polya extension
\[
F\left(x\right)=\begin{cases}
1-\left|x\right| & \mbox{if }\left|x\right|<\frac{1}{2}\\
{\displaystyle \frac{1}{3}\left(2-\left|x\right|\right)} & \mbox{if }\frac{1}{2}\leq\left|x\right|<2\\
0 & \mbox{if }\left|x\right|\geq2
\end{cases}
\]

\begin{figure}[H]
\includegraphics[scale=0.5]{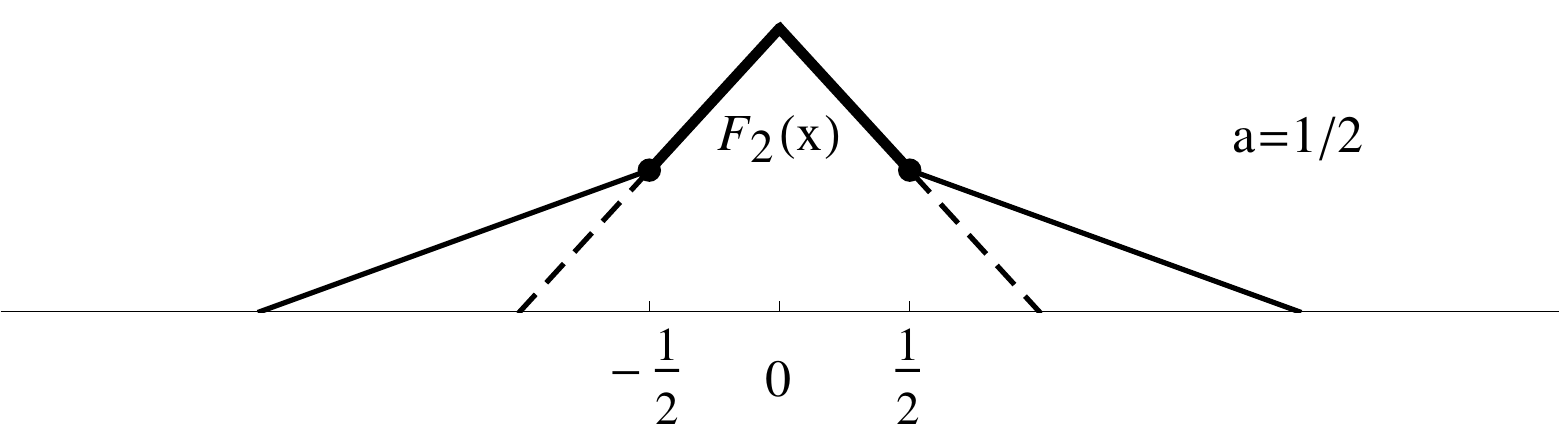}

\protect\caption{\label{fig:spline2}Extension of $F_{2}\left(x\right)=1-\left|x\right|$;
$\Omega=\left(0,\frac{1}{2}\right)$}
\end{figure}

This is a p.d. spline extension which is convex on $\mathbb{R}_{+}$.
The corresponding measure $\mu\in Ext\left(F\right)$ has the following
form $d\mu\left(\lambda\right)=\Phi\left(\lambda\right)d\lambda$
where $d\lambda=$ Lebesgue measure on $\mathbb{R}$, and where
\[
\Phi\left(\lambda\right)=\begin{cases}
\frac{3}{4\pi} & \mbox{if }\lambda=0\\
\frac{1}{3\pi\lambda^{2}}\left(3-2\cos\left(\lambda/2\right)-\cos\left(2\lambda\right)\right) & \mbox{if }\lambda\neq0.
\end{cases}
\]
This solution is in $Ext_{2}\left(F\right)$ where $F\left(x\right):=1-\left|x\right|$
for $x\in\left(-\frac{1}{2},\frac{1}{2}\right)$. By contrast, the
measure $\mu_{2}$ in Table \ref{tab:Table-3} satisfies $\mu_{2}\in Ext_{1}(F).$
See  \subref{ExtSpace}.
\end{example}

\begin{example}[Ornstein-Uhlenbeck]
\label{eg:F3}$F_{3}\left(x\right)=e^{-\left|x\right|}$; $\left|x\right|<1$.
A p.d. spline extension which is convex on $\mathbb{R}_{+}$.\index{Ornstein-Uhlenbeck}

\begin{figure}[H]
\includegraphics[scale=0.5]{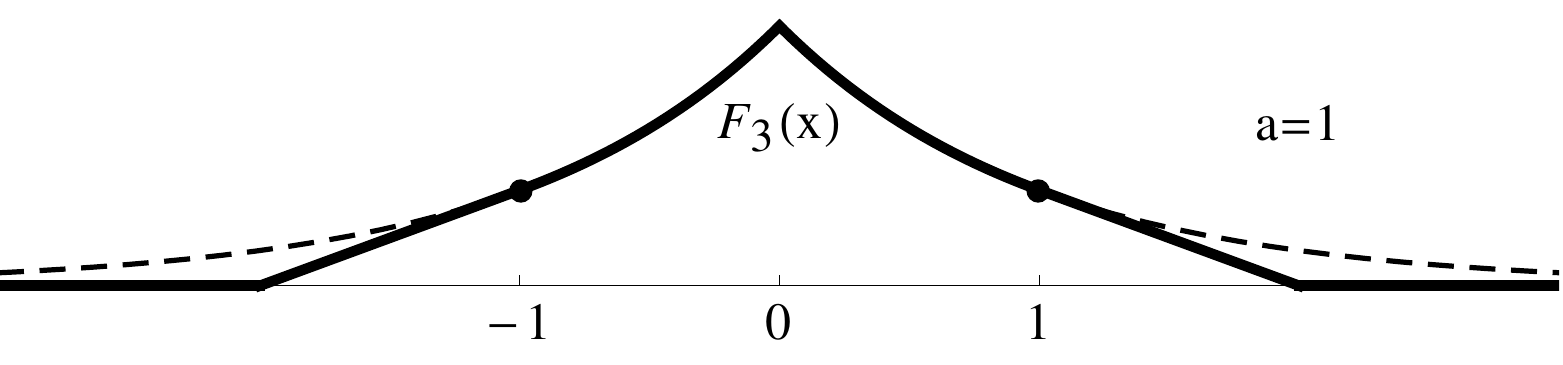}

\protect\caption{\label{fig:spline3}Extension of $F_{3}\left(x\right)=e^{-\left|x\right|}$;
$\Omega=\left(0,1\right)$}
\end{figure}

\end{example}

\begin{example}[Shannon]
\label{eg:F4}$F_{4}\left(x\right)=\left({\displaystyle \frac{\sin\pi x}{\pi x}}\right)^{2}$;
$\left|x\right|<\frac{1}{2}$. $F_{4}$ is concave near $x=0$.\index{Shannon sampling}

\begin{figure}[H]
\includegraphics[scale=0.5]{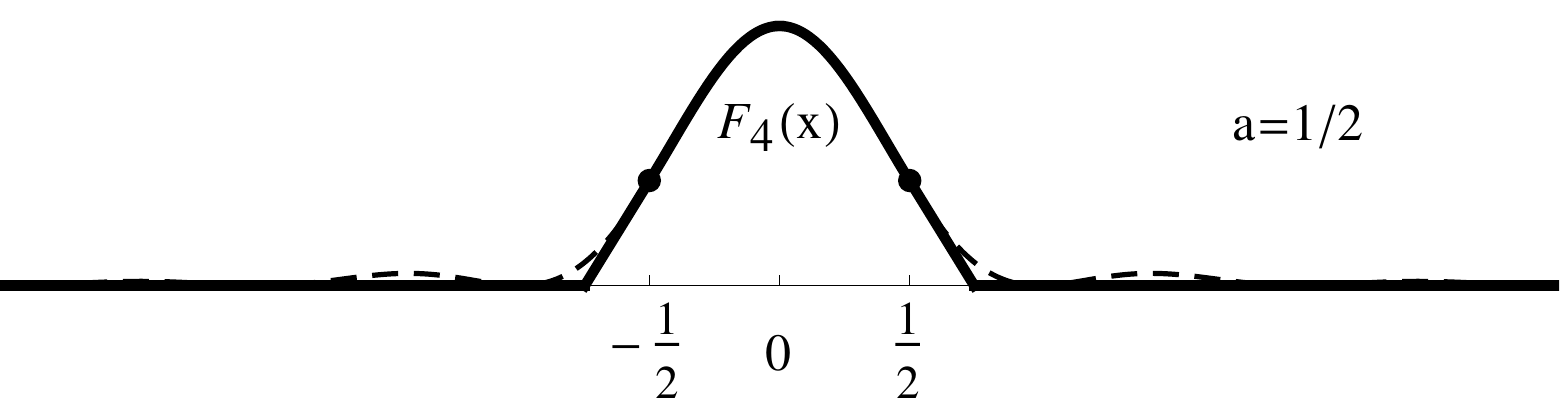}

\protect\caption{\label{fig:spline4}Extension of $F_{4}\left(x\right)=\left(\tfrac{\sin\pi x}{\pi x}\right)^{2}$;
$\Omega=\left(0,\frac{1}{2}\right)$}
\end{figure}

\end{example}

\begin{example}[Gaussian distribution]
\label{eg:F5}$F_{5}\left(x\right)=e^{-x^{2}/2}$; $\left|x\right|<1$.
$F_{5}$ is concave in $-1<x<1$. 

\begin{figure}[H]
\includegraphics[scale=0.5]{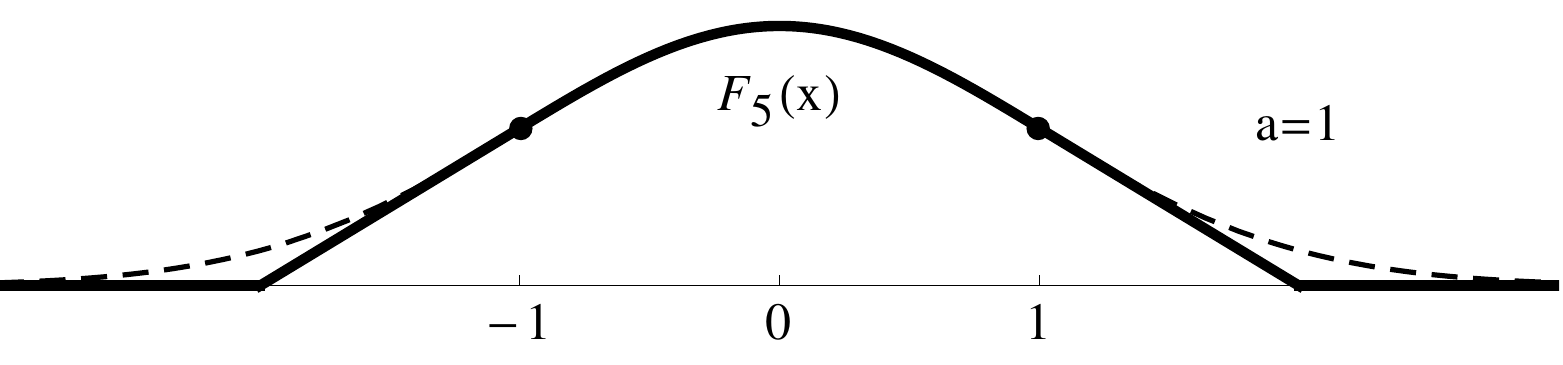}

\protect\caption{\label{fig:spline5}Extension $F_{5}\left(x\right)=e^{-x^{2}/2}$;
 $\Omega=\left(0,1\right)$}
\end{figure}
\end{example}
\begin{rem}
Some spline extensions may not be positive definite. Start with a
fixed continuous p.d. function $F$ on some finite interval $\left(-a,a\right)$,
and consider spline extensions say $F_{ex}$ such that for some $c>a$,
$F_{ex}\left(x\right)=0$, for $\left|x\right|>c$. 

In order for Polya\textquoteright s theorem to be applicable, the
spline extended function $F_{ex}$ to $\mathbb{R}$ must be convex
on $\mathbb{R}_{+}$. By construction, then our extension to $\mathbb{R}$
is by mirror symmetry around $x=0$. So if we start with a symmetric
p.d. function $F$ in $\left(-a,a\right)$ which is concave near $x=0$,
then the spline extension does not satisfy the premise in Polya\textquoteright s
theorem.

Now Polya\textquoteright s theorem of course only accounts for some
p.d. functions, not all.

But it is not surprising that the spline extension in the list of
figures \ref{fig:spline0}-\ref{fig:spline5} may not be p.d., e.g.,
Figure \ref{fig:spline1}, and \ref{fig:spline4}. Indeed, it is easy
to check that the two partially defined functions $F$ in Figure \ref{fig:spline1},
and \ref{fig:spline4} are concave near $x=0$, just calculate the
double derivative $F''$.

Of the p.d. functions $F_{1},F_{2},\ldots F_{6}$, from table \ref{tab:F1-F6},
we note that for the following $F_{1},F_{4},F_{5}$, and $F_{6}$
satisfy this: there is a $c>0$ such that the function in question
is concave in the interval $\left[0,c\right]$, the value of c varies
from one to the next.

So these four cases from table \ref{tab:F1-F6} do not yield spline
extensions $F_{ex}$ which are convex when restricted to $\mathbb{R}_{+}$.
Polya\textquoteright s theorem only applies when convexity holds on
$\mathbb{R}_{+}$. In that case, the spline extensions will be p.d..
And so Polya\textquoteright s theorem only accounts for those spline
extensions $F_{ex}$ which are convex when restricted to $\mathbb{R}_{+}$.
Now there may be p.d. spline extensions that are not convex when restricted
to $\mathbb{R}_{+}$, and we leave open the discovery of those.
\end{rem}
Above we introduced a special class of positive definite (p.d.) continuous
extensions using a spline technique based on a theorem by Polya \cite{pol49}:
More precisely, we extend from a fixed and finite open interval, extending
a given locally defined p.d. function $F$ to the whole real line.
At the endpoints of the interval we construct a linear spline extension
thereby creating a real valued p.d. extension of $F$ which is symmetric
around $x=0$, and is convex on the positive half-line. This spline
extension is supported in a larger interval on $\mathbb{R}$ which
we are free to adjust. We will refer to these p.d. extensions as Polya
extensions. 

For all of these examples we can use Polya's trick to generate families
of continuous positive definite\index{positive definite} extensions
from $\left(-a,a\right)$ to all of the real line $\mathbb{R}$. See
Figure \ref{fig:polya}, \ref{fig:spline2}, and \ref{fig:spline3}.
Some of these figures (\ref{fig:spline0}-\ref{fig:spline6}) indicate
spline extensions that are not p.d..

We get the nicest p.d. extensions if we make the derivative $F'=\frac{dF}{dx}$
a spline at the break-points. In Example \ref{eg:F1}-\ref{eg:F5},
we compute $F'\left(a\right)$. We then use symmetry for the left-hand-side
of the figure:
\begin{align*}
F_{1}'\left(1\right) & =-\frac{1}{2}\\
F_{2}'\left(1/2\right) & =-1\\
F_{3}'\left(1\right) & =-e^{-1}\\
F_{4}'\left(1/2\right) & =-16\pi^{-2}\\
F_{5}'\left(1\right) & =-e^{-1/2}
\end{align*}
So we use these slopes in making the straight line extension. 

We can use Polya's theorem on each of the five functions $F_{i}$,
$i=1,\ldots,5$, defined initially only on $\Omega-\Omega=\left(-a,a\right)$.
There we take $a=1$, or $\frac{1}{2}$, etc. 

We then get a deficiency index-problem in the RKHSs\index{RKHS} $\mathscr{H}_{F_{i}}$,
$i=1,\ldots5$, for the operator $D^{\left(F_{i}\right)}F_{\varphi}^{\left(i\right)}=F_{\varphi'}^{\left(i\right)}$,
$\forall\varphi\in C_{c}^{\infty}\left(0,a\right)$. And all the five
skew-Hermitian\index{operator!skew-Hermitian} operators in $\mathscr{H}_{Fi}$
will have deficiency indices\index{deficiency indices} $\left(1,1\right)$. 

Following is an example with deficiency indices $\left(0,0\right)$
\begin{example}
\label{eg:F6}$F_{6}\left(x\right)=\cos\left(x\right)$; $\left|x\right|<\frac{\pi}{4}$

\begin{figure}[H]
\includegraphics[scale=0.5]{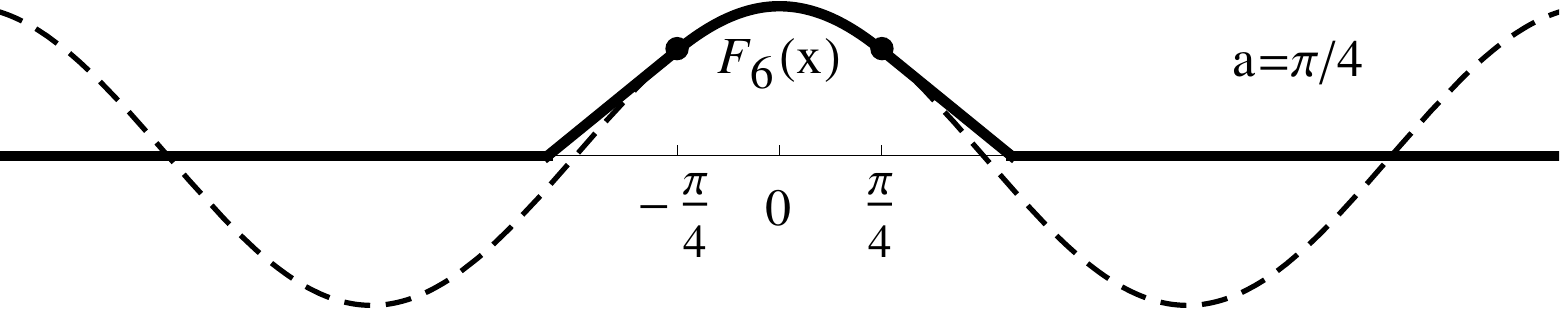}

\protect\caption{\label{fig:spline6}Extension of $F_{6}\left(x\right)=\cos\left(x\right)$;
$\Omega=\left(0,\frac{\pi}{4}\right)$}
\end{figure}
\end{example}
\begin{lem}
The function $\cos x$ is positive definite on $\mathbb{R}$. \end{lem}
\begin{proof}
For all finite system of coefficients $\left\{ c_{j}\right\} $ in
$\mathbb{C}$, we have 
\begin{eqnarray}
\sum_{j}\sum_{k}\overline{c_{j}}c_{k}\cos\left(x_{j}-x_{k}\right) & = & \sum_{j}\sum_{k}\overline{c_{j}}c_{k}\left(\cos x_{j}\cos x_{k}+\sin x_{j}\sin x_{k}\right)\nonumber \\
 & = & \Bigl|\sum_{j}c_{j}\cos x_{j}\Bigr|^{2}+\Bigl|\sum_{j}c_{j}\sin x_{j}\Bigr|^{2}\geq0.\label{eq:ext-1}
\end{eqnarray}
Hence Polya also applies here, and we get more than one continuous
positive definite extension from $\Omega-\Omega=\left(-\frac{\pi}{4},\frac{\pi}{4}\right)$
to all of $\mathbb{R}$. The other extension is $x\mapsto\cos x$,
for all $x\in\mathbb{R}$.\end{proof}
\begin{lem}
$\mathscr{H}_{F_{6}}$ is 2-dimensional.\end{lem}
\begin{proof}
For all $\varphi\in C_{c}^{\infty}\left(0,\frac{\pi}{4}\right)$,
we have:
\begin{align}
 & \int_{\Omega}\int_{\Omega}\overline{\varphi\left(x\right)}\varphi\left(y\right)F_{6}\left(x-y\right)dxdy\nonumber \\
= & \left|\int_{\Omega}\varphi\left(x\right)\cos xdx\right|^{2}+\left|\int_{\Omega}\varphi\left(x\right)\sin xdx\right|^{2}\nonumber \\
= & \left|\widehat{\varphi}^{\left(c\right)}\left(1\right)\right|^{2}+\left|\widehat{\varphi}^{\left(s\right)}\left(1\right)\right|^{2}\label{eq:ext-2}
\end{align}
where $\widehat{\varphi}^{\left(c\right)}=$ the cosine-transform,
and $\widehat{\varphi}^{\left(s\right)}=$ the sine-transform.
\end{proof}
So the deficiency indices only account from some of the extension
of a given positive definite function $F$ on $\Omega-\Omega$ at
the starting point. 

In all five examples above $\mathscr{H}_{F_{i}}$, ($i=1,\ldots,5$),
is infinite-dimensional; but $\mathscr{H}_{F_{6}}$ is 2-dimensional.

In the given five examples, we have p.d. continuous extensions to
$\mathbb{R}$ of the following form, $\widehat{d\mu_{i}}\left(\cdot\right)$,
$i=1,\ldots,5$, where these measures \index{measure!Lebesgue}are
as follows:
\begin{eqnarray*}
d\mu_{1}\left(\lambda\right) & = & \frac{1}{2}e^{-\left|\lambda\right|}d\lambda\\
d\mu_{2}\left(\lambda\right) & = & \left(\frac{\sin\pi\lambda}{\pi\lambda}\right)^{2}d\lambda\\
d\mu_{3}\left(\lambda\right) & = & \frac{d\lambda}{\pi\left(1+\lambda^{2}\right)}\\
d\mu_{4}\left(\lambda\right) & = & \chi_{\left(-1,1\right)}\left(\lambda\right)\left(1-\left|\lambda\right|\right)d\lambda\\
d\mu_{5}\left(\lambda\right) & = & \frac{1}{\sqrt{2\pi}}e^{-\lambda^{2}/2}d\lambda
\end{eqnarray*}
with $\lambda\in\mathbb{R}$; see Table \ref{tab:meas} and Figure
\ref{fig:meas} below.

\begin{cor}
In all five examples above, we get isometries as follows
\[
T^{\left(i\right)}:\mathscr{H}_{F_{i}}\rightarrow L^{2}\left(\mathbb{R},\mu_{i}\right)
\]
\[
T^{\left(i\right)}\left(F_{\varphi}^{\left(i\right)}\right)=\widehat{\varphi}
\]
for all $\varphi\in C_{c}^{\infty}\left(\Omega\right)$, where we
note that $\widehat{\varphi}\in L^{2}\left(\mathbb{R},\mu_{i}\right)$,
$i=1,\ldots,5$; and
\[
\left\Vert F_{\varphi}^{\left(i\right)}\right\Vert _{\mathscr{H}_{F_{i}}}^{2}=\left\Vert \widehat{\varphi}\right\Vert _{L^{2}\left(\mu\right)}^{2}=\int_{\mathbb{R}}\left|\widehat{\varphi}\right|^{2}d\mu_{i},\; i=1,\ldots,5;
\]
but note that $T^{\left(i\right)}$ is only isometric into $L^{2}\left(\mu_{i}\right)$. 

For the adjoint operator:
\[
\left(T^{\left(i\right)}\right)^{*}:L^{2}\left(\mathbb{R},\mu_{i}\right)\rightarrow\mathscr{H}_{F_{i}}
\]
we have 
\[
\left(T^{\left(i\right)}\right)^{*}f=\left(fd\mu_{i}\right)^{\vee},\:\forall f\in L^{2}\left(\mathbb{R},\mu_{i}\right).
\]

\end{cor}
Here is an infinite-dimensional example as a version of $F_{6}$.
Fix some positive $p$, $0<p<1$, and set 
\[
\prod_{n=1}^{\infty}\cos\left(2\pi p^{n}x\right)=F_{p}\left(x\right)
\]
then this is a continuous positive definite function on $\mathbb{R}$,
and the law is the corresponding Bernoulli measure $\mu_{p}$ satisfying
$F_{p}=\widehat{d\mu_{p}}$. Note that some of those measures $\mu_{p}$
are fractal measures.\index{measure!Bernoulli}\index{measure!fractal}

For fixed $p\in\left(0,1\right)$, the measure $\mu_{p}$ is the law
of the following random power series
\[
X_{p}\left(w\right):=\sum_{n=1}^{\infty}\left(\pm\right)p^{n}
\]
where $w\in\prod{}_{1}^{\infty}\left\{ \pm1\right\} $ (= infinite
Cartesian product) and where the distribution of each factor is $\left\{ -\frac{1}{2},\frac{1}{2}\right\} $,
and statically independent. For relevant references on random power
series, see \cite{Neu13,Lit99}.

The extensions we generate with the application of Polya's theorem
are realized in a bigger Hilbert space. The deficiency indices are
computed for the RKHS $\mathscr{H}_{F}$, i.e., for the ``small''
p.d. function $F:\Omega-\Omega\rightarrow\mathbb{C}$. \index{RKHS}\index{deficiency indices}
\begin{example}
$F_{6}$ on $\left(-\frac{\pi}{4},\frac{\pi}{4}\right)$ has the obvious
extension $\mathbb{R}\ni x\rightarrow\cos x$, with a 2-dimensional
Hilbert space; but the other p.d. extensions (from Polya) will be
in infinite-dimensional Hilbert spaces. See Figure \ref{fig:F6}.
\end{example}
\begin{figure}
\begin{align*}
\xymatrix{ &  & \underset{2-\dim}{\mbox{RKHS}(\cos\left(x\right))}\ar[dd]\\
F_{\varphi}\in\mathscr{H}_{F\mbox{ on }\left(-\frac{\pi}{4},\frac{\pi}{4}\right)}\ar@/^{2pc}/[rru]\\
\varphi\in C_{c}^{\infty}\left(-\frac{\pi}{4},\frac{\pi}{4}\right)\ar[rr]\sp(0.4){T\left(F_{\varphi}\right)=\widehat{\varphi}} &  & \mbox{RKHS}\left(\mbox{Polya}\right)\underset{\infty-\dim}{\simeq}L^{2}\Bigl(\mathbb{R},d\mu_{\mbox{Polya}}\Bigr)
}
\end{align*}

\protect\caption{\label{fig:F6}$\dim\left(\mbox{RKHS}\left(\cos x\;\mbox{on }\mathbb{R}\right)\right)=2$;
but RKHS(Polya ext. to $\mathbb{R}$) is $\infty$-dimensional.}
\end{figure}

We must make a distinction between two classes of p.d. extensions
of $F:\Omega-\Omega\rightarrow\mathbb{C}$ to continuous p.d. functions
on $\mathbb{R}$. \index{unitary representation}

\textbf{Case 1.} There exists a unitary representation $U\left(t\right):\mathscr{H}_{F}\rightarrow\mathscr{H}_{F}$
such that
\begin{equation}
F\left(t\right)=\left\langle \xi_{0},U\left(t\right)\xi_{0}\right\rangle _{\mathscr{H}_{F}},\: t\in\Omega-\Omega\label{eq:ext-3}
\end{equation}

\textbf{Case 2.} (e.g., Polya extension) There exist a Hilbert space
$\mathscr{K}$, and an isometry\index{isometry} $J:\mathscr{H}_{F}\rightarrow\mathscr{K}$,
and a unitary representation $U\left(t\right):\mathscr{K}\rightarrow\mathscr{K}$,
such that 
\begin{equation}
F\left(t\right)=\left\langle J\xi_{0},U\left(t\right)J\xi_{0}\right\rangle _{\mathscr{K}},\: t\in\Omega-\Omega\label{eq:ext-4}
\end{equation}

In both cases, $\xi_{0}=F\left(0-\cdot\right)\in\mathscr{H}_{F}$.

In case 1, the unitary representation is realized in $\mathscr{H}_{\left(F,\Omega-\Omega\right)}$,
while, in case 2, the unitary representation $U\left(t\right)$ lives
in the expanded Hilbert space $\mathscr{K}$.

Note that the RHS in both (\ref{eq:ext-3}) and (\ref{eq:ext-4})
is defined for all $t\in\mathbb{R}$.\index{GNS}
\begin{lem}
Let $F_{ex}$ be one of the Polya extensions if any. Then by the Galfand-Naimark-Segal
(GNS) construction applied to $F_{ext}:\mathbb{R}\rightarrow\mathbb{R}$,
there is a Hilbert space $\mathscr{K}$ and a vector $v_{0}\in\mathscr{K}$
and a unitary representation $\left\{ U\left(t\right)\right\} _{t\in\mathbb{R}};$
$U\left(t\right):\mathscr{K}\rightarrow\mathscr{K}$, such that 
\begin{equation}
F_{ex}\left(t\right)=\left\langle v_{0},U\left(t\right)v_{0}\right\rangle _{\mathscr{K}},\:\forall t\in\mathbb{R}.\label{eq:ext-5}
\end{equation}

Setting $J:\mathscr{H}_{F}\rightarrow\mathscr{K}$, $J\xi_{0}=v_{0}$,
then $J$ defines (by extension) an isometry such that 
\begin{equation}
U\left(t\right)J\xi_{0}=J\left(\mbox{local translation in }\Omega\right)\label{eq:ext-6}
\end{equation}
holds locally (i.e., for $t$ sufficiently close to $0$.)

Moreover, the function 
\begin{equation}
\mathbb{R}\ni t\mapsto U\left(t\right)J\xi_{0}=U\left(t\right)v_{0}\label{eq:ext-7}
\end{equation}
is compactly supported.\end{lem}
\begin{proof}
The existence of $\mathscr{K}$, $v_{0}$, and $\left\{ U\left(t\right)\right\} _{t\in\mathbb{R}}$
follows from the GNS-construction.

The conclusions in (\ref{eq:ext-6}) and (\ref{eq:ext-7}) follow
from the given data, i.e., $F:\Omega-\Omega\rightarrow\mathbb{R}$,
and the fact that $F_{ex}$ is a spline-extension, i.e., it is of
compact support; but by (\ref{eq:ext-5}), this means that (\ref{eq:ext-7})
is also compactly supported.
\end{proof}
Example \ref{eg:F2} gives a p.d. $F$ in $\left(-\tfrac{1}{2},\tfrac{1}{2}\right)$
with $D^{(F)}$ of index $(1,1)$ and explicit measures in $Ext_{1}(F)$
and in $Ext_{2}(F).$

We have the following:

\textbf{Deficiency $\left(0,0\right)$:} The p.d. extension of type
1\index{type 1} is unique; see (\ref{eq:ext-3}); but there may still
be p.d. extensions of type 2\index{type 2}; see (\ref{eq:ext-4}).

\textbf{Deficiency $\left(1,1\right)$: }This is a one-parameter family
of extensions of type 1; and some more p.d. extensions are type 2.

So we now divide 
\[
Ext\left(F\right)=\left\{ \mu\in\mbox{Prob}\left(\mathbb{R}\right)\:\big|\:\widehat{d\mu}\mbox{ is an extension of }F\right\} 
\]
up in subsets
\[
Ext\left(F\right)=Ext_{type1}\left(F\right)\cup Ext_{type2}\left(F\right);
\]
where $Ext_{type2}\left(F\right)$ corresponds to the Polya extensions.\index{operator!unbounded}

\renewcommand{\arraystretch}{2}

\begin{table}
\begin{tabular}{|l|c|c|}
\hline 
$F:\left(-a,a\right)\rightarrow\mathbb{C}$ & Indices & The Operator $D^{\left(F\right)}$\tabularnewline
\hline 
\hline 
$F_{1}\left(x\right)=\frac{1}{1+x^{2}}$, $\left|x\right|<1$  & $\left(0,0\right)$ & $D^{\left(F\right)}$ unbounded, skew-adjoint\tabularnewline
\hline 
$F_{2}\left(x\right)=1-\left|x\right|$, $\left|x\right|<\frac{1}{2}$ & $\left(1,1\right)$ & $D^{\left(F\right)}$ has unbounded sk. adj. extensions\tabularnewline
\hline 
$F_{3}\left(x\right)=e^{-\left|x\right|}$, $\left|x\right|<1$ & $\left(1,1\right)$ & $D^{\left(F\right)}$ has unbounded sk. adj. extensions\tabularnewline
\hline 
$F_{4}\left(x\right)=\left(\frac{\sin\pi x}{\pi x}\right)^{2}$, $\left|x\right|<\frac{1}{2}$ & $\left(0,0\right)$ & $D^{\left(F\right)}$ bounded, skew-adjoint\tabularnewline
\hline 
$F_{5}\left(x\right)=e^{-x^{2}/2}$, $\left|x\right|>1$ & $\left(0,0\right)$ & $D^{\left(F\right)}$ unbounded, skew-adjoint\tabularnewline
\hline 
$F_{6}\left(x\right)=\cos x$, $\left|x\right|<\frac{\pi}{4}$ & $\left(0,0\right)$ & $D^{\left(F\right)}$ is rank-one, $\dim\left(\mathscr{H}_{F_{6}}\right)=2$ \tabularnewline
\hline 
\end{tabular}

\bigskip{}

\protect\caption{\label{tab:F1-F6}The deficiency indices\index{deficiency indices}
of $D^{\left(F\right)}:F_{\varphi}\mapsto F_{\varphi'}$ in examples
\ref{eg:F1}-\ref{eg:F6}}
\end{table}

\renewcommand{\arraystretch}{1}

\renewcommand{\arraystretch}{2}

\begin{table}
\begin{tabular}{cc}
\includegraphics[scale=0.35]{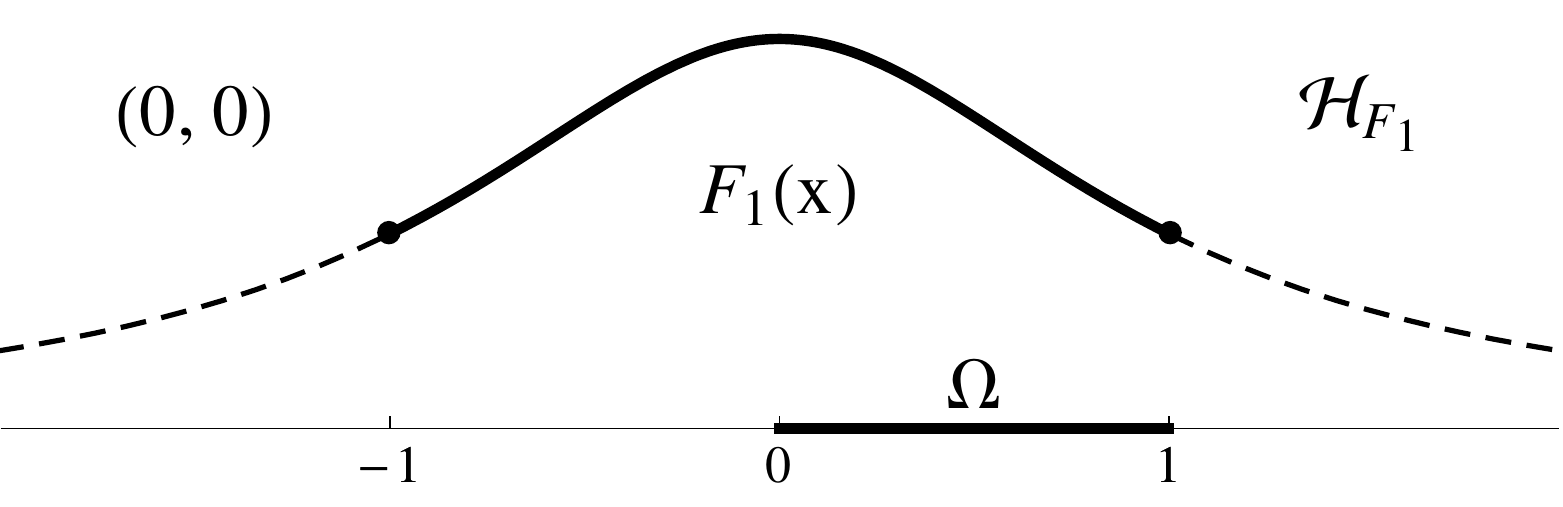} & \includegraphics[scale=0.35]{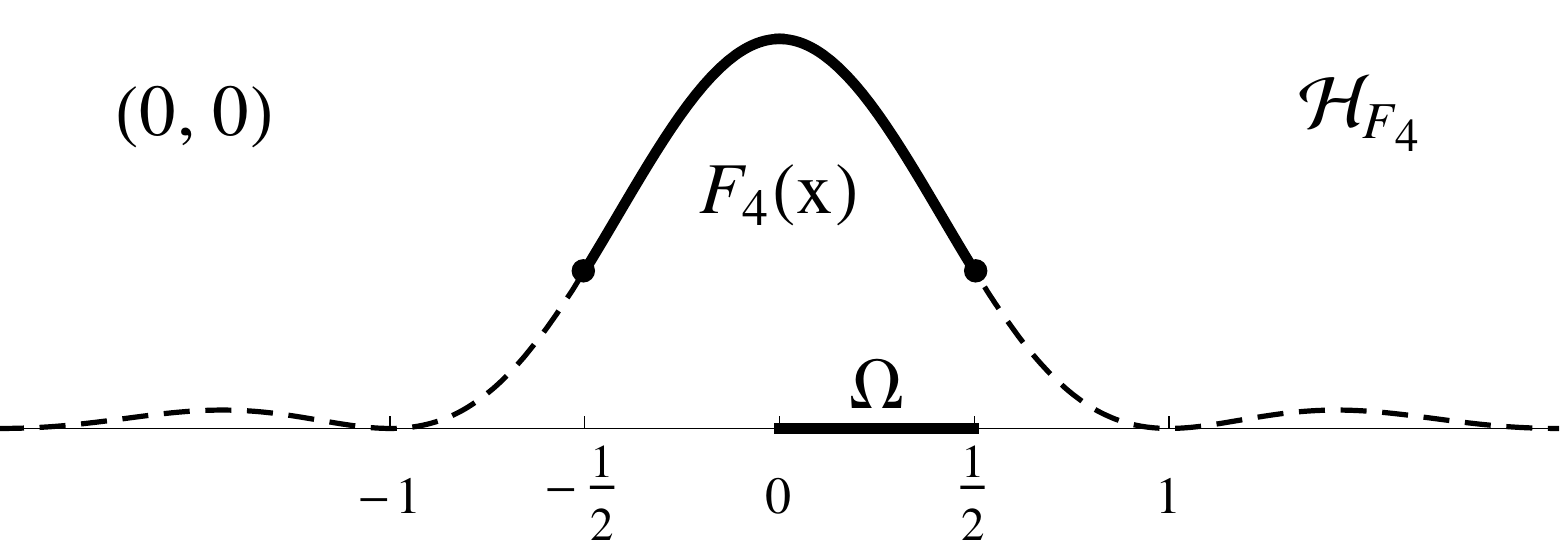}\tabularnewline
\includegraphics[scale=0.35]{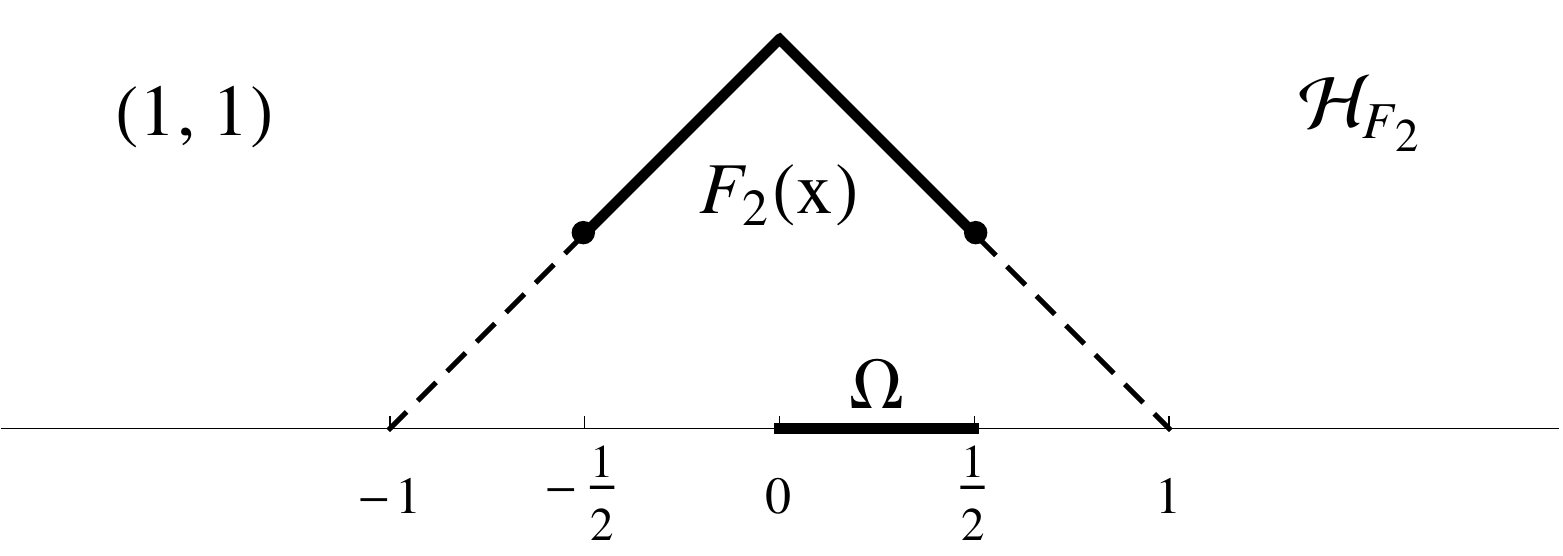} & \includegraphics[scale=0.35]{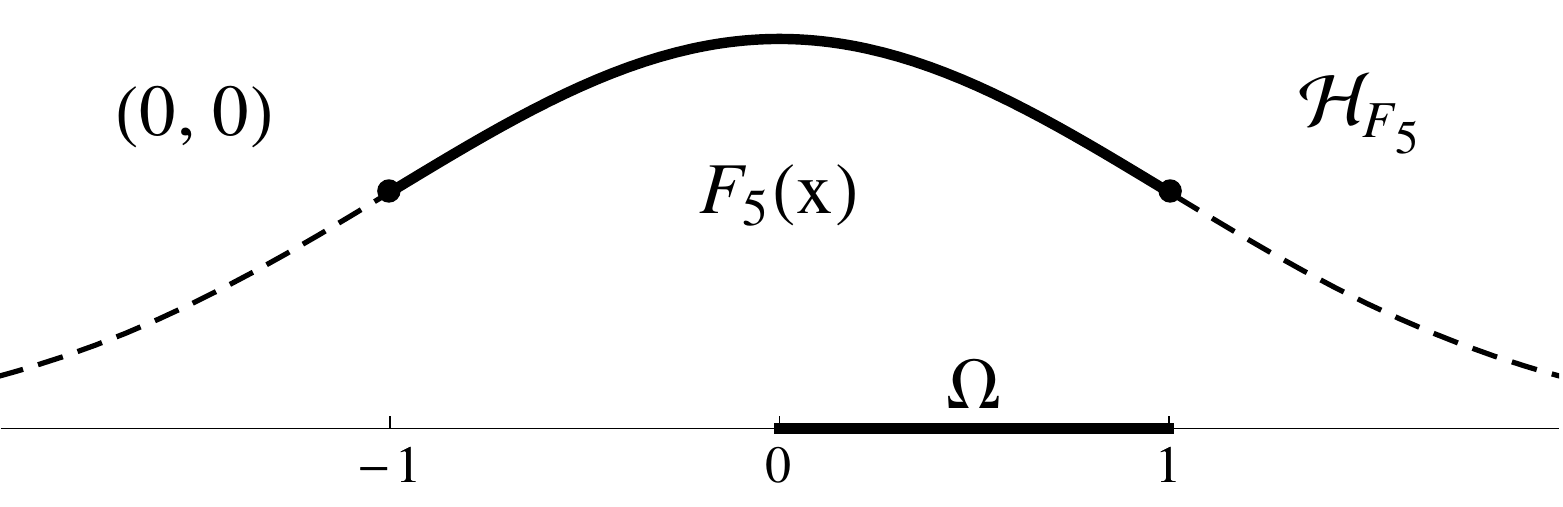}\tabularnewline
\includegraphics[scale=0.35]{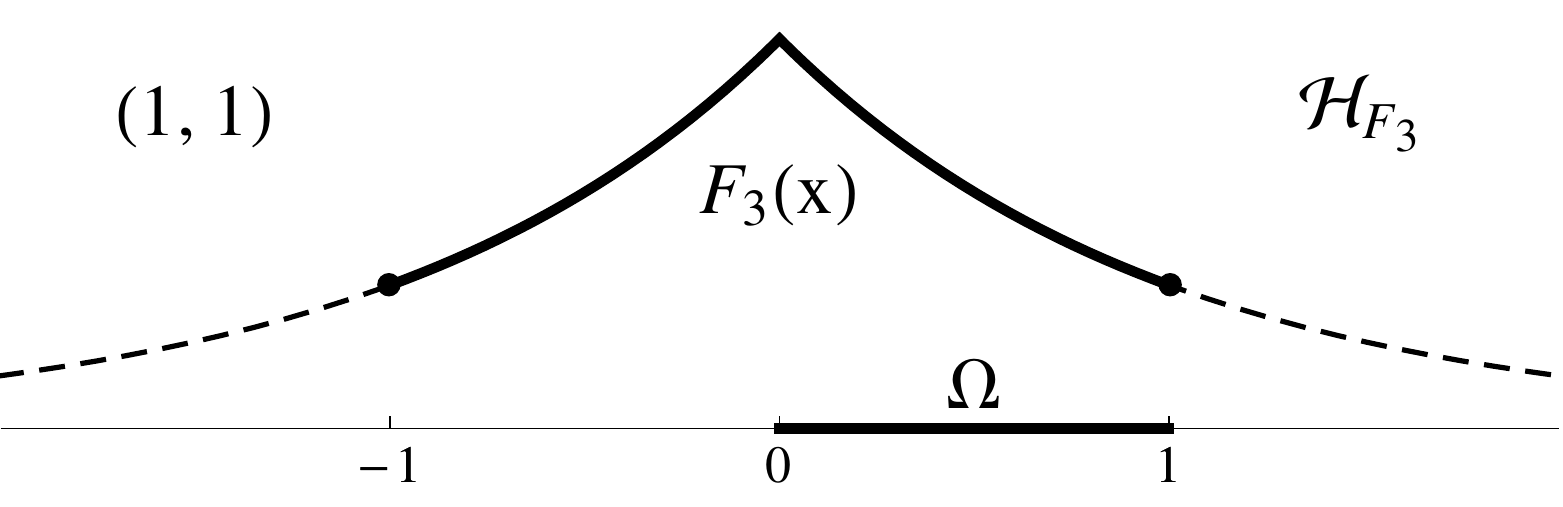} & \includegraphics[scale=0.35]{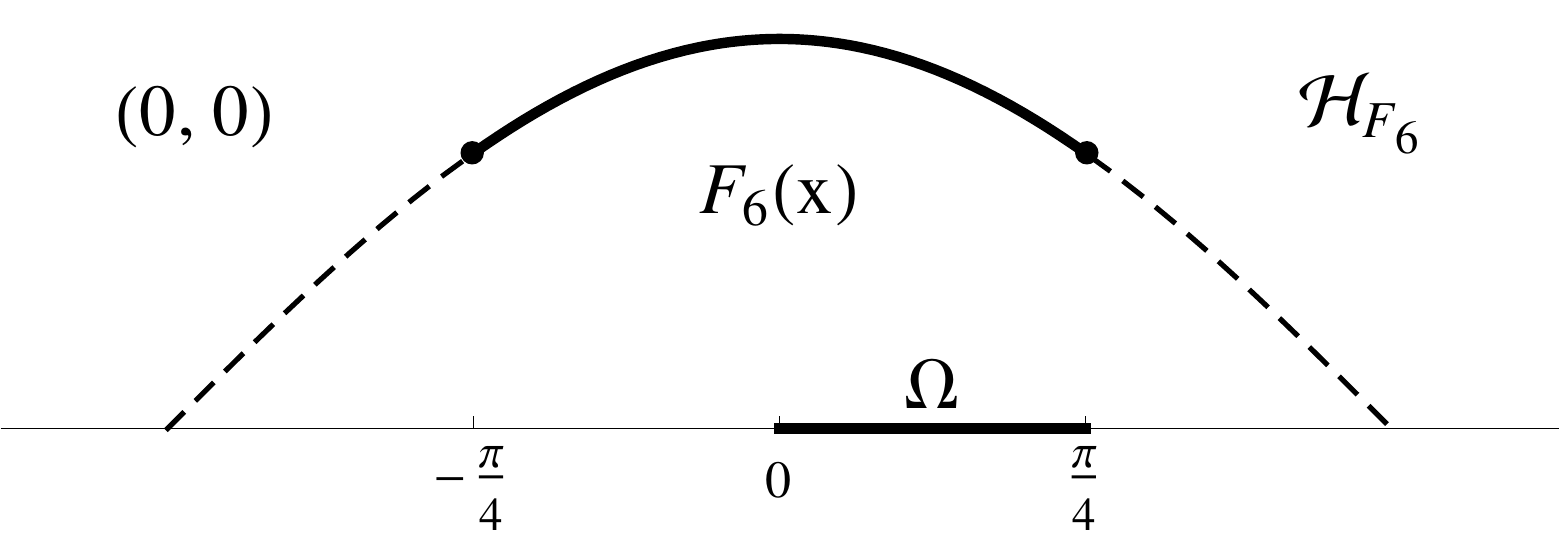}\tabularnewline
\end{tabular}

\protect\caption{\label{tab:Table-2}Type II extensions. Six cases of p.d. continuous
functions $F_{i}$ defined on a finite interval $\left(-a,a\right)$.}
\end{table}

\renewcommand{\arraystretch}{1}

\renewcommand{\arraystretch}{2}

\begin{table}
\begin{tabular}{|l|l|}
\hline 
$d\mu_{1}\left(\lambda\right)=\frac{1}{2}e^{-\left|\lambda\right|}d\lambda$ & $d\mu_{4}\left(\lambda\right)=\chi_{\left(-1,1\right)}\left(\lambda\right)\left(1-\left|\lambda\right|\right)d\lambda$,
cpt. support\tabularnewline
\hline 
$d\mu_{2}\left(\lambda\right)=\left(\frac{\sin\pi\lambda}{\pi\lambda}\right)^{2}d\lambda$,
Shannon & $d\mu_{5}\left(\lambda\right)=\frac{1}{\sqrt{2\pi}}e^{-\lambda^{2}/2}d\lambda$,
Gaussian\tabularnewline
\hline 
$d\mu_{3}\left(\lambda\right)=\frac{d\lambda}{\pi\left(1+\lambda^{2}\right)}$
, Cauchy & $d\mu_{6}\left(\lambda\right)=\frac{1}{2}\left(\delta_{1}+\delta_{-1}\right)$,
\index{atom}atomic; two Dirac masses\tabularnewline
\hline 
\end{tabular}

\bigskip{}

\protect\caption{\label{tab:Table-3}The canonical isometric embeddings: $\mathscr{H}_{F_{i}}\hookrightarrow L^{2}\left(\mathbb{R},d\mu_{i}\right)$,
$i=1,\ldots,6$.}
\end{table}

\renewcommand{\arraystretch}{1}

Return to a continuous p.d. function $F:\left(-a,a\right)\rightarrow\mathbb{C}$,
we take for the RKHS $\mathscr{H}_{F}$, and the skew-Hermitian operator
\[
D\left(F_{\varphi}\right)=F_{\varphi'},\:\varphi'=\frac{d\varphi}{dx}
\]
 If $D\subseteq A$, $A^{*}=-A$ in $\mathscr{H}_{F}$ then there
exists an isometry $J:\mathscr{H}_{F}\rightarrow L^{2}\left(\mathbb{R},\mu\right)$,
where $d\mu\left(\cdot\right)=\left\Vert P_{U}\left(\cdot\right)\xi_{0}\right\Vert ^{2}$,
\[
U_{A}\left(t\right)=e^{tA}=\int_{\mathbb{R}}e^{it\lambda}P_{U}\left(d\lambda\right),
\]
$\xi_{0}=F\left(\cdot-0\right)\in\mathscr{H}_{F}$, $J\xi_{0}=1\in L^{2}\left(\mu\right)$.

\section{\label{sec:F3-Mercer}The Example \ref{eg:F3}}

Here, we study Example \ref{eg:F3}, and compute the spectral date
of the corresponding Mercer operator. Recall that 
\[
F\left(x\right):=\begin{cases}
e^{-\left|x\right|} & \left|x\right|<1\\
e^{-1}\left(2-\left|x\right|\right) & 1\leq\left|x\right|<2\\
0 & \left|x\right|\geq2
\end{cases}
\]
See \figref{PF} below.\index{operator!Mercer} 

\begin{figure}[H]
\includegraphics[scale=0.4]{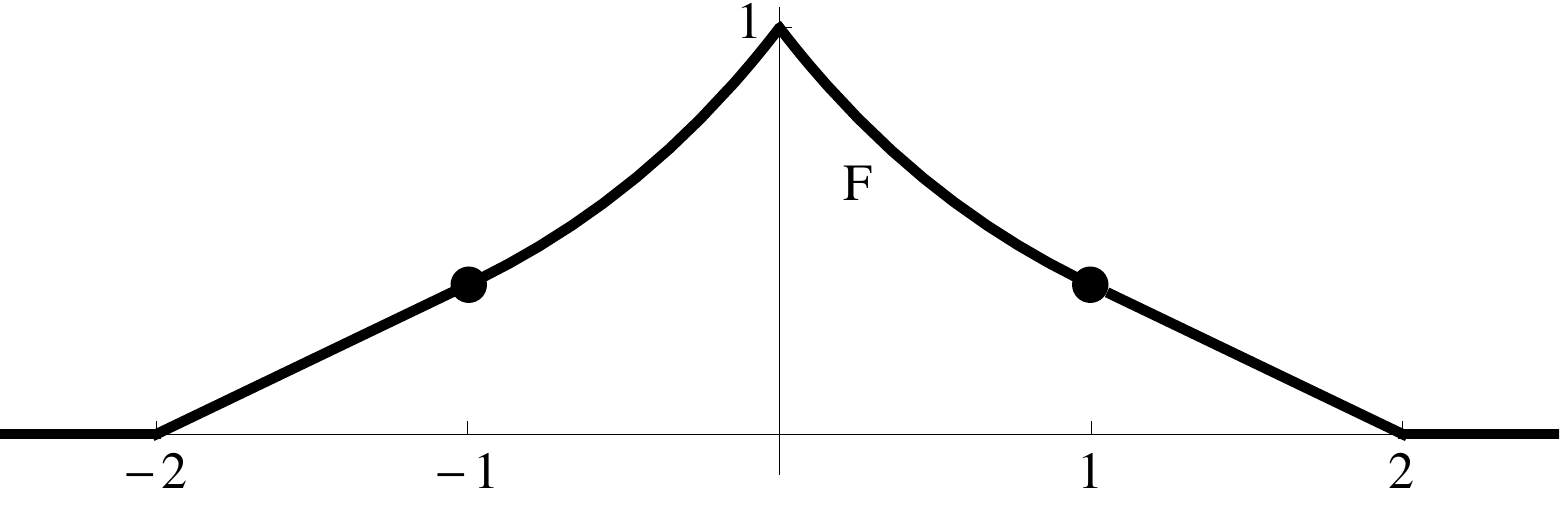}

\protect\caption{\label{fig:PF} A Polya (spline) extension of $F_{3}=e^{-\left|x\right|}$;
$\Omega=\left(-1,1\right)$.}

\end{figure}

The distributional derivative of $F$ satisfies 
\begin{eqnarray*}
F'' & = & F-2\delta_{0}+e^{-1}\left(\delta_{2}+\delta_{-2}\right)\\
 & \Updownarrow\\
\left(1-\triangle\right)F & = & 2\delta_{0}-e^{-1}\left(\delta_{2}+\delta_{-2}\right)
\end{eqnarray*}
This can be verified directly, using Schwartz' theory of distributions.
See \lemref{distd} and \cite{Tre06}. \index{Schwartz!distribution}

Thus, for 
\[
F_{x}\left(\cdot\right):=F\left(x-\cdot\right)
\]
we have the translation 
\begin{equation}
\left(1-\triangle\right)F_{x}=2\delta_{x}-e^{-1}\left(\delta_{x-2}+\delta_{x+2}\right)\label{eq:FP-1}
\end{equation}

Consider the Mercer operator (see \secref{mercer}): 
\[
L^{2}\left(0,2\right)\ni g\longmapsto\int_{0}^{2}F_{x}\left(y\right)g\left(y\right)dy=\left\langle F_{x},g\right\rangle .
\]
Suppose $\lambda\in\mathbb{R}$, and 
\begin{equation}
\left\langle F_{x},g\right\rangle =\lambda g\left(x\right).\label{eq:FP-2}
\end{equation}
Applying $\left(1-\triangle\right)$ to both sides of (\ref{eq:FP-2}),
we get
\begin{equation}
\left\langle \left(1-\triangle\right)F_{x},g\right\rangle =\lambda\left(g\left(x\right)-g''\left(x\right)\right),\;0<x<2.\label{eq:FP-3}
\end{equation}

By (\ref{eq:FP-1}), we also have 
\begin{equation}
\left\langle \left(1-\triangle\right)F_{x},g\right\rangle =2g\left(x\right),\;0<x<2;\label{eq:FP-4}
\end{equation}
using the fact the two Dirac masses in (\ref{eq:FP-1}), i.e., $\delta_{x\pm2}$,
are supported outside the open interval $\left(0,2\right)$. 

Therefore, combining (\ref{eq:FP-3})-(\ref{eq:FP-4}), we have 
\[
g''=\frac{\lambda-2}{\lambda}g,\;\forall x\in\left(0,2\right).
\]
By Mercer's theorem, $0<\lambda<2$. Setting 
\[
k:=\sqrt{\frac{2-\lambda}{\lambda}}\;\left(\Leftrightarrow\lambda=\frac{2}{1+k^{2}}\right)
\]
we have 
\[
g''=-k^{2}g,\;\forall x\in\left(0,2\right).
\]

\uline{Boundary conditions:}

In (\ref{eq:FP-3}), set $x=0$, and $x=2$, we get 
\begin{eqnarray}
2g\left(0\right)-e^{-1}g\left(2\right) & = & \lambda\left(g\left(0\right)-g''\left(0\right)\right)\label{eq:FP-bd-1}\\
2g\left(2\right)-e^{-1}g\left(0\right) & = & \lambda\left(g\left(2\right)-g''\left(2\right)\right)\label{eq:FP-bd-2}
\end{eqnarray}
Now, assume 
\[
g\left(x\right)=Ae^{ikx}+Be^{-ikx};
\]
where 
\begin{eqnarray*}
g\left(0\right) & = & A+B\\
g\left(2\right) & = & Ae^{ik2}+Be^{-ik2}\\
g''\left(0\right) & = & -k^{2}\left(A+B\right)\\
g''\left(2\right) & = & -k^{2}\left(Ae^{ik2}+Be^{-ik2}\right).
\end{eqnarray*}
Therefore, for (\ref{eq:FP-bd-1}), we have 
\begin{eqnarray*}
2\left(A+B\right)-e^{-1}\left(Ae^{ik2}+Be^{-ik2}\right) & = & \lambda\left(1+k^{2}\right)\left(A+B\right)\\
 & = & 2\left(A+B\right)
\end{eqnarray*}
i.e.,
\begin{equation}
Ae^{ik2}+Be^{-ik2}=0.\label{eq:FP-bd-3}
\end{equation}
Now, from (\ref{eq:FP-bd-2}) and using (\ref{eq:FP-bd-3}), we have
\begin{equation}
A=-B.\label{eq:FP-bd-4}
\end{equation}
Combining (\ref{eq:FP-bd-3})-(\ref{eq:FP-bd-4}), we conclude that
\[
\sin2k=0\Longleftrightarrow k=\frac{\pi n}{2},\; n\in\mathbb{Z};
\]
i.e., 
\begin{equation}
\lambda_{n}:=\frac{2}{1+\left(\dfrac{n\pi}{2}\right)^{2}},\; n\in\mathbb{N}.\label{eq:FP-5}
\end{equation}

The associated ONB in $L^{2}\left(0,2\right)$ is 
\begin{equation}
\xi_{n}\left(x\right)=\sin\left(\frac{n\pi x}{2}\right),\; n\in\mathbb{N}.\label{eq:FP-6}
\end{equation}
And the corresponding ONB in $\mathscr{H}_{F}$ consists of the functions
$\left\{ \sqrt{\lambda_{n}}\xi_{n}\right\} _{n\in\mathbb{N}}$, i.e.,
\begin{equation}
\left\{ \frac{\sqrt{2}}{\bigl(1+\bigl(\frac{n\pi}{2}\bigr)^{2}\bigr)^{1/2}}\sin\left(\frac{n\pi x}{2}\right)\right\} _{n\in\mathbb{N}}\label{eq:FP-7}
\end{equation}

\chapter{\label{chap:Ext1}The Computation of $Ext_{1}\left(F\right)$}

Given a continuous positive definite function $F$ on the open interval
$\left(-1,1\right)$. We are concerned with the set $Ext\left(F\right)$
of its extensions to continuous positive definite function defined
on all of $\mathbb{R}$, as well as a certain subset $Ext_{1}\left(F\right)$
of $Ext\left(F\right)$. Since every such p.d. extension of $F$ is
a Bochner transform of a unique positive and finite Borel measure
on $\mathbb{R}$, we will speak of the set $Ext\left(F\right)$ as
a set of finite positive Borel measure on $\mathbb{R}$. The purpose
of this chapter is to gain insight into the nature and properties
of $Ext_{1}\left(F\right)$ of $Ext\left(F\right)$.

In \secref{mercer}, we study a certain trace class integral operator
(the Mercer operator $T_{F}$). We use it to identify natural Bessel
frame in the RKHS $\mathscr{H}_{F}$; and we further introduce a notion
of Shannon sampling of finite Borel measures on $\mathbb{R}$. We
then use this in Corollary \ref{cor:shan} to give a necessary and
sufficient condition for a given finite Borel measure $\mu$ to fall
in the convex set $Ext\left(F\right)$: The measures in $Ext\left(F\right)$
are precisely those whose Shannon sampling recover the given p.d.
function $F$ on the interval $\left(-1,1\right)$.

\section{\label{sec:mercer}Mercer Operators, and Bessel Frames}

In the considerations below, we shall be primarily concerned with
the case when a fixed continuous p.d. function $F$ is defined on
a finite interval $\left(-a,a\right)\subset\mathbb{R}$. In this case,
by a Mercer operator, we mean an operator $T_{F}$ in $L^{2}\left(0,a\right)$
where $L^{2}\left(0,a\right)$ is defined from Lebesgue measure on
$\left(0,a\right)$. We have\index{operator!Mercer}
\begin{equation}
\left(T_{F}\varphi\right)\left(x\right)=\int_{0}^{a}\varphi\left(y\right)F\left(x-y\right)dy,\;\forall\varphi\in L^{2}\left(0,a\right),\forall x\in\left(0,a\right).\label{eq:mer-1}
\end{equation}

\begin{lem}
\label{lem:mer-1}Under the assumptions stated above, the Mercer operator
$T_{F}$ is trace class in $L^{2}\left(0,a\right)$; and if $F\left(0\right)=1$,
then 
\begin{equation}
trace\left(T_{F}\right)=a.\label{eq:mer-2}
\end{equation}
\end{lem}
\begin{proof}
This is an application of Mercer's theorem \cite{LP89,FR42,FM13}
to the integral operator $T_{F}$ in (\ref{eq:mer-1}). But we must
check that $F$, on $\left(-a,a\right)$, extends uniquely by limit
to a continuous p.d. function $F_{ex}$ on $\left[-a,a\right]$, the
closed interval. This is true, and easy to verify; also see Lemma
\ref{lem:F-bd}.\end{proof}
\begin{cor}
\label{cor:mer1}Let $F$ and $\left(-a,a\right)$ be as in \lemref{mer-1}.
Then there is a sequence $\left(\lambda_{n}\right)_{n\in\mathbb{N}}$,
$\lambda_{n}>0$, s.t. $\sum_{n\in\mathbb{N}}\lambda_{n}=a$, and
a system of orthogonal functions\index{orthogonal} $\left\{ \xi_{n}\right\} \subset L^{2}\left(0,a\right)\cap\mathscr{H}_{F}$
such that
\begin{equation}
F\left(x-y\right)=\sum_{n\in\mathbb{N}}\lambda_{n}\xi_{n}\left(x\right)\overline{\xi_{n}\left(y\right)},\mbox{ and}\label{eq:mer-3}
\end{equation}
\begin{equation}
\int_{0}^{a}\overline{\xi_{n}\left(x\right)}\xi_{m}\left(x\right)dx=\delta_{n,m},\; n,m\in\mathbb{N}.\label{eq:mer-4}
\end{equation}
\end{cor}
\begin{proof}
An application of Mercer's theorem \cite{LP89,FR42,FM13}.

\end{proof}
\begin{cor}
Let $a\in\mathbb{R}_{+}$, and let $F:\left(-a,a\right)\rightarrow\mathbb{C}$
be positive definite and continuous, $F\left(0\right)=1$. Let $D^{\left(F\right)}$
be the skew-Hermitian\index{operator!skew-Hermitian} operator in
$ $$\mathscr{H}_{F}$, i.e., $D^{\left(F\right)}\left(F_{\varphi}\right):=F_{\varphi'}$,
$\forall\varphi\in C_{c}^{\infty}\left(0,a\right)$. Let $z\in\mathbb{C}\backslash\left\{ 0\right\} $,
and $DEF_{F}\left(z\right)\subset\mathscr{H}_{F}$ be the corresponding
defect space. 

Let $\left\{ \xi_{n}\right\} _{n\in\mathbb{N}}\subset\mathscr{H}_{F}\cap L^{2}\left(\Omega\right)$,
$\Omega:=\left(0,a\right)$; and $\left\{ \lambda_{n}\right\} _{n\in\mathbb{N}}$
s.t. $\lambda_{n}>0$, and $\sum_{n=1}^{\infty}\lambda_{n}=a$, be
one Mercer--system as in Corollary \ref{cor:mer1}. 

Then $DEF_{F}\left(z\right)\neq0$ if and only if 
\begin{equation}
\sum_{n=1}^{\infty}\frac{1}{\lambda_{n}}\left|\int_{0}^{a}\overline{\xi_{n}\left(x\right)}e^{zx}dx\right|^{2}<\infty.\label{eq:mer-1-1}
\end{equation}
\end{cor}
\begin{proof}
We saw in Corollary \ref{cor:mer1} that if $\left\{ \xi_{n}\right\} _{n\in\mathbb{N}}$,
$\left\{ \lambda_{n}\right\} _{n\in\mathbb{N}}$, is a Mercer system
(i.e., the spectral data for the Mercer operator $T_{F}$ in $L^{2}\left(0,a\right)$),
then $\xi_{n}\in\mathscr{H}_{F}\cap L^{2}\left(\Omega\right)$; and
$\left(\sqrt{\lambda_{n}}\xi_{n}\left(\cdot\right)\right)_{n\in\mathbb{N}}$
is an ONB in $\mathscr{H}_{F}$.

But (\ref{eq:mer-1-1}) in the corollary is merely stating that the
function $e_{z}\left(x\right):=e^{zx}$ has a finite $l^{2}$-expansion
relative to this ONB. The rest is an immediate application of Parseval's
identity (w.r.t. this ONB.)\end{proof}
\begin{rem}
The conclusion in the corollary applies more generally: It shows that
a continuous function $f$ on $\left[0,a\right]$ is in $\mathscr{H}_{F}$
iff 
\[
\sum_{n=1}^{\infty}\frac{1}{\lambda_{n}}\left|\int_{0}^{a}\overline{\xi_{n}\left(x\right)}f\left(x\right)dx\right|^{2}<\infty.
\]

\end{rem}
In Theorem \ref{thm:mer1} below, we revisit the example $F_{2}$
(in $\left|x\right|<\frac{1}{2}$) from Table \ref{tab:F1-F6}. This
example has a number of intriguing properties that allow us to compute
the eigenvalues and the eigenvectors for the Mercer operator from
Corollary \ref{cor:mer1}, (so $a=\frac{1}{2}$.).
\begin{thm}
\label{thm:mer1}Set $E\left(x,y\right)=x\wedge y=\min\left(x,y\right)$,
$x,y\in\left(0,\frac{1}{2}\right)$, and 
\begin{equation}
\left(T_{E}\varphi\right)\left(x\right)=\int_{0}^{\frac{1}{2}}\varphi\left(y\right)x\wedge y\: dy\label{eq:thm:mer1-1}
\end{equation}
then the spectral resolution of $T_{E}$ in $L^{2}\left(0,\frac{1}{2}\right)$
(Mercer operator) is as follows:
\begin{equation}
E\left(x,y\right)=\sum_{n=1}^{\infty}\frac{4}{\left(\pi\left(2n-1\right)\right)^{2}}\sin\left(\left(2n-1\right)\pi x\right)\sin\left(\left(2n-1\right)\pi y\right)\label{eq:thm:mer1-2}
\end{equation}
for all $\forall\left(x,y\right)\in\left(0,\frac{1}{2}\right)\times\left(0,\frac{1}{2}\right)$.

Setting
\begin{equation}
u\left(x\right)=\left(T_{E}\varphi\right)\left(x\right),\;\mbox{for }\varphi\in C_{c}^{\infty}\left(0,\tfrac{1}{2}\right),\;\mbox{we get}\label{eq:thm:mer1-3}
\end{equation}
\begin{eqnarray}
u'\left(x\right) & = & \int_{x}^{\frac{1}{2}}\varphi\left(y\right)dy,\mbox{ and}\label{eq:thm:mer1-4}\\
u''\left(x\right) & = & -\varphi\left(x\right);\label{eq:mer-5-1}
\end{eqnarray}
moreover, $u$ satisfies the boundary condition
\begin{equation}
\begin{cases}
u\left(0\right) & =0\\
u'\left(\tfrac{1}{2}\right) & =0
\end{cases}\label{eq:mer-6-1}
\end{equation}

\end{thm}
Note that $E\left(x,y\right)$ is a p.d. kernel, but not a positive
definite\index{positive definite} function in the sense of Definition
\ref{def:pdf}.

In particular, $E$ is a function of two variables, as opposed to
one. The purpose of Theorem \ref{thm:MerF2-1}, and Lemmas \ref{lem:F2kernel},
and \ref{lem:F2kernel2} is to show that the Mercer operator $T_{F}$
defined from $F=F_{2}$ (see table \ref{tab:F1-F6}) is a rank-1 perturbation\index{perturbation}
of the related operator $T_{E}$ defined from $E\left(x,y\right)$.
The latter is of significance in at least two ways: $T_{E}$ has an
explicit spectral representation, and $E\left(x,y\right)$ is the
covariance\index{covariance} kernel for Brownian motion;\index{Brownian motion}
see also Figure \ref{fig:bm1}. By contrast, we show in Lemma \ref{lem:bbridge}
that the Mercer operator $T_{F}$ is associated with pinned Brownian
motion. 

The connection between the two operators $T_{F}$ and $T_{E}$ is
reflecting a more general feature of boundary conditions for domains
$\Omega$ in Lie groups; -- a topic we consider in \chapref{types}
and \chapref{spbd} below.

Rank-one perturbations play a role in spectral theory in different
problems; see e.g., \cite{Yos12,DJ10,Ion01,DRSS94,TW86}.
\begin{proof}[Proof of Theorem \ref{thm:mer1}]
We verify directly that (\ref{eq:thm:mer1-4})-(\ref{eq:mer-5-1})
hold. 

Consider the Hilbert space $L^{2}\left(0,\frac{1}{2}\right)$. The
operator $\triangle_{E}:=T_{E}^{-1}$ is a selfadjoint extension\index{selfadjoint extension}
of $-\left(\frac{d}{dx}\right)^{2}\Big|_{C_{c}^{\infty}\left(0,\frac{1}{2}\right)}$
in $L^{2}\left(0,\frac{1}{2}\right)$; and under the ONB $f_{n}\left(x\right)=2\sin\left(\left(2n-1\right)\pi x\right)$,
$n\in\mathbb{N}$, and the boundary condition (\ref{eq:mer-6-1}),
$\triangle_{E}$ is diagonalized as 
\begin{equation}
\triangle_{E}f_{n}=\left(\left(2n-1\right)\pi\right)^{2}f_{n},\; n\in\mathbb{N}.\label{eq:thm:mer1-5}
\end{equation}
We conclude that
\begin{align}
\triangle_{E} & =\sum_{n=1}^{\infty}\left(\left(2n-1\right)\pi\right)^{2}\left|f_{n}\left\rangle \right\langle f_{n}\right|\label{eq:thm:mer1-6}\\
T_{E} & =\triangle_{E}^{-1}=\sum_{n=1}^{\infty}\frac{1}{\left(\left(2n-1\right)\pi\right)^{2}}\left|f_{n}\left\rangle \right\langle f_{n}\right|\label{eq:thm:mer1-7}
\end{align}
where $P_{n}:=\left|f_{n}\left\rangle \right\langle f_{n}\right|=$
Dirac's rank-1 projection\index{projection} in $L^{2}\left(0,\frac{1}{2}\right)$,
and 
\begin{eqnarray}
\left(P_{n}\varphi\right)\left(x\right) & = & \left\langle f_{n},\varphi\right\rangle f_{n}\left(x\right)\nonumber \\
 & = & \left(\int_{0}^{\frac{1}{2}}f_{n}\left(y\right)\varphi\left(y\right)dy\right)f_{n}\left(x\right)\nonumber \\
 & = & 4\sin\left(\left(2n-1\right)\pi x\right)\int_{0}^{\frac{1}{2}}\sin\left(\left(2n-1\right)\pi y\right)\varphi\left(y\right)dy.\label{eq:thm:mer1-8}
\end{eqnarray}
Combing (\ref{eq:thm:mer1-7}) and (\ref{eq:thm:mer1-8}), we get
\begin{equation}
\left(T_{E}\varphi\right)\left(x\right)=\sum_{n=1}^{\infty}\frac{4\sin\left(\left(2n-1\right)\pi x\right)}{\left(\left(2n-1\right)\pi\right)^{2}}\int_{0}^{\frac{1}{2}}\sin\left(\left(2n-1\right)\pi y\right)\varphi\left(y\right)dy.\label{eq:thm:mer1-9}
\end{equation}

Note the normalization considered in (\ref{eq:thm:mer1-2}) and (\ref{eq:thm:mer1-9})
is consistent with the condition:
\[
\sum_{n=1}^{\infty}\lambda_{n}=Trace\left(T_{E}\right)
\]
in Corollary \ref{cor:mer1} for the Mercer eigenvalues $\left(\lambda_{n}\right)_{n\in\mathbb{N}}$.
Indeed,
\begin{align*}
Trace\left(T_{E}\right) & =\int_{0}^{\frac{1}{2}}x\wedge x\: dx=\frac{1}{8}\;\mbox{and}\\
\sum_{n=1}^{\infty}\frac{1}{\left(\left(2n-1\right)\pi\right)^{2}} & =\frac{1}{8}.
\end{align*}
 \end{proof}
\begin{thm}
\label{thm:MerF2-1}Set $F\left(x-y\right)=1-\left|x-y\right|$, $x,y\in\left(0,\frac{1}{2}\right)$;
$K^{\left(E\right)}\left(x,y\right)=x\wedge y=\min\left(x,y\right)$,
and let 
\begin{equation}
\left(T_{E}\varphi\right)\left(x\right)=\int_{0}^{\frac{1}{2}}\varphi\left(y\right)K^{\left(E\right)}\left(x,y\right)dy;\label{eq:f2-1}
\end{equation}
then 
\begin{equation}
K^{\left(E\right)}\left(x,y\right)=\sum_{n=1}^{\infty}\frac{4\sin\left(\left(2n-1\right)\pi x\right)\sin\left(\left(2n-1\right)\pi y\right)}{\left(\pi\left(2n-1\right)\right)^{2}}\label{eq:f2-2}
\end{equation}
and 
\begin{equation}
F\left(x-y\right)=1-x-y+2\sum_{n=1}^{\infty}\frac{4\sin\left(\left(2n-1\right)\pi x\right)\sin\left(\left(2n-1\right)\pi y\right)}{\left(\pi\left(2n-1\right)\right)^{2}}.\label{eq:f2-3}
\end{equation}
That is, 
\begin{equation}
F\left(x-y\right)=1-x-y+2K^{\left(E\right)}\left(x,y\right).\label{eq:f2-2-1}
\end{equation}
\end{thm}
\begin{rem}
Note the trace normalization $trace\left(T_{F}\right)=\frac{1}{2}$
holds. Indeed, from (\ref{eq:f2-3}), we get 
\begin{eqnarray*}
Trace\left(T_{F}\right) & = & \int_{0}^{\frac{1}{2}}\left(1-2x+2\sum_{n=1}^{\infty}\frac{4\sin^{2}\left(\left(2n-1\right)\pi x\right)}{\left(\pi\left(2n-1\right)\right)^{2}}\right)dx\\
 & = & \frac{1}{2}-\frac{1}{4}+2\sum_{n=1}^{\infty}\frac{1}{\left(\pi\left(2n-1\right)\right)^{2}}\\
 & = & \frac{1}{2}-\frac{1}{4}+2\cdot\frac{1}{8}=\frac{1}{2};
\end{eqnarray*}
where $\frac{1}{2}$ on the RHS is the right endpoint of the interval
$\left[0,\frac{1}{2}\right]$. \end{rem}
\begin{proof}[Proof of Theorem \ref{thm:MerF2-1}]
The theorem follows from lemma \ref{lem:F2kernel} and lemma \ref{lem:F2kernel2}.\end{proof}
\begin{lem}
\label{lem:F2kernel}Consider the two integral kernels:
\begin{eqnarray}
F\left(x-y\right) & = & 1-\left|x-y\right|,\; x,y\in\Omega;\label{eq:f2-4}\\
K^{\left(E\right)}\left(x,y\right) & = & x\wedge y=\min\left(x,y\right),\; x,y\in\Omega.\label{eq:f2-5}
\end{eqnarray}
We take $\Omega=\left(0,\frac{1}{2}\right)$. Then
\begin{equation}
F\left(x-y\right)=2K^{\left(E\right)}\left(x,y\right)+1-x-y;\mbox{ and}\label{eq:f2-6}
\end{equation}
\begin{equation}
\left(F_{x}-2K_{x}^{\left(E\right)}\right)''\left(y\right)=-2\delta\left(0-y\right).\label{eq:f2-7}
\end{equation}
\end{lem}
\begin{proof}
A calculation yields:
\[
x\wedge y=\frac{x+y-1+F\left(x-y\right)}{2},
\]
and therefore, solving for $F\left(x-y\right)$, we get (\ref{eq:f2-6}).

To prove (\ref{eq:f2-7}), we calculate the respective Schwartz derivatives
(in the sense of distributions). Let $H_{x}=$ the Heaviside function
at $x$, with $x$ fixed. Then 
\begin{eqnarray}
\left(K_{x}^{\left(E\right)}\right)' & = & H_{0}-H_{x},\mbox{ and}\label{eq:f2-8}\\
\left(K_{x}^{\left(E\right)}\right)'' & = & \delta_{0}-\delta_{x}\label{eq:f2-9}
\end{eqnarray}
combining (\ref{eq:f2-9}) with $\left(F_{x}\right)''=-2\delta_{x}$,
we get 
\begin{equation}
\left(F_{x}-2K_{x}^{\left(E\right)}\right)''=-2\delta_{x}-2\left(\delta_{0}-\delta_{x}\right)=-2\delta_{0}\label{eq:f2-10}
\end{equation}
which is the desired conclusion (\ref{eq:f2-7}).\end{proof}
\begin{rem}
From (\ref{eq:f2-6}), we get the following formula for three integral
operators
\begin{equation}
T_{F}=2T_{E}+L,\mbox{ where}\label{eq:f2-4-1}
\end{equation}
\begin{equation}
\left(L\varphi\right)\left(x\right)=\int_{0}^{\frac{1}{2}}\varphi\left(y\right)\left(1-x-y\right)dy.\label{eq:f2-4-2}
\end{equation}
Now in Lemma \ref{lem:F2kernel}, we diagonalize $T_{E}$, but the
two Hermitian operators on the RHS in (\ref{eq:f2-4-1}), do \uline{not}
commute. But the perturbation\index{perturbation} $L$ in (\ref{eq:f2-4-1})
is still relatively harmless; it is a rank-1 operator with just one
eigenfunction: $\varphi\left(x\right)=a+bx$, where $a$ and $b$
are determined from $\left(L\varphi\right)\left(x\right)=\lambda\varphi\left(x\right)$;
and 
\begin{eqnarray*}
\left(L\varphi\right)\left(x\right) & = & \left(1-x\right)\frac{a}{2}-\frac{1}{8}\left(a+\frac{b}{3}\right)\\
 & = & \left(\frac{3}{8}a-\frac{b}{24}\right)-\left(\frac{a}{2}\right)x=\lambda\left(a+bx\right)
\end{eqnarray*}
thus the system of equations 
\[
\begin{cases}
\left(\frac{3}{8}-\lambda\right)a-\frac{1}{24}b=0\\
-\frac{1}{2}a-\lambda b=0.
\end{cases}
\]
It follows that 
\begin{eqnarray*}
\lambda & = & \frac{1}{48}\left(9+\sqrt{129}\right)\\
b & = & -\frac{1}{2}\left(\sqrt{129}-9\right)a.
\end{eqnarray*}

\end{rem}

\begin{rem}[A dichotomy for integral kernel operators]
 Note the following dichotomy for the two integral kernel-operators,
one with the kernel $L\left(x,y\right)=1-x-y$, a rank-one operator;
and the other $T_{F}$ corresponding to $F=F_{2}$, i.e., with kernel
$F\left(x\lyxmathsym{\textendash}y\right)=1-\left|x\lyxmathsym{\textendash}y\right|$.
And by contrast, $T_{F}$ is an infinite dimensional integral-kernel
operator. Denoting both the kernel $L$, and the rank-one operator,
by the same symbol, we then establish the following link between the
two integral operators: The two operators $T_{F}$ and $L$ satisfy
the identity $T_{F}=L+2T_{E}$, where $T_{E}$ is the integral kernel-operator
defined from the covariance\index{covariance} function of Brownian
motion. For more applications of rank-one perturbations, see e.g.,
\cite{Yos12,DJ10,Ion01,DRSS94,TW86}.\end{rem}
\begin{lem}
\label{lem:F2kernel2}Set 
\begin{equation}
\left(T_{E}\varphi\right)\left(x\right)=\int_{0}^{\frac{1}{2}}K^{\left(E\right)}\left(x,y\right)\varphi\left(y\right)dy,\;\varphi\in L^{2}\left(0,\tfrac{1}{2}\right),x\in\left(0,\tfrac{1}{2}\right)\label{eq:f2-11}
\end{equation}
Then $s_{n}\left(x\right):=\sin\left(\left(2n-1\right)\pi x\right)$
satisfies 
\begin{equation}
T_{E}s_{n}=\frac{1}{\left(\left(2n-1\right)\pi\right)^{2}}s_{n};\label{eq:f2-12}
\end{equation}
and we have
\begin{equation}
K^{\left(E\right)}\left(x,y\right)=\sum_{n\in\mathbb{N}}\frac{4}{\left(\left(2n-1\right)\pi\right)^{2}}\sin\left(\left(2n-1\right)\pi x\right)\sin\left(\left(2n-1\right)\pi y\right)\label{eq:f2-13}
\end{equation}
\end{lem}
\begin{proof}
Let $\Omega=\left(0,\frac{1}{2}\right)$. Setting 
\begin{equation}
s_{n}\left(x\right):=\sin\left(\left(2n-1\right)\pi x\right);\; x\in\Omega,n\in\mathbb{N}.\label{eq:f2-14}
\end{equation}
Using (\ref{eq:f2-9}), we get 
\[
\left(T_{E}s_{n}\right)\left(x\right)=\frac{1}{\left(\left(2n-1\right)\pi\right)^{2}}s_{n}\left(x\right),\; x\in\Omega,n\in\mathbb{N},
\]
where $T_{E}$ is the integral operator with kernel $K^{\left(E\right)}$,
and $s_{n}$ is as in (\ref{eq:f2-14}). Since 
\[
\int_{0}^{\frac{1}{2}}\sin^{2}\left(\left(2n-1\right)\pi x\right)dx=\frac{1}{4}
\]
the desired formula (\ref{eq:f2-13}) holds.\end{proof}
\begin{cor}
Let $T_{F}$ and $T_{E}$ be the integral operators in $L^{2}\left(0,\frac{1}{2}\right)$
defined in the lemmas; i.e., $T_{F}$ with kernel $F\left(x-y\right)$;
and $T_{E}$ with kernel $x\wedge y$. Then the selfadjoint operator
$\left(T_{F}-2T_{E}\right)^{-1}$ is well-defined, and it is the Friedrichs
extension of $-\frac{1}{2}\left(\frac{d}{dx}\right)^{-1}\Big|_{C_{c}^{\infty}\left(0,\frac{1}{2}\right)}$
as a Hermitian and semibounded operator in $L^{2}\left(0,\frac{1}{2}\right)$. 
\end{cor}
\index{operator!selfadjoint}\index{Friedrichs extension}\index{operator!semibounded}
\begin{proof}
Formula (\ref{eq:f2-7}) in the lemma translates into 
\begin{equation}
\left(T_{F}\varphi-2T_{E}\varphi\right)''=-2\varphi\label{eq:f2-15}
\end{equation}
for all $\varphi\in C_{c}^{\infty}\left(0,\frac{1}{2}\right)$. Hence
$\left(T_{F}-2T_{E}\right)^{-1}$ is well-defined as an unbounded
selfadjoint operator in $L^{2}\left(0,\frac{1}{2}\right)$; and 
\[
\left(T_{F}-2T_{E}\right)^{-1}\varphi=-\frac{1}{2}\varphi'',\;\forall\varphi\in C_{c}^{\infty}\left(0,\tfrac{1}{2}\right).
\]
Since the Friedrichs extension is given by Dirichlet boundary condition
in $L^{2}\left(0,\frac{1}{2}\right)$, the result follows. \end{proof}
\begin{thm}
\label{thm:shannon}Let $F:\left(-1,1\right)\rightarrow\mathbb{C}$
be a fixed continuous and positive definite function, and let 
\begin{equation}
T_{F}:L^{2}\left(0,1\right)\rightarrow\mathscr{H}_{F}\label{eq:m-4-1}
\end{equation}
be the corresponding Mercer operator, where
\begin{itemize}
\item $L^{2}\left(0,1\right)=$ the $L^{2}$-space of $\left[0,1\right]$
w.r.t. Lebesgue measure restricted to $\left[0,1\right]$; and 
\item $\mathscr{H}_{F}$ is the RKHS from $F$.
\end{itemize}

Let $\mu\in Ext\left(F\right)$; then the range of $T_{F}$ as a subspace
of $\mathscr{H}_{F}$ admits the following representation:
\begin{equation}
Ran\left(T_{F}\right)=\left\{ \sum_{n\in\mathbb{Z}}c_{n}f_{n}\left(x\right)\right\} \quad\mbox{where}\label{eq:m-4-2}
\end{equation}

\begin{itemize}
\item $\left(c_{n}\right)\in l^{2}\left(\mathbb{Z}\right)$, i.e., $\sum_{n\in\mathbb{Z}}\left|c_{n}\right|^{2}<\infty$;
and
\item $\left\{ f_{n}\:|\: n\in\mathbb{Z}\right\} \subset\mathscr{H}_{F}$
is given by 
\begin{equation}
f_{n}\left(x\right)=\int_{\mathbb{R}}e^{i\lambda x}\frac{\sin\pi\left(\lambda-n\right)}{\pi\left(\lambda-n\right)}e^{-i\pi\left(\lambda-n\right)/2}d\mu\left(\lambda\right)\label{eq:m-4-3}
\end{equation}
for all $n\in\mathbb{Z}$, and all $x\in\left[0,1\right]$.
\end{itemize}

Setting 
\begin{equation}
Sha\left(\xi\right)=e^{i\xi/2}\frac{\sin\xi}{\xi},\mbox{\,\  then}\label{eq:m-5-1}
\end{equation}
\begin{equation}
f_{n}\left(x\right)=\int_{\mathbb{R}}Sha\left(\pi\left(\lambda-n\right)\right)e^{i\lambda x}d\mu\left(x\right).\label{eq:m-5-2}
\end{equation}

\end{thm}
\begin{proof}
We first show $\subseteq$ in (\ref{eq:m-4-2}). If $\varphi\in L^{2}\left(0,1\right)$,
then we may apply Shannon's sampling theorem \cite{KT09} to \index{Shannon sampling}
\begin{equation}
\widehat{\varphi}\left(\lambda\right)=\int_{0}^{1}e^{-i\lambda x}\varphi\left(x\right)dx.\label{eq:m-4-4}
\end{equation}
Hence (Shannon sampling), 
\begin{eqnarray}
\widehat{\varphi}\left(\lambda\right) & = & \sum_{n\in\mathbb{Z}}\widehat{\varphi}\left(n\right)\frac{\sin\pi\left(\lambda-n\right)}{\pi\left(\lambda-n\right)}e^{-i\pi\left(\lambda-n\right)/2}\nonumber \\
 & = & \sum_{n\in\mathbb{Z}}\widehat{\varphi}\left(n\right)Sha\left(\pi\left(\lambda-n\right)\right)e^{-i\pi\left(\lambda-n\right)/2}\quad\mbox{where}\label{eq:m-4-5}
\end{eqnarray}
\[
\int_{0}^{1}\left|\varphi\left(x\right)\right|^{2}dx=\sum_{n\in\mathbb{Z}}\left|\widehat{\varphi}\left(n\right)\right|^{2}
\]
where the kernel on the RKS in (\ref{eq:m-4-5}) is the Shannon integral
kernel.

But we already proved that 
\begin{equation}
\left(T_{F}\varphi\right)\left(x\right)=\int_{\mathbb{R}}e^{i\lambda x}\widehat{\varphi}\left(\lambda\right)d\mu\left(\lambda\right)\label{eq:m-4-6}
\end{equation}
holds for all $\varphi\in L^{2}\left(0,1\right)$, $x\in\left[0,1\right]$.
Now substituting (\ref{eq:m-4-5}) into the RHS of (\ref{eq:m-4-6}),
we arrive at the desired conclusion (\ref{eq:m-4-2}). The interchange
of integration and summation is justified by (\ref{eq:m-4-6}) and
Fubini.

It remains to prove $\supseteq$ in (\ref{eq:m-4-2}). Let $\left(c_{n}\right)\in l^{2}\left(\mathbb{Z}\right)$
be given. By Parseval, it follows that 
\[
\varphi\left(x\right)=\sum_{n\in\mathbb{Z}}c_{n}e^{i2\pi nx}\in L^{2}\left(0,1\right),
\]
and that (\ref{eq:m-4-6}) holds. Now the same argument (as above)
with Fubini proves that 
\[
\left(T_{F}\varphi\right)\left(x\right)=\sum_{n\in\mathbb{Z}}c_{n}f_{n}\left(x\right),\; x\in\left[0,1\right].
\]
\end{proof}
\begin{example}
Let $F=F_{3}$ from Table \ref{tab:F1-F6}, i.e., $F\left(x\right):=e^{-\left|x\right|}$,
$x\in\left(-1,1\right)$; then the generating function system $\left\{ f_{n}\right\} _{n\in\mathbb{Z}}$
in $\mathscr{H}_{F}$ (from Shannon sampling) is as follows:

\begin{eqnarray*}
\Re\left\{ f_{n}\right\} \left(x\right) & = & \frac{e^{x-1}+e^{-x}-2\cos\left(2\pi nx\right)}{1+\left(2\pi n\right)^{2}},\;\mbox{and}\\
\Im\left\{ f_{n}\right\} \left(x\right) & = & \frac{\left(e^{x-1}-e^{-x}\right)2\pi n-2\sin\left(2\pi nx\right)}{1+\left(2\pi n\right)^{2}},\;\forall n\in\mathbb{Z},x\in\left[0,1\right].
\end{eqnarray*}
\end{example}
\begin{rem}
While this system is explicit, it has a drawback compared to the eigenfunctions
for the Mercer operator, in that $\left\{ f_{n}\right\} _{n\in\mathbb{Z}}$
is \uline{not} orthogonal. The boundary values are as in Table
\ref{tab:shan}. 
\end{rem}
\renewcommand{\arraystretch}{2}

\begin{table}
\begin{tabular}{|c|c|c|}
\hline 
 & $\Re\left\{ f_{n}\right\} \left(x\right)$ & $\Im\left\{ f\right\} \left(x\right)$\tabularnewline
\hline 
$x=0$ & $\dfrac{e^{-1}+1-2}{1+\left(2\pi n\right)^{2}}$ & $-\dfrac{2\pi n\left(1-e^{-1}\right)}{1+\left(2\pi n\right)^{2}}$\tabularnewline
\hline 
$x=1$ & $\dfrac{e^{-1}+1-2}{1+\left(2\pi n\right)^{2}}$ & $\dfrac{2\pi n\left(1-e^{-1}\right)}{1+\left(2\pi n\right)^{2}}$\tabularnewline
\hline 
\end{tabular}

\bigskip{}
\protect\caption{\label{tab:shan}Boundary values of the Shannon functions, s.t. $\Re\left\{ f_{n}\right\} \left(1\right)=\Re\left\{ f_{n}\right\} \left(0\right)$,
and $\Im\left\{ f_{n}\right\} \left(1\right)=-\Im\left\{ f_{n}\right\} \left(0\right)$. }

\end{table}

\renewcommand{\arraystretch}{1}
\begin{cor}
\label{cor:shan}Let $F:\left(-1,1\right)\rightarrow\mathbb{C}$ be
continuous and p.d., $F\left(0\right)=1$, and let $\mu\in\mathscr{M}_{+}\left(\mathbb{R}\right)$;
then the following two conditions are equivalent:
\begin{enumerate}
\item \label{enu:m-6-1}$\mu\in Ext\left(F\right)$
\item \label{enu:m-6-2}For all $x\in\left(-1,1\right)$, 
\begin{equation}
\sum_{n\in\mathbb{Z}}\int_{\mathbb{R}}e^{i\lambda x}Sha\left(\pi\left(\lambda-n\right)\right)d\mu\left(\lambda\right)=F\left(x\right).\label{eq:m-6-1}
\end{equation}

\end{enumerate}
\end{cor}
\begin{proof}
In the assertion (\ref{enu:m-6-2}) on the LHS, we use ``$Sha$''
to denote Shannon's sampling kernel; see (\ref{eq:m-5-1})-(\ref{eq:m-5-2}).

The implication (\ref{enu:m-6-2})$\Longrightarrow$(\ref{enu:m-6-1})
follows since $\mbox{LHS}_{\left(\ref{eq:m-6-1}\right)}$ is a continuous
p.d. function defined on all of $\mathbb{R}$; and so it is an extension
as required. It remains to show that $\mbox{LHS}_{\left(\ref{eq:m-6-1}\right)}=\widehat{d\mu}$
(Bochner transform), but his follows from the proof of Theorem \ref{thm:shannon}.

The argument for (\ref{enu:m-6-1})$\Longrightarrow$(\ref{enu:m-6-2})
is as follows: Let $x\in\left(-1,1\right)$; then 
\begin{eqnarray*}
F\left(x\right) & = & \sum_{n\in\mathbb{Z}}\int_{\mathbb{R}}e^{i\lambda x}Sha\left(\pi\left(\lambda-n\right)\right)d\mu\left(\lambda\right)\quad\mbox{(Parseval on \ensuremath{L^{2}})}\\
 & = & \widehat{d\mu}\left(x\right);
\end{eqnarray*}
which is (\ref{enu:m-6-2}).
\end{proof}
In the remaining of this section, we turn to domains in Lie groups.
\begin{rem}
\label{rem:mer1}A trace class Mercer operator may be defined for
any continuous p.d. function $F:\Omega^{-1}\cdot\Omega\rightarrow\mathbb{C}$
where $\Omega\subset G$ is a subset in a Lie group satisfying the
following:
\begin{enumerate}[label=B\arabic{enumi}.,ref=B\arabic{enumi}]
\item \label{enu:mer1} $\Omega\neq\phi$,
\item \label{enu:mer2}$\Omega$ is open and connected,
\item \label{enu:mer3}the closure $\overline{\Omega}$ is compact,
\item \label{enu:mer4}the boundary of $\overline{\Omega}$ has Haar measure
zero.
\end{enumerate}
\begin{flushleft}
We then set 
\begin{equation}
\left(T_{F}\varphi\right)\left(x\right)=\int_{\Omega}\varphi\left(y\right)F\left(y^{-1}x\right)dy\label{eq:mer0}
\end{equation}
where $dy$ is the restriction of Haar measure on $G$ to the subset
$\Omega$, or equivalently to $\overline{\Omega}$.
\par\end{flushleft}
\end{rem}
\index{measure!Haar}

Note that with Lemma \ref{lem:F-bd} and assumption \ref{enu:mer3}
in Remark \ref{rem:mer1}, we conclude that 
\begin{equation}
\mathscr{H}_{F}\subseteq C\left(\overline{\Omega}\right)\subset L^{2}\left(\Omega\right)\label{eq:m-3-1}
\end{equation}
so it is natural to ask: ``What is the orthogonal complement of $\mathscr{H}_{F}$
in the larger Hilbert space $L^{2}\left(\Omega\right)$?'' The answer
is as follows:
\begin{cor}
We have 
\begin{equation}
L^{2}\left(\Omega\right)\ominus\mathscr{H}_{F}=Ker\left(T_{F}\right)\label{eq:m-3-2}
\end{equation}
where $T_{F}$ is the Mercer operator. As an operator in $L^{2}\left(\Omega\right)$,
$T_{F}$ takes the form:

\[ 
T_{F} = \begin{blockarray}{ccc}
\mathscr{H}_{F} & Ker(T_{F}) &   \\
\begin{block}{(cc)c}
  T_{F} & 0 & \mathscr{H}_{F} \\
  0 & 0 &  Ker(T_{F}) \\ 
\end{block} 
\end{blockarray} 
\]\end{cor}
\begin{proof}
We shall first need the following:
\begin{lem}
There is a finite constant $C_{1}$ such that 
\begin{equation}
\left\Vert \xi\right\Vert _{L^{2}\left(\Omega\right)}\leq C_{1}\left\Vert \xi\right\Vert _{\mathscr{H}_{F}}\label{eq:m-3-3}
\end{equation}
for all $\xi\in\mathscr{H}_{F}$. \end{lem}
\begin{proof}
Since $\mathscr{H}_{F}\subset L^{2}\left(\Omega\right)$ by (\ref{eq:m-3-1}),
the inclusion mapping, $\mathscr{H}_{F}\rightarrow L^{2}\left(\Omega\right)$,
is closed and therefore bounded; by the Closed-Graph Theorem; and
so the estimate (\ref{eq:m-3-3}) follows. (See also Theorem \ref{thm:mer4}
below for an explicit bound $C_{1}$.)
\end{proof}

\emph{Proof of the Corollary continued:}

Let $\xi\in\mathscr{H}_{F}$, and pick $\varphi_{n}\in C_{c}\left(\Omega\right)$
such that 
\begin{equation}
\left\Vert \xi-F_{\varphi_{n}}\right\Vert _{\mathscr{H}_{F}}\rightarrow0,\; n\rightarrow\infty.\label{eq:m-3-4}
\end{equation}
Note $F_{\varphi_{n}}=T_{F}\varphi_{n}$; and we get 
\[
\left\Vert \xi-T_{F}\varphi_{n}\right\Vert _{L^{2}\left(\Omega\right)}\underset{\left(\ref{eq:m-3-3}\right)}{\leq}C_{1}\left\Vert \xi-F_{\varphi_{n}}\right\Vert _{\mathscr{H}_{F}}\underset{\left(\ref{eq:m-3-4}\right)}{\longrightarrow}0,\; n\rightarrow\infty.
\]
Therefore, if $f\in L^{2}\left(\Omega\right)$ is given, we conclude
that the following properties are equivalent:
\begin{eqnarray*}
\left\langle f,\xi\right\rangle _{L^{2}\left(\Omega\right)} & = & 0,\;\forall\xi\in\mathscr{H}_{F}\\
 & \Updownarrow\\
\left\langle f,F_{\varphi}\right\rangle _{L^{2}\left(\Omega\right)} & = & 0,\;\forall\varphi\in C_{c}\left(\Omega\right)\\
 & \Updownarrow\\
\left\langle f,T_{F}\varphi\right\rangle _{L^{2}\left(\Omega\right)} & = & 0,\;\forall\varphi\in C_{c}\left(\Omega\right)\\
 & \Updownarrow & \text{(since \ensuremath{T_{F}}is selfadjoint as an operator in \ensuremath{L^{2}\left(\Omega\right)})}\\
\left\langle T_{F}f,\varphi\right\rangle _{L^{2}\left(\Omega\right)} & = & 0,\;\forall\varphi\in C_{c}\left(\Omega\right)\\
 & \Updownarrow & \text{(since \ensuremath{C_{c}}\ensuremath{\left(\Omega\right)} is dense in \ensuremath{L^{2}\left(\Omega\right)})}\\
T_{F}f & = & 0
\end{eqnarray*}
\end{proof}
\begin{rem}
Note that $Ker\left(T_{F}\right)$ may be infinite-dimensional. This
happens for example, in the cases studied in  \secref{expT}, and
in the case of $F_{6}$ in Table \ref{tab:F1-F6} (\secref{Polya}).
On the other hand, these examples are rather degenerate since they
have $\dim\mathscr{H}_{F}<\infty$.

\textbf{Convention.} When computing $T_{F}\left(f\right)$ for $f\in L^{2}\left(\Omega\right)$,
we may use (\ref{eq:m-3-1}) to write $f$ as $f=f^{\left(F\right)}+f^{\left(K\right)}$,
where $f^{\left(F\right)}\in\mathscr{H}_{F}$, and $f^{\left(K\right)}\in\mbox{Ker}\left(T_{F}\right)$,
and $\left\langle f^{\left(F\right)},f^{\left(K\right)}\right\rangle _{L^{2}\left(\Omega\right)}=0$.
As a result,
\[
T_{F}\left(f\right)=T_{F}f^{\left(F\right)}+T_{F}f^{\left(K\right)}=T_{F}f^{\left(F\right)},
\]
and $\mbox{Ker}\left(T_{F}\right)$ is not a problem: For example,
when we write $T_{F}^{-1}\left(f\right)$, we mean $T_{F}^{-1}f^{\left(F\right)}$.\end{rem}
\begin{cor}
Let $G$, $\Omega$, $F$ and $\mathscr{H}_{F}$ be as above, assume
that $\Omega\subset G$ satisfies \ref{enu:mer1}-\ref{enu:mer4}
in Remark \ref{rem:mer1}. Suppose in addition that $G=\mathbb{R}^{k}$,
and that there exists $\mu\in Ext\left(F\right)$ such that the support
$suppt\left(\mu\right)$ contains a non-empty open subset in $\mathbb{R}^{k}$;
then $Ker\left(T_{F}\right)=0$.\end{cor}
\begin{proof}
We saw in Corollary \ref{cor:lcg-isom} that there is an isometry
$\mathscr{H}_{F}\rightarrow L^{2}\left(\mathbb{R}^{k},\mu\right)$
given on the dense subset $\left\{ F_{\varphi}\:\big|\:\varphi\in C_{c}\left(\Omega\right)\right\} $,
by $F_{\varphi}\mapsto\widehat{\varphi}$, s.t. $\left\Vert F_{\varphi}\right\Vert _{\mathscr{H}_{F}}=\left\Vert \widehat{\varphi}\right\Vert _{L^{2}\left(\mathbb{R}^{k},\mu\right)}$;
where $\widehat{\varphi}$ denotes Fourier transform in $\mathbb{R}^{k}$.
Since $F_{\varphi}=T_{F}\left(\varphi\right)$, for all $\varphi\in C_{c}\left(\Omega\right)$,
and $C_{c}\left(\Omega\right)$ is dense in $L^{2}\left(\Omega\right)$,
we get 
\begin{equation}
\bigl\Vert T_{F}f\bigr\Vert_{\mathscr{H}_{F}}=\bigl\Vert\widehat{f}\bigr\Vert_{L^{2}\left(\mathbb{R}^{k},\mu\right)}\label{eq:m-3-6}
\end{equation}
holds also for $f\in L^{2}\left(\Omega\right)$. So $f\in Ker\left(T_{F}\right)\Rightarrow\widehat{f}\equiv0$
on $suppt\left(\mu\right)$, by (\ref{eq:m-3-6}). But since $\overline{\Omega}$
is compact, we conclude by Paley-Wiener, that $\widehat{f}$ is entire
analytic. Since $suppt\left(\mu\right)$ consists a non-empty open
set, we conclude that $\widehat{f}\equiv0$; and by (\ref{eq:m-3-6}),
that therefore $f=0$. 
\end{proof}
\index{Paley-Wiener}\index{measure!Haar}
\begin{thm}
\label{thm:mer2}Let $G$, $\Omega$, $F:\Omega^{-1}\Omega\rightarrow\mathbb{C}$,
and $\mathscr{H}_{F}$ be as in Remark \ref{rem:mer1}, i.e., we assume
that \ref{enu:mer1}-\ref{enu:mer4} hold. 
\begin{enumerate}
\item \label{enu:mer-1}Let $T_{F}$ denote the corresponding Mercer operator
$T_{F}:L^{2}\left(\Omega\right)\rightarrow\mathscr{H}_{F}$. Assume
further that $\Omega$ has \uline{finite} Haar measure. Then $T_{F}^{*}$
is also a densely defined operator on $\mathscr{H}_{F}$ as follows:
\begin{equation}
\xymatrix{\mathscr{H}_{F}\ar@/^{1pc}/[rr]^{j=T_{F}^{*}} &  & L^{2}\left(\Omega\right)\ar@/^{1pc}/[ll]^{T_{F}}}
\label{eq:mer18}
\end{equation}

\item \label{enu:mer-2}Moreover, every $\xi\in\mathscr{H}_{F}$ has a realization
$j\left(\xi\right)$ as a bounded uniformly continuous function on
$\Omega$, and 
\begin{equation}
T_{F}^{*}\xi=j\left(\xi\right),\;\forall\xi\in\mathscr{H}_{F}.\label{eq:mer19}
\end{equation}

\item \label{enu:mer-3}Finally, the operator $j$ from (\ref{enu:mer-1})-(\ref{enu:mer-2})
satisfies 
\begin{equation}
\mbox{Ker}\left(j\right)=0.\label{eq:mer20}
\end{equation}

\end{enumerate}
\end{thm}
\begin{proof}
We begin with the claims relating (\ref{eq:mer18}) and (\ref{eq:mer19})
in the theorem.

Using the reproducing property in the RKHS $\mathscr{H}_{F}$, we
get the following two estimates, valid for all $\xi\in\mathscr{H}_{F}$:
\begin{equation}
\left|\xi\left(x\right)-\xi\left(y\right)\right|^{2}\leq2\left\Vert \xi\right\Vert _{\mathscr{H}_{F}}^{2}\left(F\left(e\right)-\Re\left(F\left(x^{-1}y\right)\right)\right),\;\forall x,y\in\Omega;\label{eq:mer21}
\end{equation}
and
\begin{equation}
\left|\xi\left(x\right)\right|\leq\left\Vert \xi\right\Vert _{\mathscr{H}_{F}}\sqrt{F\left(e\right)},\;\forall x\in\Omega.\label{eq:mer22}
\end{equation}

In the remaining of the proof, we shall adopt the normalization $F\left(e\right)=1$;
so the assertion in (\ref{eq:mer22}) states that point-evaluation
for $\xi\in\mathscr{H}_{F}$ is contractive.
\begin{rem*}
In the discussion below, we give conditions which yield boundedness
of the operators in (\ref{eq:mer18}). If bounded, then, by general
theory, we have
\begin{equation}
\left\Vert j\right\Vert _{\mathscr{H}_{F}\rightarrow L^{2}\left(\Omega\right)}=\left\Vert T_{F}\right\Vert _{L^{2}\left(\Omega\right)\rightarrow\mathscr{H}_{F}}\label{eq:mer23}
\end{equation}
for the respective operator-norms. While boundedness holds in ``most''
cases, it does not in general.
\end{rem*}
The assertion in (\ref{eq:mer19}) in Part (\ref{enu:mer-2}) is a
statement about the adjoint of an operator mapping between different
Hilbert spaces; and it is a result of the following:
\begin{equation}
\int_{\Omega}\overline{j\left(\xi\right)\left(x\right)}\varphi\left(x\right)dx=\left\langle \xi,T_{F}\varphi\right\rangle _{\mathscr{H}_{F}}\left(=\left\langle \xi,F_{\varphi}\right\rangle _{\mathscr{H}_{F}}\right)\label{eq:mer24}
\end{equation}
for all $\xi\in\mathscr{H}_{F}$, and $\varphi\in C_{c}\left(\Omega\right)$.
But eq. (\ref{eq:mer24}), in turn, is immediate from (\ref{eq:mer0})
and the reproducing property in $\mathscr{H}_{F}$.

The assertion in Part \ref{enu:mer-3} of the theorem follows from:
\begin{eqnarray*}
\mbox{Ker}\left(j\right) & = & \left(\mbox{Ran}\left(j^{*}\right)\right)^{\perp}\\
 & \underset{\left(\text{by \ensuremath{\left(\mbox{\ref{enu:mer-2}}\right)}}\right)}{\begin{array}[t]{c}
=\end{array}} & \left(\mbox{Ran}\left(T_{F}\right)\right)^{\perp}=0;
\end{eqnarray*}
where we used that $\mbox{Ran}\left(T_{F}\right)$ is dense in $\mathscr{H}_{F}$.\end{proof}
\begin{rem}
The RHS in (\ref{eq:mer19}) is subtle because of boundary values
for the function $\xi:\overline{\Omega}\rightarrow\mathbb{C}$ which
represent some vector (also denoted $\xi$) in $\mathscr{H}_{F}$.
We refer to Lemma \ref{lem:F-bd} for specifics.

We showed that if $\varphi\in C_{c}\left(\Omega\right)$, then (\ref{eq:mer24})
holds; but if some $\varphi\in L^{2}\left(\Omega\right)$ is not in
$C_{c}\left(\Omega\right)$, then we pick up a boundary term when
computing $\left\langle \xi,T_{F}\varphi\right\rangle _{\mathscr{H}_{F}}$.
Specifically, we show in  \subref{lie} (in the case of connected
Lie groups) that there is a measure $\beta$ on the boundary $\partial\Omega=\overline{\Omega}\backslash\Omega$
such that \index{measure!boundary} 
\begin{equation}
\left\langle \xi,T_{F}\varphi\right\rangle _{\mathscr{H}_{F}}=\int_{\Omega}\overline{\xi\left(x\right)}\varphi\left(x\right)dx+\int_{\partial\Omega}\overline{\xi\Big|_{\partial\Omega}\left(\sigma\right)}\left(T_{F}\varphi\right)_{n}\left(\sigma\right)d\beta\left(\sigma\right)\label{eq:m-2-1}
\end{equation}
for all $\varphi\in L^{2}\left(\Omega\right)$. 

In (\ref{eq:m-2-1}), on the RHS, we must make the following restrictions:
\begin{enumerate}
\item $\overline{\Omega}$ is compact (and of course $\Omega$ is assumed
open connected);
\item $\partial\Omega$ is a differentiable manifold of dimension $=\dim\left(G\right)-1$.
\item The function $T_{F}\varphi\in\mathscr{H}_{F}\subset C\left(\Omega\right)$
has a well-defined inward normal vector field; and $\left(T_{F}\varphi\right)_{n}$
denotes the corresponding normal derivative.
\end{enumerate}

It follows from Lemma \ref{lem:F-bd} that the term $b_{\xi}\left(\sigma\right):=\xi\big|_{\partial\Omega}\left(\sigma\right)$
on the RHS in (\ref{eq:m-2-1}) satisfies 
\begin{equation}
\left|b_{\xi}\left(\sigma\right)\right|\leq\left\Vert \xi\right\Vert _{\mathscr{H}_{F}},\;\forall\sigma\in\partial\Omega.\label{eq:m-2-2}
\end{equation}

\end{rem}
\begin{cor}[Renormalization]
\label{cor:mer2} Let $G$, $\Omega$, $F$, $\mathscr{H}_{F}$,
and $T_{F}$ be as in the statement of Theorem \ref{thm:mer2}. Let
$\left\{ \xi_{n}\right\} _{n\in\mathbb{N}}$, $\left\{ \lambda_{n}\right\} _{n\in\mathbb{N}}$
be the spectral data for the Mercer operator $T_{F}$, i.e., $\lambda_{n}>0$,\index{Renormalization}
\begin{equation}
\sum_{n=1}^{\infty}\lambda_{n}=\left|\Omega\right|,\mbox{ and }\left\{ \xi_{n}\right\} _{n\in\mathbb{N}}\subset L^{2}\left(\Omega\right)\cap\mathscr{H}_{F},\mbox{ satisfying }\label{eq:mer25}
\end{equation}
\begin{equation}
T_{F}\xi_{n}=\lambda_{n}\xi_{n},\:\int_{\Omega}\overline{\xi_{n}\left(x\right)}\xi_{m}\left(x\right)dx=\delta_{n,m},\; n,m\in\mathbb{N},\label{eq:mer26}
\end{equation}
and 
\begin{equation}
F\left(x^{-1}y\right)=\sum_{n=1}^{\infty}\lambda_{n}\overline{\xi_{n}\left(x\right)}\xi_{n}\left(y\right),\;\forall x,y\in\Omega;\label{eq:mer27}
\end{equation}
then
\begin{equation}
\left\{ \sqrt{\lambda_{n}}\xi_{n}\left(\cdot\right)\right\} _{n\in\mathbb{N}}\label{eq:mer28}
\end{equation}
\textup{is an ONB in $\mathscr{H}_{F}$.}\end{cor}
\begin{proof}
This is immediate for the theorem. To stress the idea, we include
the proof that $\left\Vert \sqrt{\lambda_{n}}\xi_{n}\right\Vert _{\mathscr{H}_{F}}=1$,
$\forall n\in\mathbb{N}$. 

Clearly, 
\[
\left\Vert \sqrt{\lambda_{n}}\xi_{n}\right\Vert _{\mathscr{H}_{F}}^{2}=\lambda_{n}\left\langle \xi_{n},\xi_{n}\right\rangle _{\mathscr{H}_{F}}=\left\langle \xi_{n},T_{F}\xi_{n}\right\rangle _{\mathscr{H}_{F}}.
\]
But since $\xi_{n}\in L^{2}\left(\Omega\right)\cap\mathscr{H}_{F}$,
\[
\left\langle \xi_{n},T_{F}\xi_{n}\right\rangle _{\mathscr{H}_{F}}=\int_{\Omega}\overline{\xi_{n}\left(x\right)}\xi_{n}\left(x\right)dx=\left\Vert \xi_{n}\right\Vert _{L^{2}\left(\Omega\right)}^{2}=1\;\mbox{by }(\ref{eq:mer26}),
\]
and the result follows.\end{proof}
\begin{thm}
\label{thm:mer4}Let $G$, $\Omega$, $F$, $\mathscr{H}_{F}$, $\left\{ \xi_{n}\right\} $,
and $\left\{ \lambda_{n}\right\} $ be as specified in Corollary \ref{cor:mer2}. 
\begin{enumerate}
\item \label{enu:m-1-1}Then $\mathscr{H}_{F}\subseteq L^{2}\left(\Omega\right)$,
and there is a finite constant $C_{1}$ such that 
\begin{equation}
\int_{\Omega}\left|g\left(x\right)\right|^{2}dx\leq C_{1}\left\Vert g\right\Vert _{\mathscr{H}_{F}}\label{eq:m-1-1}
\end{equation}
holds for all $g\in\mathscr{H}_{F}$. Indeed, $C_{1}=\left\Vert \lambda_{1}\right\Vert _{\infty}$
will do.
\item \label{enu:m-1-2}Let $\left\{ \xi_{k}\right\} $, $\left\{ \lambda_{k}\right\} $
be as above, let $N\in\mathbb{N}$, set 
\begin{equation}
\mathscr{H}_{F}\left(N\right)=span\left\{ \xi_{k}\:\big|\: k=1,2,3,\ldots,N\right\} ;\label{eq:mer-4-1}
\end{equation}
and let $Q_{N}$ be the $\mathscr{H}_{F}$-orthogonal projection onto
$\mathscr{H}_{F}\left(N\right)$; and let $P_{N}$ be the $L^{2}\left(\Omega\right)$-orthogonal
projection onto $span_{1\leq k\leq N}\left\{ \xi_{k}\right\} $; then
\index{projection} 
\begin{equation}
Q_{N}\geq\left(\frac{1}{\lambda_{1}}\right)P_{N}\label{eq:mer-4-2}
\end{equation}
where $\leq$ in (\ref{eq:mer-4-2}) is the order of Hermitian operators,
and where $\lambda_{1}$ is the largest eigenvalue.
\end{enumerate}
\end{thm}
\begin{proof}
Part (\ref{enu:m-1-1}). Pick $\lambda_{n}$, $\xi_{n}$ as in (\ref{eq:mer26})-(\ref{eq:mer28}).
We saw that $\mathscr{H}_{F}\subset L^{2}\left(\Omega\right)$; recall
$\overline{\Omega}$ is compact. Hence
\begin{eqnarray*}
\left\Vert g\right\Vert _{L^{2}\left(\Omega\right)}^{2} & = & \sum_{n=1}^{\infty}\left|\left\langle \xi_{n},g\right\rangle _{2}\right|^{2}\\
 & = & \sum_{n=1}^{\infty}\left|\left\langle T_{F}\xi_{n},g\right\rangle _{\mathscr{H}_{F}}\right|^{2}\\
 & = & \sum_{n=1}^{\infty}\left|\left\langle \lambda_{n}\xi_{n},g\right\rangle _{\mathscr{H}_{F}}\right|^{2}\\
 & = & \sum_{n=1}^{\infty}\lambda_{n}\left|\left\langle \sqrt{\lambda_{n}}\xi_{n},g\right\rangle _{\mathscr{H}_{F}}\right|^{2}\\
 & \leq & \left(\sup_{n\in\mathbb{N}}\left\{ \lambda_{n}\right\} \right)\sum_{n=1}^{\infty}\left|\left\langle \sqrt{\lambda_{n}}\xi_{n},g\right\rangle _{\mathscr{H}_{F}}\right|^{2}\\
 & = & \left(\sup_{n\in\mathbb{N}}\left\{ \lambda_{n}\right\} \right)\left\Vert g\right\Vert _{\mathscr{H}_{F}}^{2},\mbox{ by (\ref{eq:mer28}) and Parseval.}
\end{eqnarray*}
This proves (\ref{enu:m-1-1}) with $C_{1}=\sup_{n\in\mathbb{N}}\left\{ \lambda_{n}\right\} $. 

Part (\ref{enu:m-1-2}). Let $f\in L^{2}\left(\Omega\right)\cap\mathscr{H}_{F}$.
Arrange the eigenvalues $\left(\lambda_{n}\right)$ s.t. 
\begin{equation}
\lambda_{1}\geq\lambda_{2}\geq\lambda_{3}\cdots>0.\label{eq:m-1-3}
\end{equation}
Then 
\begin{eqnarray*}
\left\langle f,Q_{N}f\right\rangle _{\mathscr{H}_{F}}=\left\Vert Q_{N}f\right\Vert _{\mathscr{H}_{F}}^{2} & \overset{\text{(\ref{eq:mer28})}}{=} & \sum_{n=1}^{N}\left|\left\langle \sqrt{\lambda_{n}}\xi_{n},f\right\rangle _{\mathscr{H}_{F}}\right|^{2}\\
 & = & \sum_{n=1}^{N}\frac{1}{\lambda_{n}}\left|\left\langle \lambda_{n}\xi_{n},f\right\rangle _{\mathscr{H}_{F}}\right|^{2}\\
 & = & \sum_{n=1}^{N}\frac{1}{\lambda_{n}}\left|\left\langle T_{F}\xi_{n},f\right\rangle _{\mathscr{H}_{F}}\right|^{2}\\
 & = & \sum_{n=1}^{N}\frac{1}{\lambda_{n}}\left|\left\langle \xi_{n},f\right\rangle _{L^{2}\left(\Omega\right)}\right|^{2}\\
 & \overset{\text{(\ref{eq:m-1-3})}}{\geq} & \frac{1}{\lambda_{1}}\sum_{n=1}^{N}\left|\left\langle \xi_{n},f\right\rangle _{L^{2}\left(\Omega\right)}\right|^{2}\\
 & \overset{\text{(Parseval)}}{=} & \frac{1}{\lambda_{1}}\left\Vert P_{N}f\right\Vert _{L^{2}\left(\Omega\right)}^{2}=\frac{1}{\lambda_{1}}\left\langle f,P_{N}f\right\rangle _{L^{2}\left(\Omega\right)};
\end{eqnarray*}
and so the \emph{a priori} estimate (\ref{enu:m-1-2}) holds.\end{proof}
\begin{rem}
The estimate (\ref{eq:m-1-1}) in Theorem \ref{thm:mer4} is related
to, but different from a classical Poincaré-inequality \cite{Maz11}.
The latter \emph{a priori }is as follows:

\index{Poincaré-inequality}

Let $\Omega\subset\mathbb{R}^{k}$ satisfying in \ref{enu:mer1}-\ref{enu:mer4}
in Remark \ref{rem:mer1}, and let $\left|\Omega\right|_{k}$ denote
the $k$-dimensional Lebesgue measure of $\Omega$, i.e., 
\[
\left|\Omega\right|_{k}=\int_{\mathbb{R}^{k}}\chi_{\Omega}\left(x\right)\underset{dx}{\underbrace{dx_{1}\cdots dx_{k}}};
\]
let $\nabla=\left(\frac{\partial}{\partial x_{1}},\ldots,\frac{\partial}{\partial x_{k}}\right)$
be the gradient, and 
\begin{equation}
\left\Vert \nabla f\right\Vert _{L^{2}\left(\Omega\right)}^{2}:=\sum_{i=1}^{k}\int_{\Omega}\left|\frac{\partial f}{\partial x_{i}}\right|^{2}dx.\label{eq:f2-3-1}
\end{equation}
Finally, let $\lambda_{1}\left(N\right)=$ the finite eigenvalue for
the Neumann problem on $\Omega$ (NBP$\Omega$). Then,
\begin{equation}
\left\Vert f-\frac{1}{\left|\Omega\right|_{k}}\int_{\Omega}fdx\right\Vert _{L^{2}\left(\Omega\right)}^{2}\leq\frac{1}{\lambda_{1}\left(N\right)}\left\Vert \nabla f\right\Vert _{L^{2}\left(\Omega\right)}^{2}\label{eq:f2-3-2}
\end{equation}
holds for all $f\in L^{2}\left(\Omega\right)$ such that $\frac{\partial f}{\partial x_{i}}\in L^{2}\left(\Omega\right)$,
$1\leq i\leq k$.\end{rem}
\begin{defn}
A system of functions $\left\{ f_{n}\right\} _{n\in\mathbb{Z}}\subset\mathscr{H}_{F}$
is said to be a Bessel frame \cite{CM13} if there is a finite constant
$A$ such that \index{Bessel frame} 
\begin{equation}
\sum_{n\in\mathbb{Z}}\left|\left\langle f_{n},\xi\right\rangle _{\mathscr{H}_{F}}\right|^{2}\leq A\left\Vert \xi\right\Vert _{\mathscr{H}_{F}}^{2},\;\forall\xi\in\mathscr{H}_{F}.\label{eq:f2-5-1}
\end{equation}
\end{defn}
\begin{thm}
Let $F:\left(-1,1\right)\rightarrow\mathbb{C}$ be continuous and
p.d., and let $\left\{ f_{n}\right\} _{n\in\mathbb{Z}}\subset\mathscr{H}_{F}$
be the system (\ref{eq:m-4-3}) obtained by Shannon sampling in Fourier-domain;
then $\left\{ f_{n}\right\} $ is a Bessel frame, where we may take
$A=\lambda_{1}$ as frame bound in (\ref{eq:f2-5-1}). ($\lambda_{1}=$
the largest eigenvalue of the Mercer operator $T_{F}$.)\end{thm}
\begin{proof}
Let $\xi\in\mathscr{H}_{F}$, then 
\begin{eqnarray*}
\sum_{n\in\mathbb{Z}}\left|\left\langle f_{n},\xi\right\rangle _{\mathscr{H}_{F}}\right|^{2} & = & \sum_{n\in\mathbb{Z}}\left|\left\langle Te_{n},\xi\right\rangle _{\mathscr{H}_{F}}\right|^{2}\\
 & \underset{\text{Cor. \ensuremath{\left(\ref{cor:mer2}\right)}}}{=} & \sum_{n\in\mathbb{Z}}\left|\left\langle e_{n},\xi\right\rangle _{L^{2}\left(0,1\right)}\right|^{2}\\
 & \underset{\text{(Parseval)}}{=} & \left\Vert \xi\right\Vert _{L^{2}\left(0,1\right)}^{2}\\
 & \underset{\text{(Thm. \ensuremath{\left(\ref{thm:mer4}\right)})}}{\leq} & \lambda_{1}\left\Vert \xi\right\Vert _{\mathscr{H}_{F}}^{2}
\end{eqnarray*}
which is the desired conclusion.

At the start of the estimate above we used the Fourier basis $e_{n}\left(x\right)=e^{i2\pi nx}$,
$n\in\mathbb{Z}$, an ONB in $L^{2}\left(0,1\right)$; and the fact
that $f_{n}=T_{F}\left(e_{n}\right)$, $n\in\mathbb{Z}$; see the
details in the proof of Theorem \ref{thm:shannon}. \end{proof}
\begin{cor}
For every $f\in L^{2}\left(\Omega\right)\cap\mathscr{H}_{F}$, and
every $x\in\overline{\Omega}$ (including boundary points), we have
the following estimate:
\begin{equation}
\left|f\left(x\right)\right|\leq\left\Vert f\right\Vert _{\mathscr{H}_{F}}\label{eq:m-1-4}
\end{equation}
(Note that this is independent of $x$ and of $f$.)\end{cor}
\begin{proof}
The estimate in (\ref{eq:m-1-4}) follows from Lemma \ref{lem:F-bd};
and (\ref{eq:mer-4-2}) in part (\ref{enu:m-1-2}) of Theorem \ref{thm:mer4}. \end{proof}
\begin{cor}
Let $\left\{ \xi_{n}\right\} $, $\left\{ \lambda_{n}\right\} $,
$T_{F}$, $\mathscr{H}_{F}$ and $L^{2}\left(\Omega\right)$ be as
above; and assume $\lambda_{1}\geq\lambda_{2}\geq\cdots$; then we
have the following details for operator norms
\begin{eqnarray}
\left\Vert T_{F}^{-1}\Big|_{span\left\{ \xi_{k}:k=1,\ldots,N\right\} }\right\Vert _{\mathscr{H}_{F}\rightarrow\mathscr{H}_{F}} & = & \left\Vert T_{F}^{-1}\Big|_{span\left\{ \xi_{k}:k=1,\ldots,N\right\} }\right\Vert _{L^{2}\left(\Omega\right)\rightarrow L^{2}\left(\Omega\right)}\label{eq:mer-1-5}\\
 & = & \frac{1}{\lambda_{N}}\longrightarrow\infty,\mbox{ as }N\rightarrow\infty.\nonumber 
\end{eqnarray}

\end{cor}
Since we shall not have occasion to use this more general version
of the Mercer-operators we omit details below, and restrict attention
to the case of finite interval in $\mathbb{R}$.
\begin{rem}
Some of the conclusions in Theorem \ref{thm:mer2} hold even if conditions
\ref{enu:mer1}-\ref{enu:mer4} are relaxed. But condition \ref{enu:mer3}
ensures that $L^{2}\left(\Omega\right)$ has a realization as a subspace
of $\mathscr{H}_{F}$; see eq. (\ref{eq:mer18}). By going to unbounded
sets $\Omega$ we give up this. 

Even if $\Omega\subset G$ is unbounded, then the operator $T_{F}$
in (\ref{eq:mer0}) is still well-defined; and it may be considered
as a possibly unbounded linear operator as follow:
\begin{equation}
L^{2}\left(\Omega\right)\stackrel{T_{F}}{\longrightarrow}\mathscr{H}_{F}\label{eq:mer-2-1}
\end{equation}
with dense domain $C_{c}\left(\Omega\right)$ in $L^{2}\left(\Omega\right)$.
(If $G$ is a Lie group, we may take $C_{c}^{\infty}\left(\Omega\right)$
as dense domain for $T_{F}$ in (\ref{eq:mer-2-1}).)\end{rem}
\begin{lem}
\label{lem:mer-2-2}Let $\left(\Omega,F\right)$ be as above, but
now assume only conditions \ref{enu:mer1} and \ref{enu:mer2} for
the subset $\Omega\subset G$. 

Then the operator $T_{F}$ in (\ref{eq:mer-2-1}) is a \uline{closable}
operator from $L^{2}\left(\Omega\right)$ into $\mathscr{H}_{F}$;
i.e., the closure of the graph of $T_{F}$, as a subspace in $L^{2}\left(\Omega\right)\times\mathscr{H}_{F}$,
is the graph of a (closed) operator $\overline{T_{F}}$ from $L^{2}\left(\Omega\right)$
into $\mathscr{H}_{F}$; still with $dom\left(\overline{T_{F}}\right)$
dense in $L^{2}\left(\Omega\right)$. \index{operator!closed}\end{lem}
\begin{proof}
Using a standard lemma on unbounded operators, see \cite[ch.13]{Rud73},
we need only show that the following implication holds:\index{operator!unbounded}

Given $\left\{ f_{n}\right\} \subset C_{c}\left(\Omega\right)\left(\subset L^{2}\left(\Omega\right)\right)$,
suppose $\exists\,\xi\in\mathscr{H}_{F}$; and suppose the following
two limits holds:
\begin{align}
 & \lim_{n\rightarrow\infty}\left\Vert f_{n}\right\Vert _{L^{2}\left(\Omega\right)}=0,\mbox{ and}\label{eq:mer-2-2}\\
 & \lim_{n\rightarrow\infty}\left\Vert \xi-T_{F}\left(f_{n}\right)\right\Vert _{\mathscr{H}_{F}}=0.\label{eq:mer-2-3}
\end{align}
Then, it follows that $\xi=0$ in $\mathscr{H}_{F}$.

Now assume $\left\{ f_{n}\right\} $ and $\xi$ satisfying (\ref{eq:mer-2-2})-(\ref{eq:mer-2-3});
the by the reproducing property in $\mathscr{H}_{F}$, we have 
\begin{equation}
\left\langle \xi,T_{F}f_{n}\right\rangle _{\mathscr{H}_{F}}=\int_{\Omega}\overline{\xi\left(x\right)}f_{n}\left(x\right)dx\label{eq:mer-2-4}
\end{equation}
Using (\ref{eq:mer-2-3}), we get 
\[
\lim_{n\rightarrow\infty}\left(\mbox{LHS}\right)_{\left(\ref{eq:mer-2-4}\right)}=\left\langle \xi,\xi\right\rangle _{\mathscr{H}_{F}}=\left\Vert \xi\right\Vert _{\mathscr{H}_{F}}^{2};
\]
and using (\ref{eq:mer-2-2}), we get 
\[
\lim_{n\rightarrow\infty}\left(\mbox{RHS}\right)_{\left(\ref{eq:mer-2-4}\right)}=0.
\]
The domination here is justified by (\ref{eq:mer-2-3}). Indeed, if
(\ref{eq:mer-2-3}) holds, $\exists\, n_{0}$ s.t.
\[
\left\Vert \xi-T_{F}\left(f_{n}\right)\right\Vert _{\mathscr{H}_{F}}\leq1,\;\forall n\geq n_{0},
\]
and therefore,
\begin{equation}
\sup_{n\in\mathbb{N}}\left\Vert T_{F}\left(f_{n}\right)\right\Vert _{\mathscr{H}_{F}}\leq\max_{n\leq n_{0}}\Bigl(1+\left\Vert T_{F}\left(f_{n}\right)\right\Vert _{\mathscr{H}_{F}}\Bigr)<\infty.\label{eq:mer-2-5}
\end{equation}
The desired conclusion follows; we get $\xi=0$ in $\mathscr{H}_{F}$.\end{proof}
\begin{rem}
The conclusion in Lemma \ref{lem:mer-2-2} is the assertion that the
closure of the graph of $T_{F}$ is again the graph of a closed operator,
called the closure. Hence the importance of \textquotedblleft closability.\textquotedblright{}
Once we have existence of the closure of the operator $T_{F}$, as
a closed operator, we will denote this closed operator also by the
same $T_{F}$. This helps reduce the clutter in operator symbols to
follow. From now on, $T_{F}$ will be understood to be the closed
operator obtained in Lemma \ref{lem:mer-2-2}.\index{operator!closed}
\end{rem}

\begin{rem}
If in Lemma \ref{lem:mer-2-2}, for $\left(F,\Omega\right)$ the set
$\Omega$ also satisfies \ref{enu:mer3}-\ref{enu:mer4}, then the
operator $T_{F}$ in (\ref{eq:mer-2-1}) is in fact bounded; but in
general it is not; see the example below with $G=\mathbb{R}$, and
$\Omega=\mathbb{R}_{+}$.\end{rem}
\begin{cor}
\label{cor:mer-2-4}Let $\Omega$, $G$ and $F$ be as in Lemma \ref{lem:mer-2-2},
i.e., with $\Omega$ possibly unbounded, and let $T_{F}$ denote the
closed operator obtained from (\ref{eq:mer-2-1}), and the conclusion
in the lemma. Then we get the following two conclusions:\index{operator!selfadjoint}
\begin{enumerate}
\item \label{enu:mer-2-1}$T_{F}^{*}T_{F}$ is selfadjoint with dense domain
in $L^{2}\left(\Omega\right)$, and 
\item \label{enu:mer-2-2}$T_{F}T_{F}^{*}$ is selfadjoint with dense domain
in $\mathscr{H}_{F}$.
\end{enumerate}
\end{cor}
\begin{proof}
This is an application of the fundamental theorem for closed operators;
see \cite[Theorem 13.13]{Rud73}.\end{proof}
\begin{rem}
\label{rem:mer-2-5}The significance of the conclusions (\ref{enu:mer-2-1})-(\ref{enu:mer-2-2})
in the corollary is that we may apply the spectral theorem to the
respective selfadjoint operators in order to get that $\left(T_{F}^{*}T_{F}\right)^{1/2}$
is a well-defined selfadjoint operator in $L^{2}\left(\Omega\right)$;
and that $\left(T_{F}T_{F}^{*}\right)^{1/2}$ well-defined and selfadjoint
in $\mathscr{H}_{F}$. 

Moreover, by the polar decomposition applied to $T_{F}$ (see \cite[ch. 13]{Rud73}),
we conclude that:\index{polar decomposition}
\begin{equation}
\mbox{spec}\left(T_{F}^{*}T_{F}\right)\backslash\left\{ 0\right\} =\mbox{spec}\left(T_{F}T_{F}^{*}\right)\backslash\left\{ 0\right\} .\label{eq:mer-2-6}
\end{equation}
\end{rem}
\begin{thm}
\label{thm:mer3}Assume $F$, $G$, $\Omega$, and $T_{F}$, are as
above, where $T_{F}$ denotes the closed operator $L^{2}\left(\Omega\right)\stackrel{T}{\longrightarrow}\mathscr{H}_{F}$.
We are assuming that $G$ is a Lie group, $\Omega$ satisfies \ref{enu:mer1}-\ref{enu:mer2}.
Let $X$ be a vector in the Lie algebra of $G$, $X\in La$$\left(G\right)$,
and define $D_{X}^{\left(F\right)}$ as a skew-Hermitian operator
in $\mathscr{H}_{F}$ as follows:
\begin{align}
dom\left(D_{X}^{\left(F\right)}\right) & =\left\{ T_{F}\varphi\:\big|\:\varphi\in C_{c}^{\infty}\left(\Omega\right)\right\} ,\mbox{ and}\label{eq:mer-2-7}\\
D_{X}^{\left(F\right)}\left(F_{\varphi}\right) & =F_{\widetilde{X}\varphi}\nonumber 
\end{align}
where
\begin{align}
\bigl(\widetilde{X}\varphi\bigr)\left(g\right)= & \lim_{t\rightarrow0}\frac{1}{t}\left(\varphi\left(\exp\left(-tX\right)g\right)-\varphi\left(g\right)\right)\label{eq:mer-2-8}
\end{align}
for all $\varphi\in C_{c}^{\infty}\left(\Omega\right)$, and all $g\in\Omega$. 

Then $\widetilde{X}$ defines a skew-Hermitian operator in $L^{2}\left(\Omega\right)$
with dense domain $C_{c}^{\infty}\left(\Omega\right)$. (It is closable,
and we shall denote its closure also by $\widetilde{X}$.)\index{operator!skew-Hermitian}

We get 
\begin{equation}
D_{X}^{\left(F\right)}T_{F}=T_{F}\widetilde{X}\label{eq:mer-2-9}
\end{equation}
on the domain of $\widetilde{X}$; or equivalently
\begin{equation}
D_{X}^{\left(F\right)}=T_{F}\widetilde{X}T_{F}^{-1}.\label{eq:mer-2-10}
\end{equation}
\end{thm}
\begin{proof}
By definition, for all $\varphi\in C_{c}^{\infty}\left(\Omega\right)$,
we have 
\[
\left(D_{X}^{\left(F\right)}T_{F}\right)\left(\varphi\right)=D_{X}^{\left(F\right)}F_{\varphi}=F_{\widetilde{X}\varphi}=\left(T_{F}\widetilde{X}\right)\left(\varphi\right).
\]
Since $\left\{ F_{\varphi}\:\big|\:\varphi\in C_{c}^{\infty}\left(\Omega\right)\right\} $
is a core-domain, (\ref{eq:mer-2-9}) follows. Then the conclusions
in the theorem follow from a direct application of Lemma \ref{lem:mer-2-2},
and Corollary \ref{cor:mer-2-4}; see also Remark \ref{rem:mer-2-5}.\end{proof}
\begin{cor}
\label{cor:mer-2-10}For the respective adjoint operators in (\ref{eq:mer-2-10}),
we have 
\begin{equation}
\bigl(D_{X}^{\left(F\right)}\bigr)^{*}=T_{F}^{*-1}\widetilde{X}^{*}T_{F}^{*}.\label{eq:mer-2-11}
\end{equation}
\end{cor}
\begin{proof}
The formula (\ref{eq:mer-2-11}) in the corollary results from applying
the adjoint operation to both sides of eq (\ref{eq:mer-2-10}), and
keeping track of the domains of the respective operators in the product
on the RHS in eq (\ref{eq:mer-2-10}). Only after checking domains
of the three respective operators, occurring as factors in the product
on the RHS in eq (\ref{eq:mer-2-10}), may we then use the algebraic
rules for adjoint of a product of operators. In this instance, we
conclude that adjoint of the product on the RHS in eq (\ref{eq:mer-2-10})
is the product of the adjoint of the factors, but now composed in
the reverse order; so product from left to right, becomes product
of the adjoints from right to left; hence the result on the RHS in
eq (\ref{eq:mer-2-11}).

Now the fact that the domain issues work out follows from application
of Corollary \ref{cor:mer-2-4}, Remark \ref{rem:mer-2-5}, and Theorem
\ref{thm:mer3}; see especially eqs (\ref{eq:mer-2-6}), and (\ref{eq:mer-2-7}).
The rules for adjoint of a product of operators, where some factors
are unbounded are subtle, and we refer to \cite[chapter 13]{Rud73}
and \cite[Chapter 11-12]{DS88b}. Care must be exercised when the
unbounded operators in the product map between different Hilbert spaces.
The fact that our operator $T_{F}$ is closed as a linear operator
from $L^{2}\left(\Omega\right)$ into $\mathscr{H}_{F}$ is crucial
in this connection; see Lemma \ref{lem:mer-2-2}.
\end{proof}

\begin{cor}
\label{cor:mer3}Let $G$, $\Omega$, $F$, and $T_{F}$ be as above;
then the RKHS $\mathscr{H}_{F}$ consists precisely of the continuous
functions $\xi$ on $\Omega$ such that $\xi\in dom\bigl(\bigl(T_{F}^{*}T_{F}\bigr)^{-1/4}\bigr)$,
and then 
\[
\bigl\Vert\xi\bigr\Vert_{\mathscr{H}_{F}}=\bigl\Vert\bigl(T_{F}^{*}T_{F}\bigr)^{-1/4}\xi\bigr\Vert_{L^{2}\left(\Omega\right)}.
\]
\end{cor}
\begin{proof}
An immediate application of Corollary \ref{cor:mer-2-4}; and the
polar decomposition, applied to the closed operator $T_{F}$ from
Lemma \ref{lem:mer-2-2}.\end{proof}
\begin{example}[Application]
Let $G=\mathbb{R}$, $\Omega=\mathbb{R}_{+}=\left(0,\infty\right)$;
so that $\Omega-\Omega=\mathbb{R}$; let $F\left(x\right)=e^{-\left|x\right|}$,
$\forall x\in\mathbb{R}$, and let $D^{\left(F\right)}$ be the skew-Hermitian
operator from Corollary \ref{cor:mer-2-10}. Then $D^{\left(F\right)}$
has deficiency indices $\left(1,0\right)$ in $\mathscr{H}_{F}$.
\index{deficiency indices}\end{example}
\begin{proof}
From Corollary \ref{cor:mer-2-4}, we conclude that $\mathscr{H}_{F}$
consists of all continuous functions $\xi$ on $\mathbb{R}_{+}\left(=\Omega\right)$
such that $\xi$ and $\xi'=\frac{d\xi}{dx}$ are in $L^{2}\left(\mathbb{R}_{+}\right)$;
and then 
\begin{equation}
\bigl\Vert\xi\bigr\Vert_{\mathscr{H}_{F}}^{2}=\frac{1}{2}\left(\int_{0}^{\infty}\left|\xi\left(x\right)\right|^{2}dx+\int_{0}^{\infty}\left|\xi'\left(x\right)\right|^{2}dx\right)+\int_{0}^{1}\overline{\xi_{n}}\xi\, d\beta;\label{eq:mer-2-13}
\end{equation}
where $\xi_{n}$ denote its inward normal derivative, and $d\beta$
is the corresponding boundary measure. Indeed, $d\beta=-\frac{1}{2}\delta_{0}$,
with $\delta_{0}:=\delta\left(\cdot-0\right)=$ Dirac mass at $x=0$.
See sections \ref{sub:F2}-\ref{sub:F3} for details. 

We now apply Corollary \ref{cor:mer-2-10} to the operator $D_{0}=\dfrac{d}{dx}$
in $L^{2}\left(\mathbb{R}_{+}\right)$ with $dom\left(D_{0}\right)=C_{c}^{\infty}\left(\mathbb{R}_{+}\right)$.
It is well known that $D_{0}$ has deficiency indices $\left(1,0\right)$;
and the $+$ deficiency space is spanned by $\xi_{+}\left(x\right):=e^{-x}\in L^{2}\left(\mathbb{R}_{+}\right)$,
i.e., $x>0$.

Hence, using (\ref{eq:mer-2-11}), we only need to show that $\xi_{+}\in\mathscr{H}_{F}$;
but this is immediate from (\ref{eq:mer-2-13}); in fact
\[
\bigl\Vert\xi_{+}\bigr\Vert_{\mathscr{H}_{F}}^{2}=1.
\]
Setting $\xi_{-}\left(x\right):=e^{x}$, the same argument shows that
$\mbox{RHS}_{\left(\ref{eq:mer-2-13}\right)}=\infty$, so the index
conclusion $\left(1,0\right)$ follows.
\end{proof}
We now return to the case for $G=\mathbb{R}$, and $F$ is fixed continuous
positive definite function on some finite interval $\left(-a,a\right)$,
i.e., the case where $\Omega=\left(0,a\right)$.
\begin{cor}
If $G=\mathbb{R}$ and if $\Omega=\left(0,a\right)$ is a bounded
interval, $a<\infty$, then the operator $D^{\left(F\right)}$ has
equal indices for all given $F:\left(-a,a\right)\rightarrow\mathbb{C}$
which is p.d. and continuous.\end{cor}
\begin{proof}
We showed in Theorem \ref{thm:mer2}, and Corollary \ref{cor:mer2}
that if $\Omega=\left(0,a\right)$ is bounded, then $T_{F}:L^{2}\left(0,a\right)\rightarrow\mathscr{H}_{F}$
is bounded. By Corollary \ref{cor:mer-2-4}, we get that $T_{F}^{-1}:\mathscr{H}_{F}\rightarrow L^{2}\left(0,a\right)$
is closed. Moreover, as an operator in $L^{2}\left(0,a\right)$, $T_{F}$
is positive and selfadjoint. 

Since 
\begin{equation}
D_{0}=\dfrac{d}{dx}\Big|{}_{C_{c}^{\infty}\left(0,a\right)}\label{eq:mer-2-14}
\end{equation}
has indices $\left(1,1\right)$ in $L^{2}\left(0,a\right)$, it follows
from (\ref{eq:mer-2-11}) applied to (\ref{eq:mer-2-14}) that $D^{\left(F\right)}$,
as a skew-Hermitian operator in $\mathscr{H}_{F}$, must have indices
$\left(0,0\right)$ or $\left(1,1\right)$.

To finish the proof, use that a skew Hermitian operator with indices
$\left(1,0\right)$ must generate a semigroup of isometries; one that
is non-unitary. If such an isometry semigroup were generated by the
particular skew Hermitian operator $D^{\left(F\right)}$ then this
would be inconsistent with Corollary \ref{cor:mer3}; see especially
the formula for the norm in $\mathscr{H}_{F}$.
\end{proof}
To simplify notation, we now assume that the endpoint $a$ in \eqref{mer-1}
is $a=1$.
\begin{prop}
\label{prop:mer1}Let $F$ be p.d. continuos on $I=\left(-1,1\right)\subset\mathbb{R}$.
Assume $\mu\in Ext\left(F\right)$, and $\mu\ll d\lambda$, i.e.,
$\exists M\in L^{1}\left(\mathbb{R}\right)$ s.t.
\begin{equation}
d\mu\left(\lambda\right)=M\left(\lambda\right)d\lambda,\;\mbox{where }d\lambda=\mbox{Lebesgue measure on }\mathbb{R}.\label{eq:mer1}
\end{equation}
Set $\mathscr{L}=\left(2\pi\mathbb{Z}\right)$ (period lattice), and
\begin{equation}
\widehat{\varphi_{I}}\left(\xi\right)=\int_{0}^{1}e^{-i\xi y}\varphi\left(y\right)dy,\;\forall\varphi\in C_{c}\left(0,1\right);\label{eq:mer2}
\end{equation}
then the Mercer operator is as follows:
\begin{equation}
\left(T_{F}\varphi\right)\left(x\right)=\sum_{l\in\mathscr{L}}M\left(l\right)\widehat{\varphi_{I}\left(l\right)}e^{ilx}.\label{eq:mer3}
\end{equation}
\end{prop}
\begin{proof}
Let $x\in\left(0,1\right)$, then 
\begin{eqnarray*}
\mbox{RHS}_{\mbox{\ensuremath{\left(\ref{eq:mer3}\right)}}}\left(x\right) & = & \sum_{l\in\mathscr{L}}M\left(l\right)\left(\int_{0}^{1}e^{-ily}\varphi\left(y\right)dy\right)e^{ilx}\\
 & \underset{\left(\mbox{Fubini}\right)}{=} & \int_{0}^{1}\varphi\left(y\right)\underset{\mbox{Poisson summation}}{\underbrace{\left(\sum_{l\in\mathscr{L}}M\left(l\right)e^{il\left(x-y\right)}\right)}}dy\\
 & = & \int_{0}^{1}\varphi\left(y\right)F\left(x-y\right)dy\\
 & = & \left(T_{F}\varphi\right)\left(x\right)\\
 & = & \mbox{LHS}_{\left(\ref{eq:mer3}\right)},
\end{eqnarray*}
where we use that $F=\widehat{d\mu}\Big|_{\left(-1,1\right)}$, and
(\ref{eq:mer1}).\end{proof}
\begin{example}
Application to Table \vpageref{tab:F1-F6}: $\mathscr{L}=2\pi\mathbb{Z}$. \end{example}
\begin{proof}
Application of Proposition \ref{prop:mer1}. See Table \ref{tab:mercer}
below.
\end{proof}

\renewcommand{\arraystretch}{2}

\begin{table}[H]
\begin{tabular}{|c|c|c|}
\hline 
p.d. Function & $\left(T_{F}\varphi\right)\left(x\right)$ & $M\left(\lambda\right)$, $\lambda\in\mathbb{R}$\tabularnewline
\hline 
$F_{1}$ & ${\displaystyle \frac{1}{2}\sum_{l\in\mathscr{L}}e^{-\left|l\right|}\widehat{\varphi}_{I}\left(l\right)e^{ilx}}$ & ${\displaystyle \frac{1}{2}e^{-\left|l\right|}}$ \tabularnewline
\hline 
$F_{3}$ & ${\displaystyle \sum_{l\in\mathscr{L}}\frac{1}{\pi\left(1+l^{2}\right)}\widehat{\varphi}_{I}\left(l\right)e^{ilx}}$ & ${\displaystyle \frac{1}{\pi\left(1+l^{2}\right)}}$\tabularnewline
\hline 
$F_{5}$ & ${\displaystyle \sum_{l\in\mathscr{L}}\frac{1}{\sqrt{2\pi}}e^{-l^{2}/2}\widehat{\varphi}_{I}\left(l\right)e^{ilx}}$ & ${\displaystyle \frac{1}{\sqrt{2\pi}}e^{-l^{2}/2}}$\tabularnewline
\hline 
\end{tabular}

\protect\caption{\label{tab:mercer}Application of Proposition \ref{prop:mer1} to
Table \ref{tab:F1-F6}.}
\end{table}

\renewcommand{\arraystretch}{1}
\begin{cor}
Let $F:\left(-1,1\right)\rightarrow\mathbb{C}$ be a continuous positive
definite function on the interval $\left(-1,1\right)$, and assume:
\begin{enumerate}[label=(\roman{enumi}),ref=\roman{enumi}]
\item  $F\left(0\right)=1$
\item $\exists\mu\in Ext_{1}\left(F\right)$ s.t. $\mu\ll d\lambda$, i.e.,
$\exists M\in L^{1}\left(\mathbb{R}\right)$ s.t. $d\mu\left(\lambda\right)=M\left(\lambda\right)d\lambda$
on $\mathbb{R}$. 
\end{enumerate}
Now consider the Mercer operator 
\begin{equation}
\left(T_{F}\varphi\right)\left(x\right)=\int_{0}^{1}\varphi\left(y\right)F\left(x-y\right)dy,\;\varphi\in L^{2}\left(0,1\right),x\in\left(0,1\right).\label{eq:mer11}
\end{equation}
Then the following two conditions (bd-1) and (bd-2) are equivalent,
where 
\begin{equation}
\mathscr{L}=\left(2\pi\mathbb{Z}\right)=\widehat{\mathbb{T}},\mbox{ and}\label{eq:mer12}
\end{equation}
\begin{eqnarray*}
 & \text{(bd-1)} & \qquad b_{M}:=\sup_{\lambda\in\left[0,1\right]}\sum_{l\in\mathscr{L}}M\left(\lambda+l\right)<\infty,\;\mbox{and}\\
 & \Updownarrow\\
 & \text{(bd-2)} & \qquad T_{F}\left(L^{2}\left(0,1\right)\right)\subseteq\mathscr{H}_{F}.
\end{eqnarray*}
If (bd-1) ($\Leftrightarrow$ (bd-2)) holds, then, for the corresponding
operator-norm, we then have
\begin{equation}
\left\Vert T_{F}\right\Vert _{L^{2}\left(0,1\right)\rightarrow\mathscr{H}_{F}}=\sqrt{b_{M}}\;\mbox{in }\text{(bd-1)}.\label{eq:mer13}
\end{equation}
\end{cor}
\begin{rem}
Condition (bd-1) is automatically satisfied in all interesting cases
(at least from the point of view of our present Memoir.)\end{rem}
\begin{proof}
The key step in the proof of ``$\Longleftrightarrow$'' was the
Parseval duality, 
\begin{equation}
\int_{0}^{1}\left|f\left(x\right)\right|^{2}dx=\sum_{l\in\mathscr{L}}\left|\widehat{f_{I}}\left(l\right)\right|^{2},\mbox{ where}\label{eq:mer14}
\end{equation}
$\left[0,1\right)\simeq\mathbb{T}=\mathbb{R}/\mathbb{Z}$, $\widehat{\mathbb{T}}\simeq\mathscr{L}$. 

Let $F$, $T_{F}$, and $M$ be as in the statement of the corollary.
Then for $\varphi\in C_{c}\left(0,1\right)$, we compute the $\mathscr{H}_{F}$-norm
of 
\begin{equation}
T_{F}\left(\varphi\right)=F_{\varphi}\label{eq:mer15}
\end{equation}
with the use of (\ref{eq:mer11}), and Proposition \ref{prop:mer1}.

We return to 
\begin{equation}
\widehat{\varphi_{I}}\left(l\right)=\int_{0}^{1}e^{-ily}\varphi\left(y\right)dy,\; l\in\mathscr{L};\label{eq:mer16}
\end{equation}
and we now compute $\widehat{\left(T_{F}\varphi\right)_{I}}\left(l\right)$,
$l\in\mathscr{L}$; starting with $T_{F}\varphi$ from (\ref{eq:mer11}).
The result is
\[
\widehat{\left(T_{F}\varphi\right)_{I}}\left(l\right)=M\left(l\right)\widehat{\varphi_{I}}\left(l\right),\;\forall l\in\mathscr{L}\left(=2\pi\mathbb{Z}.\right)
\]
And further, using \chapref{conv}, we have: 
\begin{eqnarray*}
\left\Vert F_{\varphi}\right\Vert _{\mathscr{H}_{F}}^{2} & = & \left\Vert T_{F}\varphi\right\Vert _{\mathscr{H}_{F}}^{2}\\
 & \begin{array}[t]{c}
=\\
\text{(Cor. \ensuremath{\left(\ref{cor:lcg-isom}\right)})}
\end{array} & \int_{\mathbb{R}}\left|\widehat{\varphi}\left(\lambda\right)\right|^{2}M\left(\lambda\right)d\lambda\\
 & = & \int_{0}^{1}\sum_{l\in\mathscr{L}}\left|\widehat{\varphi}\left(\lambda+l\right)\right|^{2}M\left(\lambda+l\right)d\lambda\\
 & \begin{array}[t]{c}
\leq\\
\text{\ensuremath{\left(\ref{eq:mer16}\right)}}
\end{array} & \left(\sum_{l\in\mathscr{L}}\left|\widehat{\varphi}\left(l\right)\right|^{2}\right)\sup_{\lambda\in\left[0,1\right)}\sum_{l\in\mathscr{L}}M\left(\lambda+l\right)\\
 & \begin{array}[t]{c}
=\\
\text{\ensuremath{\left(\ref{eq:mer14}\right)}}\\
\text{and (bd-1)}
\end{array} & \left(\int_{0}^{1}\left|\varphi\left(x\right)\right|^{2}dx\right)\cdot b_{M}=\left\Vert \varphi\right\Vert _{L^{2}\left(0,1\right)}^{2}\cdot b_{M}.
\end{eqnarray*}

Hence, if $b_{M}<\infty$, (bd-2) holds, with
\begin{equation}
\Bigl\Vert T_{F}\Bigr\Vert_{L^{2}\left(0,1\right)\rightarrow\mathscr{H}_{F}}\leq\sqrt{b_{M}}.\label{eq:mer17}
\end{equation}
Using standard Fourier duality, one finally sees that ``$\leq$''
in (\ref{eq:mer17}) is in fact ``$=$''. \index{Fourier duality}\end{proof}
\begin{rem}
A necessary condition for boundedness of $T_{F}:L^{2}\left(0,1\right)\rightarrow\mathscr{H}_{F}$,
is $M\in L^{\infty}\left(\mathbb{R}\right)$ when the function $M\left(\cdot\right)$
is as specified in (ii) of the corollary.\end{rem}
\begin{proof}
Let $\varphi\in C_{c}\left(0,1\right)$, then

\begin{eqnarray*}
\Bigl\Vert T_{F}\varphi\Bigr\Vert_{\mathscr{H}_{F}}^{2}=\Bigl\Vert F_{\varphi}\Bigr\Vert_{\mathscr{H}_{F}}^{2} & = & \int_{\mathbb{R}}\left|\widehat{\varphi}\left(\lambda\right)\right|^{2}M\left(\lambda\right)d\lambda\\
 & \leq & \bigl\Vert M\bigr\Vert{}_{\infty}\cdot\int_{\mathbb{R}}\left|\widehat{\varphi}\left(\lambda\right)\right|^{2}d\lambda\\
 & = & \bigl\Vert M\bigr\Vert{}_{\infty}\cdot\int_{\mathbb{R}}\left|\varphi\left(x\right)\right|^{2}d\lambda\,\,\,\,\,\,(\mbox{Parseval})\\
 & = & \bigl\Vert M\bigr\Vert{}_{\infty}\bigl\Vert\varphi\bigr\Vert{}_{L^{2}\left(0,1\right)}^{2}.
\end{eqnarray*}
\end{proof}
\begin{thm}
Let $F$ be as in Proposition \ref{prop:mer1}, and $\mathscr{H}_{F}$
the corresponding RKHS. Define the skew-Hermitian operator $D^{\left(F\right)}\left(F_{\varphi}\right)=\frac{1}{i}F_{\varphi'}$
on 
\[
dom\bigl(D^{\left(F\right)}\bigr)=\left\{ F_{\varphi}\:\Big|\:\varphi\in C_{c}^{\infty}\left(0,1\right)\right\} \subset\mathscr{H}_{F}
\]
as before. Let $A\supset D^{\left(F\right)}$ be a selfadjoint extension
of $D^{\left(F\right)}$, i.e., 
\[
D^{\left(F\right)}\subset A\subset\bigl(D^{\left(F\right)}\bigr)^{*},\; A=A^{*}.
\]
Let $P=P_{A}$ be the projection valued measure (PVM) of $A$, and
\begin{equation}
U_{t}^{\left(A\right)}=e^{tA}=\int_{\mathbb{R}}e^{it\lambda}P_{A}\left(d\lambda\right),\: t\in\mathbb{R}\label{eq:mer4}
\end{equation}
be the one-parameter unitary group; and for all $f$ measurable on
$\mathbb{R}$, set (the Spectral Theorem applied to $A$)
\begin{equation}
f\left(A\right)=\int_{\mathbb{R}}f\left(\lambda\right)P_{A}\left(d\lambda\right);\label{eq:mer5}
\end{equation}
then we get the following
\begin{equation}
\left(T_{F}\varphi\right)\left(x\right)=\left(M\widehat{\varphi_{I}}\right)\left(A\right)U_{x}^{\left(A\right)}=U_{x}^{\left(A\right)}\left(M\widehat{\varphi_{I}}\right)\left(A\right)\label{eq:mer6}
\end{equation}
for the Mercer operator $\left(T_{F}\varphi\right)\left(x\right)=\int_{0}^{1}F\left(x-y\right)\varphi\left(y\right)$,
$\varphi\in L^{2}\left(0,1\right)$.

\index{measure!PVM}\index{selfadjoint extension}\end{thm}
\begin{proof}
Using (\ref{eq:mer5}), we get 
\begin{eqnarray}
\widehat{\varphi_{I}}\left(A\right)U_{x}^{\left(A\right)} & = & \int_{\mathbb{R}}\int_{0}^{1}\varphi\left(y\right)e^{-i\lambda y}P_{A}\left(d\lambda\right)U^{A}\left(x\right)\nonumber \\
 & \underset{\text{(Fubini)}}{=} & \int_{0}^{1}\varphi\left(y\right)\left(\int_{\mathbb{R}}e^{i\lambda\left(x-y\right)}P_{A}\left(d\lambda\right)\right)dy\nonumber \\
 & \underset{\text{\ensuremath{\left(\ref{eq:mer4}\right)}}}{=} & \int_{0}^{1}U^{A}\left(x-y\right)\varphi\left(y\right)dy\nonumber \\
 & = & U_{x}^{\left(A\right)}U^{\left(A\right)}\left(\varphi\right),\;\mbox{where}\label{eq:mer7}
\end{eqnarray}
\begin{equation}
U^{\left(A\right)}\left(\varphi\right)=\int_{0}^{1}\varphi\left(y\right)U^{\left(A\right)}\left(-y\right)dy,\label{eq:mer8}
\end{equation}
all operators in the RKHS $\mathscr{H}_{F}$.

We have a selfadjoint extension $A$ corresponding to $\mu=\mu^{\left(A\right)}\in Ext\left(F\right)$,
and a cyclic vector $v_{0}$: 
\begin{eqnarray}
F^{\left(A\right)}\left(t\right) & = & \left\langle v_{0},U^{\left(A\right)}\left(t\right)v_{0}\right\rangle \nonumber \\
 & = & \int_{\mathbb{R}}e^{i\lambda t}d\mu_{A}\left(\lambda\right),\mbox{ where }d\mu_{A}\left(\lambda\right)=\left\Vert P_{A}\left(d\lambda\right)v_{0}\right\Vert ^{2},\label{eq:mer9}
\end{eqnarray}
and from (\ref{eq:mer8}):
\begin{equation}
\left(T_{F}\varphi\right)\left(x\right)=\bigl(F_{\varphi}^{\left(A\right)}\bigr)\left(x\right),\;\forall\varphi\in C_{c}\left(0,1\right),\mbox{ and }\forall x\in\left(0,1\right).\label{eq:mer10}
\end{equation}

\end{proof}

\section{Positive Definite Functions, Green's Functions, and Boundary}

In this section, we consider a correspondence and interplay between
a class of boundary value problems on the one hand, and spectral theoretic
properties of extension operators on the other.

Fix a bounded domain $\Omega\subset\mathbb{R}^{n}$, open and connected.
Let $F:\Omega-\Omega\rightarrow\mathbb{C}$ be a continuous positive
definite (p.d.) function. We consider a special case when $F$ occurs
as the Green's function of certain linear operator. \index{positive definite}\index{operator!selfadjoint}\index{Schwartz!test function}
\begin{lem}
\label{lem:gr-1}Let $\mathscr{D}$ be a Hilbert space, a Fréchet
space or an LF-space (see \cite{Tre06}), such that $\mathscr{D}\underset{j}{\hookrightarrow}L^{2}\left(\Omega\right)$;
and such that the inclusion mapping $j$ is continuous relative to
the respective topologies on $\mathscr{D}$, and on $L^{2}\left(\Omega\right)$.
Let $\mathscr{D}^{*}:=$ the dual of $\mathscr{D}$ when $\mathscr{D}$
is given its Fréchet (LF, or Hilbert) topology; then there is a natural
``inclusion'' mapping $j^{*}$ from $L^{2}\left(\Omega\right)$
to $\mathscr{D}^{*}$, i.e., we get 
\begin{equation}
\mathscr{D}\underset{j}{\hookrightarrow}L^{2}\left(\Omega\right)\underset{j^{*}}{\hookrightarrow}\mathscr{D}^{*}.\label{eq:gr-2-1}
\end{equation}
\end{lem}
\begin{proof}
It is immediate from the assumptions, and the fact that $L^{2}\left(\Omega\right)$
is its own dual. See also \cite{Tre06}.\end{proof}
\begin{rem}
In the following we shall use \lemref{gr-1} in two cases:
\begin{enumerate}[leftmargin=*]
\item Let $A$ be a selfadjoint operator (unbounded in the non-trivial
cases) acting in $L^{2}\left(\Omega\right)$; and with dense domain.
For $\mathscr{D}=\mathscr{D}_{A}$, we may choose the domain of $A$
with its graph topology. 
\item Let $\mathscr{D}$ be a space of Schwartz test functions, e.g., $C_{c}^{\infty}\left(\Omega\right)$,
given its natural LF-topology, see \cite{Tre06}; then the inclusion
\begin{equation}
C_{c}^{\infty}\left(\Omega\right)\underset{j}{\hookrightarrow}L^{2}\left(\Omega\right)\label{eq:gr-2-2}
\end{equation}
satisfies the condition in \lemref{gr-1}.
\end{enumerate}
\end{rem}
\begin{cor}
\label{cor:gr-1}Let $\mathscr{D}\subset L^{2}\left(\Omega\right)$
be a subspace satisfying the conditions in \lemref{gr-1}; and consider
the triple of spaces (\ref{eq:gr-2-1}); then the inner product in
$L^{2}\left(\Omega\right)$, here denoted $\left\langle \cdot,\cdot\right\rangle _{2}$,
extends by closure to a sesquilinear function $\left\langle \cdot,\cdot\right\rangle $
(which we shall also denote by $\left\langle \cdot,\cdot\right\rangle _{2}$):
\begin{equation}
\left\langle \cdot,\cdot\right\rangle :L^{2}\left(\Omega\right)\times\mathscr{D}^{*}\rightarrow\mathbb{C}.\label{eq:gr-2-3}
\end{equation}
\end{cor}
\begin{proof}
This is a standard argument based on dual topologies; see \cite{Tre06}.\end{proof}
\begin{example}[Application]
If $\mathscr{D}=C_{c}^{\infty}\left(\Omega\right)$ in (\ref{eq:gr-2-1}),
then $\mathscr{D}^{*}=$ the space of all Schwartz-distributions on
$\Omega$, including the Dirac masses. Referring to (\ref{eq:gr-2-3}),
we shall write $\left\langle \delta_{x},f\right\rangle _{2}$ to mean
$f\left(x\right)$, when $f\in C\left(\overline{\Omega}\right)\cap L^{2}\left(\Omega\right)$. 
\end{example}
Adopting the constructions from Lemma \lemref{gr-1} and Corollary
\ref{cor:gr-1}, we now turn to calculus of positive definite functions: 
\begin{defn}
If $F:\Omega-\Omega\rightarrow\mathbb{C}$ is a function, or a distribution,
then we say that $F$ is positive definite iff\index{positive definite!distribution}
\begin{equation}
\left\langle F,\overline{\varphi}\otimes\varphi\right\rangle \geq0\label{eq:gr-1-1}
\end{equation}
for all $\varphi\in C_{c}^{\infty}\left(\Omega\right)$. The meaning
of (\ref{eq:gr-1-1}) is the distribution $K_{F}:=F\left(x-y\right)$
acting on $\left(\overline{\varphi}\otimes\varphi\right)\left(x,y\right):=\overline{\varphi\left(x\right)}\varphi\left(y\right)$,
$x,y\in\Omega$. 
\end{defn}
Let 
\begin{equation}
\triangle:=\sum_{j=1}^{k}\Bigl(\frac{\partial}{\partial x_{j}}\Bigr)^{2}\label{eq:gr-1-2}
\end{equation}
and consider an open domain $\Omega\subset\mathbb{R}^{k}$.

In $\mathscr{H}_{F}$, set 
\begin{equation}
D_{j}^{\left(F\right)}\left(F_{\varphi}\right):=F_{\frac{\partial\varphi}{\partial x_{j}}},\;\varphi\in C_{c}^{\infty}\left(\Omega\right),\; j=1,\ldots,k.\label{eq:gr-1-3}
\end{equation}
Then this is a system of commuting skew-Hermitian operators with dense
domain in $\mathscr{H}_{F}$. 
\begin{lem}
Let $F:\Omega-\Omega\rightarrow\mathbb{C}$ be a positive definite
function (or a distribution); and set 
\begin{equation}
M:=-\triangle F\label{eq:gr-1-4}
\end{equation}
where $\triangle F$ on the RHS in (\ref{eq:gr-1-4}) is in the sense
of distributions. Then $M$ is also positive definite and 
\begin{equation}
\left\langle M_{\varphi},M_{\psi}\right\rangle _{\mathscr{H}_{M}}=\sum_{j=1}^{k}\left\langle D_{j}^{\left(F\right)}F_{\varphi},D_{j}^{\left(F\right)}F_{\psi}\right\rangle _{\mathscr{H}_{F}}\label{eq:gr-1-5}
\end{equation}
for all $\varphi,\psi\in C_{c}^{\infty}\left(\Omega\right)$. In particular,
setting $\varphi=\psi$ in (\ref{eq:gr-1-5}), we have 
\begin{equation}
\Bigl\Vert M_{\varphi}\Bigr\Vert_{\mathscr{H}_{M}}^{2}=\sum_{j=1}^{k}\Bigl\Vert D_{j}^{\left(F\right)}F_{\varphi}\Bigr\Vert_{\mathscr{H}_{F}}^{2}.\label{eq:gr-1-6}
\end{equation}
\end{lem}
\begin{proof}
We must show that $M$ satisfies (\ref{eq:gr-1-1}), i.e., that 
\begin{equation}
\left\langle M,\overline{\varphi}\otimes\varphi\right\rangle \geq0;\label{eq:gr-1-7}
\end{equation}
and moreover that (\ref{eq:gr-1-5}), or equivalently (\ref{eq:gr-1-4}),
holds.

For $\mbox{LHS}_{\left(\ref{eq:gr-1-7}\right)}$, we have 
\[
\left\langle M,\overline{\varphi}\otimes\varphi\right\rangle =\left\langle -\triangle F,\overline{\varphi}\otimes\varphi\right\rangle =-\sum_{j=1}^{k}\left\langle \Bigl(\frac{\partial}{\partial x_{j}}\Bigr)^{2}F,\overline{\varphi}\otimes\varphi\right\rangle ;
\]
and using the action of $\frac{\partial}{\partial x_{j}}$ in the
sense of distributions, we get, 
\[
\mbox{LHS}_{\left(\ref{eq:gr-1-7}\right)}=\sum_{j=1}^{k}\left\langle F,\overline{\frac{\partial\varphi}{\partial x_{j}}}\otimes\frac{\partial\varphi}{\partial x_{j}}\right\rangle \underset{\left(\text{by \ensuremath{\left(\ref{eq:gr-1-3}\right)}}\right)}{=}\sum_{j=1}^{k}\Bigl\Vert D_{j}^{\left(F\right)}\left(F_{\varphi}\right)\Bigr\Vert_{\mathscr{H}_{F}}^{2}\geq0,
\]
which yields of the desired conclusions.\end{proof}
\begin{example}
For $k=1$, consider the functions $F_{2}$ and $F_{3}$ from Table
\ref{tab:F1-F6}. 
\begin{enumerate}
\item Let $F=F_{2}$, $\Omega=\left(-\frac{1}{2},\frac{1}{2}\right)$, then
\begin{equation}
M=-F''=2\delta\label{eq:gr-1-8}
\end{equation}
where $\delta$ is the Dirac mass at $x=0$, i.e., $\delta=\delta\left(x-0\right)$. 
\item Let $F=F_{3}$, $\Omega=\left(-1,1\right)$, then 
\begin{equation}
M=-F''=2\delta-F\label{eq:gr-1-9}
\end{equation}
 \end{enumerate}
\begin{proof}
The proof of the assertions in the two examples follows directly from
sections \ref{sub:F2} and \ref{sub:F3}.
\end{proof}
\end{example}
Now we return to the p.d. function $F:\Omega-\Omega\rightarrow\mathbb{C}$.
Suppose $A:L^{2}\left(\Omega\right)\rightarrow L^{2}\left(\Omega\right)$
is an unbounded positive linear operator, i.e., $A\geq c>0$, for
some constant $c$. Further assume that $A^{-1}$ has the integral
kernel (Green's function) $F$, i.e., 
\begin{equation}
\left(A^{-1}f\right)\left(x\right)=\int_{\Omega}F\left(x-y\right)f\left(y\right)dy,\;\forall f\in L^{2}\left(\Omega\right).\label{eq:gr-1}
\end{equation}
For all $x\in\Omega$, define 
\begin{equation}
F_{x}\left(\cdot\right):=F\left(x-\cdot\right).\label{eq:gr-2}
\end{equation}

Here $F_{x}$ is the fundamental solution to the following equation
\[
Au=f
\]
where $u\in dom\left(A\right)$, and $f\in L^{2}\left(\Omega\right)$.
Hence, in the sense of distribution, we have
\begin{eqnarray*}
AF_{x}\left(\cdot\right) & = & \delta_{x}\\
 & \Updownarrow\\
A\left(\int_{\Omega}F\left(x,y\right)f\left(y\right)dy\right) & = & \int\left(AF_{x}\left(y\right)\right)f\left(y\right)dy\\
 & = & \int\delta_{x}\left(y\right)f\left(y\right)dy\\
 & = & f\left(x\right).
\end{eqnarray*}
 Note that $A^{-1}\geq0$ iff $F$ is a p.d. kernel. 

Let $\mathscr{H}_{A}=$ the completion of $C_{c}^{\infty}\left(\Omega\right)$
in the bilinear form 
\begin{equation}
\left\langle f,g\right\rangle _{A}:=\left\langle Af,g\right\rangle _{2};\label{eq:gr-3}
\end{equation}
where the RHS extends the inner product in $L^{2}\left(\Omega\right)$
as in (\ref{eq:gr-2-3}).
\begin{lem}
$\mathscr{H}_{A}$ is a RKHS and the reproducing kernel is $F_{x}$.
\index{RKHS}\end{lem}
\begin{proof}
Since $A\geq c>0$, in the usual ordering of Hermitian operator, (\ref{eq:gr-3})
is a well-defined inner product, so $\mathscr{H}_{A}$ is a Hilbert
space. For the reproducing property, we check that 
\[
\left\langle F_{x},g\right\rangle _{A}=\left\langle AF_{x},g\right\rangle _{2}=\left\langle \delta_{x},g\right\rangle _{2}=g\left(x\right).
\]
\end{proof}
\begin{lem}
Let $\mathscr{H}_{F}$ be the RKHS corresponding to $F$, i.e., the
completion of $span\left\{ F_{x}\::\: x\in\Omega\right\} $ in the
inner product 
\begin{equation}
\left\langle F_{y},F_{x}\right\rangle _{F}:=F_{x}\left(y\right)=F\left(x-y\right)\label{eq:gr-4}
\end{equation}
extending linearly. Then we have the isometric embedding $\mathscr{H}_{F}\hookrightarrow\mathscr{H}_{A}$,
via the map,
\begin{equation}
F_{x}\mapsto F_{x}.\label{eq:gr-5}
\end{equation}
\end{lem}
\begin{proof}
We check directly that 
\begin{align*}
\left\Vert F_{x}\right\Vert _{F}^{2} & =\left\langle F_{x},F_{x}\right\rangle _{F}=F_{x}\left(x\right)=F\left(0\right)\\
\left\Vert F_{x}\right\Vert _{A}^{2} & =\left\langle F_{x},F_{x}\right\rangle _{A}=\left\langle AF_{x},F_{x}\right\rangle _{L^{2}}=\left\langle \delta_{x},F_{x}\right\rangle _{L^{2}}=F_{x}\left(x\right)=F\left(0\right).
\end{align*}
\end{proof}
\begin{rem}
Now consider $\mathbb{R}$, and let $\Omega=\left(0,a\right)$. Recall
the \uline{Mercer operator} 
\[
T_{F}:L^{2}\left(\Omega\right)\rightarrow L^{2}\left(\Omega\right),\mbox{ by}
\]
\begin{eqnarray}
\left(T_{F}g\right)\left(x\right) & := & \int_{0}^{a}F_{x}\left(y\right)g\left(y\right)dy\label{eq:gr-10}\\
 & = & \left\langle F_{x},g\right\rangle _{2},\;\forall g\in L^{2}\left(0,a\right).\nonumber 
\end{eqnarray}
By Corollary \ref{cor:mer1}, $T_{F}$ can be diagonalized in $L^{2}\left(0,a\right)$
by
\[
T_{F}\xi_{n}=\lambda_{n}\xi_{n},\;\lambda_{n}>0
\]
where $\left\{ \xi_{n}\right\} _{n\in\mathbb{N}}$ is an ONB in $L^{2}\left(0,a\right)$;
further $\xi_{n}\subset\mathscr{H}_{F}$, for all $n\in\mathbb{N}$.

From (\ref{eq:gr-10}), we then have 
\begin{equation}
\left\langle F_{x},\xi_{n}\right\rangle _{2}=\lambda_{n}\xi_{n}\left(x\right).\label{eq:gr-11}
\end{equation}
Applying $A$ on both sides of (\ref{eq:gr-11}) yields 
\begin{eqnarray*}
\mbox{LHS}_{\left(\ref{eq:gr-11}\right)} & = & \left\langle AF_{x},\xi_{n}\right\rangle _{2}=\left\langle \delta_{x},\xi_{n}\right\rangle _{2}=\xi_{n}\left(x\right)\\
\mbox{RHS}_{\left(\ref{eq:gr-11}\right)} & = & \lambda_{n}\left(A\xi_{n}\right)\left(x\right)
\end{eqnarray*}
therefore, $A\xi_{n}=\frac{1}{\lambda_{n}}\xi_{n}$, i.e., 
\begin{equation}
A=T_{F}^{-1}.\label{eq:fr-12}
\end{equation}
Consequently, 
\[
\left\langle \xi_{n},\xi_{m}\right\rangle _{A}=\left\langle A\xi_{n},\xi_{m}\right\rangle _{2}=\frac{1}{\lambda_{n}}\left\langle \xi_{n},\xi_{m}\right\rangle _{2}=\frac{1}{\lambda_{n}}\delta_{n,m}.
\]
And we conclude that $\left\{ \sqrt{\lambda_{n}}\xi_{n}\right\} _{n\in\mathbb{N}}$
is an ONB in $\mathscr{H}_{A}=\mathscr{H}_{T_{F}^{-1}}$. 

See \secref{F3-Mercer}, where $F=$ Polya extension of $F_{3}$,
and a specific construction of $\mathscr{H}_{T_{F}^{-1}}$. 
\end{rem}

\subsection{Connection to the Energy Space Hilbert Space}

Now consider $A=1-\triangle$ defined on $C_{c}^{\infty}\left(\Omega\right)$.
There is a connection between the RKHS $\mathscr{H}_{A}$ and the
energy space as follows:

For $f,g\in\mathscr{H}_{A}$, we have (restricting to real valued
functions), 
\begin{eqnarray*}
\left\langle f,g\right\rangle _{A} & = & \left\langle \left(1-\triangle\right)f,g\right\rangle _{L^{2}}\\
 & = & \int_{\Omega}fg-\int_{\Omega}\left(\triangle f\right)g\\
 & = & \underset{\mbox{energy inner product}}{\underbrace{\int_{\Omega}fg+\int_{\Omega}Df\cdot Dg}}+\text{boundary corrections};
\end{eqnarray*}
So we define 
\begin{equation}
\left\langle f,g\right\rangle _{Energy}:=\int_{\Omega}fg+\int_{\Omega}Df\cdot Dg;\label{eq:gr-6}
\end{equation}
and then 
\begin{equation}
\left\langle f,g\right\rangle _{A}=\left\langle f,g\right\rangle _{Energy}+\mbox{boundary corrections.}\label{eq:gr-7}
\end{equation}

\begin{rem}
The $A$-inner product on the LHS of (\ref{eq:gr-7}) incorporates
the boundary information.\end{rem}
\begin{example}
Consider $L^{2}\left(0,1\right)$, $F\left(x\right)=e^{-\left|x\right|}\big|_{\left(-1,1\right)}$,
and $A=\frac{1}{2}\bigl(1-\bigl(\frac{d}{dx}\bigr)^{2}\bigr)$. We
have 

\begin{eqnarray*}
\left\langle f,g\right\rangle _{A} & = & \frac{1}{2}\left\langle f-f'',g\right\rangle _{L^{2}}\\
 & = & \frac{1}{2}\int_{0}^{1}fg-\frac{1}{2}\int_{0}^{1}f''g\\
 & = & \frac{1}{2}\left(\int_{0}^{1}fg+\int_{0}^{1}f'g'\right)+\frac{\left(f'g\right)\left(0\right)-\left(f'g\right)\left(1\right)}{2}\\
 & = & \left\langle f,g\right\rangle _{Energy}+\frac{\left(f'g\right)\left(0\right)-\left(f'g\right)\left(1\right)}{2}.
\end{eqnarray*}
Here, the boundary term
\begin{equation}
\frac{\left(f'g\right)\left(0\right)-\left(f'g\right)\left(1\right)}{2}\label{eq:gr-8}
\end{equation}
contains the inward normal derivative of $f'$ at $x=0$ and $x=1$.
\begin{enumerate}[leftmargin=*]
\item \begin{flushleft}
We proceed to check the reproducing property w.r.t. the $A$-inner
product:
\[
2\left\langle e^{-\left|x-\cdot\right|},g\right\rangle _{Energy}=\int_{0}^{1}e^{-\left|x-y\right|}g\left(y\right)dy+\int_{0}^{1}\left(\frac{d}{dy}e^{-\left|x-y\right|}\right)g'\left(y\right)dy
\]
where
\begin{eqnarray*}
 &  & \int_{0}^{1}\left(\frac{d}{dy}e^{-\left|x-y\right|}\right)g'\left(y\right)dy\\
 & = & \int_{0}^{x}e^{-\left(x-y\right)}g'\left(y\right)dy-\int_{x}^{1}e^{-\left(y-x\right)}g'\left(y\right)dy\\
 & = & 2g\left(x\right)-g\left(0\right)e^{-x}-g\left(1\right)e^{-\left(1-x\right)}-\int_{0}^{1}e^{-\left|x-y\right|}g\left(y\right)dy;
\end{eqnarray*}
it follows that 
\begin{equation}
\left\langle e^{-\left|x-\cdot\right|},g\right\rangle _{Energy}=g\left(x\right)-\frac{g\left(0\right)e^{-x}+g\left(1\right)e^{-\left(1-x\right)}}{2}\label{eq:f3-2-1}
\end{equation}

\par\end{flushleft}
\item \begin{flushleft}
It remains to check the boundary term in (\ref{eq:f3-2-1}) comes
from the inward normal derivative of $e^{-\left|x-\cdot\right|}$.
Indeed, set $f\left(\cdot\right)=e^{-\left|x-\cdot\right|}$ in (\ref{eq:gr-8}),
then 
\[
f'\left(0\right)=e^{-x},\qquad f'\left(1\right)=-e^{-\left(1-x\right)}
\]
therefore,
\par\end{flushleft}
\end{enumerate}
\[
\frac{\left(f'g\right)\left(0\right)-\left(f'g\right)\left(1\right)}{2}=\frac{e^{-x}g\left(0\right)+e^{-\left(1-x\right)}g\left(1\right)}{2}.
\]

\end{example}

\begin{example}
Consider $L^{2}\left(0,\frac{1}{2}\right)$, $F\left(x\right)=1-\left|x\right|$
with $\left|x\right|<\frac{1}{2}$, and let $A=-\frac{1}{2}\left(\frac{d}{dx}\right)^{2}$.
Then the $A$-inner product yields
\begin{eqnarray*}
\left\langle f,g\right\rangle _{A} & = & -\frac{1}{2}\left\langle f'',g\right\rangle _{L^{2}}\\
 & = & \frac{1}{2}\int_{0}^{\frac{1}{2}}f'g'-\frac{\left(f'g\right)\left(\frac{1}{2}\right)-\left(f'g\right)\left(0\right)}{2}\\
 & = & \left\langle f,g\right\rangle _{Energy}+\frac{\left(f'g\right)\left(0\right)-\left(f'g\right)\left(\frac{1}{2}\right)}{2}
\end{eqnarray*}
where we set 
\[
\left\langle f,g\right\rangle _{Energy}:=\frac{1}{2}\int_{0}^{\frac{1}{2}}f'g';
\]
and the corresponding boundary term is 
\begin{equation}
\frac{\left(f'g\right)\left(0\right)-\left(f'g\right)\left(\frac{1}{2}\right)}{2}\label{eq:gr-9}
\end{equation}

\begin{enumerate}[leftmargin=*]
\item \begin{flushleft}
To check the reproducing property w.r.t. the $A$-inner product:
Set $F_{x}\left(y\right):=1-\left|x-y\right|$, $x,y\in\left(0,\frac{1}{2}\right)$;
then 
\begin{eqnarray}
\left\langle F_{x},g\right\rangle _{Energy} & = & \frac{1}{2}\int_{0}^{\frac{1}{2}}F_{x}\left(y\right)'g'\left(y\right)dy\nonumber \\
 & = & =\frac{1}{2}\left(\int_{0}^{x}g'\left(y\right)dy-\int_{x}^{\frac{1}{2}}g'\left(y\right)dy\right)\nonumber \\
 & = & g\left(x\right)-\frac{g\left(0\right)+g\left(\frac{1}{2}\right)}{2}.\label{eq:F2-3-1}
\end{eqnarray}

\par\end{flushleft}
\item \begin{flushleft}
Now we check the second term on the RHS of (\ref{eq:F2-3-1}) contains
the inward normal derivative of $F_{x}$. Note that 
\begin{eqnarray*}
F_{x}'\left(0\right) & = & \frac{d}{dy}\Big|_{y=0}\left(1-\left|x-y\right|\right)=1\\
F_{x}'\left(\frac{1}{2}\right) & = & \frac{d}{dy}\Big|_{y=\frac{1}{2}}\left(1-\left|x-y\right|\right)=-1
\end{eqnarray*}
Therefore, 
\[
\frac{\left(f'g\right)\left(0\right)-\left(f'g\right)\left(\frac{1}{2}\right)}{2}=\frac{g\left(0\right)+g\left(\frac{1}{2}\right)}{2};
\]
which verifies the boundary term in (\ref{eq:gr-9}).
\par\end{flushleft}
\end{enumerate}
\end{example}

\section{\label{sec:F2F3}The RKHSs for the Two Examples $F_{2}$ and $F_{3}$
in Table \ref{tab:F1-F6}}

In this section, we revisit cases $F_{2}$, and $F_{3}$ (from \tabref{F1-F6})
and their associated RKHSs. We show that they are (up to isomorphism)
also the Hilbert spaces used in stochastic integration for Brownian
motion, and for the Ornstein-Uhlenbeck process (see e.g., \cite{Hi80}),
respectively. As reproducing kernel Hilbert spaces, they have an equivalent
and more geometric form, of use in for example analysis of Gaussian
processes. Analogous results for the respective RKHSs also hold for
other positive definite function systems $\left(F,\Omega\right)$,
but for the present two examples $F_{2}$, and $F_{3}$, the correspondences
involved are explicit. As a bonus, we get an easy and transparent
proof that the deficiency-indices for the respective operators $D^{\left(F\right)}$
are $\left(1,1\right)$ in both these examples.

The purpose of the details below are two-fold. First we show that
the respective RKHSs corresponding to $F_{2}$ and $F_{3}$ in \tabref{F1-F6}
are naturally isomorphic to more familiar RKHSs which are used in
the study of Gaussian processes, see e.g., \cite{AJL11,AL10,AJ12};
and secondly, to give an easy (and intuitive) proof that the deficiency
indices in these two cases are $\left(1,1\right)$. Recall for each
p.d. function $F$ in an interval $\left(-a,a\right)$, we study 
\begin{equation}
D^{\left(F\right)}\left(F_{\varphi}\right):=F_{\varphi'},\;\varphi\in C_{c}^{\infty}\left(0,a\right)\label{eq:RKHS-eg-1}
\end{equation}
as a skew-Hermitian operator in $\mathscr{H}_{F}$; see Lemma \ref{lem:lcg-F_varphi}.

\index{Brownian motion}

\index{Ornstein-Uhlenbeck}

\index{RKHS}

\index{positive definite}

\index{Gaussian processes}

\index{deficiency indices}

\index{operator!skew-Hermitian}

\subsection{\label{sub:green}Green's Functions}
\begin{lem}
~
\begin{enumerate}
\item \label{enu:F2-D}For $F_{2}\left(x\right)=1-\left|x\right|$, $\left|x\right|<\frac{1}{2}$,
let $\varphi\in C_{c}^{\infty}\left(0,\frac{1}{2}\right)$, then $ $$u\left(x\right):=\left(T_{F_{2}}\varphi\right)\left(x\right)$
satisfies 
\begin{equation}
\varphi=-\frac{1}{2}\bigl(\frac{d}{dx}\bigr)^{2}u.\label{eq:F2-D}
\end{equation}
Hence, 
\begin{equation}
T_{F_{2}}^{-1}\supset-\frac{1}{2}\bigl(\frac{d}{dx}\bigr)^{2}\Big|_{C_{c}^{\infty}\left(0,\frac{1}{2}\right)}.\label{eq:F2-D-ext}
\end{equation}

\item \label{enu:F3-D}For $F_{3}\left(x\right)=e^{-\left|x\right|}$, $\left|x\right|<1$,
let $\varphi\in C_{c}^{\infty}\left(0,1\right)$, then 
\begin{equation}
\varphi=\frac{1}{2}\bigl(I-\bigl(\frac{d}{dx}\bigr)^{2}\bigr)u.\label{eq:F3-D}
\end{equation}
Hence,
\begin{equation}
T_{F_{3}}^{-1}\supset\frac{1}{2}\bigl(I-\bigl(\frac{d}{dx}\bigr)^{2}\bigr)\Big|_{C_{c}^{\infty}\left(0,1\right)}.\label{eq:F3-D-ext}
\end{equation}

\end{enumerate}
\end{lem}
\begin{proof}
The computation for $F=F_{2}$ is as follows: Let $\varphi\in C_{c}^{\infty}\left(0,\frac{1}{2}\right)$,
then 
\begin{align*}
u\left(x\right)=\left(T_{F_{2}}\varphi\right)\left(x\right) & =\int_{0}^{\frac{1}{2}}\varphi\left(y\right)\left(1-\left|x-y\right|\right)dy\\
 & =\int_{0}^{x}\varphi\left(y\right)\left(1-\left(x-y\right)\right)dy+\int_{x}^{\frac{1}{2}}\varphi\left(y\right)\left(1-\left(y-x\right)\right)dy;
\end{align*}
and 
\[
\int_{0}^{x}\varphi\left(y\right)\left(1-\left(x-y\right)\right)dy=\int_{0}^{x}\varphi\left(y\right)dy-x\int_{0}^{x}\varphi\left(y\right)dy+\int_{0}^{x}y\varphi\left(y\right)dy
\]
\begin{eqnarray*}
u'\left(x\right) & = & -\int_{0}^{x}\varphi\left(y\right)+\varphi\left(x\right)+\int_{x}^{\frac{1}{2}}\varphi\left(y\right)dy-\varphi\left(x\right)\\
u''\left(x\right) & = & -2\varphi\left(x\right).
\end{eqnarray*}
Thus, $\varphi=-\frac{1}{2}u''$, and the desired result follows.

For $F_{3}$, let $\varphi\in C_{c}^{\infty}\left(0,1\right)$, then
\begin{eqnarray*}
u\left(x\right)=\left(T_{F_{3}}\varphi\right)\left(x\right) & = & \int_{0}^{1}e^{-\left|x-y\right|}\varphi\left(y\right)dy\\
 & = & \int_{0}^{x}e^{-\left(x-y\right)}\varphi\left(y\right)dy+\int_{x}^{1}e^{-\left(y-x\right)}\varphi\left(y\right)dy.
\end{eqnarray*}
Now, 
\begin{eqnarray*}
u'\left(x\right)=\left(T_{F_{3}}\varphi\right)'\left(x\right) & = & -e^{-x}\int_{0}^{x}e^{-y}\varphi\left(y\right)dy+\varphi\left(x\right)\\
 &  & +e^{x}\int_{x}^{1}e^{-y}\varphi\left(y\right)dy-\varphi\left(x\right)\\
 & = & -e^{-x}\int_{0}^{x}e^{y}\varphi\left(y\right)dy+e^{x}\int_{x}^{1}e^{-y}\varphi\left(y\right)dy\\
u'' & = & e^{-x}\int_{0}^{x}e^{y}\varphi\left(y\right)dy-\varphi\left(x\right)\\
 &  & +e^{x}\int_{x}^{1}e^{-y}\varphi\left(y\right)dy-\varphi\left(x\right)\\
 & = & -2\varphi+\int_{0}^{1}e^{-\left|x-y\right|}\varphi\left(y\right)dy\\
 & = & -2\varphi+T_{F_{3}}\left(\varphi\right);
\end{eqnarray*}
and then 
\begin{eqnarray*}
u''\left(x\right) & = & e^{-x}\int_{0}^{x}e^{y}\varphi\left(y\right)dy-\varphi\left(x\right)\\
 &  & +e^{x}\int_{x}^{1}e^{-y}\varphi\left(y\right)dy-\varphi\left(x\right)\\
 & = & -2\varphi\left(x\right)+\int_{0}^{1}e^{-\left|x-y\right|}\varphi\left(y\right)dy\\
 & = & -2\varphi+u\left(x\right).
\end{eqnarray*}
Thus, $\varphi=\frac{1}{2}$$\left(u-u''\right)=\frac{1}{2}\left(I-\frac{1}{2}\left(\frac{d}{dx}\right)^{2}\right)u$.
This proves (\ref{eq:F3-D}).\end{proof}
\begin{summary}[Conclusions for the two examples]

The computation for $F=F_{2}$ is as follows: If $\varphi\in L^{2}\left(0,\frac{1}{2}\right)$,
then $u\left(x\right):=\left(T_{F}\varphi\right)\left(x\right)$ satisfies
\[
\begin{array}[t]{cccc}
(F_{2}) &  &  & \varphi=\frac{1}{2}\bigl(-\bigl(\frac{d}{dx}\bigr)^{2}\bigr)u;\end{array}
\]
while, for $F=F_{3}$, the corresponding computation is as follows:
If $\varphi\in L^{2}\left(0,1\right)$, then $u\left(x\right)=\left(T_{F}\varphi\right)\left(x\right)$
satisfies
\[
\begin{array}[t]{cccc}
(F_{3}) &  &  & \varphi=\frac{1}{2}\bigl(I-\bigl(\frac{d}{dx}\bigr)^{2}\bigr)u;\end{array}
\]
For the operator $D^{\left(F\right)}$, in the case of $F=F_{2}$,
it follows that the Mercer operator $T_{F}$ plays the following role:
$T_{F}^{-1}$ is a selfadjoint extension of $-\frac{1}{2}\bigl(D^{\left(F\right)}\bigr)^{2}$.
In the case of $F=F_{3}$ the corresponding operator $T_{F}^{-1}$
(in the RKHS $\mathscr{H}_{F_{3}}$) is a selfadjoint extension of
$\frac{1}{2}\bigl(I-\bigl(D^{\left(F\right)}\bigr)^{2}\bigr)$; in
both cases, they are the Friedrichs extensions.\end{summary}
\begin{rem}
When solving boundary values for elliptic operators in a bounded domain,
say $\Omega\subset\mathbb{R}^{n}$, one often ends up with Green's
functions (= integral kernels) which are positive definite kernels,
so $K\left(x,y\right)$, defined on $\Omega\times\Omega$, not necessarily
of the form $K\left(x,y\right)=F\left(x\lyxmathsym{\textendash}y\right)$. 

But many of the questions we ask in the special case of p.d. functions,
so when the kernel is $K\left(x,y\right)=F\left(x-y\right)$ will
also make sense for p.d. kernels.
\end{rem}
\index{Friedrichs extension}

\index{operator!Mercer}

\subsection{\label{sub:F2}The Case of $F_{2}\left(x\right)=1-\left|x\right|$,
$x\in\left(-\frac{1}{2},\frac{1}{2}\right)$}

Let $F=F_{2}$. Fix $x\in\left(0,\frac{1}{2}\right)$, and set 
\begin{equation}
F_{x}\left(y\right)=F\left(x-y\right),\;\mbox{for }x,y\in{\textstyle \left(0,\frac{1}{2}\right)};\label{eq:RKHS-eg-2}
\end{equation}
where $F_{x}\left(\cdot\right)$ and its derivative (in the sense
of distributions) are as in \figref{Fx} (sketched for two values
of $x$).

Consider the Hilbert space
\begin{align}
\mathscr{H}_{F} & :=\begin{Bmatrix}h; & \text{continuous on \ensuremath{\left(0,\tfrac{1}{2}\right)}, and }h'=\tfrac{dh}{dx}\in L^{2}\left(0,\tfrac{1}{2}\right)\\
 & \text{where the derivative is in the weak sense}\qquad
\end{Bmatrix}\label{eq:RKHS-eg-3}
\end{align}
modulo constants; and let the norm, and inner-product, in $\mathscr{H}_{F}$
be given by 
\begin{equation}
\left\Vert h\right\Vert _{\mathscr{H}_{F}}^{2}=\frac{1}{2}\int_{0}^{\frac{1}{2}}\left|h'\left(x\right)\right|^{2}dx+\int_{\partial\Omega}\overline{h_{n}}h\, d\beta.\label{eq:RKHS-eg-4}
\end{equation}
On the RHS of (\ref{eq:RKHS-eg-4}), $d\beta$ denotes the corresponding
boundary measure, and $h_{n}$ is the inward normal derivative of
$h$. See Theorem \ref{thm:F2-bd} below. 

Then the reproducing kernel property is as follows:
\begin{equation}
\left\langle F_{x},h\right\rangle _{\mathscr{H}_{F}}=h\left(x\right),\;\forall h\in\mathscr{H}_{F},\forall x\in\left(0,\tfrac{1}{2}\right);\label{eq:RKHS-eg-5}
\end{equation}
and it follows that $\mathscr{H}_{F}$ is naturally isomorphic to
the RKHS for $F_{2}$ from  \subref{lcg}. 

\begin{figure}
\begin{tabular}{cc}
\includegraphics[scale=0.6]{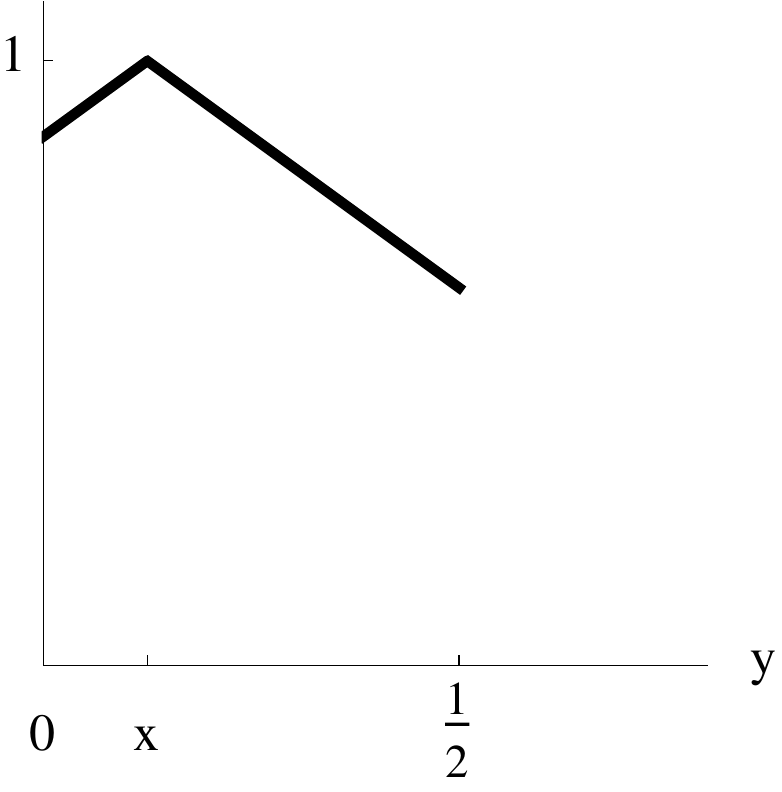} & \includegraphics[scale=0.6]{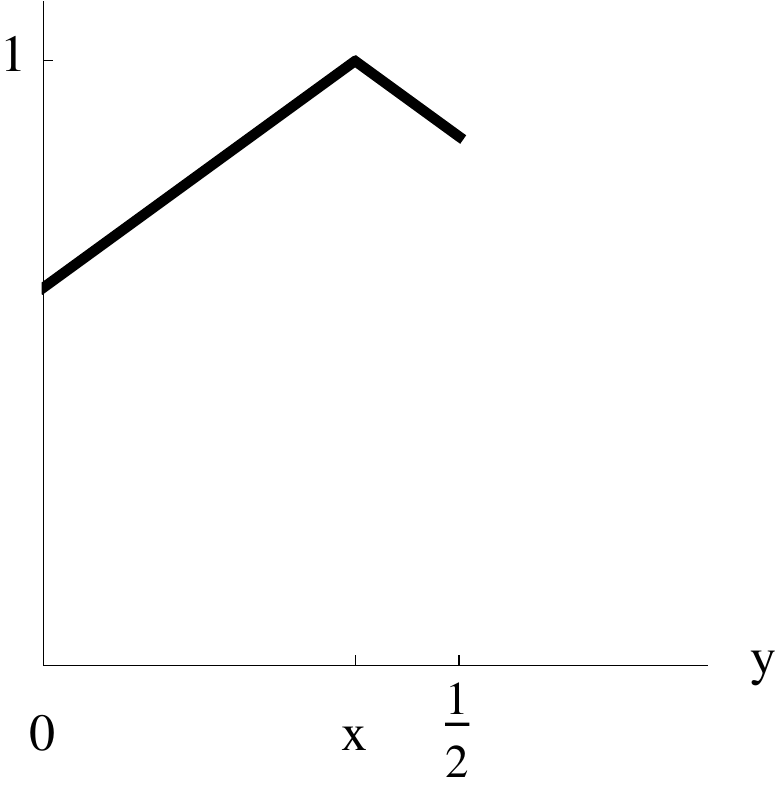}\tabularnewline
\includegraphics[scale=0.6]{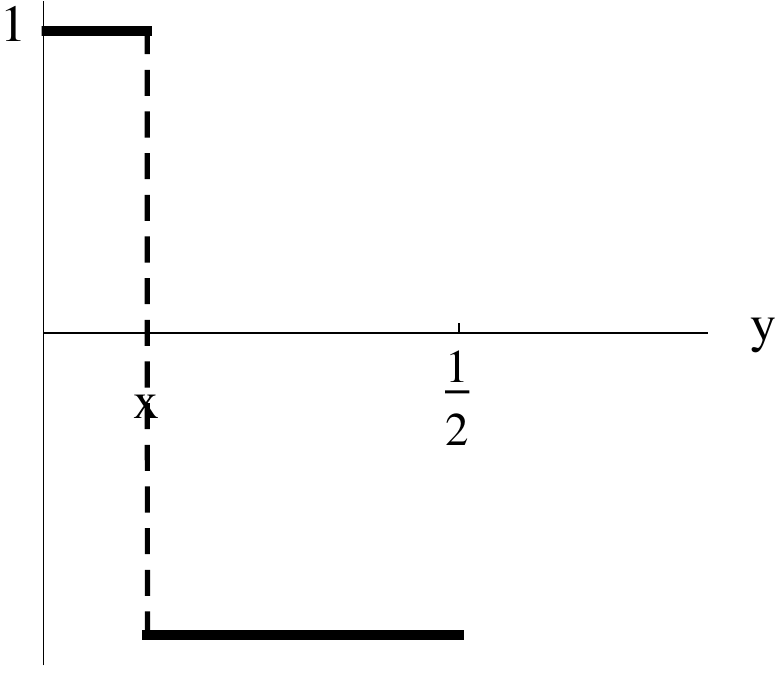} & \includegraphics[scale=0.6]{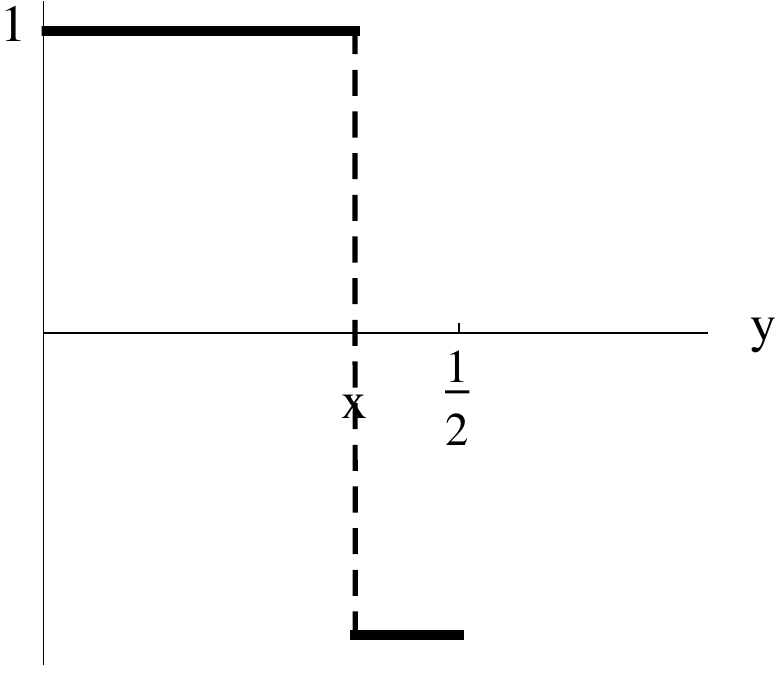}\tabularnewline
\end{tabular}

\protect\caption{\label{fig:Fx}The kernel $F_{x}$ and its derivative (the case of
$F_{2}$)}
\end{figure}

\begin{thm}
\label{thm:F2-bd}The boundary measure for $F=F_{2}$ (see (\ref{eq:RKHS-eg-4}))
is 
\[
\beta=\frac{1}{2}\left(\delta_{0}+\delta_{1/2}\right).
\]
\end{thm}
\begin{proof}
Set 
\begin{equation}
E\left(\xi\right):=\frac{1}{2}\int_{0}^{\frac{1}{2}}\left|\xi'\left(x\right)\right|^{2}dx,\;\forall\xi\in\mathscr{H}_{F}.\label{eq:en-2-1}
\end{equation}
And let $F_{x}\left(\cdot\right):=1-\left|x-\cdot\right|$ defined
on $\left[0,\frac{1}{2}\right]$, for all $x\in\left(0,\frac{1}{2}\right)$.
Then 
\begin{eqnarray*}
E\left(F_{x},\xi\right) & = & \frac{1}{2}\int_{0}^{\frac{1}{2}}F_{x}'\left(y\right)\xi'\left(y\right)dy\\
 & = & \frac{1}{2}\left(\int_{0}^{x}\xi'\left(y\right)dy-\int_{x}^{\frac{1}{2}}\xi'\left(y\right)dy\right)\\
 & = & \xi\left(x\right)-\frac{\xi\left(0\right)+\xi\left(\frac{1}{2}\right)}{2}.\;(\mbox{see Fig. \ref{fig:Fx}})
\end{eqnarray*}
Since
\[
\left\Vert \xi\right\Vert _{\mathscr{H}_{F}}^{2}=E\left(\xi\right)+\int\left|\xi\right|^{2}d\beta
\]
we get 
\[
\left\langle F_{x},\xi\right\rangle _{\mathscr{H}_{F}}=\xi\left(x\right),\;\forall\xi\in\mathscr{H}_{F}.
\]
We conclude that 
\begin{equation}
\left\langle \xi,\eta\right\rangle _{\mathscr{H}_{F}}=E\left(\xi,\eta\right)+\int_{\partial\Omega}\overline{\xi_{n}}\eta d\beta;\label{eq:en-3-1}
\end{equation}
note the boundary in this case consists two points, $x=0$ and $x=\frac{1}{2}$.\end{proof}
\begin{rem}
\label{rem:bm}The energy form in (\ref{eq:en-2-1}) also defines
a RKHS (Wiener's energy form for Brownian motion, see figure \ref{fig:bm2})
as follows: 

On the space of all continuous functions, $\mathscr{C}\left([0,\frac{1}{2}]\right)$,
set 
\begin{equation}
\mathscr{H}_{\mathscr{E}}:=\left\{ f\in\mathscr{C}\left([0,\tfrac{1}{2}]\right)\:\big|\:\mathscr{E}\left(f\right)<\infty\right\} \label{eq:en-2-2}
\end{equation}
modulo constants, where
\begin{equation}
\mathscr{E}\left(f\right)=\int_{0}^{\frac{1}{2}}\left|f'\left(x\right)\right|^{2}dx.\label{eq:en-2-3}
\end{equation}
For $x\in\left[0,\frac{1}{2}\right]$, set 
\[
E_{x}\left(y\right)=x\wedge y=\min\left(x,y\right),\; y\in\left(0,\tfrac{1}{2}\right);
\]
see \figref{bm1}; then $E_{x}\in\mathscr{H}_{\mathscr{E}}$, and
\begin{equation}
\left\langle E_{x},f\right\rangle _{\mathscr{H}_{\mathscr{E}}}=f\left(x\right),\;\forall f\in\mathscr{H}_{\mathscr{E}},\forall x\in\left[0,\tfrac{1}{2}\right].\label{eq:en-2-4}
\end{equation}

\end{rem}
\begin{figure}[H]
\includegraphics[scale=0.5]{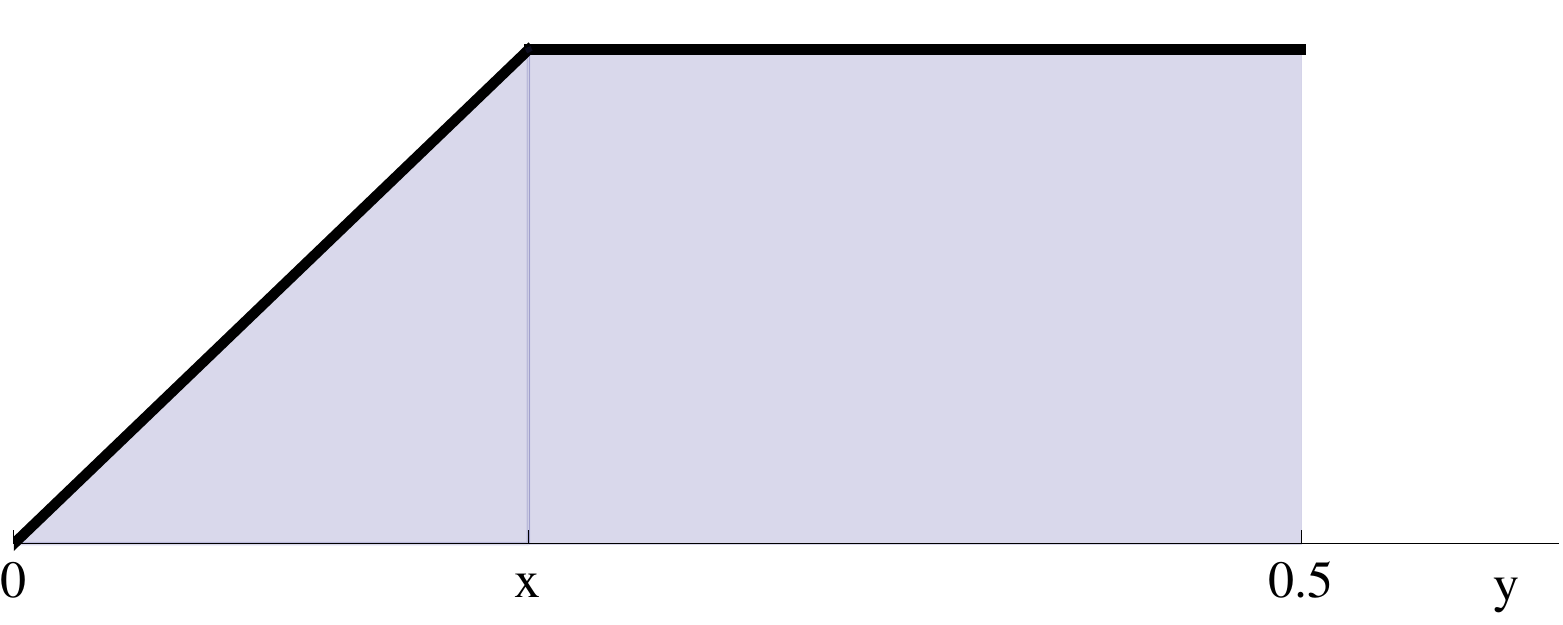}

\protect\caption{\label{fig:bm1}The covariance function $E_{x}\left(\cdot\right)=\min\left(x,\cdot\right)$
of Brownian motion.}
\end{figure}

\begin{proof}
For the reproducing property (\ref{eq:en-2-4}): Let $f$ and $x$
be as stated; then 
\begin{eqnarray*}
\left\langle E_{x},f\right\rangle _{\mathscr{H}_{\mathscr{E}}} & = & \int_{0}^{\frac{1}{2}}E_{x}'\left(y\right)f'\left(y\right)dy\\
 & \underset{\text{Fig. \ensuremath{\left(\ref{fig:bm1}\right)}}}{=} & \int_{0}^{\frac{1}{2}}\chi_{\left[0,x\right]}\left(y\right)f'\left(y\right)dy\\
 & = & \int_{0}^{x}f'\left(y\right)dy=f\left(x\right)-f\left(0\right).
\end{eqnarray*}
Note in (\ref{eq:en-2-2}) we define $\mathscr{H}_{\mathscr{E}}$
modulo constants; alternatively, we may stipulate $f\left(0\right)=0$. 
\end{proof}
Note that the Brownian motion\index{Brownian motion} RKHS is not
defined by a p.d. function, but rather by a p.d. kernel. Nonetheless
the remark explains its connection to our present RKHS $\mathscr{H}_{F}$
which is defined by the p.d. function, namely the p.d. function $F_{2}$.

\begin{figure}[H]
\includegraphics[scale=0.6]{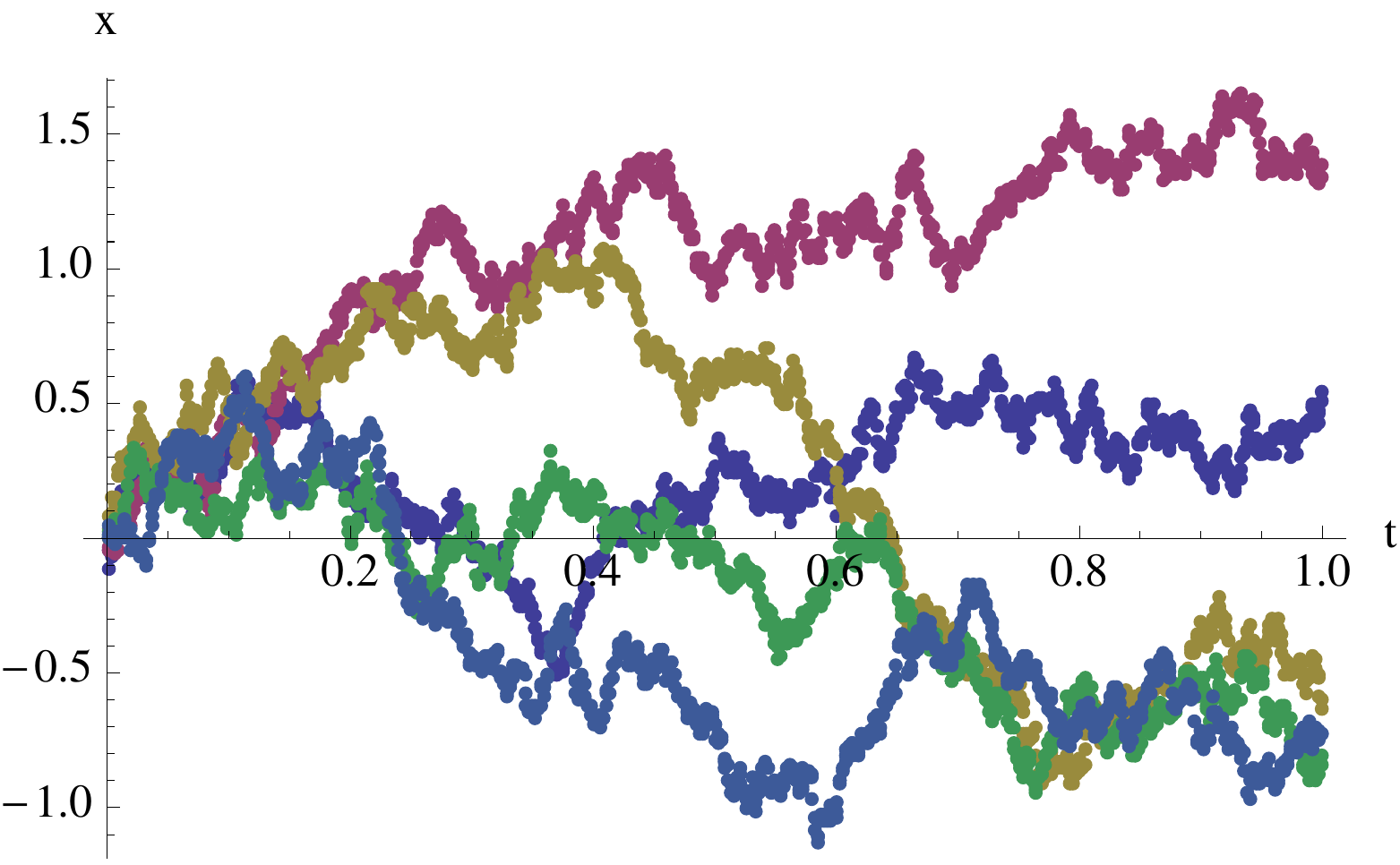}

\protect\caption{\label{fig:bm2}Monte-Carlo simulation of Brownian motion starting
at $x=0$, with 5 sample paths. \index{Brownian motion}}
\end{figure}

\textbf{Pinned Brownian Motion. }\index{Brownian bridge}

We illustrate the boundary term in Theorem \ref{thm:F2-bd}, eq. (\ref{eq:en-3-1})
with pinned Brownian motion (also called ``Brownian bridge.'') In
order to simplify constructions, we pin the Brownian motion at the
following two points in $\left(t,x\right)$ space, $\left(t,x\right)=\left(0,0\right)$,
and $\left(t,x\right)=\left(1,1\right)$; see \figref{bbridge}. To
simplify computations further, we restrict attention to real valued
functions only. 

\begin{figure}[H]
\includegraphics[scale=0.6]{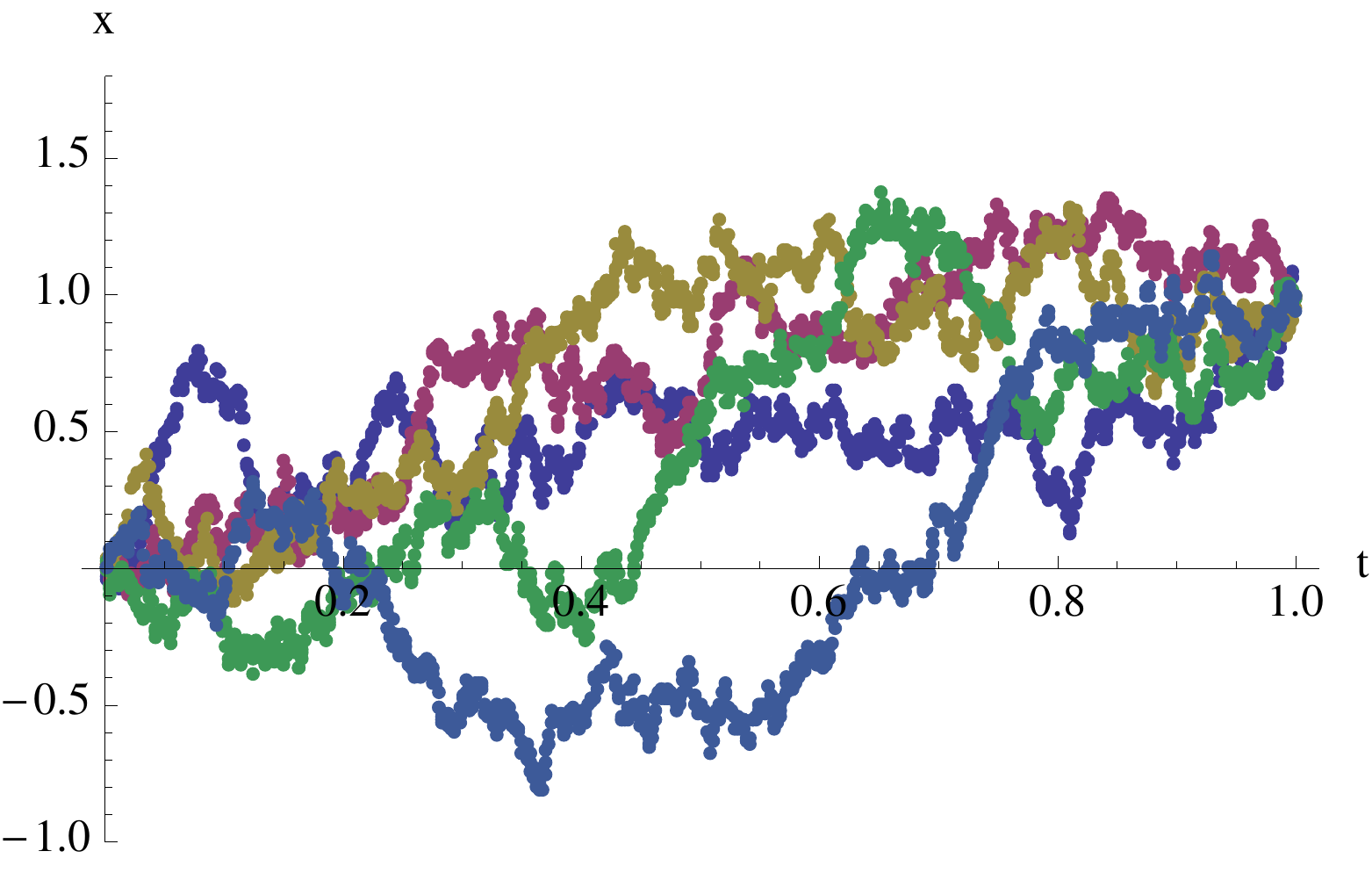}

\protect\caption{\label{fig:bbridge}Monte-Carlo simulation of Brownian bridge pinned
at $\left(0,0\right)$ and $\left(1,1\right)$, with 5 sample paths.}
\end{figure}

For background on stochastic processes, the book \cite{IkWa89} serves
our purpose; see especially p. 243-244. \index{stochastic processes}

For the pinning down the process $X_{t}$ at the two points $\left(0,0\right)$
and $\left(1,1\right)$ as in \figref{bbridge}, we have 
\begin{equation}
X_{t}=t+\left(1-t\right)\int_{0}^{t}\frac{dB_{s}}{1-s},\;0<t<1\label{eq:bb-1}
\end{equation}
where $dB_{s}$ on the RHS of (\ref{eq:bb-1}) refers to the standard
Brownian motion $dB_{s}$, and the second term in (\ref{eq:bb-1})
is the corresponding Ito-integral; so we have
\begin{eqnarray}
\mathbb{E}\left(\int_{0}^{t}\frac{dB_{s}}{1-s}\right) & = & 0,\mbox{ and}\label{eq:bb-2}\\
\mathbb{E}\left(\left|\int_{0}^{t}\frac{dB_{s}}{1-s}\right|^{2}\right) & = & \int_{0}^{t}\frac{ds}{\left(1-s\right)^{2}}=\frac{t}{1-t};\label{eq:bb-3}
\end{eqnarray}
where $\mathbb{E}\left(\cdots\right)=$ expectation with respect to
the underlying probability space $\left(\Omega,\mathscr{F},\mathbb{P}\right)$,
i.e., $\mathbb{E}\left(\cdots\right)=\int_{\Omega}\cdots d\mathbb{P}$. 

Hence, for the mean and covariance\index{covariance}, of the process
$X_{t}$ in (\ref{eq:bb-1}), we get 
\begin{equation}
\mathbb{E}\left(X_{t}\right)=t,\;\forall t,\,0<t<1;\mbox{ and }\label{eq:bb-4}
\end{equation}
\begin{equation}
Cov\left(X_{t}X_{s}\right)=\mathbb{E}\left(\left(X_{t}-t\right)\left(X_{s}-s\right)\right)=s\wedge t-st;\label{eq:bb-5}
\end{equation}
where $s\wedge t=\min\left(s,t\right)$, and $s,t\in\left(0,1\right)$. 

And it follows in particular that the function on the RHS in eq (\ref{eq:bb-5})
is a positive definite kernel. Its connection to $F_{2}$ is given
in the next lemma.

Now return to the p.d. function $F=F_{2}$; i.e., $F\left(t\right)=1-\left|t\right|$,
and therefore, $F\left(s-t\right)=1-\left|s-t\right|$, $s,t\in\left(0,1\right)$,
and let $\mathscr{H}_{F}$ be the corresponding RKHS. 

Let $\left\{ F_{t}\right\} _{t\in\left(0,1\right)}$ denote the kernels
in $\mathscr{H}_{F}$, i.e., $\left\langle F_{t},g\right\rangle _{\mathscr{H}_{F}}=g\left(t\right)$,
$t\in\left(0,1\right)$, $g\in\mathscr{H}_{F}$. 
\begin{lem}
\label{lem:bbridge}Let $\left(X_{t}\right)_{t\in\left(0,1\right)}$
denote the pinned Brownian motion (\ref{eq:bb-1}); then 
\begin{equation}
\left\langle F_{s},F_{t}\right\rangle _{\mathscr{H}_{F}}=Cov\left(X_{s}X_{t}\right);\label{eq:bb-6}
\end{equation}
see (\ref{eq:bb-5}); and 
\begin{equation}
\left\langle F_{s},F_{t}\right\rangle _{energy}=s\wedge t;\label{eq:bb-7}
\end{equation}
while the boundary term 
\begin{equation}
bdr\left(s,t\right)=-st.\label{eq:bb-8}
\end{equation}
 \end{lem}
\begin{proof}
With our normalization from Fig \ref{fig:bbridge}, we must take the
energy form as follows:
\begin{equation}
\left\langle f,g\right\rangle _{energy}=\int_{0}^{1}f'\left(x\right)g'\left(x\right)dx.\label{eq:bb-9}
\end{equation}

Set $F_{s}\left(y\right)=s\wedge y$, see Figure \ref{fig:bm1}. For
the distributional derivative we have 
\[
F_{s}'\left(y\right)=\chi_{\left[0,s\right]}\left(y\right)=\begin{cases}
1 & \mbox{if }0\leq y\leq s\\
0 & \mbox{else},
\end{cases}
\]
then 
\begin{eqnarray*}
\left\langle F_{s},F_{t}\right\rangle _{energy} & = & \int_{0}^{1}F_{s}'\left(y\right)F_{t}'\left(y\right)dy\\
 & = & \int_{0}^{1}\chi_{\left[0,s\right]}\left(y\right)\chi_{\left[0,t\right]}\left(y\right)dy\\
 & = & \left|\left[0,s\right]\cap\left[0,t\right]\right|_{\text{Lebesgue measure}}\\
 & = & s\wedge t.
\end{eqnarray*}
The desired conclusions (\ref{eq:bb-7})-(\ref{eq:bb-8}) in the lemma
now follow. See Remark \ref{rem:bm}.

The verification of (\ref{eq:bb-6}) uses Ito-calculus as follows:
Note that (\ref{eq:bb-1}) for $X_{t}$ is the solution to the following
Ito-equation:
\begin{equation}
dX_{t}=\left(\frac{X_{t}-1}{t-1}\right)dt+dB_{t};\label{eq:bb-11}
\end{equation}
and by Ito's lemma, therefore, 
\begin{equation}
\left(dX_{t}\right)^{2}=\left(dB_{t}\right)^{2}.\label{eq:bb-12}
\end{equation}

As a result, if $f:\mathbb{R}\rightarrow\mathbb{R}$ is a function
in the energy-Hilbert space defined from (\ref{eq:bb-9}), then 
\[
\mathbb{E}\left(\left|\left(df\right)\left(X_{t}\right)\right|^{2}\right)=\mathbb{E}\left(\left|\left(df\right)\left(B_{t}\right)\right|^{2}\right)=\int_{0}^{1}\left|f'\left(t\right)\right|^{2}dt=\mathscr{E}_{2}\left(f\right).
\]
\end{proof}
\begin{rem}
In (\ref{eq:bb-11})-(\ref{eq:bb-12}), we use standard conventions
for Brownian motions $B_{t}$: Let $\left(\Omega,\mathscr{F},\mathbb{P}\right)$
be a choice of probability space for $\left\{ B_{t}\right\} _{t\in\mathbb{R}}$,
(or $t\in\left[0,1\right]$.) With $E\left(\cdots\right)=\int_{\Omega}\cdots d\mathbb{P}$,
we have 
\begin{equation}
E\left(B_{s}B_{t}\right)=s\wedge t=\min\left(s,t\right),\; t\in\left[0,1\right].\label{eq:bm-1-1}
\end{equation}
If $f:\mathbb{R}\rightarrow\mathbb{R}$ is a $C^{1}$-function, we
write $f\left(B_{t}\right)$ for $f\circ B_{t}$; and $df\left(B_{t}\right)$
refers to Ito-calculus.
\end{rem}

\subsection{\label{sub:F3}The Case of $F_{3}\left(x\right)=e^{-\left|x\right|}$,
$x\in\left(-1,1\right)$}

Let $F_{x}\left(y\right)=F\left(x-y\right)$, for all $x,y\in\left(0,1\right)$;
and consider the Hilbert space 
\begin{align}
\mathscr{H}_{F} & :=\begin{Bmatrix}h; & \text{continuous on \ensuremath{\left(0,1\right)}, }h\in L^{2}\left(0,1\right)\text{, and}\\
 & \text{the weak derivative }h'\in L^{2}\left(0,1\right)\,\,\,\,\,\,\,\,\,\,\,\,\,\,\,
\end{Bmatrix};\label{eq:RKHS-eg-6}
\end{align}
and let the $\mathscr{H}_{F}$-norm, and inner product, be given by
\begin{equation}
\left\Vert h\right\Vert _{\mathscr{H}_{F}}=\frac{1}{2}\left(\int_{0}^{1}\left|h'\left(x\right)\right|^{2}dx+\int_{0}^{1}\left|h\left(x\right)\right|^{2}dx\right)+\int_{\partial\Omega}\overline{h_{n}}h\, d\beta.\label{eq:RKHS-eg-7}
\end{equation}
Here, $d\beta$ on the RHS of (\ref{eq:RKHS-eg-7}) denotes the corresponding
boundary measure, and $h_{n}$ is the inward normal derivative of
$h$. See Theorem \ref{thm:F3bd} below.

Then a direct verification yields:
\begin{equation}
\left\langle F_{x},h\right\rangle _{\mathscr{H}_{F}}=h\left(x\right),\;\forall h\in\mathscr{H}_{F},\forall x\in\left(0,1\right);\label{eq:RKHS-eg-8}
\end{equation}
and it follows that $\mathscr{H}_{F}$ is naturally isomorphic to
RKHS for $F_{3}$ from \subref{lcg}.

For details of (\ref{eq:RKHS-eg-8}), see also \cite{Jor81}.
\begin{cor}
\label{cor:RKHS-eg-1}In both $\mathscr{H}_{F_{i}}$, $i=2,3$, the
deficiency indices are $\left(1,1\right)$.\end{cor}
\begin{proof}
In both cases, we are referring to the skew-Hermitian operator $D^{\left(F_{i}\right)}$
in $\mathscr{H}_{Fi}$, $i=2,3$; see (\ref{eq:RKHS-eg-1}) above.
But it follows from (\ref{eq:RKHS-eg-4}) and (\ref{eq:RKHS-eg-7})
for the respective inner products, that the functions $e^{\pm x}$
have finite positive norms in the respective RKHSs.\end{proof}
\begin{thm}
\label{thm:F3bd}Let $F=F_{3}$ as before. Consider the energy Hilbert
space
\begin{equation}
\mathscr{H}_{E}:=\begin{Bmatrix}\ensuremath{f\in C\left[0,1\right]}\,\big|\, f,f'\in L^{2}\left(0,1\right)\text{ where }\\
f'\text{ is the weak derivative of }f
\end{Bmatrix}\label{eq:en-1}
\end{equation}
with 
\begin{align}
\left\langle f,g\right\rangle _{E} & =\frac{1}{2}\left(\int_{0}^{1}\overline{f'\left(x\right)}g'\left(x\right)dx+\int_{0}^{1}\overline{f\left(x\right)}g\left(x\right)dx\right);\;\text{and so}\label{eq:en-2}\\
\left\Vert f\right\Vert _{E}^{2} & =\frac{1}{2}\left(\int_{0}^{1}\left|f'\left(x\right)\right|^{2}dx+\int_{0}^{1}\left|f\left(x\right)\right|^{2}dx\right),\;\forall f,g\in\mathscr{H}_{E}.\label{eq:en-3}
\end{align}
Set 
\begin{align}
P\left(f,g\right) & =\int_{0}^{1}\overline{f_{n}\left(x\right)}g\left(x\right)d\beta\left(x\right),\;\text{where }\label{eq:en-4}\\
\beta & :=\frac{\delta_{0}+\delta_{1}}{2},\text{ i.e., Dirac masses at endpoints}\label{eq:en-5}
\end{align}
where $g_{n}$ denotes the inward normal derivative at endpoints. 

Let 
\begin{equation}
\mathscr{H}_{F}:=\begin{Bmatrix}f\in C\left[0,1\right]\,\big|\,\left\Vert f\right\Vert _{E}^{2}+P_{2}\left(f\right)<\infty\text{ where}\\
P_{2}\left(f\right):=P\left(f,f\right)=\int_{0}^{1}\left|f\left(x\right)\right|^{2}d\beta\left(x\right)
\end{Bmatrix}\label{eq:en-6}
\end{equation}
Then we have following:
\begin{enumerate}
\item As a Green-Gauss-Stoke principle, we have 
\begin{align}
\left\Vert \cdot\right\Vert _{F}^{2} & =\left\Vert \cdot\right\Vert _{E}^{2}+P_{2}\:\text{i.e.,}\label{eq:en-7}\\
\left\Vert f\right\Vert _{F}^{2} & =\left\Vert f\right\Vert _{E}^{2}+\int_{0}^{1}\left|f\left(x\right)\right|^{2}d\beta\left(x\right).\label{eq:en-8}
\end{align}

\item Moreover, 
\begin{equation}
\left\langle F_{x},g\right\rangle _{E}=g\left(x\right)-\frac{e^{-x}g\left(0\right)+e^{-\left(1-x\right)}g\left(1\right)}{2}.\label{eq:en-9}
\end{equation}

\item As a result of (\ref{eq:en-9}), eqs. (\ref{eq:en-7})-(\ref{eq:en-8})
is the reproducing property in $\mathscr{H}_{F}$. Specifically, we
have
\begin{equation}
\left\langle F_{x},g\right\rangle _{\mathscr{H}_{F}}=g\left(x\right),\;\forall g\in\mathscr{H}_{F}.\:\text{(see (\ref{eq:en-6}))}\label{eq:en-10}
\end{equation}

\end{enumerate}
\end{thm}
\begin{proof}
Given $g\in\mathscr{H}_{F}$, one checks directly that 
\begin{eqnarray*}
P\left(F_{x},g\right) & = & \int_{0}^{1}F_{x}\left(y\right)g\left(y\right)d\beta\left(y\right)\\
 & = & \frac{F_{x}\left(0\right)g\left(0\right)+F_{x}\left(1\right)g\left(1\right)}{2}\\
 & = & \frac{e^{-x}g\left(0\right)+e^{-\left(1-x\right)}g\left(1\right)}{2};
\end{eqnarray*}
and, using integration by parts, we have 
\begin{eqnarray*}
\left\langle F_{x},g\right\rangle _{E} & = & g\left(x\right)-\frac{e^{-x}g\left(0\right)+e^{-\left(1-x\right)}g\left(1\right)}{2}\\
 & = & \left\langle F_{x},g\right\rangle _{F}-P\left(F_{x},g\right).
\end{eqnarray*}

Now, using 
\begin{equation}
\int_{\partial\Omega}\left(F_{x}\right)_{n}gd\beta:=\frac{e^{-x}g\left(0\right)+e^{-\left(1-x\right)}g\left(1\right)}{2},\label{eq:en-13}
\end{equation}
and $\left\langle F_{x},g\right\rangle _{\mathscr{H}_{F}}=g\left(x\right)$,
$\forall x\in\left(0,1\right)$, and using (\ref{eq:en-9}), the desired
conclusion
\begin{equation}
\left\langle F_{x},g\right\rangle _{\mathscr{H}_{F}}=\underset{\text{energy term}}{\underbrace{\left\langle F_{x},g\right\rangle _{E}}}+\int_{\partial\Omega}\left(F_{x}\right)_{n}gd\beta\label{eq:en-1-1}
\end{equation}
follows. Since the $\mathscr{H}_{F}$-norm closure of the span of
$\left\{ F_{x}\:\big|\: x\in\left(0,1\right)\right\} $ is all of
$\mathscr{H}_{F}$, from (\ref{eq:en-1-1}), we further conclude that
\begin{equation}
\left\langle f,g\right\rangle _{\mathscr{H}_{F}}=\left\langle f,g\right\rangle _{E}+\int_{\partial\Omega}\overline{f_{n}}gd\beta\label{eq:en-1-2}
\end{equation}
holds for all $f,g\in\mathscr{H}_{F}$.

In (\ref{eq:en-1-1}) and (\ref{eq:en-1-2}), we used $f_{n}$ to
denote the inward normal derivative, i.e., $f_{n}\left(0\right)=f'\left(0\right)$,
and $f_{n}\left(1\right)=-f'\left(1\right)$, $\forall f\in\mathscr{H}_{E}$. \end{proof}
\begin{example}
Fix $p$, $0<p\leq1$, and set 
\begin{equation}
F_{p}\left(x\right):=1-\left|x\right|^{p},\; x\in\left(-\tfrac{1}{2},\tfrac{1}{2}\right);\label{eq:Fp-1}
\end{equation}
see \figref{Fp}. Then $F_{p}$ is positive definite and continuous;
and so $-F''_{p}\left(x-y\right)$ is a p.d. kernel, so $-F_{p}''$
is a positive definite distribution on $\left(-\frac{1}{2},\frac{1}{2}\right)$. 

We saw that if $p=1$, then 
\begin{equation}
-F_{1}''=2\delta\label{eq:Fp-2}
\end{equation}
where $\delta$ is the Dirac mass at $x=0$. But for $0<p<1$, $-F_{p}''$
does not have the form (\ref{eq:Fp-2}). We illustrate this if $p=\frac{1}{2}$.
Then 
\begin{equation}
-F_{\frac{1}{2}}''=\chi_{\left\{ x\neq0\right\} }\frac{1}{4}\left|x\right|^{-\frac{3}{2}}+\frac{1}{4}\delta'',\label{eq:Fp-3}
\end{equation}
where $\delta''$ is the double derivative of $\delta$ in the sense
of distributions. 

\begin{figure}
\includegraphics[scale=0.4]{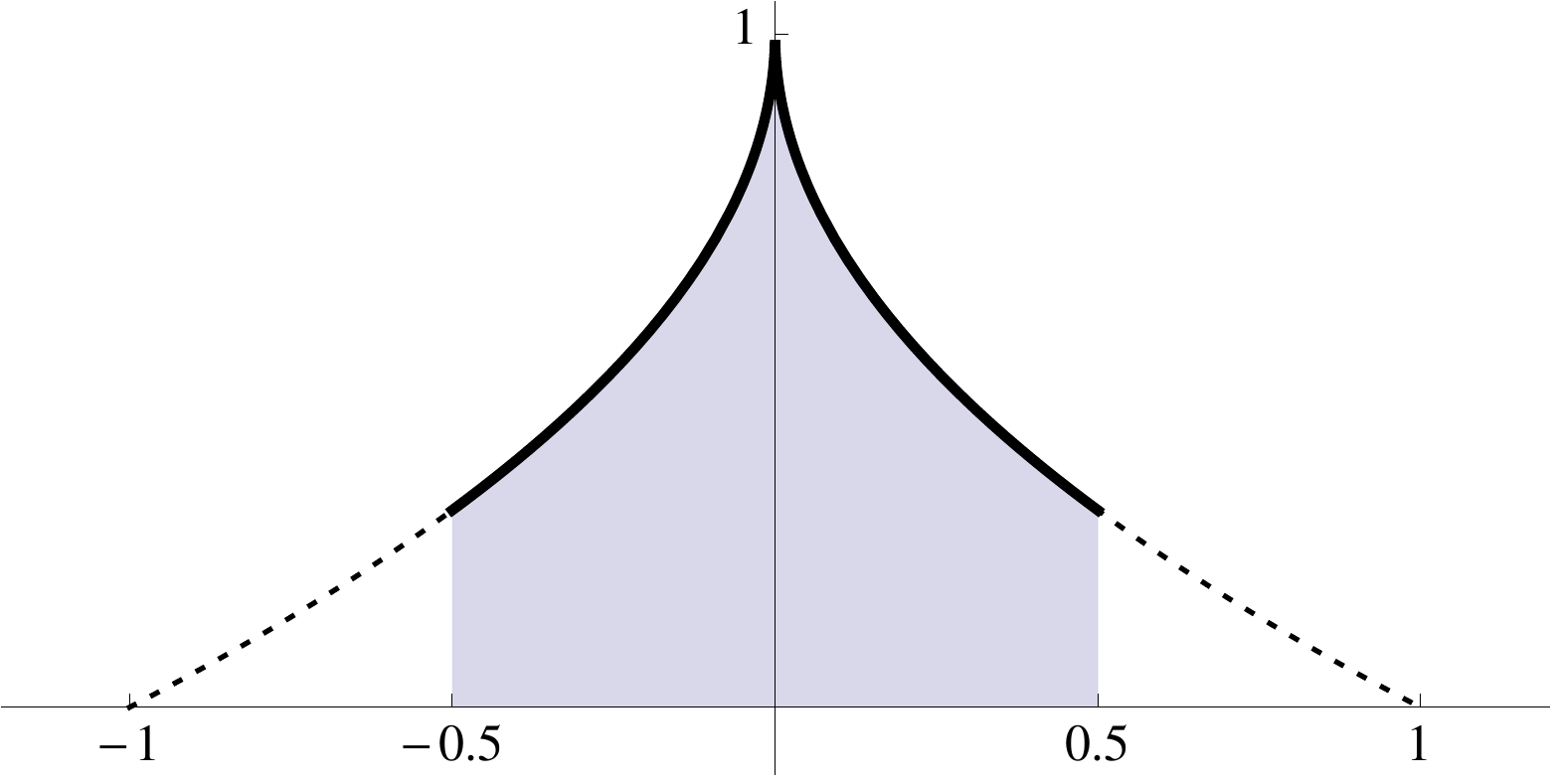}

\protect\caption{\label{fig:Fp}$F_{p}\left(x\right)=1-\left|x\right|^{p}$, $x\in\left(-\tfrac{1}{2},\tfrac{1}{2}\right)$
and $0<p\leq1$.}

\end{figure}
\end{example}
\begin{rem}
There is a notion of boundary measure in potential theory. Boundary
measures exist for any potential theory. In our case it works even
more generally, whenever p.d. $F_{ex}$ for bounded domains $\Omega\subset\mathbb{R}^{n}$,
and even Lie groups. But in the example of $F_{3}$, the boundary
is two points.\index{measure!boundary}

In all cases, we get $\mathscr{H}_{F}$ as a perturbation of the energy
form:\index{perturbation} 
\begin{equation}
\left\Vert \cdot\right\Vert _{\mathscr{H}_{F}}^{2}=\text{energy form }+\text{perturbation}\label{eq:en-11}
\end{equation}
It is a Green-Gauss-Stoke principle. There is still a boundary measure
for ALL bounded domains $\Omega\subset\mathbb{R}^{n}$, and even Lie
groups. 

And RKHS form 
\begin{equation}
\left\Vert \cdot\right\Vert _{\mathscr{H}_{F}}^{2}=\text{energy form }+\int_{\partial\Omega}f_{n}f\, d\mu_{\text{bd meas.}}\label{eq:en-12}
\end{equation}

For the general theory of boundary measures and their connection to
the Green-Gauss-Stoke principle, we refer readers to \cite{JoPe13,JoPe13b,Mo12,Bat90,Tel83,ACF09}. 

The approach via $\left\Vert \cdot\right\Vert _{\mathscr{H}_{F}}^{2}=$
\textquotedbl{}energy term + boundary term\textquotedbl{}, does not
fail to give a RKHS, but we must replace \textquotedbl{}energy term\textquotedbl{}
with an abstract Dirichlet form; see Refs \cite{HS12,Tre88}.
\end{rem}

\subsection{Integral Kernels and Positive Definite Functions}

Let $0<H<1$ be given, and set \index{fractional Brownian motion}
\begin{equation}
K_{H}\left(x,y\right)=\frac{1}{2}\left(\left|x\right|^{2H}+\left|y\right|^{2H}-\left|x-y\right|^{2H}\right),\;\forall x,y\in\mathbb{R}.\label{eq:pdH-1-1}
\end{equation}
It is known that $K_{H}$$\left(\cdot,\cdot\right)$ is the covariance
kernel for \uline{fractional} Brownian motion \cite{AJ12,AJL11,AL08,Aur11}.
The special case $H=\frac{1}{2}$ is Brownian motion; and if $H=\frac{1}{2}$,
then 
\[
K_{\frac{1}{2}}\left(x,y\right)=\left|x\right|\wedge\left|y\right|=\min\left(\left|x\right|,\left|y\right|\right).
\]
See Figure \ref{fig:Hker}. 

Set 
\begin{eqnarray*}
\widetilde{F}_{H}\left(x,y\right) & := & 1-\left|x-y\right|^{2H}\\
F_{H}\left(x\right) & := & 1-\left|x\right|^{2H}
\end{eqnarray*}
and we recover $F\left(x\right)=1-\left|x\right|\left(=F_{2}\right)$
as a special case of $H=\frac{1}{2}$. 

\begin{figure}[H]
\begin{tabular}{c}
\includegraphics[scale=0.4]{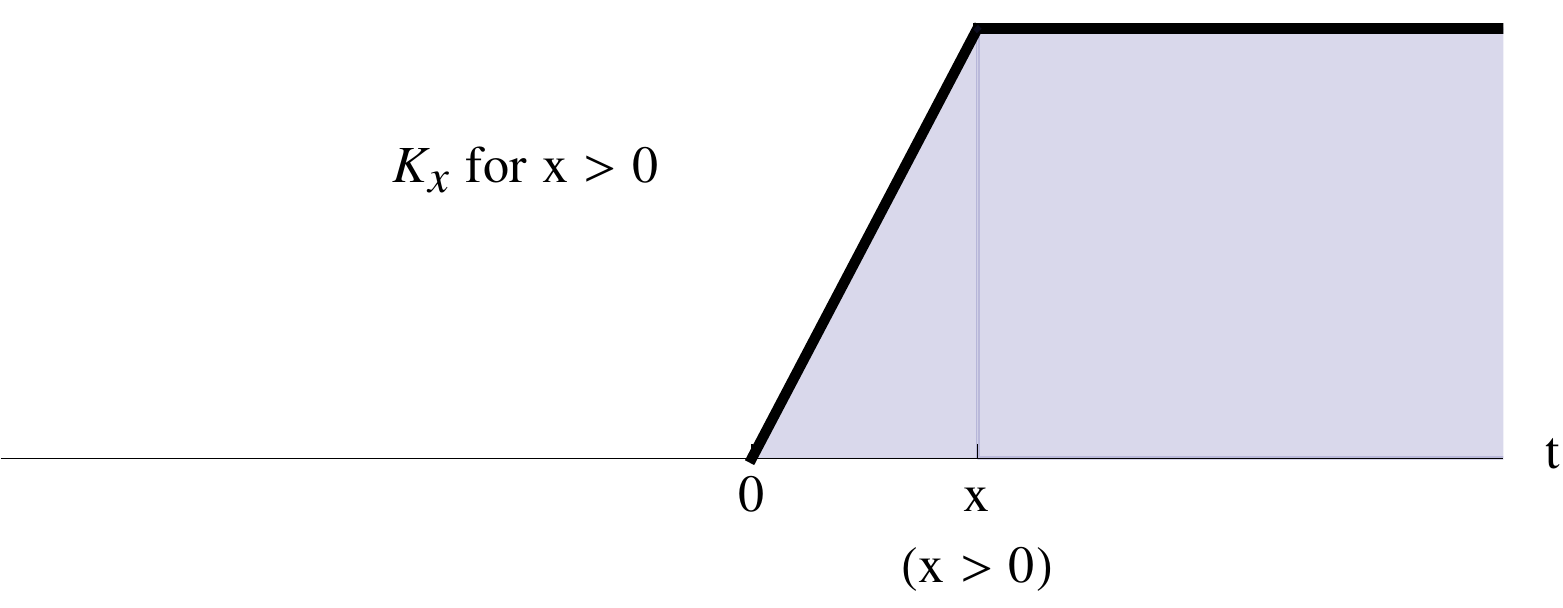}\tabularnewline
\includegraphics[scale=0.4]{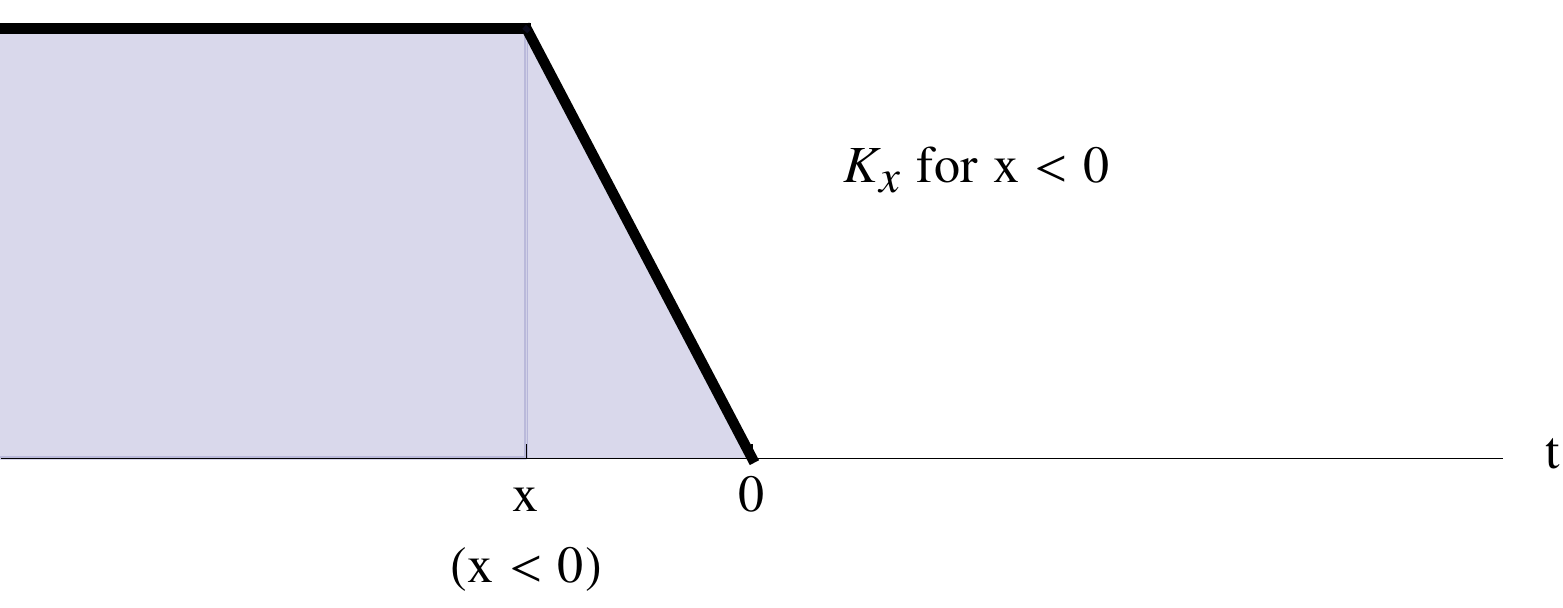}\tabularnewline
\end{tabular}

\protect\caption{\label{fig:Hker}The integral kernel $K_{\frac{1}{2}}\left(x,y\right)=\left|x\right|\wedge\left|y\right|$.}
\end{figure}

Solving for $F_{H}$ in (\ref{eq:pdH-1-1}) ($0<H<1$ fixed), we then
get 
\begin{equation}
F_{H}\left(x-y\right)=\underset{L_{H}}{\underbrace{1-\left|x\right|^{2H}-\left|y\right|^{2H}}}+2K_{H}\left(x,y\right);\label{eq:pdH-1-2}
\end{equation}
and specializing to $x,y\in\left[0,1\right]$; we then get 
\[
F\left(x-y\right)=\underset{L}{\underbrace{1-x-y}}+2\left(x\wedge y\right);
\]
or written differently:
\begin{equation}
K_{F}\left(x,y\right)=L\left(x,y\right)+2E\left(x,y\right)\label{eq:pdH-1-3}
\end{equation}
where $E\left(x,y\right)=x\wedge y$ is the familiar covariance kernel
for Brownian motion.

Introducing the Mercer integral kernels corresponding to (\ref{eq:pdH-1-2}),
we therefore get:
\begin{eqnarray}
\left(T_{F_{H}}\varphi\right)\left(x\right) & = & \int_{0}^{1}\varphi\left(y\right)F_{H}\left(x-y\right)dy\label{eq:pdH-1-5}\\
\left(T_{L_{H}}\varphi\right)\left(x\right) & = & \int_{0}^{1}\varphi\left(y\right)L_{H}\left(x,y\right)dy\label{eq:pdH-1-6}\\
\left(T_{K_{H}}\varphi\right)\left(x\right) & = & \int_{0}^{1}\varphi\left(y\right)K_{H}\left(x,y\right)dy\label{eq:pdH-1-7}
\end{eqnarray}
for all $\varphi\in L^{2}\left(0,1\right)$, and all $x\in\left[0,1\right]$.
Note the special case of (\ref{eq:pdH-1-7}) for $H=\frac{1}{2}$
is 
\[
\left(T_{E}\varphi\right)\left(x\right)=\int_{0}^{1}\varphi\left(y\right)\left(x\wedge y\right)\, dy,\;\varphi\in L^{2}\left(0,1\right),x\in\left[0,1\right].
\]

We have the following lemma for these Mercer operators:
\begin{lem}
Let $0<H<1$, and let $F_{H}$, $L_{H}$ and $K_{H}$ be as in (\ref{eq:pdH-1-2}),
then the corresponding Mercer operators satisfy:
\begin{equation}
T_{F_{H}}=T_{L_{H}}+2T_{K_{H}}.\label{eq:pdH-1-9}
\end{equation}
\end{lem}
\begin{proof}
This is an easy computation, using (\ref{eq:pdH-1-2}), and (\ref{eq:pdH-1-5})-(\ref{eq:pdH-1-7}).
\end{proof}

\subsection{Ornstein-Uhlenbeck}

The reproducing kernel Hilbert space $\mathscr{H}_{F_{2}}$ in (\ref{eq:RKHS-eg-3})
is used in computations of Ito-integrals of Brownian motion; while
the corresponding RKHS $\mathscr{H}_{F3}$ from (\ref{eq:RKHS-eg-6})-(\ref{eq:RKHS-eg-7})
is used in calculations of stochastic integration with the Ornstein-Uhlenbeck
process.

Motivated by Newton's second law of motion, the Ornstein-Uhlenbeck\index{Ornstein-Uhlenbeck}
velocity process is proposed to model a random external driving force.
In 1D, the process is the solution to the following stochastic differential
equation 
\begin{equation}
dv_{t}=-\gamma v_{t}dt+\beta dB_{t},\:\gamma,\beta>0.\label{eq:ou-1}
\end{equation}
Here, $-\gamma v_{t}$ is the dissipation, $\beta dB_{t}$ denotes
a random fluctuation, and $B_{t}$ is the standard Brownian motion.

Assuming the particle starts at $t=0$. The solution to (\ref{eq:ou-1})
is a Gaussian stochastic process\index{stochastic processes} such
that
\begin{eqnarray*}
\mathbb{E}\left[v_{t}\right] & = & v_{0}e^{-\gamma t}\\
var\left[v_{t}\right] & = & \frac{\beta^{2}}{2\gamma}\left(1-e^{-2\gamma t}\right);
\end{eqnarray*}
with $v_{0}$ being the initial velocity. See Figure \ref{fig:OU}.
Moreover, the process has the following covariance function 
\[
c\left(s,t\right)=\frac{\beta^{2}}{2\gamma}\left(e^{-\gamma\left|t-s\right|}-e^{-\gamma\left|s+t\right|}\right).
\]
If we wait long enough, it turns to a stationary process such that
\[
c\left(s,t\right)\sim\frac{\beta^{2}}{2\gamma}e^{-\gamma\left|t-s\right|}.
\]
This corresponds to the function $F_{3}$. 

\begin{figure}
\includegraphics[scale=0.6]{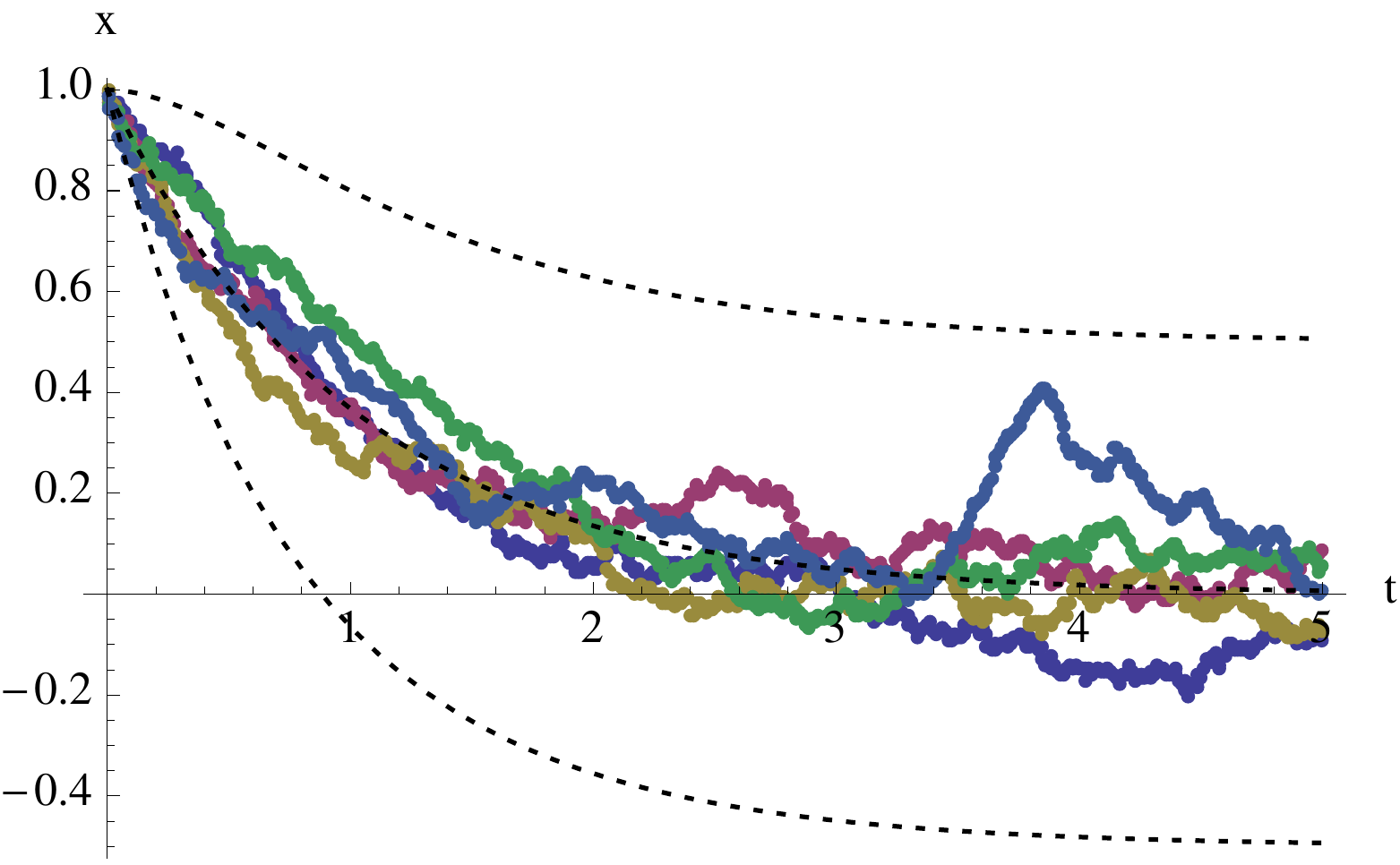}

\protect\caption{\label{fig:OU}Monte-Carlo simulation of Ornstein-Uhlenbeck process
with 5 sample paths. $\beta=\gamma=1$, $v_{0}=1$, $t\in\left[0,5\right]$}
\end{figure}

\subsection{An Overview of the Two Cases: $F_{2}$ and $F_{3}$.}

In Table \ref{tab:F23} below, we give a list of properties for the
particular two cases $F_{2}$ and $F_{3}$ from table \ref{tab:F1-F6}.
Special attention to these examples is merited; first they share a
number of properties with much wider families of locally defined positive
definite functions; and these properties are more transparent in their
simplest form. Secondly, there are important differences between cases
$F_{2}$ and $F_{3}$, and the table serves to highlight both differences
and similarities. A particular feature that is common for the two
is that, when the Mercer operator $T_{F}$ is introduced, then its
inverse $T_{F}^{-1}$ exists as an unbounded positive and selfadjoint
operator in $\mathscr{H}_{F}$. Moreover, in each case, this operator
coincides with the Friedrichs extension of a certain second order
Hermitian semibounded operator (calculated from $D^{\left(F\right)}$),
with dense domain in $\mathscr{H}_{F}$. 

\index{Friedrichs extension}

\index{operator!unbounded}

\index{operator!Mercer}

\index{operator!semibounded}

\section{\label{sec:hdim}Higher Dimensions}

Our function above for $F_{3}$ (sect. \ref{sub:green} and \ref{sub:F3})
admits a natural extension to $\mathbb{R}^{n}$, $n>1$, as follows.

Let $\Omega\subset\mathbb{R}^{n}$ be a subset satisfying $\Omega\neq\phi$,
$\Omega$ open and connected; and assume $\overline{\Omega}$ is compact.
Let $\triangle=\sum_{i=1}^{n}\left(\partial/\partial x_{i}\right)^{2}$
be the usual Laplacian in $n$-variables.
\begin{lem}
Let $F:\Omega-\Omega\rightarrow\mathbb{C}$ be continuous and positive
definite; and let $\mathscr{H}_{F}$ be the corresponding RKHS. In
$\mathscr{H}_{F}$, we set
\begin{align}
L^{\left(F\right)}\left(F_{\varphi}\right) & :=F_{\triangle\varphi},\;\forall\varphi\in C_{c}^{\infty}\left(\Omega\right),\:\mbox{with}\\
dom\bigl(L^{\left(F\right)}\bigr) & =\left\{ F_{\varphi}\:\big|\:\varphi\in C_{c}^{\infty}\left(\Omega\right)\right\} ;
\end{align}
then $L^{\left(F\right)}$ is semibounded in $\mathscr{H}_{F}$, $L^{\left(F\right)}\leq0$
in the order of Hermitian operators.

Let $T_{F}$ be the Mercer operator; then there is a positive constant
$c\:(\:=c_{\left(F,\Omega\right)}>0)$ such that $T_{F}^{-1}$ \textup{is
a selfadjoint extension of the densely defined operator $c\left(I-L^{\left(F\right)}\right)$
in $\mathscr{H}_{F}$. }\end{lem}
\begin{proof}
The ideas for the proof are contained in section \ref{sub:green}
and \ref{sub:F3} above. To get the general conclusions, we may combine
these considerations with the general theory of Green's functions
for elliptic linear PDEs (partial differential equations); see also
\cite{Nel57,JLW69}.
\end{proof}
\index{operator!Mercer}\index{operator!partial differential}

\begin{table}
\begin{tabular}{|c|c|}
\hline 
\begin{minipage}[t]{0.47\columnwidth}%
1. $F_{2}\left(x\right)=1-\left|x\right|$, $\left|x\right|<\frac{1}{2}$

2. Mercer operator 
\[
T_{F_{2}}:L^{2}\left(0,\tfrac{1}{2}\right)\rightarrow L^{2}\left(0,\tfrac{1}{2}\right)
\]

3. $T_{F_{2}}^{-1}$ is unbounded, selfadjoint \index{operator!selfadjoint}
\begin{proof}
Since $T_{F}$ is positive, bounded, and trace class, it follows that
$T_{F_{2}}^{-1}$ is positive, unbounded, and selfadjoint.
\end{proof}
4. $T_{F_{2}}^{-1}=$ Friedrichs extension of 
\[
-\tfrac{1}{2}\bigl(\tfrac{d}{dx}\bigr)^{2}\Big|_{C_{c}^{\infty}\left(0,\frac{1}{2}\right)}
\]
as a selfadjoint operator on $L^{2}\left(0,\frac{1}{2}\right)$. 

\emph{Sketch of proof.} 

Setting 
\[
u\left(x\right)=\int_{0}^{\frac{1}{2}}\varphi\left(y\right)\left(1-\left|x-y\right|\right)dy
\]
then 
\begin{eqnarray*}
u'' & = & -2\varphi\\
 & \Updownarrow\\
u & = & \left(-\tfrac{1}{2}\bigl(\tfrac{d}{dx}\bigr)^{2}\right)^{-1}\varphi
\end{eqnarray*}
and so 
\[
T_{F_{2}}=\left(-\tfrac{1}{2}\bigl(\tfrac{d}{dx}\bigr)^{2}\right)^{-1}.
\]
And we get a selfadjoint extension 
\[
T_{F_{2}}^{-1}\supset-\tfrac{1}{2}\bigl(\tfrac{d}{dx}\bigr)^{2}
\]
in $L^{2}\left(0,\frac{1}{2}\right)$, where the containment refers
to operator graphs.%
\end{minipage} & %
\begin{minipage}[t]{0.47\columnwidth}%
1. $F_{3}\left(x\right)=e^{-\left|x\right|}$, $\left|x\right|<1$

2. Mercer operator 
\[
T_{F_{3}}:L^{2}\left(0,1\right)\rightarrow L^{2}\left(0,1\right)
\]

3. $T_{F_{3}}^{-1}$ is unbounded, selfadjoint
\begin{proof}
Same argument as in the proof for $T_{F_{2}}^{-1}$; also follows
from Mercer's theorem.
\end{proof}
4. $T_{F_{2}}^{-1}=$ Friedrichs extension of 
\[
\tfrac{1}{2}\bigl(I-\bigl(\tfrac{d}{dx}\bigr)^{2}\bigr)\Big|_{C_{c}^{\infty}\left(0,1\right)}
\]
as a selfadjoint operator on $L^{2}\left(0,1\right)$. 

\emph{Sketch of proof.} 

Setting
\[
u\left(x\right)=\int_{0}^{1}\varphi\left(y\right)e^{-\left|x-y\right|}dy
\]
then 
\begin{eqnarray*}
u'' & = & u-2\varphi\\
 & \Updownarrow\\
u & = & \left(\tfrac{1}{2}\bigl(1-\bigl(\tfrac{d}{dx}\bigr)^{2}\bigr)\right)^{-1}\varphi
\end{eqnarray*}
and so 
\[
T_{F_{3}}=\left(\tfrac{1}{2}\left(I-\bigl(\tfrac{d}{dx}\bigr)^{2}\right)\right)^{-1}.
\]
Now, a selfadjoint extension\index{selfadjoint extension} 
\[
T_{F_{3}}^{-1}\supset\tfrac{1}{2}\left(I-\bigl(\tfrac{d}{dx}\bigr)^{2}\right)
\]
in $L^{2}\left(0,1\right)$.%
\end{minipage}\tabularnewline
\hline 
\end{tabular}

\protect\caption{\label{tab:F23}An overview of two cases: $F_{2}$ v.s. $F_{3}$.}
\end{table}

\chapter{\label{chap:CompareFK}Comparing Different RKHSs $\mathscr{H}_{F}$
and $\mathscr{H}_{K}$}

Before we turn to comparison of pairs of RKHSs, we will prove a lemma
which accounts for two uniformity principles for positive definite
continuous functions defined on open subsets in locally compact groups:\index{RKHS}\index{positive definite}
\begin{lem}
\label{lem:F-bd}Let $G$ be a locally compact group, and let $\Omega\subset G$
be an open and connected subset, $\Omega\neq\phi$. Let $F:\Omega^{-1}\Omega\rightarrow\mathbb{C}$
be a continuous positive definite function satisfying $F\left(e\right)=1$,
where $e\in G$ is the unit for the group operation. 
\begin{enumerate}
\item Then $F$ extends by limit to a continuous p.d. function
\begin{equation}
F^{\left(ex\right)}:\overline{\Omega}^{-1}\overline{\Omega}\longrightarrow\mathbb{C}.
\end{equation}

\item Moreover, the two p.d. functions $F$ and $F^{\left(ex\right)}$ have
the same RKHS consisting of continuous functions $\xi$ on $\overline{\Omega}$
such that, $\exists\,0<A<\infty$, $A=A_{\xi}$, s.t.
\begin{equation}
\left|\int_{\overline{\Omega}}\overline{\xi\left(x\right)}\varphi\left(x\right)dx\right|^{2}\leq A\left\Vert F_{\varphi}\right\Vert _{\mathscr{H}_{F}}^{2},\;\forall\varphi\in C_{c}\left(\Omega\right)\label{eq:lcg-1-1}
\end{equation}
where $dx=$ Haar measure on $G$, \index{measure!Haar} 
\[
F_{\varphi}\left(\cdot\right)=\int_{\Omega}\varphi\left(y\right)F\left(y^{-1}\cdot\right)dy,\;\varphi\in C_{c}\left(\Omega\right);
\]
and $\left\Vert F_{\varphi}\right\Vert _{\mathscr{H}_{F}}$ denotes
the $\mathscr{H}_{F}$-norm of $F_{\varphi}$. 
\item Every $\xi\in\mathscr{H}_{F}$ satisfies the following \uline{a
priori} estimate:
\begin{equation}
\left|\xi\left(x\right)-\xi\left(y\right)\right|^{2}\leq2\left\Vert \xi\right\Vert _{\mathscr{H}_{F}}^{2}\left(1-\Re\left\{ F\left(x^{-1}y\right)\right\} \right)
\end{equation}
for all $\xi\in\mathscr{H}_{F}$, and all $x,y\in\overline{\Omega}$. 
\end{enumerate}
\end{lem}
\begin{proof}
The arguments in the proof only use standard tools from the theory
of reproducing kernel Hilbert spaces. We covered a special case in
\secref{mercer}, we shall omit the details here.
\end{proof}
Now add the further assumption on the subset $\Omega\subset G$ from
the lemma: Assume in addition that $\Omega$ has compact closure,
so $\overline{\Omega}$ is compact. Let $\partial\Omega=\overline{\Omega}\backslash\Omega$
be the boundary. With this assumption, we get automatically the inclusion
\[
C\left(\overline{\Omega}\right)\subset L^{2}\left(\overline{\Omega}\right)
\]
since continuous functions on $\overline{\Omega}$ are automatically
uniformly bounded, and Haar measure on $G$ has the property that
$\left|\overline{\Omega}\right|<\infty$.
\begin{defn}
Let $\left(\Omega,F\right)$ be as above, i.e., 
\[
F:\overline{\Omega}^{-1}\overline{\Omega}\longrightarrow\mathbb{C}
\]
is continuous and positive definite\index{positive definite}. Assume
$G$ is a Lie group.\index{group!Lie} Recall, extension by limit
to $\overline{\Omega}$ is automatic by the limit. Let $\beta\in\mathscr{M}\left(\partial\Omega\right)$
be a positive finite measure on the boundary $\partial\Omega$. We
say that $\beta$ is a boundary measure \index{measure!boundary}iff
\begin{equation}
\left\langle T_{F}f,\xi\right\rangle _{\mathscr{H}_{F}}-\int_{\Omega}\overline{f\left(x\right)}\xi\left(x\right)dx=\int_{\partial\Omega}\overline{\left(T_{F}f\right)_{n}\left(\sigma\right)}\xi\left(\sigma\right)d\beta\left(\sigma\right)
\end{equation}
holds for all $f\in C\left(\overline{\Omega}\right)$, $\forall\xi\in\mathscr{H}_{F}$,
where $\left(\cdot\right)_{n}=$ normal derivative computed on the
boundary $\partial\Omega$ of $\Omega$.\end{defn}
\begin{rem}
For the example $G=\mathbb{R}$, $\overline{\Omega}=\left[0,1\right]$,
$\partial\Omega=\left\{ 0,1\right\} $, and $F=F_{3}$ on $\left[0,1\right]$,
where 
\[
F\left(x\right)=e^{-\left|x\right|},\;\forall x\in\left[-1,1\right],
\]
the boundary measure is
\begin{equation}
\beta=\frac{1}{2}\left(\delta_{0}+\delta_{1}\right).
\end{equation}

\end{rem}
Let $G$ be a locally compact group with left-Haar measure, and let
$\Omega\subset G$ be a non-empty subset satisfying: $\Omega$ is
open and connected; and is of finite Haar measure; write $\left|\Omega\right|<\infty$.
The Hilbert space $L^{2}\left(\Omega\right)=L^{2}\left(\Omega,dx\right)$
is the usual $L^{2}$-space of measurable functions $f$ on $\Omega$
such that 
\begin{equation}
\left\Vert f\right\Vert _{L^{2}\left(\Omega\right)}^{2}:=\int_{\Omega}\left|f\left(x\right)\right|^{2}dx<\infty.\label{eq:FK-1}
\end{equation}

\begin{defn}
\label{def:FK-1}Let $F$ and $K$ be two continuous and positive
definite functions defined on 
\begin{equation}
\Omega^{-1}\cdot\Omega:=\left\{ x^{-1}y\:\Big|\: x,y\in\Omega\right\} .\label{eq:FK-2}
\end{equation}
We say that $K\ll F$ iff there is a finite constant $A$ such that
for all finite subsets $\left\{ x_{i}\right\} _{i=1}^{N}\subset\Omega$,
and all systems $\left\{ c_{i}\right\} _{i=1}^{N}\subset\mathbb{C}$,
we have:
\begin{equation}
\sum_{i}\sum_{j}\overline{c_{i}}c_{j}K\left(x_{i}^{-1}x_{j}\right)\leq A\sum_{i}\sum_{j}\overline{c_{i}}c_{j}F\left(x_{i}^{-1}x_{j}\right).\label{eq:FK-3}
\end{equation}
\end{defn}
\begin{lem}
\label{lem:FK-1}Let $F$ and $K$ be as above; then $K\ll F$ if
and only if there is a finite constant $A\in\mathbb{R}_{+}$ such
that
\begin{equation}
\int_{\Omega}\int_{\Omega}\overline{\varphi\left(x\right)}\varphi\left(y\right)K\left(x^{-1}y\right)dxdy\leq A\int_{\Omega}\int_{\Omega}\overline{\varphi\left(x\right)}\varphi\left(y\right)F\left(x^{-1}y\right)dxdy\label{eq:FK-4}
\end{equation}
holds for all $\varphi\in C_{c}\left(\Omega\right)$. The constant
$A$ in (\ref{eq:FK-3}) and (\ref{eq:FK-4}) will be the same.\end{lem}
\begin{proof}
Easy; use an approximate identity\index{approximate identity} in
$G$, see e.g., \cite{Rud73,Ru90}.
\end{proof}
Setting 
\begin{equation}
F_{\varphi}\left(x\right)=\int_{\Omega}\varphi\left(y\right)F\left(y^{-1}x\right)dy,\label{eq:FK-5}
\end{equation}
and similarly for $K_{\varphi}=\int_{\Omega}\varphi\left(y\right)K\left(y^{-1}\cdot\right)dy$,
we note that $K\ll F$ if and only if:
\begin{equation}
\left\Vert K_{\varphi}\right\Vert _{\mathscr{H}_{K}}\leq\sqrt{A}\left\Vert F_{\varphi}\right\Vert _{\mathscr{H}_{F}},\;\forall\varphi\in C_{c}\left(\Omega\right).\label{eq:FK-6}
\end{equation}
Further, note that, if $G$ is also a \uline{Lie group}, then (\ref{eq:FK-6})
follows from checking it only for all $\varphi\in C_{c}^{\infty}\left(G\right)$.
See Lemma \ref{lem:li-meas}.
\begin{thm}
\label{thm:FK-1}Let $\Omega$, $F$ and $K$ be as in Definition
\ref{def:FK-1}, i.e., both continuous and p.d. on the set $\Omega^{-1}\cdot\Omega$
in (\ref{eq:FK-2}); then the following two conditions are equivalent:
\begin{enumerate}[label=(\roman{enumi})]
\item  $K\ll F$, and
\item $\mathscr{H}_{K}$ \uline{is} a closed subspace of $\mathscr{H}_{F}$.
\end{enumerate}
\end{thm}
\begin{proof}
$\Downarrow$ Assume $K\ll F$, we then define a linear operator $l:\mathscr{H}_{F}\rightarrow\mathscr{H}_{K}$,
setting 
\begin{equation}
l\left(F_{\varphi}\right):=K_{\varphi},\;\forall\varphi\in C_{c}\left(\Omega\right).\label{eq:FK-7}
\end{equation}
We now use (\ref{eq:FK-6}), plus the fact that $\mathscr{H}_{F}$
is the $\left\Vert \cdot\right\Vert _{\mathscr{H}_{F}}$-completion
of 
\[
\left\{ F_{\varphi}\:\Big|\:\varphi\in C_{c}\left(\Omega\right)\right\} ;
\]
and similarly for $\mathscr{H}_{K}$.

As a result of (\ref{eq:FK-7}) and (\ref{eq:FK-6}), we get a canonical
extension of $l$ to a bounded linear operator, also denoted $l:\mathscr{H}_{F}\rightarrow\mathscr{H}_{K}$,
and 
\begin{equation}
\left\Vert l\left(\xi\right)\right\Vert _{\mathscr{H}_{K}}\leq\sqrt{A}\left\Vert \xi\right\Vert _{\mathscr{H}_{F}},\mbox{ for all }\xi\in\mathscr{H}_{K}.\label{eq:FK-8}
\end{equation}

\index{operator!bounded}

We interrupt the proof to give a lemma:\end{proof}
\begin{lem}
Let $F$, $K$, $\Omega$ be as above. Assume $K\ll F$, and let $l:\mathscr{H}_{F}\rightarrow\mathscr{H}_{K}$
be the bounded operator introduced in (\ref{eq:FK-7}) and (\ref{eq:FK-8}).
Then the adjoint operator $l^{*}:\mathscr{H}_{K}\rightarrow\mathscr{H}_{F}$
satisfies:
\begin{equation}
\left(l^{*}\left(\xi\right)\right)\left(x\right)=\xi\left(x\right),\mbox{ for all }\xi\in\mathscr{H}_{K},x\in\Omega.\label{eq:FK-9}
\end{equation}

\begin{proof}
By the definite of $l^{*}$, as the adjoint of a bounded linear operator
between Hilbert spaces, we get 
\begin{equation}
\left\Vert l^{*}\right\Vert _{\mathscr{H}_{K}\rightarrow\mathscr{H}_{F}}=\left\Vert l\right\Vert _{\mathscr{H}_{F}\rightarrow\mathscr{H}_{K}}\leq\sqrt{A}\label{eq:FK-10}
\end{equation}
for the respective operator norms; and 
\begin{equation}
\left\langle l^{*}\left(\xi\right),F_{\varphi}\right\rangle _{\mathscr{H}_{F}}=\left\langle \xi,K_{\varphi}\right\rangle _{\mathscr{H}_{K}},\;\forall\varphi\in C_{c}\left(\Omega\right).\label{eq:FK-11}
\end{equation}

Using now the reproducing property in the two RKHSs, we get:
\begin{eqnarray*}
\left(\mbox{LHS}\right)_{\left(\ref{eq:FK-11}\right)} & = & \int_{\Omega}\overline{l^{*}\left(\xi\right)\left(x\right)}\varphi\left(x\right)dx,\mbox{ and}\\
\left(\mbox{RHS}\right)_{\left(\ref{eq:FK-11}\right)} & = & \int_{\Omega}\overline{\xi\left(x\right)}\varphi\left(x\right)dx,\;\mbox{ for all }\varphi\in C_{c}\left(\Omega\right).
\end{eqnarray*}
Taking now approximations in $C_{c}\left(\Omega\right)$ to the Dirac
masses $\{\delta_{x}\:|\: x\in\Omega\}$, the desired conclusion (\ref{eq:FK-9})
follows.
\end{proof}
\end{lem}
\begin{proof}[Proof of Theorem \ref{thm:FK-1} resumed ]
 Assume $K\ll F$, the lemma proves that $\mathscr{H}_{K}$ identifies
with a linear space of continuous functions $\xi$ on $\Omega$, and
if $\xi\in\mathscr{H}_{K}$, then it is also in $\mathscr{H}_{F}$.

We claim that this is a closed subspace in $\mathscr{H}_{F}$ relative
to the $\mathscr{H}_{F}$-norm. 

\uline{Step 1}. Let $\left\{ \xi_{n}\right\} \subset\mathscr{H}_{K}$
satisfying 
\[
\lim_{n,m\rightarrow\infty}\left\Vert \xi_{n}-\xi_{m}\right\Vert _{\mathscr{H}_{F}}=0.
\]
By (\ref{eq:FK-8}) and (\ref{eq:FK-9}), the lemma; we get
\[
\lim_{n,m\rightarrow\infty}\left\Vert \xi_{n}-\xi_{m}\right\Vert _{\mathscr{H}_{K}}=0.
\]

\uline{Step 2}. Since $\mathscr{H}_{K}$ is complete, we get $\chi\in\mathscr{H}_{K}$
such that 
\begin{equation}
\lim_{n\rightarrow\infty}\left\Vert \xi_{n}-\chi\right\Vert _{\mathscr{H}_{K}}=0.\label{eq:FK-12}
\end{equation}

\uline{Step 3}. We claim that this $\mathscr{H}_{K}$-limit $\chi$
also defines a unique element in $\mathscr{H}_{F}$, and it is therefore
the $\mathscr{H}_{F}$-limit. 

We have for all $\varphi\in C_{c}\left(\Omega\right)$:
\begin{eqnarray*}
\left|\int_{\Omega}\overline{\chi\left(x\right)}\varphi\left(x\right)dx\right| & \leq & \left\Vert \chi\right\Vert _{\mathscr{H}_{K}}\left\Vert K_{\varphi}\right\Vert _{\mathscr{H}_{K}}\\
 & \leq & \left\Vert \chi\right\Vert _{\mathscr{H}_{K}}\sqrt{A}\left\Vert F_{\varphi}\right\Vert _{\mathscr{H}_{F}};
\end{eqnarray*}
and so $\chi\in\mathscr{H}_{F}$.

We now turn to the converse implication of Theorem \ref{thm:FK-1}:

$\Uparrow$ Assume $F$ and $K$ are as in the statement of the theorem;
and that $\mathscr{H}_{K}$ is a close subspace in $\mathscr{H}_{F}$
via identification of the respective continuous functions on $\Omega$.
We then prove that $K\ll F$. 

Now let $P_{K}$ denote the orthogonal projection\index{projection}
of $\mathscr{H}_{F}$ onto the closed subspace $\mathscr{H}_{K}$.
We claim that 
\begin{equation}
P_{K}\left(F_{\varphi}\right)=K_{\varphi},\;\forall\varphi\in C_{c}\left(\Omega\right).\label{eq:FK-13}
\end{equation}
Using the uniqueness of the projection $P_{K}$, we need to verify
that $F_{\varphi}-K_{\varphi}\in\mathscr{H}_{F}\ominus\mathscr{H}_{K}$;
i.e., that 
\begin{equation}
\left\langle F_{\varphi}-K_{\varphi},\xi_{K}\right\rangle _{\mathscr{H}_{F}}=0,\;\mbox{for all }\xi_{K}\in\mathscr{H}_{K}.\label{eq:FK-14}
\end{equation}
But since $\mathscr{H}_{K}\subset\mathscr{H}_{F}$, we have 
\[
\text{LHS}_{\left(\ref{eq:FK-14}\right)}=\int_{\Omega}\overline{\varphi\left(x\right)}\xi_{K}\left(x\right)dx-\int_{\Omega}\overline{\varphi\left(x\right)}\xi_{K}\left(x\right)dx=0,
\]
for all $\varphi\in C_{c}\left(\Omega\right)$. This proves (\ref{eq:FK-13}).

To verify $K\ll F$, we use the criterion (\ref{eq:FK-6}) from Lemma
\ref{lem:FK-1}. Indeed, consider $K_{\varphi}\in\mathscr{H}_{K}$.
Since $\mathscr{H}_{K}\subset\mathscr{H}_{F}$, we get 
\[
l\left(F_{\varphi}\right)=P_{K}\left(F_{\varphi}\right)=K_{\varphi},\;\mbox{and}
\]
\[
\left\Vert K_{\varphi}\right\Vert _{\mathscr{H}_{K}}=\left\Vert l\left(F_{\varphi}\right)\right\Vert _{\mathscr{H}_{K}}\leq\sqrt{A}\left\Vert F_{\varphi}\right\Vert _{\mathscr{H}_{F}}
\]
which is the desired estimate (\ref{eq:FK-6}).
\end{proof}

\section{Applications}

Below we give an application of Theorem \ref{thm:FK-1} to the deficiency-index
problem, and to the computation of the deficiency spaces; see also
Lemma \ref{lem:exp-1}, and Lemma \ref{lem:def1}.

As above, we will consider two given continuous p.d. functions $F$
and $K$, but the group now is $G=\mathbb{R}$: We pick $a,b\in\mathbb{R}_{+}$,
$0<a<b$, such that $F$ is defined on $\left(-b,b\right)$, and $K$
on $\left(-a,a\right)$. The corresponding two RKHSs will be denoted
$\mathscr{H}_{F}$ and $\mathscr{H}_{K}$. We say that $K\ll F$ iff
there is a finite constant $A$ such that\index{RKHS}
\begin{equation}
\left\Vert K_{\varphi}\right\Vert _{\mathscr{H}_{K}}^{2}\leq A\left\Vert F_{\varphi}\right\Vert _{\mathscr{H}_{F}}^{2}\label{eq:FK-app-1}
\end{equation}
for all $\varphi\in C_{c}\left(0,a\right)$. Now this is a slight
adaptation of our Definition \ref{def:FK-1} above, but this modification
will be needed; for example in computing the indices of two p.d. functions
$F_{2}$ and $F_{3}$ from Table \ref{tab:F1-F6}; see also  \secref{F2F3}
below. In fact, a simple direct checking shows that
\begin{equation}
F_{2}\ll F_{3}\quad\left(\mbox{see Table }\ref{tab:F1-F6}\right),\label{eq:FK-app-2}
\end{equation}
and we now take $a=\frac{1}{2}$, $b=1$. 

Here, $F_{2}\left(x\right)=1-\left|x\right|$ in $\left|x\right|<\frac{1}{2}$;
and $F_{3}\left(x\right)=e^{-\left|x\right|}$ in $\left|x\right|<1$;
see \figref{FK-1}.

\begin{figure}[H]
\begin{tabular}{cc}
\includegraphics[scale=0.5]{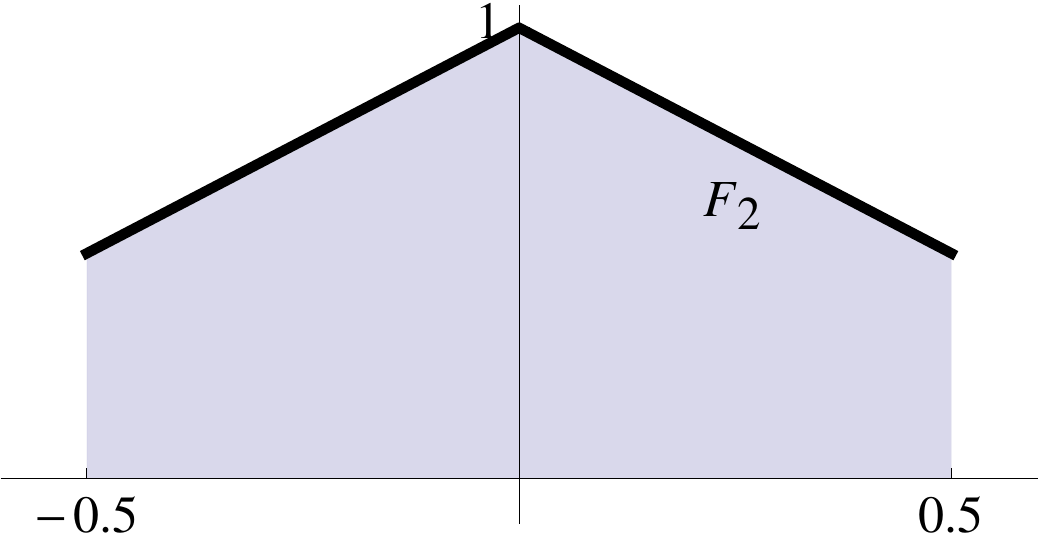} & \includegraphics[scale=0.5]{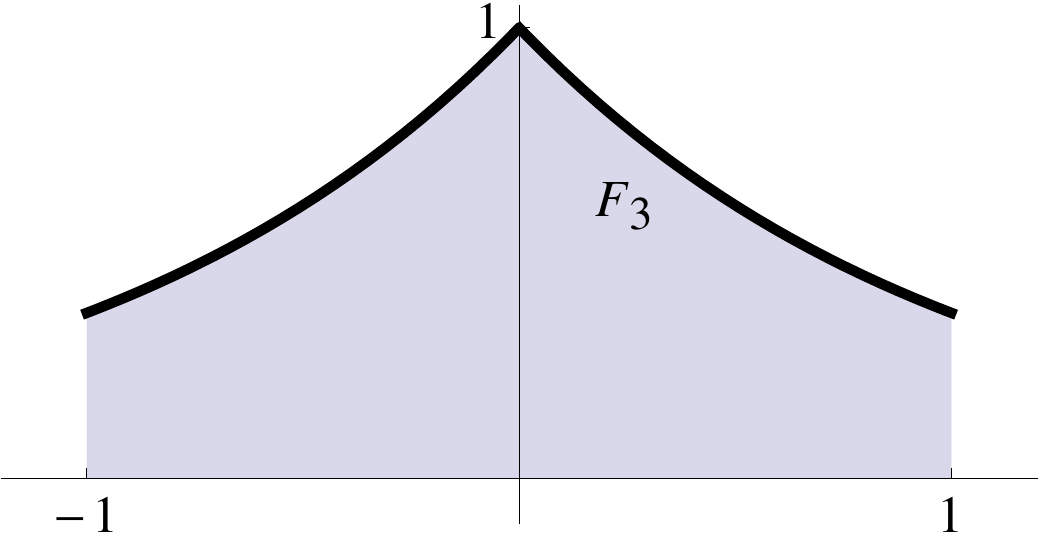}\tabularnewline
\end{tabular}

\protect\caption{\label{fig:FK-1}The examples of  $F_{2}$ and $F_{3}$.}

\end{figure}

We wish to compare the respective skew-Hermitian operators, $D^{\left(F\right)}$
in $\mathscr{H}_{F}$; and $D^{\left(F\right)}$ in $\mathscr{H}_{K}$;
see  \subref{euclid}, i.e., \index{operator!skew-Hermitian} 
\begin{eqnarray}
D^{\left(F\right)}\left(F_{\varphi}\right) & = & F_{\varphi'},\;\forall\varphi\in C_{c}^{\infty}\left(0,b\right);\;\mbox{and}\label{eq:FK-app-3}\\
D^{\left(K\right)}\left(K_{\varphi}\right) & = & K_{\varphi'},\;\forall\varphi\in C_{c}^{\infty}\left(0,a\right).\label{eq:FK-app-4}
\end{eqnarray}

Let $z\in\mathbb{C}$; and we set 
\begin{eqnarray}
DEF_{F}\left(z\right) & = & \left\{ \xi\in dom\bigl(\bigl(D^{\left(F\right)}\bigr)^{*}\bigr)\:\Big|\:\bigl(D^{\left(F\right)}\bigr)^{*}\xi=z\xi\right\} ,\;\mbox{and}\label{eq:FK-app-5}\\
DEF_{K}\left(z\right) & = & \left\{ \xi\in dom\bigl(\bigl(D^{\left(K\right)}\bigr)^{*}\bigr)\:\Big|\:\bigl(D^{\left(K\right)}\bigr)^{*}\xi=z\xi\right\} .\label{eq:FK-app-6}
\end{eqnarray}

\begin{thm}
\label{thm:FK-app-1}Let two continuous p.d. functions $F$ and $K$
be specified as above, and suppose 
\begin{equation}
K\ll F;\label{eq:FK-app-7}
\end{equation}
then 
\begin{equation}
DEF_{K}\left(z\right)=DEF_{F}\left(z\right)\Big|_{\left(0,a\right)}\label{eq:FK-app-8}
\end{equation}
i.e., restriction to the smaller interval.\end{thm}
\begin{proof}
Since (\ref{eq:FK-app-7}) is assumed, it follows from Theorem \ref{thm:FK-1},
that $\mathscr{H}_{K}$ is a subspace of $\mathscr{H}_{F}$. 

If $\varphi\in C_{c}^{\infty}\left(0,b\right)$, and $\xi\in dom\bigl(\bigl(D^{\left(F\right)}\bigr)^{*}\bigr)$,
then 
\[
\left\langle \bigl(D^{\left(F\right)}\bigr)^{*}\xi,F_{\varphi}\right\rangle _{\mathscr{H}_{F}}=\left\langle \xi,F_{\varphi'}\right\rangle _{\mathscr{H}_{F}}=\int_{0}^{b}\overline{\xi\left(x\right)}\varphi'\left(x\right)dx;
\]
and it follows that functions $\xi$ in $DEF_{F}\left(z\right)$ must
be multiples of 
\begin{equation}
\left(0,b\right)\ni x\longmapsto e_{z}\left(x\right)=e^{-zx}.\label{eq:FK-app-9}
\end{equation}
Hence, by Theorem \ref{thm:FK-1}, we get 
\[
DEF_{K}\left(z\right)\subseteq DEF_{F}\left(z\right),
\]
and by (\ref{eq:FK-app-9}), we see that(\ref{eq:FK-app-8}) must
hold. 

Conversely, if $DEF_{F}\left(z\right)\neq0$, then $l\left(DEF_{F}\left(z\right)\right)\neq0$,
and its restriction to $\left(0,a\right)$ is contained in $DEF_{K}\left(z\right)$.
The conclusion in the theorem follows.\end{proof}
\begin{rem}
The spaces $DEF_{\left(F\right)}\left(z\right)$, $z\in\mathbb{C}$,
are also discussed in Theorems \ref{thm:Eigenspaces-for-the-adjoint}-\ref{cor:eigen}.\end{rem}
\begin{example}[Application]
 Consider the two functions $F_{2}$ and $F_{3}$ in Table \ref{tab:F1-F6}.
Both of the operators $D^{\left(F_{i}\right)}$, $i=2,3$, have deficiency
indices\index{deficiency indices} $\left(1,1\right)$. \end{example}
\begin{proof}
One easily checks that $F_{2}\ll F_{3}$. And it is also easy to check
directly that $D^{\left(F_{2}\right)}$ has indices $\left(1,1\right)$.
Hence, by (\ref{eq:FK-app-8}) in the theorem, it follows that $D^{\left(F_{3}\right)}$
also must have indices $\left(1,1\right)$. (The latter conclusion
is not as easy to verify by direct means!)
\end{proof}

\section{Radially Symmetric Positive Definite Functions}

Among other subclasses of positive definite\index{positive definite}
functions we have radially symmetric p.d. functions. If a given p.d.
function happens to be radially symmetric, then there are a number
of simplifications available, and the analysis in higher dimension
often simplifies. This is to a large extend due to theorems of I.
J. Schöenberg\index{Schöenberg} and D. V. Widder. Below we sketch
two highpoints, but we omit details and application to interpolation
and to geometry. These themes are in the literature, see e.g. \cite{Sch38,ScWh53,Sch64,Wid41,WeWi75}.
\begin{rem}
In some cases, the analysis in one dimension yields insight into the
possibilities in $\mathbb{R}^{k}$, $k>1$. This leads for example
for functions $F$ on $\mathbb{R}^{k}$ which are radial, i.e., of
the form $F\left(x\right)=\Phi\bigl(\left\Vert x\right\Vert ^{2}\bigr)$,
where $\left\Vert x\right\Vert ^{2}=\sum_{i=1}^{k}x_{i}^{2}$. 
\end{rem}
A function $q$ on $\mathbb{R}_{+}$, $q:\mathbb{R}_{+}\rightarrow\mathbb{R}$,
is said to be \emph{completely monotone} iff $q\in C\left(\left[0,\infty\right)\right)\cap C^{\infty}\left(\left(0,\infty\right)\right)$
and 
\begin{equation}
\left(-1\right)^{n}q^{\left(n\right)}\left(r\right)\geq0,\; r\in\mathbb{R}_{+},n\in\mathbb{N}_{0}.\label{eq:rad-1}
\end{equation}

\begin{example}
~
\begin{eqnarray*}
q\left(r\right) & = & e^{-\alpha r},\qquad\alpha\geq0;\\
q\left(r\right) & = & \frac{\alpha}{r^{1-\alpha}},\qquad\alpha\leq1;\\
q\left(r\right) & = & \frac{1}{\left(r+\alpha^{2}\right)^{\beta}},\;\alpha>0,\beta\geq0.
\end{eqnarray*}
\end{example}
\begin{thm}[Schöenberg (1938)]
 A function $q:\mathbb{R}_{+}\rightarrow\mathbb{R}$ is completely
monotone iff the corresponding function $F_{q}\left(x\right)=q\bigl(\left\Vert x\right\Vert ^{2}\bigr)$
is positive definite and radial on $\mathbb{R}^{k}$ for all $k\in\mathbb{N}$. \end{thm}
\begin{proof}
We omit details, but the proof uses:\end{proof}
\begin{thm}[Bernstein-Widder]
 A function $q:\mathbb{R}_{+}\rightarrow\mathbb{R}$ is completely
monotone iff there is a finite positive Borel measure \index{measure!Borel}on
$\mathbb{R}_{+}$ s.t. 
\[
q\left(r\right)=\int_{0}^{\infty}e^{-rt}d\mu\left(t\right),\; r\in\mathbb{R}_{+},
\]
i.e., $q$ is the Laplace transform of a finite positive measure \index{measure!positive}$\mu$
on $\mathbb{R}_{+}$.\end{thm}
\begin{rem}
The condition that the function $q$ in (\ref{eq:rad-1}) be in $C^{\infty}\left(\mathbb{R}_{+}\right)$
may be relaxed; and then (\ref{eq:rad-1}) takes the following alternative
form:
\begin{equation}
\sum_{k=1}^{n}\left(-1\right)^{k}\binom{n}{k}q\left(r+k\delta\right)\geq0
\end{equation}
for all $n\in\mathbb{N}$, all $\delta>0$, and $x\in\bigl[0,\infty\bigr)$;
i.e., 
\begin{gather*}
q\left(r\right)-q\left(r+\delta\right)\geq0\\
q\left(r\right)-2q\left(r+\delta\right)+q\left(r+2\delta\right)\geq0\;\;\mbox{e.t.c. }
\end{gather*}

It is immediate that every completely monotone function $q$ on $\bigl[0,\infty\bigr)$
is convex.
\end{rem}

\section{\label{sec:FFbar}Connecting $F$ and $\overline{F}$ When $F$ is
a Positive Definite Function}

Let $F:\left(-1,1\right)\rightarrow\mathbb{C}$ be continuous and
positive definite\index{positive definite}, and let $\overline{F}$
be the complex conjugate, i.e., $\overline{F}$$\left(x\right)=F\left(-x\right)$,
$\forall x\in\left(-1,1\right)$. Below, we construct a contractive-linking
operator $\mathscr{H}_{F}\rightarrow\mathscr{H}_{\overline{F}}$ between
the two RKHSs. \index{RKHS}
\begin{lem}
\label{lem:conj-1-1}Let $\mu$ and $\mu^{\left(s\right)}$ be as
before, $\mu^{\left(s\right)}=\mu\circ s$, $s\left(x\right)=-x$;
and set 
\begin{equation}
g=\sqrt{\frac{d\mu^{\left(s\right)}}{d\mu}};\label{eq:conj-1-1}
\end{equation}
(the square root of the Radon-Nikodym\index{Radon-Nikodym} derivative)
then the $g$-multiplication operator is isometric between the respective
Hilbert spaces; $L^{2}\left(\mu^{\left(s\right)}\right)$ and $L^{2}\left(\mu\right)$
as follows:
\begin{equation}
\xymatrix{L^{2}\left(\mathbb{R},\mu^{\left(s\right)}\right)\ar@{|->}[rr]\sp(0.6){M_{g}}\sb(0.6){h\mapsto gh} &  & L^{2}\left(\mathbb{R},\mu\right)}
.\label{eq:conj-1-2}
\end{equation}
\end{lem}
\begin{proof}
Let $h\in L^{2}\left(\mu^{\left(s\right)}\right)$, then 
\begin{eqnarray*}
\int_{\mathbb{R}}\left|gh\right|^{2}d\mu & = & \int_{\mathbb{R}}\left|h\right|^{2}\frac{d\mu^{\left(s\right)}}{d\mu}d\mu\\
 & = & \int_{\mathbb{R}}\left|h\right|^{2}d\mu^{\left(s\right)}=\left\Vert h\right\Vert _{L^{2}\left(\mu^{\left(s\right)}\right)}^{2}.
\end{eqnarray*}
\end{proof}
\begin{lem}
If $F:\left(-1,1\right)\rightarrow\mathbb{C}$ is a given continuous
p.d. function, and if $\mu\in Ext\left(F\right)$, then 
\begin{equation}
\xymatrix{\mathscr{H}_{F}\ni F_{\varphi}\ar@{|->}[rr]\sp(0.5){V^{\left(F\right)}} &  & \widehat{\varphi}\in L^{2}\left(\mathbb{R},\mu\right)}
\label{eq:conj-1-3}
\end{equation}
extends by closure to an isometry\index{isometry}. \end{lem}
\begin{proof}
For $\varphi\in C_{c}\left(0,1\right)$, we have: 
\begin{eqnarray*}
\left\Vert F_{\varphi}\right\Vert _{\mathscr{H}_{F}}^{2} & = & \int_{0}^{1}\int_{0}^{1}\overline{\varphi\left(x\right)}\varphi\left(y\right)F\left(x-y\right)dxdy\\
 & \underset{\text{\ensuremath{\mu\in}Ext(F)}}{=} & \int_{0}^{1}\int_{0}^{1}\overline{\varphi\left(x\right)}\varphi\left(y\right)\left(\int_{\mathbb{R}}e^{i\left(x-y\right)\lambda}d\mu\left(\lambda\right)\right)dxdy\\
 & \underset{\left(\text{Fubini}\right)}{=} & \int_{\mathbb{R}}\left|\widehat{\varphi}\left(\lambda\right)\right|^{2}d\mu\left(\lambda\right)\\
 & \underset{\text{(\ref{eq:conj-1-3})}}{=} & \left\Vert V^{\left(F\right)}\left(F_{\varphi}\right)\right\Vert _{L^{2}\left(\mathbb{R},\mu\right)}^{2}.
\end{eqnarray*}
\end{proof}
\begin{defn}
Set 
\begin{eqnarray}
\left(\varphi\ast g^{\vee}\right)\left(x\right) & := & \int_{0}^{1}\varphi\left(x\right)g^{\vee}\left(x-y\right)dy\label{eq:conj-1-4}\\
 & = & \left(g^{\vee}\ast\varphi\right)\left(x\right),\; x\in\left(0,1\right),\varphi\in C_{c}\left(0,1\right).\nonumber 
\end{eqnarray}
\end{defn}
\begin{thm}
We have 
\begin{equation}
\left(V^{\left(F\right)*}M_{g}V^{\left(\overline{F}\right)}\right)\left(\overline{F}_{\varphi}\right)=T_{F}\left(g^{\vee}\ast\varphi\right),\;\forall\varphi\in C_{c}\left(0,1\right).\label{eq:conj-1-5}
\end{equation}
\end{thm}
\begin{proof}
Let $\varphi\in C_{c}\left(0,1\right)$, we will then compute the
two sides in (\ref{eq:conj-1-5}), where $g^{\vee}:=$ inverse Fourier
transform:
\[
\mbox{LHS}_{\left(\ref{eq:conj-1-5}\right)}=V^{\left(F\right)*}\left(g\widehat{\varphi}\right);
\]
($\widehat{\varphi}\in L^{2}\left(\mu^{\left(s\right)}\right)$, and
using that $g\widehat{\varphi}\in L^{2}\left(\mu\right)$ by \lemref{conj-1-1}
) we get: 
\begin{eqnarray*}
\mbox{LHS}_{\left(\ref{eq:conj-1-5}\right)} & = & \left(V^{\left(F\right)}\right)^{*}\underset{\in L^{2}\left(\mu\right)}{\left(\underbrace{\widehat{g^{\vee}\ast\varphi}}\right)}\\
 & \underset{\left(\ref{eq:conj-1-5}\right)}{=} & T_{F}\left(g^{\vee}\ast\varphi\right)
\end{eqnarray*}
where $T_{F}$ is the Mercer operator $T_{F}:L^{2}\left(\Omega\right)\rightarrow\mathscr{H}_{F}$
defined using 
\begin{eqnarray*}
T_{F}\left(\varphi\right)\left(x\right) & = & \int_{0}^{1}\varphi\left(x\right)F\left(x-y\right)dy\\
 & = & \chi_{\left[0,1\right]}\left(x\right)\left(\widehat{\varphi}d\mu\right)^{\vee}\left(x\right).
\end{eqnarray*}
\end{proof}
\begin{cor}
Let $\mu$ and $\mu^{\left(s\right)}$ be as above, with $\mu\in Ext\left(F\right)$,
and $\mu^{\left(s\right)}\ll\mu$. Setting $g=\sqrt{\frac{d\mu^{\left(s\right)}}{d\mu}}$,
we get 
\begin{equation}
\Bigl\Vert\bigl(V^{\left(F\right)}\bigr)^{*}M_{g}V^{\left(\overline{F}\right)}\Bigr\Vert_{\mathscr{H}_{F}\rightarrow\mathscr{H}_{F}}\leq1.\label{eq:conj-2-1}
\end{equation}
\end{cor}
\begin{proof}
For the three factors in the composite operator $\left(V^{\left(F\right)}\right)^{*}M_{b}V^{\left(\overline{F}\right)}$
in (\ref{eq:conj-2-1}), we have two isometries as follows:
\begin{eqnarray*}
V^{\left(\overline{F}\right)}:\mathscr{H}_{\overline{F}} & \rightarrow & L^{2}\bigl(\mu^{\left(s\right)}\bigr),\mbox{ and}\\
M_{g}:L^{2}\bigl(\mu^{\left(s\right)}\bigr) & \rightarrow & L^{2}\left(\mu\right),
\end{eqnarray*}
and both isometries; while
\[
\bigl(V^{\left(F\right)}\bigr)^{*}:L^{2}\left(\mu\right)\rightarrow\mathscr{H}_{F}
\]
is co-isometric, and therefore contractive, i.e., 
\begin{equation}
\bigl\Vert\bigl(V^{\left(F\right)}\bigr)^{*}\bigr\Vert_{L^{2}\left(\mu\right)\rightarrow\mathscr{H}_{F}}\leq1.\label{eq:conj-2-2}
\end{equation}
But then: 
\begin{eqnarray*}
\Bigl\Vert\bigl(V^{\left(F\right)}\bigr)^{*}M_{g}V^{\left(\overline{F}\right)}\Bigr\Vert_{\mathscr{H}_{F}\rightarrow\mathscr{H}_{F}} & \leq & \Bigl\Vert\bigl(V^{\left(F\right)}\bigr)^{*}\Bigr\Vert\Bigl\Vert M_{g}\Bigr\Vert\Bigl\Vert V^{\left(\overline{F}\right)}\Bigr\Vert\\
 & = & \Bigl\Vert\bigl(V^{\left(F\right)}\bigr)^{*}\Bigr\Vert\leq1,\:\mbox{by \ensuremath{\left(\ref{eq:conj-2-2}\right)}}.
\end{eqnarray*}

\end{proof}

\section{\label{sec:imgF}The Imaginary Part of a Positive Definite Function}
\begin{lem}
\label{lem:Im-Lemma-1}Let $F:(-1,1)\to\mathbb{C}$ be a continuous
p.d. function. For $\phi$ in $C_{c}^{\infty}(0,1)$ let 
\[
\left(t(\phi)\right)(x)=\phi(1-x),\quad\text{for all }x\in(0,1).
\]
The operator $F_{\phi}\to F_{t(\phi)}$ is bounded in $\mathscr{H}_{F}$
iff 
\begin{equation}
\overline{F}\ll F\label{eq:im-1}
\end{equation}
where $\overline{F}$ is the complex conjugate of $F,$ and $\ll$
is the order on p.d. functions, i.e., there is an $A<\infty$ such
that 
\begin{equation}
\sum\sum\overline{c_{j}}c_{k}\overline{F}(x_{j}-x_{k})\leq A\sum\sum\overline{c_{j}}c_{k}F(x_{j}-x_{k}),\label{eq:im-2}
\end{equation}
for all finite systems $\left\{ c_{j}\right\} $ and $\left\{ x_{j}\right\} $,
where $c_{j}\in\mathbb{C}$, $x_{j}\in\left(0,1\right)$.\end{lem}
\begin{proof}
It follows from (\ref{eq:im-2}) that $\overline{F}\ll F$ iff there
is an $A<\infty$ such that 
\begin{equation}
\left\Vert \overline{F_{\phi}}\right\Vert _{\mathscr{H}_{\overline{F}}}\leq\sqrt{A}\left\Vert F_{\phi}\right\Vert _{\mathscr{H}_{F}},\label{eq:im-3}
\end{equation}
for all $\phi$ in $C_{c}^{\infty}(0,1).$ Since 
\begin{align}
\left\Vert F_{t(\phi)}\right\Vert _{\mathscr{H}_{F}}^{2} & =\int_{0}^{1}\int_{0}^{1}\overline{\phi(1-x)}\phi(1-y)F(x-y)dxdy\nonumber \\
 & =\int_{0}^{1}\int_{0}^{1}\overline{\phi(x)}\phi(y)F(y-x)dxdy\label{eq:im-4}\\
 & =\left\Vert \overline{F_{\phi}}\right\Vert _{\mathscr{H}_{\overline{F}}}^{2}\nonumber 
\end{align}
we have established the claim. 
\end{proof}
Let $M=\left(M_{jk}\right)$ be an $N\times N$ matrix over $\mathbb{C}.$
Set 
\[
\Re\left\{ M\right\} =\left(\Re\left\{ M_{jk}\right\} \right),\;\Im\left\{ M\right\} =\left(\Im\left\{ M_{jk}\right\} \right).
\]
Assume $M^{*}=M,$ where $M^{*}$ is the conjugate transpose of $M,$
and $M\geq0$. Recall, $M\geq0$ iff all eigenvalues of $M$ are $\geq0$
iff all sub-determinants $\mathrm{det}M_{n}\geq0,$ $n=1,\ldots,N$,
where $M_{n}=\left(M_{jk}\right)_{j,k\leq n}.$ 
\begin{defn}
Let $s(x)=-x.$ For a measure $\mu$ on $\mathbb{R}$, let $\mu^{s}=\mu\circ s.$ \end{defn}
\begin{lem}
\label{lem:Im-mu-s}If $F=\widehat{d\mu}$ then $\overline{F}=\widehat{d\mu^{s}}.$ \end{lem}
\begin{proof}
Suppose $F=\widehat{d\mu}$, then the calculation 
\begin{align*}
\overline{F(x)} & =F(-x)=\int_{\mathbb{R}}e_{\lambda}(-x)d\mu(\lambda)\\
 & =\int_{\mathbb{R}}e_{-\lambda}(x)d\mu(\lambda)\\
 & =\int_{\mathbb{R}}e_{\lambda}(x)d\mu^{s}(\lambda)=\widehat{d\mu^{s}}(x)
\end{align*}
establishes the claim.\end{proof}
\begin{cor}
\label{cor:Im-eq-3}If $F=\widehat{d\mu},$ then (\ref{eq:im-3})
takes the form 
\[
\int_{\mathbb{R}}\left|\widehat{\phi}(\lambda)\right|^{2}d\mu^{s}(\lambda)\leq A\int_{\mathbb{R}}\left|\widehat{\phi}(\lambda)\right|^{2}d\mu(\lambda),
\]
for all $\phi$ in $C_{c}^{\infty}(0,1).$ \end{cor}
\begin{proof}
A calculation show that 
\[
\left\Vert F_{\phi}\right\Vert _{\mathscr{H}_{F}}^{2}=\int_{\mathbb{R}}\left|\widehat{\phi}(\lambda)\right|^{2}d\mu(\lambda)
\]
and similarly $\left\Vert \overline{F_{\phi}}\right\Vert _{\mathscr{H}_{\overline{F}}}^{2}=\int_{\mathbb{R}}\bigl|\widehat{\phi}(\lambda)\bigr|^{2}d\mu^{s}(\lambda),$
where we used Lemma \ref{lem:Im-mu-s}.\end{proof}
\begin{example}
\label{Example:Im-example-5}If $\mu=\tfrac{1}{2}\left(\delta_{-1}+\delta_{2}\right),$
then $\mu^{s}=\tfrac{1}{2}\left(\delta_{-1}+\delta_{2}\right).$ Set
\begin{align*}
F(x) & =\widehat{\mu}(x)=\tfrac{1}{2}\left(e^{-ix}+e^{i2x}\right),\text{ then}\\
\overline{F(x)} & =\widehat{\mu^{s}}(x)=\tfrac{1}{2}\left(e^{ix}+e^{-i2x}\right).
\end{align*}
It follows from Corollary \ref{cor:Im-eq-3} and Lemma \ref{lem:Im-Lemma-1}
that $\overline{F}\not\ll F$ and $F\not\ll\overline{F}$. In fact,
\begin{align*}
\left\Vert F_{\phi}\right\Vert _{\mathscr{H}_{F}}^{2} & =\tfrac{1}{2}\left(\left|\widehat{\phi}(-1)\right|^{2}+\left|\widehat{\phi}(2)\right|^{2}\right)\\
\left\Vert \overline{F_{\phi}}\right\Vert _{\mathscr{H}_{\overline{F}}}^{2} & =\tfrac{1}{2}\left(\left|\widehat{\phi}(1)\right|^{2}+\left|\widehat{\phi}(-2)\right|^{2}\right).
\end{align*}
Fix $f\in C_{c}^{\infty},$ such that $f(0)=1,$ $f\geq0,$ and $\int f=1.$
Considering $\phi_{n}(x)=\tfrac{1}{2}\left(e^{-ix}+e^{i2x}\right)f\left(\tfrac{x}{n}\right),$
and $\psi_{n}(x)=\tfrac{1}{2}\left(e^{-ix}+e^{i2x}\right)f\left(\tfrac{x}{n}\right),$
completes the verification, since $\widehat{\phi_{n}}\to\mu$ and
$\widehat{\psi_{n}}\to\mu^{s}.$ 

And similarly, $\overline{F}\not\ll F$ and $F\not\ll\overline{F}$,
where $F$ is as in Example \ref{Example:im-14}. \end{example}
\begin{rem}
In fact, $\overline{F}\ll F$ iff $\mu^{s}\ll\mu$ with Radon-Nikodym
derivative $\frac{d\mu^{s}}{d\mu}\in L^{\infty}(\mu).$ See, \secref{FFbar}.
\index{Radon-Nikodym}\end{rem}
\begin{cor}
\label{cor:im-mu-mu-s}If $F=\widehat{d\mu},$ then 
\begin{align*}
\Re\left\{ F\right\}  & =\tfrac{1}{2}\widehat{\left(\mu+\mu^{s}\right)},\text{ and }\\
\Im\left\{ F\right\}  & =\tfrac{1}{2i}\widehat{\left(\mu-\mu^{s}\right)}.
\end{align*}

\end{cor}
We can rewrite the corollary in the form: If $F=\widehat{d\mu},$
then
\begin{align}
\Re\left\{ F\right\} (x) & =\int_{\mathbb{R}}\cos\left(\lambda x\right)\, d\mu(\lambda),\text{ and }\label{eq:im-a-1}\\
\Im\left\{ F\right\} (x) & =\int_{\mathbb{R}}\sin\left(\lambda x\right)\, d\mu(\lambda).\label{eq:im-a-2}
\end{align}

\begin{rem}
(\ref{eq:im-a-1}) simply states that if $F$ is positive definite,
so is its real part $\Re\left\{ F\right\} $. But (\ref{eq:im-a-2})
is deeper: If the function $\lambda$ is in $L^{1}(\mu),$ then 
\[
\frac{d}{dx}\Im\left\{ F\right\} (x)=\int_{\mathbb{R}}\cos(\lambda x)\lambda\, d\mu(\lambda)
\]
is the cosine transform of $\lambda d\mu(\lambda).$
\end{rem}
Suppose $F$ is p.d. on $(-a,a)$ and $\mu\in\mathrm{Ext}(F)$, i.e.,
$\mu$ is a finite positive measure satisfying 
\[
F(x)=\int_{\mathbb{R}}e_{\lambda}(x)d\mu(x).
\]
For a finite set $\{x_{j}\}\in(-a,a)$ let 
\[
M:=\left(F\left(x_{j}-x_{k}\right)\right).
\]
For $c_{j}$ in $\mathbb{C}$ consider 
\[
\overline{c^{T}}Mc=\sum\sum\overline{c_{j}}c_{k}M_{jk}.
\]
The for $\Re\left\{ F\right\} $, we have
\begin{align*}
 & \sum_{j}\sum_{k}\overline{c_{j}}c_{k}\Re\left\{ F\right\} (x_{j}-x_{k})\\
= & \sum_{j}\sum_{k}\overline{c_{j}}c_{k}\int_{\mathbb{R}}\left(\cos\left(\lambda x_{j}\right)\cos\left(\lambda x_{k}\right)+\sin\left(\lambda x_{j}\right)\sin\left(\lambda x_{k}\right)\right)d\mu(\lambda)\\
= & \int_{\mathbb{R}}\left(\left|C(\lambda)\right|^{2}+\left|S(\lambda)\right|^{2}\right)d\mu(\lambda)\geq0,
\end{align*}
where 
\begin{align*}
C(\lambda) & =C\left(\lambda,\left(x_{j}\right)\right)=\sum_{j}c_{j}\cos\left(\lambda x_{j}\right)\\
S(\lambda) & =S\left(\lambda,\left(x_{j}\right)\right)=\sum_{j}c_{j}\sin\left(\lambda x_{j}\right)
\end{align*}
for all $\lambda\in\mathbb{R}.$ Similarly, for $\Im\left\{ F\right\} $,
we have 
\begin{align*}
 & \sum_{j}\sum_{k}\overline{c_{j}}c_{k}\Im\left\{ F\right\} (x_{j}-x_{k})\\
= & \sum_{j}\sum_{k}\overline{c_{j}}c_{k}\int_{\mathbb{R}}\left(\sin\left(\lambda x_{j}\right)\cos\left(\lambda x_{k}\right)-\cos\left(\lambda x_{j}\right)\sin\left(\lambda x_{k}\right)\right)d\mu(\lambda)\\
= & \int_{\mathbb{R}}\left(\overline{S(\lambda)}C(\lambda)-\overline{C(\lambda)}S(\lambda)\right)d\mu(\lambda)\\
= & 2i\int_{\mathbb{R}}\Im\left\{ \overline{S\left(\lambda\right)}C\left(\lambda\right)\right\} d\mu(\lambda).
\end{align*}
If $\left\{ c_{j}\right\} \subset\mathbb{R},$ then $S(\lambda),C(\lambda)$
are real valued and 
\[
\sum_{j}\sum_{k}\overline{c_{j}}c_{k}\Im\left\{ F\right\} (x_{j}-x_{k})=0.
\]

\subsection{An application of Bochner\textquoteright s Theorem }
\begin{lem}
Let $F$ be a continuous positive definite function on some open interval
$(-a,a)$. Let $K$ be the real part $\Re\left\{ F\right\} $ of $F$
and let $L$ be the imaginary part $\Im\left\{ F\right\} $ of $F,$
hence $K$ and $L$ are real valued, $K$ is a continuous positive
definite real valued function, in particular $K$ is an even function,
and $L$ is an odd function.\index{Bochner}\end{lem}
\begin{proof}
The even/odd claims follow from $F(-x)=\overline{F(x)}$ for $x\in(-a,a).$
For a finite set of points $\left\{ x_{j}\right\} _{j=1}^{N}$ in
$(-a,a)$ form the matrices 
\[
M_{F}=\left(F(x_{j}-x_{k})\right)_{j,k=1}^{N},M_{K}=\left(K(x_{j}-x_{k})\right)_{j,k=1}^{N}M_{L}=\left(L(x_{j}-x_{k})\right)_{j,k=1}^{N}.
\]
Let $c=(c_{j})$ be a vector in $\mathbb{R}^{N}.$ Since $L$ is an
odd function it follows that $c^{T}M_{L}c=0,$ consequently, 
\begin{equation}
c^{T}M_{K}c=c^{T}M_{F}c\geq0.\label{eq:im-MMM}
\end{equation}
It follows that $K$ is positive definite over the real numbers and
therefore also over the complex numbers \cite{Aro50}.\end{proof}
\begin{defn}
We say a signed measure $\mu$ is \emph{even,} if $\mu(B)=\mu(-B)$
for all $\mu-$measurable sets $B,$ where $-B=\left\{ -x:x\in B\right\} .$
Similarly, we say $\mu$ is \emph{odd,} if $\mu(B)=-\mu(-B)$ for
all $\mu$--measurable sets $B.$ \end{defn}
\begin{rem}
\label{remark:im-1}Let 
\begin{align*}
\mu_{K}(B) & :=\frac{\mu(B)+\mu(-B)}{2}\;\mbox{ and}\\
\mu_{L}(B) & :=\frac{\mu(B)-\mu(-B)}{2}
\end{align*}
for all $\mu-$measurable sets $B.$ Then $\mu_{K}$ is an even probability
measure and $\mu_{L}$ is an odd real valued measure. If $F=\widehat{d\mu},$
$K=\widehat{d\mu_{K}},$ and $iL=\widehat{d\mu_{L}},$ then $K$ and
$L$ are real valued continuous functions, $F$ and $K$ are continuous
positive definite functions, $L$ is a real valued continuous odd
function and $F=K+iL.$\end{rem}
\begin{lem}
\label{lem:im-bo}Suppose $K$ as the Fourier transform of some even
probability measure $\mu_{K}$ and $iL$ as the Fourier transform
of some odd measure $\mu_{L},$ then $F:=K+iL$ is positive definite
iff $\mu:=\mu_{K}+\mu_{L}$ is a probability measure, i.e., iff $\mu(B)\geq0$
for all Borel set $B.$ \end{lem}
\begin{proof}
This is a direct consequence of Bochner's theorem.\end{proof}
\begin{cor}
If $F$ is positive definite, and $\Im\left\{ F\right\} \neq0,$ then 

(i) $F_{m}:=\Re\left\{ F\right\} +im\Im\left\{ F\right\} $, is positive
definite for all $-1\leq m\leq1$ and 

(ii) $F_{m}$ is not positive definite for sufficiently large $m.$ \end{cor}
\begin{proof}
(\emph{i}) We will use the notation from Remark \ref{remark:im-1}.
If $0<m$ and $\mu_{L}(B)<0,$ then
\[
\mu_{m}(B):=\mu_{K}(B)+m\,\mu_{L}(B)\geq\mu(B)\geq0.
\]
The cases where $m<0$ are handled by using that $\overline{F}$ is
positive definite. 

(\emph{ii}) Is established using a similar argument. \end{proof}
\begin{cor}
Let $K$ and $L$ be real valued continuous functions on $\mathbb{R}$.
Suppose $K$ positive definite and $L$ odd, and let $\mu_{k}$ and
$\mu_{L}$ be the correspond even and odd measures. If $K+im\, L$
is positive definite for some real $m\neq0,$ then the support of
$\mu_{L}$ is a subset of the support of $\mu_{K}.$ \end{cor}
\begin{proof}
Fix $m\neq0.$ If the support of $\mu_{L}$ is not contained in the
support of $\mu_{K},$ then $\mu_{K}(B)+m\,\mu_{L}(B)<0$ for some
$B.$\end{proof}
\begin{rem}
The converse fails, support containment does not imply $\mu_{k}+m\mu_{L}$
is positive for some $m>0$ since $\mu_{K}$ can ``decrease'' much
faster than $\mu_{L}.$\end{rem}
\begin{example}
\label{Example:im-14}Let $d\mu(\lambda):=\delta_{-1}+\chi_{\mathbb{R}^{+}}(\lambda)e^{-\lambda}d\lambda$
and set $F:=\restr{\widehat{d\mu}}{(-1,1)}.$ Then
\begin{align*}
\Re\left\{ F\right\} (x) & =\cos(x)+\frac{1}{1+x^{2}}\\
\Im\left\{ F\right\} (x) & =-\sin(x)+\frac{x}{1+x^{2}};
\end{align*}
and $D^{(F)}$ in $\mathscr{H}_{F}$ has deficiency indices $(0,0).$
\index{deficiency indices}\end{example}
\begin{proof}
By construction:
\begin{align}
F(x) & =\int_{\mathbb{R}}e^{i\lambda x}d\mu(\lambda)=e^{-ix}+\int_{0}^{\infty}e^{i\lambda x}-\lambda d\lambda\nonumber \\
 & =e^{-ix}+\frac{1}{1-ix};\label{eq:im-10}
\end{align}
establishing the first claim. 

Consider $u=T_{F}\phi,$ for some $\phi\in C_{c}^{\infty}(0,1).$
By (\ref{eq:im-10})
\begin{align*}
u(x) & =\int_{0}^{1}\phi(y)F(x-y)dy\\
 & =\widehat{\phi}(-1)e^{-ix}+\int_{0}^{1}\phi(y)\frac{1}{1-i(x-y)}dy.
\end{align*}
Taking two derivatives we get 
\[
u''(x)=-\widehat{\phi}(-1)e^{-ix}+\int_{0}^{1}\phi(y)\frac{-2}{\left(1-i(x-y)\right)^{3}}dy.
\]
It follows that $u''+u\to0$ as $x\to\pm\infty,$ i.e., 
\begin{equation}
\lim_{|x|\to\infty}\left|u''(x)+u(x)\right|=0.\label{eq:im-9}
\end{equation}
A standard approximation argument shows that (\ref{eq:im-9}) holds
for all $u\in\mathscr{H}_{F}.$ 

Equation (\ref{eq:im-9}) rules out that either of $e^{\pm x}$ is
on $\mathscr{H}_{F}$, hence the deficiency indices are $(0,0)$ as
claimed. 
\end{proof}
If $\mu_{K}$ is an even probability measure and $f\left(x\right)$
is an odd function, s.t. $-1\leq f(x)\leq1$, then $d\mu(x):=\left(1+f(x)\right)d\mu_{K}(x)$
is a probability measure. Conversely, we have
\begin{lem}
\label{lem:RN+Hahn}Let $\mu$ be a probability measure on the Borel
sets. There is an even probability measure $\mu_{K}$ and an odd real
valued $\mu-$measurable function $f$ with $\left|f\right|\leq1,$
such that $d\mu(\lambda)=\left(1+f(\lambda)\right)d\mu_{K}(\lambda).$ \end{lem}
\begin{proof}
Let $\mu_{K}:=\tfrac{1}{2}\left(\mu+\mu^{s}\right)$ and $\mu_{L}:=\tfrac{1}{2}\left(\mu-\mu^{s}\right).$
Clearly, $\mu_{K}$ is an even probability measure and $\mu_{L}$
is an odd real valued measure. Since $\mu_{K}(B)+\mu_{L}(B)=\mu(B)\geq0,$
it follows that $\mu_{K}(B)\geq\mu_{L}(B)$ for all Borel sets $B.$ 

Applying the Hahn decomposition theorem to $\mu_{L}$ we get sets
$P$ and $N$ such that $P\cap N=\emptyset,$ $\mu_{L}(B\cap P)\geq0$
and $\mu_{L}(B\cap N)\leq0$ for all $B.$ Let
\begin{align*}
P' & :=\left\{ x\in P:-x\in N\right\} \\
N' & :=\left\{ x\in N:-x\in P\right\} \\
O' & :=\left(P\setminus P'\right)\cup\left(N\setminus N'\right),
\end{align*}
then $N'=-P'$ and $\mu_{L}(B\cap O')=0$ for all $B.$ Write 
\[
\mu_{L}\left(B\right)=\mu_{L}\left(B\cap P'\right)+\mu_{L}\left(B\cap N'\right).
\]
Then $\mu_{K}\left(B\right)\geq\mu_{K}\left(B\cap P'\right)\geq\mu_{L}\left(B\cap P'\right)$
and
\begin{align*}
0\leq-\mu_{L}\left(B\cap N'\right) & =\mu_{L}\left(-\left(B\cap N'\right)\right)\\
 & =\mu_{L}\left(-B\cap P'\right)\\
 & \leq\mu_{K}\left(-B\cap P'\right)\leq\mu_{K}(B).
\end{align*}
Hence, $\mu_{L}$ is absolutely continuous with respect to $\mu_{K}.$
Setting $f:=\frac{d\mu_{L}}{d\mu_{K}}$, the Radon-Nikodym derivative
of $\mu_{L}$ with respect to $\mu_{K},$ completes the proof. \index{absolutely continuous}\end{proof}
\begin{cor}
Let $F=\widehat{\mu}$ be a positive definite function with $F(0)=1.$
Let $\mu_{K}:=\tfrac{1}{2}\left(\mu+\mu^{s}\right)$ then $\Re\left\{ F\right\} (x)=\widehat{\mu_{K}}(x)$
and there is an odd function 
\[
-1\leq f(\lambda)\leq1,
\]
 such that $\Im\left\{ F\right\} (x)=\widehat{f\mu_{K}}(x)$.
\begin{cor}
\label{cor:RI}Let $F$ be a continuous p.d. function on $(-a,a).$
Let $\Re\left\{ F\right\} $ be the real part of $F.$ Then $\mathscr{H}_{F}$
is a subset of $\mathscr{H}_{\Re\left\{ F\right\} }.$ In particular,
if $D^{\left(\Re\left\{ F\right\} \right)}$ has deficiency indices
$(1,1)$ so does $D^{\left(F\right)}.$ \index{deficiency indices}
\end{cor}
\end{cor}
\begin{proof}
Recall, a continuous function $\xi$ is in $ $$\mathscr{H}_{F}$
iff 
\[
\left|\int_{0}^{1}\psi(y)\xi(y)dy\right|^{2}\leq A\int_{0}^{1}\int_{0}^{1}\overline{\phi(x)}\phi(y)F(x-y)dxdy.
\]
Since, 
\begin{align*}
\int_{0}^{1}\int_{0}^{1}\overline{\phi(x)}\phi(y)F(x-y)dxdy & =\int_{0}^{1}\int_{0}^{1}\int_{\mathbb{R}}\overline{\phi(x)}\phi(y)e^{-i\lambda(x-y)}d\mu(\lambda)dxdy\\
 & =\int_{\mathbb{R}}\left|\phi(\lambda)\right|^{2}d\mu(\lambda)\\
 & \leq2\int_{\mathbb{R}}\left|\phi(\lambda)\right|^{2}d\mu_{K}(\lambda)\\
 & =2\int_{0}^{1}\int_{0}^{1}\overline{\phi(x)}\phi(y)K(x-y)dxdy
\end{align*}
it follows that $\mathscr{H}_{F}$ is contained in $\mathscr{H}_{\Re\left\{ F\right\} }$. 
\end{proof}

\chapter{\label{chap:conv}Convolution Products}

A source of interesting measures in probability are constructed as
product measures or convolutions; and this includes infinite operations;
see for example \cite{IM65,Jor07,KS02,Par09}.

Below we study these operations in the contest of our positive definite
functions, defined on subsets of groups. For example, most realizations
of fractal measures arise as infinite convolutions, see e.g., \cite{DJ10,JP12,JKS12,DJ12,JKS11,JKS08}.
Motivated by these applications, we show below that, given a system
of continuous positive definite functions $F_{1},F_{2},\ldots$, defined
on an open subset of a group, we can form well defined products, including
infinite products, which are again continuous positive definite. We
further show that if finite positive measures $\mu_{i}$, $i=1,2,\ldots$,
are given, $\mu_{i}\in Ext\left(F_{i}\right)$ then the convolution
of the measures $\mu_{i}$ is in $Ext\left(F\right)$ where $F$ is
the product of the p.d. functions $F_{i}$. This will be applied later
in the Memoir.

\index{measure!product}

\index{measure!convolution}

\index{measure!fractal}

\index{group!locally compact abelian}

\index{group!circle}
\begin{defn}
Let $F$ be a continuous positive definite function defined on a subset
in $G$ (a locally compact Abelian group). Set 
\begin{equation}
Ext\left(F\right)=\left\{ \mu\in\mathscr{M}\left(\widehat{G}\right)\:\big|\:\widehat{d\mu}\mbox{ is an extension of \ensuremath{F}}\right\} .\label{eq:conv-1}
\end{equation}

\end{defn}
In order to study the set $Ext\left(F\right)$ from above, it helps
to develop tools. One such tool is convolution, which we outline below.
It is also helpful in connection with the study of symmetric spaces,
such as the case $G=\mathbb{T}=\mathbb{R}/\mathbb{Z}$ (the circle
group), versus extensions to the group $\mathbb{R}$. 

Let $G$ be a locally compact group, and let $\Omega$ be a non-empty,
connected and open subset in $G$. Now consider systems of p.d. and
continuous functions on the set $\Omega^{-1}\Omega$. Specifically,
let $F_{i}$ be two or more p.d. continuous functions on $\Omega^{-1}\Omega$;
possibly an infinite family, so $F_{1},F_{2},\ldots$, all defined
on $\Omega^{-1}\Omega$. As usual, we normalize our p.d. functions
$F_{i}\left(e\right)=1$, where $e$ is the unit element in $G$.
\begin{lem}
Form the point-wise product $F$ of any system of p.d. functions $F_{i}$
on $\Omega^{-1}\Omega$; then $F$ is again p.d. and continuous on
the set $\Omega^{-1}\Omega$.\end{lem}
\begin{proof}
This is an application of a standard lemma about p.d. kernels, see
e.g., \cite{BCR84}. From this, we conclude that $F$ is again a continuous
and positive definite function on $\Omega^{-1}\Omega$.
\end{proof}
If we further assume that $G$ is also Abelian, and so $G$ is locally
compact abelian, then the spectral theory takes a more explicit form.
\begin{lem}
Assume $Ext\left(F_{i}\right)$ for $i=1,2,\ldots$ are non-empty.
For any system of individual measures $\mu_{i}\in Ext\left(F_{i}\right)$
we get that the resulting convolution-product measure $\mu$ formed
from the factors $\mu_{i}$ by the convolution in $G$, is in $Ext\left(F\right)$\end{lem}
\begin{proof}
This is an application of our results in sections \ref{sub:lcg}-\ref{sub:euclid}. \end{proof}
\begin{rem}
In some applications the convolution $\mu_{1}\ast\mu_{2}$ makes sense
even if only one of the measures is finite.
\end{rem}
\begin{flushleft}
\textbf{Application. }The case $G=\mathbb{R}$. Let $\mu_{1}$ be
the Dirac-comb (\cite{Ch03,Cor89}) 
\[
d\mu_{1}:=\sum_{n\in\mathbb{Z}}\delta\left(\lambda-n\right),\;\lambda\in\mathbb{R};
\]
let $\Phi\geq0$, $\Phi\in L^{1}\left(\mathbb{R}\right)$, and assume
$\int_{\mathbb{R}}\Phi\left(\lambda\right)d\lambda=1$. Set $d\mu_{2}=\Phi\left(\lambda\right)d\lambda$,
where $d\lambda=$ Lebesgue measure on $\mathbb{R}$; then $\mu_{1}\ast\mu_{2}$
yields the following probability measure on $\mathbb{T=\mathbb{R}}/\mathbb{Z}$:
Set 
\[
\Phi_{per}\left(\lambda\right)=\sum_{n\in\mathbb{Z}}\Phi\left(\lambda-n\right);
\]
then $\Phi_{per}\left(\lambda\right)\in L^{1}\left(\mathbb{T},dt\right)$,
where $dt=$ Lebesgue measure on $\mathbb{T}$, i.e., if $f\left(\lambda+n\right)=f\left(n\right)$,
$\forall n\in\mathbb{Z}$, $\forall\lambda\in\mathbb{R}$, then $f$
defines a function on $\mathbb{T}$, and $\int_{\mathbb{T}}f\, dt=\int_{0}^{1}f\left(t\right)\, dt$.
We get 
\[
d\left(\mu_{1}\ast\mu_{2}\right)=\Phi_{per}\left(\cdot\right)dt\;\mbox{on }\mathbb{T}.
\]

\par\end{flushleft}
\begin{proof}
We have
\begin{align*}
1 & =\int_{-\infty}^{\infty}\Phi\left(\lambda\right)d\lambda=\sum_{n\in\mathbb{Z}}\int_{n}^{n+1}\Phi\left(\lambda\right)d\lambda\\
 & =\int_{0}^{1}\sum_{n\in\mathbb{Z}}\Phi\left(\lambda-n\right)d\lambda\\
 & =\int_{\mathbb{T}}\Phi_{per}\left(t\right)dt.
\end{align*}

\end{proof}
We now proceed to study the relations between the other items in our
analysis, the RKHSs $\mathscr{H}_{F_{i}}$, for $i=1,2,\ldots$; and
computing $\mathscr{H}_{F}$ from the RKHSs $\mathscr{H}_{F_{i}}$.

We further study the associated unitary representations of $G$ when
$Ext\left(F_{i}\right)$, $i=1,2,\ldots$ are non-empty?

As an application, we get infinite convolutions, and they are fascinating;
include many fractal measures of course.

In the case of $\mathbb{G}=\mathbb{R}$ we will study the connection
between deficiency index values in $\mathscr{H}_{F}$ as compared
to those of the factor RKHSs $F_{i}$.

\chapter{\label{chap:spbd}Models for Operator Extensions}

A special case of our extension question for continuous positive definite
(p.d.) functions on a fixed finite interval $\left|x\right|<a$ in
$\mathbb{R}$ is the following: It offers a spectral model representation
for ALL Hermitian operators with dense domain in Hilbert space and
with deficiency indices $\left(1,1\right)$.

Specifically, on $\mathbb{R}$, all the partially defined continuous
p.d. functions extend, and we can make a translation of our p.d. problem
into the problem of finding all $\left(1,1\right)$ restrictions selfadjoint
operators.

By the Spectral theorem, every selfadjoint operator with simple spectrum
has a representation as a multiplication operator $M_{\lambda}$ in
some $L^{2}\left(\mathbb{R},\mu\right)$ for some probability measure
$\mu$ on $\mathbb{R}$. So this accounts for all Hermitian restrictions
operators with deficiency indices $\left(1,1\right)$.

So the problem we have been studying for just the case of $G=\mathbb{R}$
is the case of finding spectral representations for ALL Hermitian
operators with dense domain in Hilbert space having deficiency indices
$\left(1,1\right)$.

\index{deficiency indices}

\index{spectrum}

\index{measure!probability}

\section{\label{sec:R^1}Model for Restrictions of Continuous p.d. Functions
on $\mathbb{R}$}

Let $\mathscr{H}$ be a Hilbert space, $A$ a skew-adjoint operator,
$A^{*}=-A$, which is \uline{unbounded}; let $v_{0}\in\mathscr{H}$
satisfying $\left\Vert v_{0}\right\Vert _{\mathscr{H}}=1$. Then we
get an associated p.d. continuous function $F_{A}$ defined on $\mathbb{R}$
as follows:
\begin{equation}
F_{A}\left(t\right):=\left\langle v_{0},e^{tA}v_{0}\right\rangle =\left\langle v_{0},U_{A}\left(t\right)v_{0}\right\rangle ,\: t\in\mathbb{R},\label{eq:tmp11}
\end{equation}
where $U_{A}\left(t\right)=e^{tA}$ is a unitary representation of
$\mathbb{R}$. Note that we define $U\left(t\right)=U_{A}\left(t\right)=e^{tA}$
by the spectral theorem. Note (\ref{eq:tmp11}) holds for all $t\in\mathbb{R}$.

\index{unitary representation}

\index{measure!PVM}

\index{measure!probability}

Let $P_{U}\left(\cdot\right)$ be the projection-valued measure (PVM)
of $A$, then 
\begin{equation}
U\left(t\right)=\int_{-\infty}^{\infty}e^{i\lambda t}P_{U}\left(d\lambda\right),\:\forall t\in\mathbb{R}.\label{eq:tmp12}
\end{equation}

\begin{lem}
Setting
\begin{equation}
d\mu=\left\Vert P_{U}\left(d\lambda\right)v_{0}\right\Vert ^{2}\label{eq:tmp18}
\end{equation}
we then get 
\begin{equation}
F_{A}\left(t\right)=\widehat{d\mu}\left(t\right),\:\forall t\in\mathbb{R}\label{eq:tmp14}
\end{equation}
Moreover, every probability measure $\mu$ on $\mathbb{R}$ arises
this way.\end{lem}
\begin{proof}
By (\ref{eq:tmp11}), 
\begin{eqnarray*}
F_{A}\left(t\right) & = & \int e^{it\lambda}\left\langle v_{0},P_{U}\left(d\lambda\right)v_{0}\right\rangle \\
 & = & \int e^{it\lambda}\left\Vert P_{U}\left(d\lambda\right)v_{0}\right\Vert ^{2}\\
 & = & \int e^{it\lambda}d\mu\left(\lambda\right)
\end{eqnarray*}
\end{proof}
\begin{lem}
\label{lem:iso}For Borel functions $f$ on $\mathbb{R}$, let 
\begin{equation}
f\left(A\right)=\int_{\mathbb{R}}f\left(\lambda\right)P_{U}\left(d\lambda\right)\label{eq:tmp15}
\end{equation}
be given by functional calculus. We note that 
\begin{equation}
v_{0}\in dom\left(f\left(A\right)\right)\Longleftrightarrow f\in L^{2}\left(\mu\right)\label{eq:tmp16}
\end{equation}
where $\mu$ is the measure in (\ref{eq:tmp18}). Then
\begin{equation}
\left\Vert f\left(A\right)v_{0}\right\Vert ^{2}=\int_{\mathbb{R}}\left|f\right|^{2}d\mu.\label{eq:tmp17}
\end{equation}
\end{lem}
\begin{proof}
The lemma follows from 
\begin{eqnarray*}
\left\Vert f\left(A\right)v_{0}\right\Vert ^{2} & = & \left\Vert \int_{\mathbb{R}}f\left(\lambda\right)P_{U}\left(d\lambda\right)v_{0}\right\Vert ^{2}\,\,\,\,\,\,\,\,\,\,\,\,\,\,\,(\mbox{by }(\ref{eq:tmp15}))\\
 & = & \int\left|f\left(\lambda\right)\right|^{2}\left\Vert P_{U}\left(d\lambda\right)v_{0}\right\Vert ^{2}\\
 & = & \int\left|f\left(\lambda\right)\right|^{2}d\mu\left(\lambda\right).\,\,\,\,\,\,\,\,\,\,\,\,\,\,\,\,\,\,\,\,\,\,\,\,\,\,\,\,\,(\mbox{by }(\ref{eq:tmp18}))
\end{eqnarray*}

\end{proof}
Now we consider restriction of $F_{A}$ to, say $\left(-1,1\right)$,
i.e., 
\begin{equation}
F\left(\cdot\right)=F_{A}\Big|_{\left(-1,1\right)}\left(\cdot\right)\label{eq:tmp19}
\end{equation}

\begin{lem}
Let $\mathscr{H}_{F}$ be the RKHS computed for $F$ in (\ref{eq:tmp14});
and for $\varphi\in C_{c}\left(0,1\right)$, set $F_{\varphi}=$ the
generating vectors in $\mathscr{H}_{F}$, as usual. Set 
\begin{equation}
U\left(\varphi\right):=\int_{0}^{1}\varphi\left(y\right)U\left(-y\right)dy\label{eq:tmp20}
\end{equation}
where $dy=$ Lebesgue measure on $\left(0,1\right)$; then
\begin{equation}
F_{\varphi}\left(x\right)=\left\langle v_{0},U\left(x\right)U\left(\varphi\right)v_{0}\right\rangle ,\:\forall x\in\left(0,1\right).
\end{equation}
\end{lem}
\begin{proof}
We have
\begin{eqnarray*}
F_{\varphi}\left(x\right) & = & \int_{0}^{1}\varphi\left(y\right)F\left(x-y\right)dy\\
 & = & \int_{0}^{1}\varphi\left(y\right)\left\langle v_{0},U_{A}\left(x-y\right)v_{0}\right\rangle dy\,\,\,\,\,\,\,\,\,\,(\mbox{by }(\ref{eq:tmp11}))\\
 & = & \left\langle v_{0},U_{A}\left(x\right)\int_{0}^{1}\varphi\left(y\right)U_{A}\left(-y\right)v_{0}dy\right\rangle \\
 & = & \left\langle v_{0},U_{A}\left(x\right)U\left(\varphi\right)v_{0}\right\rangle \,\,\,\,\,\,\,\,\,\,(\mbox{by }(\ref{eq:tmp20}))\\
 & = & \left\langle v_{0},U\left(\varphi\right)U_{A}\left(x\right)v_{0}\right\rangle 
\end{eqnarray*}
for all $x\in\left(0,1\right)$, and all $\varphi\in C_{c}\left(0,1\right)$.\end{proof}
\begin{cor}
Let $A$, $U\left(t\right)=e^{tA}$, $v_{0}\in\mathscr{H}$, $\varphi\in C_{c}\left(0,1\right)$,
and $F$ p.d. on $\left(0,1\right)$ be as above; let $\mathscr{H}_{F}$
be the RKHS of $F$; then, for the inner product in $\mathscr{H}_{F}$,
we have 
\begin{equation}
\left\langle F_{\varphi},F_{\psi}\right\rangle _{\mathscr{H}_{F}}=\left\langle U\left(\varphi\right)v_{0},U\left(\psi\right)v_{0}\right\rangle _{\mathscr{H}},\:\forall\varphi,\psi\in C_{c}\left(0,1\right).\label{eq:tmp-01}
\end{equation}
\end{cor}
\begin{proof}
Note that 
\begin{eqnarray*}
\left\langle F_{\varphi},F_{\psi}\right\rangle _{\mathscr{H}_{F}} & = & \int_{0}^{1}\int_{0}^{1}\overline{\varphi\left(x\right)}\psi\left(y\right)F\left(x-y\right)dxdy\\
 & = & \int_{0}^{1}\int_{0}^{1}\overline{\varphi\left(x\right)}\psi\left(y\right)\left\langle v_{0},U_{A}\left(x-y\right)v_{0}\right\rangle _{\mathscr{H}}dxdy\,\,\,\,\,\,\,\,\,\,(\mbox{by }(\ref{eq:tmp19}))\\
 & = & \int_{0}^{1}\int_{0}^{1}\left\langle \varphi\left(x\right)U_{A}\left(-x\right)v_{0},\psi\left(y\right)U_{A}\left(-y\right)v_{0}\right\rangle _{\mathscr{H}}dxdy\\
 & = & \left\langle U\left(\varphi\right)v_{0},U\left(\psi\right)v_{0}\right\rangle _{\mathscr{H}}\,\,\,\,\,\,\,\,\,\,(\mbox{by }(\ref{eq:tmp20}))
\end{eqnarray*}
\end{proof}
\begin{cor}
Set $\varphi^{\#}\left(x\right)=\overline{\varphi\left(-x\right)}$,
$x\in\mathbb{R}$, $\varphi\in C_{c}\left(\mathbb{R}\right)$, or
in this case, $\varphi\in C_{c}\left(0,1\right)$; then we have:
\begin{equation}
\left\langle F_{\varphi},F_{\psi}\right\rangle _{\mathscr{H}_{F}}=\left\langle v_{0},U\left(\varphi^{\#}\ast\psi\right)v_{0}\right\rangle _{\mathscr{H}},\;\forall\varphi,\psi\in C_{c}\left(0,1\right).\label{eq:tmp-02}
\end{equation}
\end{cor}
\begin{proof}
Immediate from (\ref{eq:tmp-01}) and Fubini. \end{proof}
\begin{cor}
Let $F$ and $\varphi\in C_{c}\left(0,1\right)$ be as above; then
in the RKHS $\mathscr{H}_{F}$ we have:
\begin{equation}
\left\Vert F_{\varphi}\right\Vert _{\mathscr{H}_{F}}^{2}=\left\Vert U\left(\varphi\right)v_{0}\right\Vert _{\mathscr{H}}^{2}=\int\left|\widehat{\varphi}\right|^{2}d\mu\label{eq:tmp-03}
\end{equation}
where $\mu$ is the measure in (\ref{eq:tmp18}). $\widehat{\varphi}=$
Fourier transform: $\widehat{\varphi}\left(\lambda\right)=\int_{0}^{1}e^{-i\lambda x}\varphi\left(x\right)dx$,
$\lambda\in\mathbb{R}$.\end{cor}
\begin{proof}
Immediate from (\ref{eq:tmp-02}); indeed:
\begin{eqnarray*}
\left\Vert F_{\varphi}\right\Vert _{\mathscr{H}_{F}}^{2} & = & \int_{0}^{1}\int_{0}^{1}\overline{\varphi\left(x\right)}\varphi\left(y\right)\int_{\mathbb{R}}e_{\lambda}\left(x-y\right)d\mu\left(\lambda\right)\\
 & = & \int_{\mathbb{R}}\left|\widehat{\varphi}\left(\lambda\right)\right|^{2}d\mu\left(\lambda\right),\;\forall\varphi\in C_{c}\left(0,1\right).
\end{eqnarray*}
\end{proof}
\begin{cor}
Every Borel probability measure $\mu$ on $\mathbb{R}$ arises this
way. \index{measure!probability}\end{cor}
\begin{proof}
We shall need to following:
\begin{lem}
Let $A$, $\mathscr{H}$, $\left\{ U_{A}\left(t\right)\right\} _{t\in\mathbb{R}}$,
$v_{0}\in\mathscr{H}$ be as above; and set
\begin{equation}
d\mu=d\mu_{A}\left(\cdot\right)=\left\Vert P_{U}\left(\cdot\right)v_{0}\right\Vert ^{2}\label{eq:tmp-04}
\end{equation}
as in (\ref{eq:tmp18}). Assume $v_{0}$ is cyclic; then $W_{\mu}f\left(A\right)v_{0}=f$
defines a unitary isomorphism $W_{\mu}:\mathscr{H}\rightarrow L^{2}\left(\mu\right)$;
and 
\begin{equation}
W_{\mu}U_{A}\left(t\right)=e^{it\cdot}W_{\mu}\label{eq:tmp-05}
\end{equation}
where $e^{it\cdot}$ is seen as a multiplication operator in $L^{2}\left(\mu\right)$.
More precisely:
\begin{equation}
\left(W_{\mu}U\left(t\right)\xi\right)\left(\lambda\right)=e^{it\lambda}\left(W_{\mu}\xi\right)\left(\lambda\right),\;\forall t,\lambda\in\mathbb{R},\forall\xi\in\mathscr{H}.\label{eq:tmp-06}
\end{equation}
(We say that the isometry $W_{\mu}$ \uline{intertwines} the two
unitary one-parameter groups.)

\index{isometry}

\index{operator!unitary one-parameter group}\end{lem}
\begin{proof}
Since $v_{0}$ is cyclic, it is enough to consider $\xi\in\mathscr{H}$
of the following form: $\xi=f\left(A\right)v_{0}$, with $f\in L^{2}\left(\mu\right)$,
see (\ref{eq:tmp16}) in Lemma \ref{lem:iso}. Then
\begin{equation}
\left\Vert \xi\right\Vert _{\mathscr{H}}^{2}=\int_{\mathbb{R}}\left|f\left(\lambda\right)\right|^{2}d\mu\left(\lambda\right),\:\mbox{so}\label{eq:tmp-07}
\end{equation}
\[
\left\Vert W_{\mu}\xi\right\Vert _{L^{2}\left(\mu\right)}=\left\Vert \xi\right\Vert _{\mathscr{H}}\:(\Longleftrightarrow(\ref{eq:tmp-07}))
\]

For the adjoint operator $W_{\mu}^{*}:L^{2}\left(\mathbb{R},\mu\right)\rightarrow\mathscr{H}$,
we have
\[
W_{\mu}^{*}f=f\left(A\right)v_{0},
\]
see (\ref{eq:tmp15})-(\ref{eq:tmp17}). Note that $f\left(A\right)v_{0}\in\mathscr{H}$
is well-defined for all $f\in L^{2}\left(\mu\right)$. Also $W_{\mu}^{*}W_{\mu}=I_{\mathscr{H}}$,
$W_{\mu}W_{\mu}^{*}=I_{L^{2}\left(\mu\right)}$. 

Proof of (\ref{eq:tmp-06}). Take $\xi=f\left(A\right)v_{0}$, $f\in L^{2}\left(\mu\right)$,
and apply the previous lemma, we have 
\[
W_{\mu}U\left(t\right)\xi=W_{\mu}U\left(t\right)f\left(A\right)_{0}=W_{\mu}\left(e^{it\cdot}f\left(\cdot\right)\right)\left(A\right)v_{0}=e^{it\cdot}f\left(\cdot\right)=e^{it\cdot}W_{\mu}\xi;
\]
or written differently:
\[
W_{\mu}U\left(t\right)=M_{e^{it\cdot}}W_{\mu},\;\forall t\in\mathbb{R}
\]
where $M_{e^{it\cdot}}$ is the multiplication operator by $e^{it\cdot}$.
\end{proof}
\end{proof}
\begin{rem}
Deficiency indices $\left(1,1\right)$ occur for probability measures
$\mu$ on $\mathbb{R}$ such that 
\begin{equation}
\int_{\mathbb{R}}\left|\lambda\right|^{2}d\mu\left(\lambda\right)=\infty.\label{eq:tmp-08}
\end{equation}
See examples below.
\end{rem}
\renewcommand{\arraystretch}{2}

\begin{table}[H]
\begin{tabular}{|c|c|c|}
\hline 
measure & condition (\ref{eq:tmp-08}) & deficiency indices\tabularnewline
\hline 
$\mu_{1}$ & $\int_{\mathbb{R}}\left|\lambda\right|^{2}e^{-\left|\lambda\right|}d\lambda<\infty$ & $\left(0,0\right)$\tabularnewline
\hline 
$\mu_{2}$ & $\int_{\mathbb{R}}\left|\lambda\right|^{2}\left(\frac{\sin\pi\lambda}{\pi\lambda}\right)^{2}d\lambda=\infty$ & $\left(1,1\right)$\tabularnewline
\hline 
$\mu_{3}$ & $\int_{\mathbb{R}}\left|\lambda\right|^{2}\frac{d\lambda}{\pi\left(1+\lambda^{2}\right)}=\infty$ & $\left(1,1\right)$\tabularnewline
\hline 
$\mu_{4}$ & $\int_{\mathbb{R}}\left|\lambda\right|^{2}\chi_{\left(-1,1\right)}\left(\lambda\right)\left(1-\left|\lambda\right|\right)d\lambda<\infty$ & $\left(0,0\right)$\tabularnewline
\hline 
$\mu_{5}$ & $\int_{\mathbb{R}}\left|\lambda\right|^{2}\frac{1}{\sqrt{2\pi}}e^{-\lambda^{2}/2}d\lambda=1<\infty$ & $\left(0,0\right)$\tabularnewline
\hline 
\end{tabular}

\protect\caption{\label{tab:meas}Application of Theorem \ref{thm:defmeas} to Table
\ref{tab:F1-F6}.}
\end{table}

\begin{figure}[H]
\begin{tabular}{cc}
\includegraphics[scale=0.35]{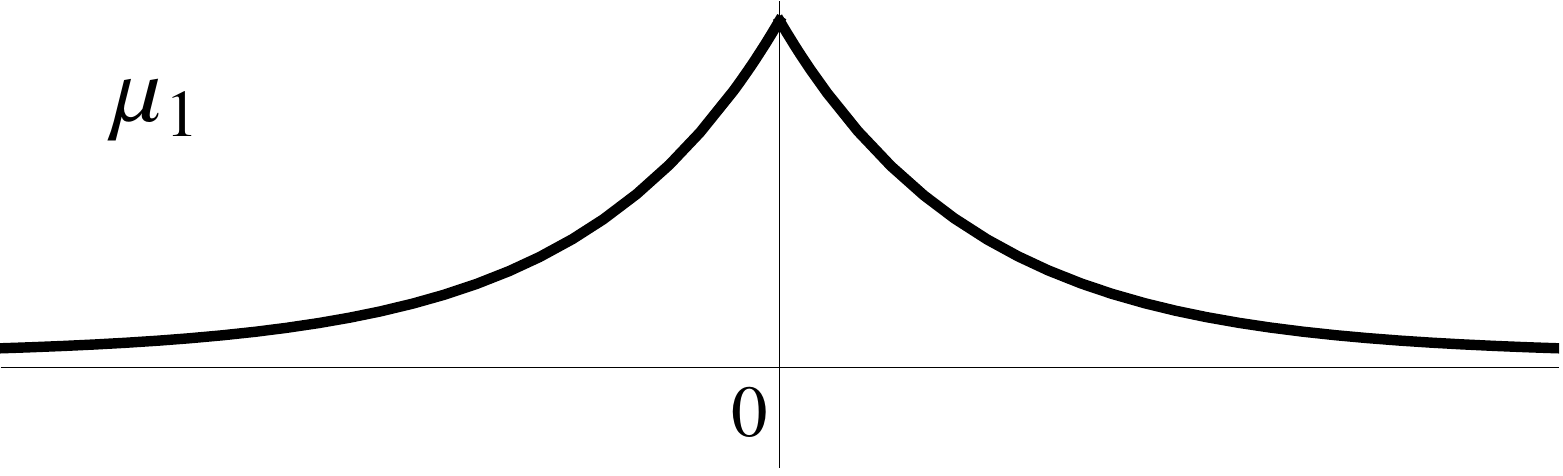} & \includegraphics[scale=0.35]{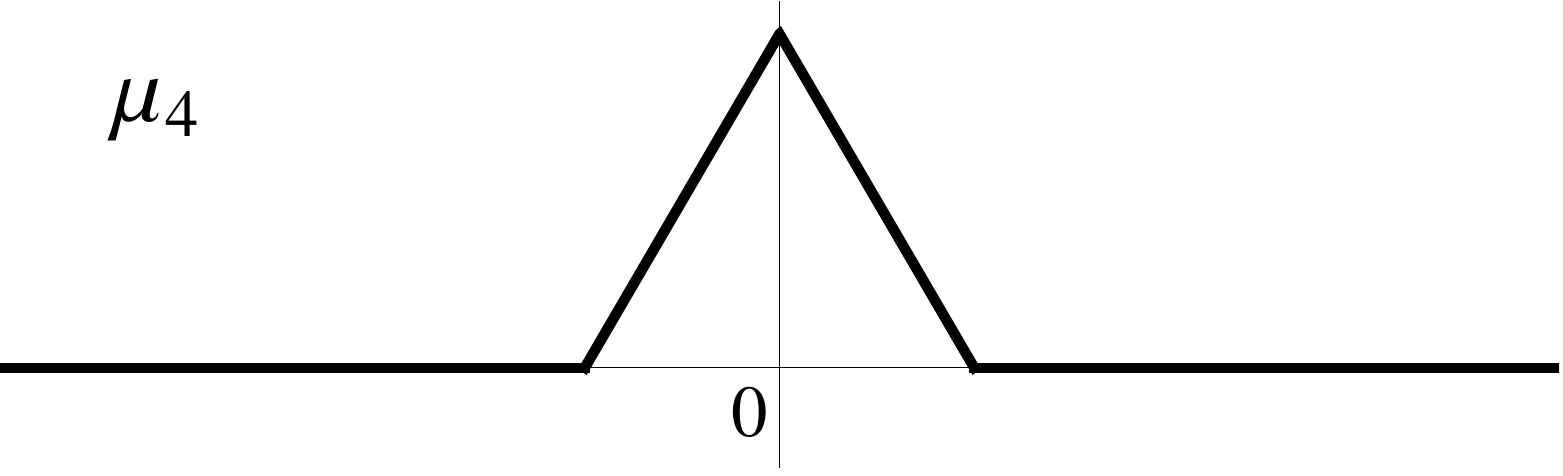}\tabularnewline
\includegraphics[scale=0.35]{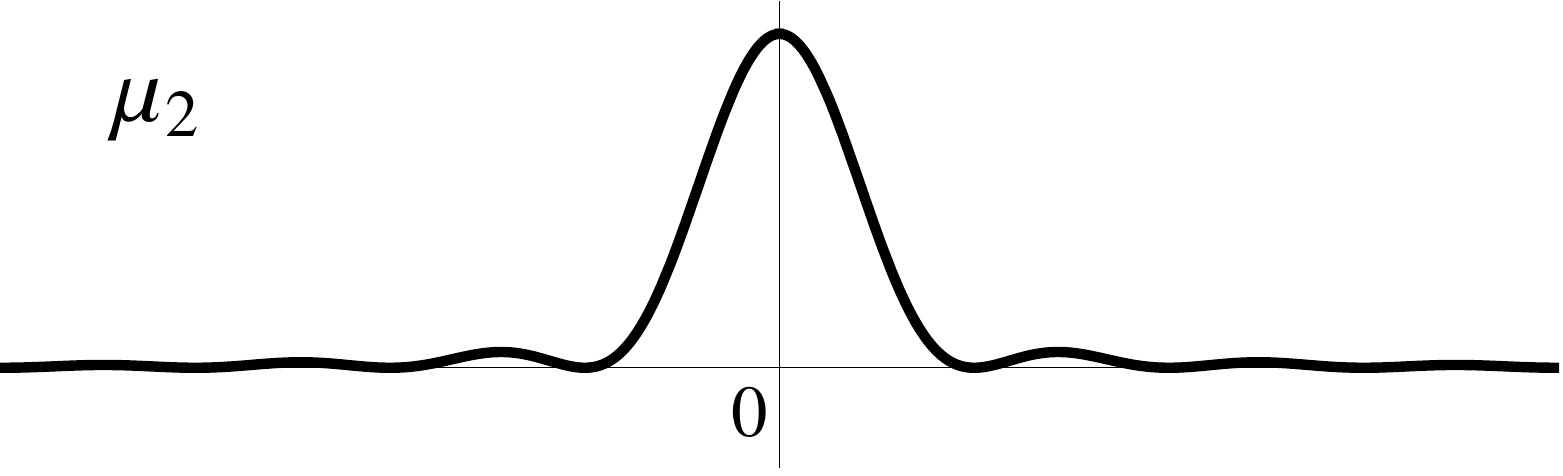} & \includegraphics[scale=0.35]{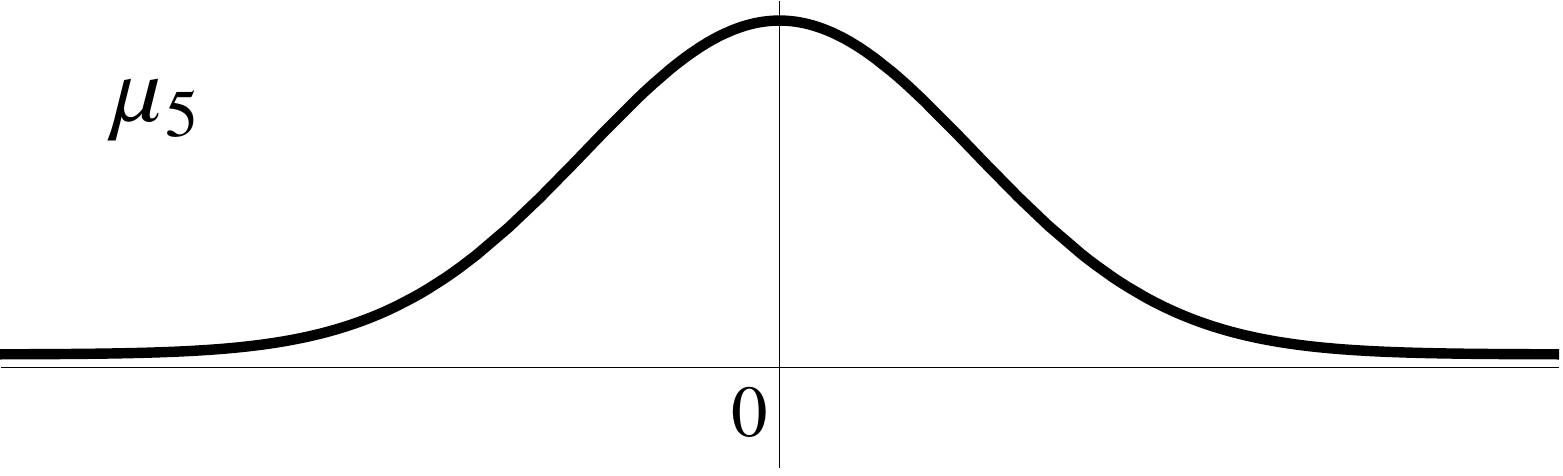}\tabularnewline
\includegraphics[scale=0.35]{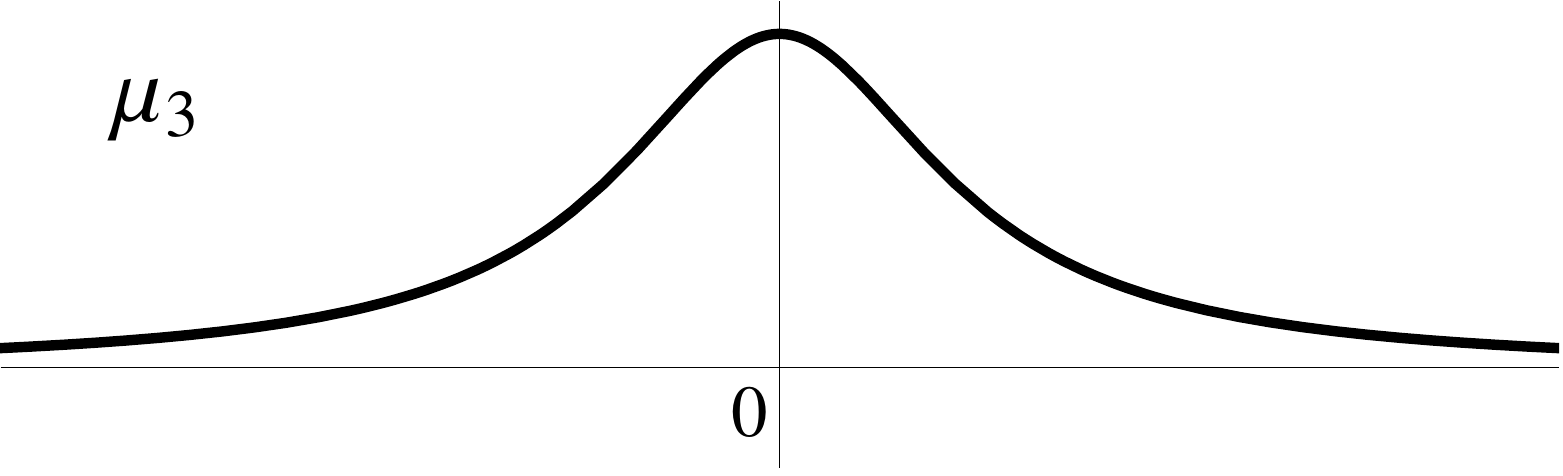} & \tabularnewline
\end{tabular}

\protect\caption{\label{fig:meas}The measures $\mu_{i}\in Ext\left(F_{i}\right)$
extending p.d. functions $F_{i}$ in Table \ref{tab:F1-F6}, $i=1,2,\ldots5$. }
\end{figure}

\renewcommand{\arraystretch}{1}

Summary: restrictions with deficiency indices $\left(1,1\right)$.

\index{measure!probability}
\begin{thm}
\label{thm:defmeas}If $\mu$ is a fixed probability measure on $\mathbb{R}$,
then the following two conditions are equivalent:
\begin{enumerate}
\item $\int_{\mathbb{R}}\lambda^{2}d\mu\left(\lambda\right)=\infty$;
\item The set 
\[
dom\left(S\right)=\left\{ f\in L^{2}\left(\mu\right)\:\Big|\:\lambda f\in L^{2}\left(\mu\right)\mbox{ and }\int_{\mathbb{R}}\left(\lambda+i\right)f\left(\lambda\right)d\mu\left(\lambda\right)=0\right\} 
\]
is the \uline{dense} domain of a restriction operator $S\subset M_{\lambda}$
with deficiency indices $\left(1,1\right)$, and the deficiency space
$DEF_{+}=\mathbb{C}1$, \emph{(}$1=$ the constant function $1$ in
$L^{2}\left(\mu\right)$.\emph{)}
\end{enumerate}
\end{thm}

\section{\label{sec:index 11}A Model of ALL Deficiency Index-$\left(1,1\right)$
Operators}
\begin{lem}
\label{lem:def1}Let $\mu$ be a Borel probability measure on $\mathbb{R}$,
and denote $L^{2}\left(\mathbb{R},d\mu\right)$ by $L^{2}\left(\mu\right)$.
then we have TFAE:
\begin{enumerate}
\item 
\begin{equation}
\int_{\mathbb{R}}\left|\lambda\right|^{2}d\mu\left(\lambda\right)=\infty\label{eq:tmp-09}
\end{equation}

\item the following two subspaces in $L^{2}\left(\mu\right)$are dense (in
the $L^{2}\left(\mu\right)$-norm):
\begin{equation}
\left\{ f\in L^{2}\left(\mu\right)\Big|\left[\left(\lambda\pm i\right)f\left(\lambda\right)\right]\in L^{2}\left(\mu\right)\mbox{ and }\int\left(\lambda\pm i\right)f\left(\lambda\right)d\mu\left(\lambda\right)=0\right\} \label{eq:tmp-10}
\end{equation}
where $i=\sqrt{-1}$.
\end{enumerate}
\end{lem}
\begin{proof}
See \cite{Jor81}.\end{proof}
\begin{rem}
If (\ref{eq:tmp-09}) holds, then the two dense subspaces $\mathscr{D}_{\pm}\subset L^{2}\left(\mu\right)$
in (\ref{eq:tmp-10}) form the dense domain of a restriction $S$
of $M_{\lambda}$ in $L^{2}\left(\mu\right)$; and this restriction
has deficiency indices $\left(1,1\right)$. Moreover, all Hermitian
operators having deficiency indices $\left(1,1\right)$ arise this
way. \index{deficiency indices}

Assume (\ref{eq:tmp-09}) holds; then the subspace 
\[
\mathscr{D}=\left\{ f\in L^{2}\left(\mu\right)\Big|\left(\lambda+i\right)f\in L^{2}\left(\mu\right)\mbox{ and }\int\left(\lambda+i\right)f\left(\lambda\right)d\mu\left(\lambda\right)=0\right\} 
\]
is a dense domain of a restricted operator of $M_{\lambda}$, so $S\subset M_{\lambda}$,
and $S$ is Hermitian. \end{rem}
\begin{lem}
\label{lem:index(1,1)}With $i=\sqrt{-1}$, set 
\begin{equation}
dom\left(S\right)=\left\{ f\in L^{2}\left(\mu\right)\Big|\lambda f\in L^{2}\left(\mu\right)\mbox{ and }\int\left(\lambda+i\right)f\left(\lambda\right)d\mu\left(\lambda\right)=0\right\} \label{eq:tmp-12}
\end{equation}
then $S\subset M_{\lambda}\subset S^{*}$; and the deficiency subspaces
$DEF_{\pm}$ are as follow:
\begin{eqnarray}
DEF_{+} & = & \mbox{the constant function in }L^{2}\left(\mu\right)=\mathbb{C}1\label{eq:tmp-11}\\
DEF_{-} & = & span\left\{ \frac{\lambda-i}{\lambda+i}\right\} _{\lambda\in\mathbb{R}}\subseteq L^{2}\left(\mu\right)\label{eq:tmp-14}
\end{eqnarray}
where $DEF_{-}$ is also a 1-dimensional subspace in $L^{2}\left(\mu\right)$.\end{lem}
\begin{proof}
Let $f\in dom\left(S\right)$, then, by definition, 
\[
\int_{\mathbb{R}}\left(\lambda+i\right)f\left(\lambda\right)d\mu\left(\lambda\right)=0\;\mbox{and so}
\]
\begin{equation}
\left\langle 1,\left(S+iI\right)f\right\rangle _{L^{2}\left(\mu\right)}=\int_{\mathbb{R}}\left(\lambda+i\right)f\left(\lambda\right)d\mu\left(\lambda\right)=0\label{eq:tmp-13}
\end{equation}
hence (\ref{eq:tmp-11}) follows.

Note we have formula (\ref{eq:tmp-12}) for $dom\left(S\right)$.
Moreover $dom\left(S\right)$ is dense in $L^{2}\left(\mu\right)$
because of (\ref{eq:tmp-10}) in Lemma \ref{lem:def1}.

Now to (\ref{eq:tmp-14}): Let $f\in dom\left(S\right)$; then
\begin{eqnarray*}
\left\langle \frac{\lambda-i}{\lambda+i},\left(S-iI\right)f\right\rangle _{L^{2}\left(\mu\right)} & = & \int_{\mathbb{R}}\left(\frac{\lambda+i}{\lambda-i}\right)\left(\lambda-i\right)f\left(\lambda\right)d\mu\left(\lambda\right)\\
 & = & \int_{\mathbb{R}}\left(\lambda+i\right)f\left(\lambda\right)d\mu\left(\lambda\right)=0
\end{eqnarray*}
again using the definition of $dom\left(S\right)$ in (\ref{eq:tmp-12}).
\end{proof}
We have established a representation for \uline{all} Hermitian
operators with dense domain in a Hilbert space, and having deficiency
indices $\left(1,1\right)$. In particular, we have justified the
answers in Table \ref{tab:F1-F6} for $F_{i}$, $i=1,\ldots,5$. 

To further emphasize to the result we need about deficiency indices
$\left(1,1\right)$, we have the following:
\begin{thm}
Let $\mathscr{H}$ be a separable Hilbert space, and let $S$ be a
Hermitian operator with dense domain in $\mathscr{H}$. Suppose the
deficiency indices of $S$ are $\left(d,d\right)$; and suppose one
of the selfadjoint extensions of $S$ has simple spectrum. \index{selfadjoint extension}

Then the following two conditions are equivalent:
\begin{enumerate}
\item \label{enu:(11)-1}$d=1$;
\item \label{enu:(11)-2}for each of the selfadjoint extensions $T$ of
$S$, we have a unitary equivalence between $\left(S,\mathscr{H}\right)$
on the one hand, and a system $\left(S_{\mu},L^{2}\left(\mathbb{R},\mu\right)\right)$
on the other, where $\mu$ is a Borel probability measure on $\mathbb{R}$.
Moreover, 
\begin{equation}
dom\left(S_{\mu}\right)=\left\{ f\in L^{2}\left(\mu\right)\Big|\lambda f\left(\cdot\right)\in L^{2}\left(\mu\right),\mbox{ and }\int_{\mathbb{R}}\left(\lambda+i\right)f\left(\lambda\right)d\mu\left(\lambda\right)=0\right\} ,\label{eq:tmp-15}
\end{equation}
and
\begin{equation}
\left(S_{\mu}f\right)\left(\lambda\right)=\lambda f\left(\lambda\right),\;\forall f\in dom\left(S_{\mu}\right),\forall\lambda\in\mathbb{R}.\label{eq:tmp-16}
\end{equation}

\end{enumerate}

In case $\mu$ satisfies condition (\ref{eq:tmp-16}), then the constant
function $\mathbf{1}$ (in $L^{2}\left(\mathbb{R},\mu\right)$) is
in the domain of $S_{\mu}^{*}$, and 
\begin{equation}
S_{\mu}^{*}\mathbf{1}=i\mathbf{1}\label{eq:tmp-17}
\end{equation}
i.e., $\left(S_{\mu}^{*}\mathbf{1}\right)\left(\lambda\right)=i$,
a.a. $\lambda$ w.r.t. $d\mu$.

\end{thm}
\begin{proof}
For the implication (\ref{enu:(11)-2})$\Rightarrow$(\ref{enu:(11)-1}),
see Lemma \ref{lem:index(1,1)}.

(\ref{enu:(11)-1})$\Rightarrow$(\ref{enu:(11)-2}). Assume that
the operator $S$, acting in $\mathscr{H}$ is Hermitian with deficiency
indices $\left(1,1\right)$. This means that each of the two subspaces
$DEF_{\pm}\subset\mathscr{H}$ is one-dimensional, where 
\begin{equation}
DEF_{\pm}=\left\{ h_{\pm}\in dom\left(S^{*}\right)\Big|S^{*}h_{\pm}=\pm ih_{\pm}\right\} .\label{eq:tmp-20}
\end{equation}
 Now pick a selfadjoint extension, say $T$, extending $S$. We have
\begin{equation}
S\subseteq T=T^{*}\subseteq S^{*}\label{eq:tmp-18}
\end{equation}
where ``$\subseteq$'' in (\ref{eq:tmp-18}) means ``containment
of the respective graphs.''

Now set $U\left(t\right)=e^{itT}$, $t\in\mathbb{R}$, and let $P_{U}\left(\cdot\right)$
be the corresponding projection-valued measure\index{measure!PVM},
i.e., we have:

\begin{equation}
U\left(t\right)=\int_{\mathbb{R}}e^{it\lambda}P_{U}\left(d\lambda\right),\;\forall t\in\mathbb{R}.\label{eq:tmp-19}
\end{equation}

Using the assumption (\ref{enu:(11)-1}), and (\ref{eq:tmp-20}),
it follows that there is a vector $h_{+}\in\mathscr{H}$ such that
$\left\Vert h_{+}\right\Vert _{\mathscr{H}}=1$, $h_{+}\in dom\left(S^{*}\right)$,
and $S^{*}h_{+}=ih_{+}$. Now set 
\begin{equation}
d\mu\left(\lambda\right):=\left\Vert P_{U}\left(d\lambda\right)h_{+}\right\Vert _{\mathscr{H}}^{2}.\label{eq:tmp-21}
\end{equation}
Using (\ref{eq:tmp-19}), we then verify that there is a unitary (and
isometric) isomorphism of $L^{2}\left(\mu\right)\overset{W}{\longrightarrow}\mathscr{H}$
given by
\begin{equation}
Wf=f\left(T\right)h_{+},\;\forall f\in L^{2}\left(\mu\right);\label{eq:tmp-22}
\end{equation}
where $f\left(T\right)=\int_{\mathbb{R}}f\left(T\right)P_{U}\left(d\lambda\right)$
is the functional calculus applied to the selfadjoint operator $T$.
Hence 
\begin{eqnarray*}
\left\Vert Wf\right\Vert _{\mathscr{H}}^{2} & = & \left\Vert f\left(T\right)h_{+}\right\Vert _{\mathscr{H}}^{2}\\
 & = & \int_{\mathbb{R}}\left|f\left(\lambda\right)\right|^{2}\left\Vert P_{U}\left(d\lambda\right)h_{+}\right\Vert ^{2}\\
 & = & \int_{\mathbb{R}}\left|f\left(\lambda\right)\right|^{2}d\mu\left(\lambda\right)\,\,\,\,\,\,\,\,\,\,(\mbox{by }\ref{eq:tmp-21})\\
 & = & \left\Vert f\right\Vert _{L^{2}\left(\mu\right)}^{2}.
\end{eqnarray*}
To see that $W$ in (\ref{eq:tmp-22}) is an isometric isomorphism
of $L^{2}\left(\mu\right)$ onto $\mathscr{H}$, we use the assumption
that $T$ has simple spectrum\index{spectrum}. 

Now set 
\begin{eqnarray}
S_{\mu} & := & W^{*}SW\label{eq:tmp-23}\\
T_{\mu} & := & W^{*}TW.\label{eq:tmp-24}
\end{eqnarray}
We note that $T_{\mu}$ is then the multiplication operator $M$ in
$L^{2}\left(\mathbb{R},\mu\right)$, given by
\begin{equation}
\left(Mf\right)\left(\lambda\right)=\lambda f\left(\lambda\right),\;\forall f\in L^{2}\left(\mu\right)\label{eq:tmp-25}
\end{equation}
such that $\lambda f\in L^{2}\left(\mu\right)$. This assertion is
immediate from (\ref{eq:tmp-22}) and (\ref{eq:tmp-21}).

To finish the proof, we compute the integral in (\ref{eq:tmp-15})
in the theorem, and we use the intertwining\index{intertwining} properties
of the isomorphism $W$ from (\ref{eq:tmp-22}). Indeed, we have 
\begin{eqnarray}
\int_{\mathbb{R}}\left(\lambda+i\right)f\left(\lambda\right)d\mu\left(\lambda\right) & = & \left\langle \mathbf{1},\left(M+iI\right)f\right\rangle _{L^{2}\left(\mu\right)}\nonumber \\
 & = & \left\langle W\mathbf{1},W\left(M+iI\right)f\right\rangle _{\mathscr{H}}\nonumber \\
 & \overset{(\ref{eq:tmp-21})}{=} & \left\langle h_{+},\left(T+iI\right)Wf\right\rangle _{\mathscr{H}}.\label{eq:tmp-26}
\end{eqnarray}
Hence $Wf\in dom\left(S\right)$ $\Longleftrightarrow$ $f\in dom\left(S_{\mu}\right)$,
by (\ref{eq:tmp-23}); and, so for $Wf\in dom\left(S\right)$, the
RHS in (\ref{eq:tmp-26}) yields $\left\langle \left(S^{*}-iI\right)h_{+},Wf\right\rangle _{\mathscr{H}}=0$;
and the assertion (\ref{enu:(11)-2}) in the theorem follows.
\end{proof}

\subsection{\label{sub:momentum}Momentum Operators in $L^{2}\left(0,1\right)$}

What about our work on momentum operators in $L^{2}\left(0,1\right)$?
They have deficiency indices $\left(1,1\right)$; and there is the
family of measures $\mu$ on $\mathbb{R}$:

Fix $0\leq\theta<1$, fix 
\begin{equation}
w_{n}>0,\;\sum_{n\in\mathbb{Z}}w_{n}=1;\label{eq:L01-1}
\end{equation}
and let $\mu=\mu_{\left(\theta,w\right)}$ as follows:
\begin{equation}
d\mu=\sum_{n\in\mathbb{Z}}w_{n}\delta_{\theta+n}\label{eq:L01-2}
\end{equation}
such that 
\begin{equation}
\sum_{n\in\mathbb{Z}}n^{2}w_{n}=\infty.\label{eq:L01-3}
\end{equation}

In this case, there is the bijective $L^{2}\left(\mathbb{R},d\mu\right)\longleftrightarrow\left\{ \xi_{n}\right\} _{n\in\mathbb{Z}}$,
$\xi\in\mathbb{C}$ s.t.
\[
\sum_{n\in\mathbb{Z}}\left|\xi_{n}\right|^{2}w_{n}<\infty.
\]
We further assume that 
\[
\sum n^{2}w_{n}=\infty.
\]
The trick is to pick a representation in which $v_{0}$ in Lemma \ref{lem:iso},
i.e., (\ref{eq:tmp15})-(\ref{eq:tmp17}): $\mathscr{H}$, cyclic
vector $v_{0}\in\mathscr{H}$, $\left\Vert v_{0}\right\Vert _{\mathscr{H}}=1$,
$\left\{ U_{A}\left(t\right)\right\} _{t\in\mathbb{R}}$.
\begin{lem}
\label{lem:(01)}The $\left(1,1\right)$ index condition is the case
of the momentum operator in $L^{2}\left(0,1\right)$. To see this
we set $v_{0}=e^{x}$, or $v_{0}=e^{2\pi x}$. \end{lem}
\begin{proof}
Defect vector in $\mathscr{H}=L^{2}\left(0,1\right)$, say $e^{2\pi x}$:
\[
v_{0}\left(x\right)=\sum_{n\in\mathbb{Z}}c_{n}e_{n}^{\theta}\left(x\right)
\]
with 
\[
w_{n}\sim\frac{\left(\cos\left(2\pi\left(\theta+n\right)\right)-e^{-2\pi}\right)^{2}+\sin^{2}\left(2\pi\left(\theta+n\right)\right)}{1+\left(\theta+n\right)^{2}}
\]
and so the $\left(1,1\right)$ condition holds, i.e., 
\[
\sum_{n\in\mathbb{Z}}n^{2}w_{n}=\infty.
\]

\end{proof}
We specialize to $\mathscr{H}=L^{2}\left(0,1\right)$, and the usual
momentum operator in $\left(0,1\right)$ with the selfadjoint extensions.
Fix $\theta$, $0\leq\theta<1$, we have an ONB in $\mathscr{H}=L^{2}\left(0,1\right)$,
$e_{k}^{\theta}\left(x\right):=e^{i2\pi\left(\theta+k\right)x}$,
and 
\begin{equation}
v_{0}\left(x\right)=\sum_{k\in\mathbb{Z}}c_{k}e_{k}^{\theta}\left(x\right),\;\sum_{k\in\mathbb{Z}}\left|c_{k}\right|^{2}=1.\label{eq:L01-4}
\end{equation}

\textbf{Fact.} $v_{0}$ is cyclic for $\left\{ U_{A_{\theta}}\left(t\right)\right\} _{t\in\mathbb{R}}$$\Longleftrightarrow$
$c_{k}\neq0$, $\forall k\in\mathbb{Z}$; then set $w_{k}=\left|c_{k}\right|^{2}$
, and the conditions (\ref{eq:L01-1})-(\ref{eq:L01-3}) holds. Reason:
the measure $\mu=\mu_{v_{0},\theta}$ depends on the choice of $v_{0}$.
The non-trivial examples here $v_{0}\longleftrightarrow\left\{ c_{k}\right\} _{k\in\mathbb{Z}}$,
$c_{k}\neq0$, $\forall k$, is cyclic, i.e., 
\[
cl\: span\left\{ U_{A_{\theta}}\left(t\right)v_{0}\Big|t\in\mathbb{R}\right\} \:\left(=\mathscr{H}=L^{2}\left(0,1\right)\right)
\]

\begin{lem}
The measure $\mu$ in (\ref{eq:L01-2}) is determined as:
\begin{eqnarray*}
F_{\theta}\left(t\right) & = & \left\langle v_{0},U_{A_{\theta}}\left(t\right)v_{0}\right\rangle \\
 & = & \sum_{k\in\mathbb{Z}}e^{it2\pi\left(\theta+k\right)}w_{k}\\
 & = & \int_{\mathbb{R}}e^{it\lambda}d\mu\left(\lambda\right)
\end{eqnarray*}
where $\mu$ is as in (\ref{eq:L01-2}); so purely atom\index{atom}ic.

Note that our boundary conditions for the selfadjoint extensions of
the minimal momentum operator are implied by (\ref{eq:L01-4}), i.e.,
\begin{equation}
f_{0}\left(x+1\right)=e^{i2\pi\theta}f_{0}\left(x\right),\;\forall x\in\mathbb{R}.\label{eq:L01-05}
\end{equation}
It is implied by choices of $\theta$ s.t.
\[
f_{0}=\sum_{n\in\mathbb{Z}}B_{n}e_{n}^{\theta}\left(x\right),\;\sum_{n\in\mathbb{Z}}\left|B_{n}\right|^{2}<\infty.
\]

\end{lem}

\subsection{Restriction Operators}

In this representation, the restriction operators are represented
by sequence $\left\{ f_{k}\right\} _{k\in\mathbb{Z}}$ s.t. $\sum\left|f_{k}\right|^{2}w_{k}<\infty$,
so use restriction the selfadjoint operator corresponds to the $\theta$-boundary
conditions (\ref{eq:L01-05}), $dom\left(S\right)$ where $S$ is
the Hermitian restriction operator, i.e., $S\subset s.a.\subset S^{*}$.
It has its dense domain $dom\left(S\right)$ as follows: $\left(f_{k}\right)\in dom\left(S\right)$
$\Longleftrightarrow$ $\left(kf_{k}\right)\in l^{2}\left(\mathbb{Z},w\right)$,
and 
\[
\sum_{k\in\mathbb{Z}}\left(\theta+k+i\right)f_{k}w_{k}=0;\; i=\sqrt{-1}.
\]

Comparison with (\ref{eq:tmp-12}) in the general case. What is special
about this case $\mathscr{H}=L^{2}\left(0,1\right)$ and the usual
boundary conditions at the endpoints is that the family of measures
$\mu$ on $\mathbb{R}$ are purely atomic\index{measure!atomic};
see (\ref{eq:L01-2})
\[
d\mu=\sum_{k\in\mathbb{Z}}w_{k}\delta_{k+\theta}.
\]

\section{\label{sec:index (d,d)}The Case of Indices $\left(d,d\right)$ where
$d>1$ }

Let $\mu$ be a Borel probability measure on $\mathbb{R}$, and let
\begin{equation}
L^{2}\left(\mu\right):=L^{2}\left(\mathbb{R},\mathscr{B},\mu\right).\label{eq:d-1}
\end{equation}
The notation $\mbox{Prob}\left(\mathbb{R}\right)$ will be used for
these measures. 

We saw that the restriction/extension problem \index{extension problem}for
continuous positive definite\index{positive definite} (p.d.) functions
$F$ on $\mathbb{R}$ may be translated into a spectral theoretic
model in some $L^{2}\left(\mu\right)$ for suitable $\mu\in\mbox{Prob}\left(\mathbb{R}\right)$.
We saw that extension from a finite open ($\neq\phi$) interval leads
to spectral representation in $L^{2}\left(\mu\right)$, and restrictions
of 
\begin{equation}
\left(M_{\mu}f\right)\left(\lambda\right)=\lambda f\left(\lambda\right),\; f\in L^{2}\left(\mu\right)\label{eq:d-2}
\end{equation}
having deficiency-indices $\left(1,1\right)$; hence the case $d=1$. 
\begin{thm}
Fix $\mu\in\mbox{Prob}\left(\mathbb{R}\right)$. There is a 1-1 bijective
correspondence between the following:
\begin{enumerate}
\item \label{enu:d-1}certain closed subspaces $\mathscr{L}\subset L^{2}\left(\mu\right)$
\item \label{enu:d-2}Hermitian restrictions $S_{\mathscr{L}}$ of $M_{\mu}$
(see (\ref{eq:d-2})) such that 
\begin{equation}
DEF_{+}\left(S_{\mathscr{L}}\right)=\mathscr{L}.\label{eq:d-3}
\end{equation}

\end{enumerate}

The closed subspaces in (\ref{enu:d-1}) are specified as follows:
\begin{enumerate}[label=(\roman{enumi}),ref=\roman{enumi}]
\item \label{enu:d-3}$\dim\left(\mathscr{L}\right)=d<\infty$
\item \label{enu:d-4}the following implication holds:
\begin{equation}
g\neq0,\:\mbox{and }g\in\mathscr{L}\Longrightarrow\left(\left[\lambda\mapsto\lambda g\left(\lambda\right)\right]\notin L^{2}\left(\mu\right)\right)\label{eq:d-4}
\end{equation}

\end{enumerate}

Then set 
\begin{equation}
dom\left(S_{\mathscr{L}}\right):=\left\{ f\in dom\left(M_{\mu}\right)\Big|\int\overline{g\left(\lambda\right)}\left(\lambda+i\right)f\left(\lambda\right)d\mu\left(\lambda\right),\forall g\in\mathscr{L}\right\} \label{eq:d-5}
\end{equation}
and set 
\begin{equation}
S_{\mathscr{L}}:=M_{\mu}\Big|_{dom\left(S_{\mathscr{L}}\right)}\label{eq:d-6}
\end{equation}
where $dom\left(S_{\mathscr{L}}\right)$ is specified as in (\ref{eq:d-5}).

\end{thm}
\begin{proof}
Note that the case $d=1$ is contained in the previous theorem. 

Proof of (\ref{enu:d-1}) $\Rightarrow$ (\ref{enu:d-2}). We will
be using an idea from \cite{Jor81}. With assumptions (\ref{enu:d-3})-(\ref{enu:d-4}),
in particular (\ref{eq:d-4}), one checks that $dom\left(S_{\mathscr{L}}\right)$as
specified in (\ref{eq:d-5}) is dense in $L^{2}\left(\mu\right)$.
In fact, the converse implication is also true. 

Now setting $S_{\mathscr{L}}$ to be the restriction in (\ref{eq:d-6}),
we conclude that
\begin{equation}
S_{\mathscr{L}}\subseteq M_{\mu}\subseteq S_{\mathscr{L}}^{*}\label{eq:d-7}
\end{equation}
where 
\begin{eqnarray*}
dom\left(S_{\mathscr{L}}^{*}\right) & = & \Biggl\{ h\in L^{2}\left(\mu\right)\:\Big|\: s.t.\:\exists C<\infty\mbox{ and}\\
 &  & \left|\int_{\mathbb{R}}\overline{h\left(\lambda\right)}\lambda f\left(\lambda\right)d\mu\left(\lambda\right)\right|^{2}\leq C\int_{\mathbb{R}}\left|f\left(\lambda\right)\right|^{2}d\mu\left(\lambda\right)\\
 &  & \mbox{holds }\forall f\in dom\left(S_{\mathscr{L}}\right)\Biggr\}
\end{eqnarray*}
 The assertions in (\ref{enu:d-2}) now follow from this.

Proof of (\ref{enu:d-2}) $\Rightarrow$ (\ref{enu:d-1}). Assume
that $S$ is a densely defined restriction of $M_{\mu}$, and let
$DEF_{+}\left(S\right)=$ the (+) deficiency space, i.e., 
\begin{equation}
DEF_{+}\left(S\right)=\left\{ g\in dom\left(S^{*}\right)\:\Big|\: S^{*}g=ig\right\} \label{eq:d-9}
\end{equation}
Assume $\dim\left(DEF_{+}\left(S\right)\right)=d$, and $1\leq d<\infty$.
Then set $\mathscr{L}:=DEF_{+}\left(S\right)$. Using \cite{Jor81},
one checks that (\ref{enu:d-1}) then holds for this closed subspace
in $L^{2}\left(\mu\right)$. 

The fact that (\ref{eq:d-4}) holds for this subspace $\mathscr{L}$
follows from the observation:
\[
DEF_{+}\left(S\right)\cap dom\left(M_{\mu}\right)=\left\{ 0\right\} 
\]
for every densely defined restriction $S$ of $M_{\mu}$.
\end{proof}

\section{\label{sec:index11}Spectral Representation of Index $\left(1,1\right)$
Hermitian Operators}

In this section we give an explicit answer to the following question:
How to go from any index $\left(1,1\right)$ Hermitian operator to
a $\left(\mathscr{H}_{F},D^{\left(F\right)}\right)$ model; i.e.,
from a given index $\left(1,1\right)$ Hermitian operator with dense
domain in a separable Hilbert space $\mathscr{H}$, we build a p.d.
continuous function $F$ on $\Omega-\Omega$, where $\Omega$ is a
finite interval $\left(0,a\right)$, $a>0$.

So far, we have been concentrating on building transforms going in
the other direction. But recall that, for a given continuous p.d.
function $F$ on $\Omega-\Omega$, it is often difficult to answer
the question of whether the corresponding operator $D^{\left(F\right)}$
in the RKHS $\mathscr{H}_{F}$ has deficiency indices $\left(1,1\right)$
or $\left(0,0\right)$. 

Now this question answers itself once we have an explicit transform
going in the opposite direction. Specifically, given any index $\left(1,1\right)$
Hermitian operator $S$ in a separable Hilbert space $\mathscr{H}$,
we then to find a pair $\left(F,\Omega\right)$, p.d. function and
interval, with the desired properties. There are two steps:

Step 1, writing down explicitly, a p.d. continuous function $F$ on
$\Omega-\Omega$, and the associated RKHS $\mathscr{H}_{F}$ with
operator $D^{\left(F\right)}$.

Step 2, constructing an intertwining\index{intertwining} isomorphism
$W:\mathscr{H}\rightarrow\mathscr{H}_{F}$, having the following properties.
$W$ will be an isometric isomorphism, intertwining the pair $\left(\mathscr{H},S\right)$
with $\left(\mathscr{H}_{F},D^{\left(F\right)}\right)$, i.e., satisfying
$WS=D^{\left(F\right)}W$; and also intertwining the respective domains
and deficiency spaces, in $\mathscr{H}$ and $\mathscr{H}_{F}$.

Moreover, starting with any $\left(1,1\right)$ Hermitian operator,
we can even arrange a normalization for the p.d. function $F$ such
that $\Omega=$ the interval $\left(0,1\right)$ will do the job.

We now turn to the details:

We will have three pairs $\left(\mathscr{H},S\right)$, $\left(L^{2}\left(\mathbb{R},\mu\right),\mbox{ restriction of }M_{\mu}\right)$,
and $\left(\mathscr{H}_{F},D^{\left(F\right)}\right)$, where:
\begin{enumerate}[leftmargin=*,label=(\roman{enumi})]
\item $S$ is a fixed Hermitian operator with dense domain $dom\left(S\right)$
in a separable Hilbert space $\mathscr{H}$, and with deficiency indices
$\left(1,1\right)$. 
\item From the given information in (i), we will construct a finite Borel
measure\index{measure!Borel} $\mu$ on $\mathbb{R}$ such that an
index-$\left(1,1\right)$ restriction of $M_{\mu}:f\mapsto\lambda f\left(\lambda\right)$
in $L^{2}\left(\mathbb{R},\mu\right)$, is equivalent to $\left(\mathscr{H},S\right)$. 
\item Here $F:\left(-1,1\right)\rightarrow\mathbb{C}$ will be a p.d. continuous
function, $\mathscr{H}_{F}$ the corresponding RKHS; and $D^{\left(F\right)}$
the usual operator with dense domain 
\[
\left\{ F_{\varphi}\:\Big|\:\varphi\in C_{c}^{\infty}\left(0,1\right)\right\} ,\mbox{ and}
\]
\begin{equation}
D^{\left(F\right)}\left(F_{\varphi}\right)=\frac{1}{i}F_{\varphi'},\;\varphi'=\frac{d\varphi}{dx}.\label{eq:d11-1}
\end{equation}

\end{enumerate}
We will accomplish the stated goal with the following system of intertwining
operators: See Figure \ref{fig:d11-1}.

But we stress that, at the outset, only (i) is given; the rest ($\mu$,
$F$ and $\mathscr{H}_{F}$) will be constructed. Further, the solutions
$\left(\mu,F\right)$ in Figure \ref{fig:d11-1} are not unique; rather
they depend on choice of selfadjoint extension in (i): Different selfadjoint
extensions of $S$ in (i) yield different solutions $\left(\mu,F\right)$.
But the selfadjoint extensions of $S$ in $\mathscr{H}$ are parameterized
by von Neumann's theory; see e.g., \cite{Rud73}. 

\begin{figure}[H]
$\xymatrix{S\mbox{ in }\mathscr{H}\ar[rr]^{W_{\mu}}\ar[ddr]_{W:=T_{\mu}^{*}W_{\mu}} &  & \ar@/^{1pc}/[ddl]^{T_{\mu}^{*}}L^{2}\left(\mathbb{R},\mu\right),\mbox{ restriction of }M_{\mu}\\
\\
 & D^{\left(F\right)}\mbox{ in }\mbox{\ensuremath{\mathscr{H}}}_{F}\ar@/^{1pc}/[uur]^{T_{\mu}}
}
$

\protect\caption{\label{fig:d11-1}A system of intertwining operators.}
\end{figure}
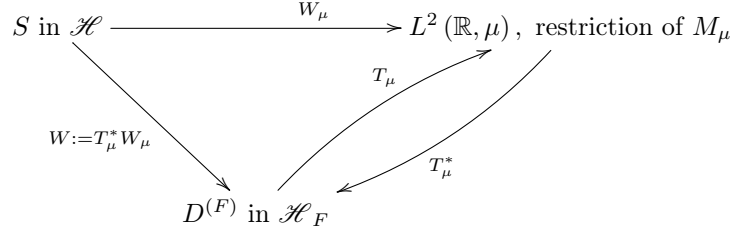

\begin{flushleft}
\index{von Neumann}
\par\end{flushleft}

\begin{flushleft}
\index{intertwining}
\par\end{flushleft}

\index{measure!probability}
\begin{rem}
In our analysis of (i)-(iii), we may without loss of generality assume
that the following normalizations hold:

$\left(z_{1}\right)$ $\mu\left(\mathbb{R}\right)=1$ , so $\mu$
is a probability measure;

$\left(z_{2}\right)$ $F\left(0\right)=1$, and the p.d. continuous
solution 

$\left(z_{3}\right)$ $F:\left(-1,1\right)\rightarrow\mathbb{C}$
is defined on $\left(-1,1\right)$; so $\Omega:=\left(0,1\right)$.

\begin{flushleft}
Further, we may assume that the operator $S$ in $\mathscr{H}$ from
(i) has simple spectrum.
\par\end{flushleft}\end{rem}
\begin{thm}
\label{thm:d11-1}Starting with $\left(\mathscr{H},S\right)$ as in
(i), there are solutions $\left(\mu,F\right)$ to (ii)-(iii), and
intertwining operators $W_{\mu}$, $T_{\mu}$ as in Figure \ref{fig:d11-1},
such that\index{intertwining}
\begin{equation}
W:=T_{\mu}^{*}W_{\mu}\label{eq:d11-3}
\end{equation}
satisfies the intertwining properties for $\left(\mathscr{H},S\right)$
and $\left(\mathscr{H}_{F},D^{\left(F\right)}\right)$.\end{thm}
\begin{proof}
Since $S$ has indices $\left(1,1\right)$, $\dim DEF_{\pm}\left(S\right)=1$,
and $S$ has selfadjoint extensions indexed by partial isometries
$DEF_{+}\overset{v}{\longrightarrow}DEF_{-}$; see \cite{Rud73,DS88b}.
We now pick $g_{+}\in DEF_{+}$, $\left\Vert g_{+}\right\Vert =1$,
and partial isometry $v$ with selfadjoint extension $S_{v}$, i.e.,
\index{selfadjoint extension} 
\begin{equation}
S\subset S_{v}\subset S_{v}^{*}\subset S^{*}.\label{eq:d11-4}
\end{equation}

Hence $\left\{ U_{v}\left(t\right)\:\Big|\: t\in\mathbb{R}\right\} $
is a strongly continuous unitary representation of $\mathbb{R}$,
acting in $\mathscr{H}$, $U_{v}\left(t\right):=e^{itS_{v}}$, $t\in\mathbb{R}$.
Let $P_{S_{v}}\left(\cdot\right)$ be the corresponding projection
valued measure (PVM) on $\mathscr{B}\left(\mathbb{R}\right)$, i.e.,
we have \index{measure!PVM} 
\begin{equation}
U_{v}\left(t\right)=\int_{\mathbb{R}}e^{it\lambda}P_{S_{v}}\left(d\lambda\right);\label{eq:d11-5}
\end{equation}
and set 
\begin{equation}
d\mu\left(\lambda\right):=d\mu_{v}\left(\lambda\right)=\left\Vert P_{S_{v}}\left(d\lambda\right)g_{+}\right\Vert _{\mathscr{H}}^{2}.\label{eq:d11-6}
\end{equation}
For $f\in L^{2}\left(\mathbb{R},\mu_{v}\right)$, set 
\begin{equation}
W_{\mu_{v}}\left(f\left(S_{v}\right)g_{+}\right)=f;\label{eq:d11-7}
\end{equation}
then $W_{\mu_{v}}:\mathscr{H}\rightarrow L^{2}\left(\mathbb{R},\mu_{v}\right)$
is isometric onto; and 
\begin{equation}
W_{\mu_{v}^{*}}\left(f\right)=f\left(S_{v}\right)g_{+},\label{eq:d11-8}
\end{equation}
where 
\begin{equation}
f\left(S_{v}\right)g_{+}=\int_{\mathbb{R}}f\left(\lambda\right)P_{S_{v}}\left(d\lambda\right)g_{+}.\label{eq:d11-9}
\end{equation}
For justification of these assertions, see e.g., \cite{Ne69}. Moreover,
$W_{\mu}$ has the intertwining properties sketched in Figure \ref{fig:d11-1}.

Returning to (\ref{eq:d11-5}) and (iii) in the theorem, we now set
$F=$ the restriction to $\left(-1,1\right)$ of 
\begin{eqnarray}
F_{\mu}\left(t\right) & := & \left\langle g_{+},U_{v}\left(t\right)g_{+}\right\rangle \label{eq:d11-10}\\
 & \underset{(\mbox{by }(\ref{eq:d11-5}))}{=} & \left\langle g_{+},\int_{\mathbb{R}}e^{it\lambda}P_{S_{v}}\left(d\lambda\right)g_{+}\right\rangle \nonumber \\
 & = & \int_{\mathbb{R}}e^{it\lambda}\left\Vert P_{S_{v}}\left(d\lambda\right)g_{+}\right\Vert ^{2}\nonumber \\
 & \underset{(\mbox{by }(\ref{eq:d11-6}))}{=} & \int_{\mathbb{R}}e^{it\lambda}d\mu_{v}\left(\lambda\right)\nonumber \\
 & = & \widehat{d\mu_{v}}\left(t\right),\;\forall t\in\mathbb{R}.\nonumber 
\end{eqnarray}

We now show that 
\begin{equation}
F:=F_{\mu}\Big|_{\left(-1,1\right)}\label{eq:d11-11}
\end{equation}
has the desired properties.

From Corollary \ref{cor:lcg-isom}, we have the isometry\index{isometry}
$T_{\mu}\left(F_{\varphi}\right)=\widehat{\varphi}$, $\varphi\in C_{c}\left(0,1\right)$,
with adjoint 
\begin{equation}
T_{\mu}^{*}\left(f\right)=\left(fd\mu\right)^{\vee},\mbox{ and}\label{eq:d11-12}
\end{equation}
\[
\xymatrix{\mathscr{H}_{F}\ar@/^{1pc}/^{T_{\mu}}[rr] &  & L^{2}\left(\mathbb{R},\mu\right)\ar@/^{1pc}/^{T_{\mu}^{*}}[ll]}
;
\]
see also Fig \ref{fig:d11-1}. 

The following properties are easily checked:
\begin{equation}
W_{\mu}\left(g_{+}\right)=\mathbf{1}\in L^{2}\left(\mathbb{R},\mu\right),\mbox{ and}\label{eq:d11-13}
\end{equation}
\begin{equation}
T_{\mu}^{*}\left(\mathbf{1}\right)=F_{0}=F\left(\cdot-0\right)\in\mathscr{H}_{F},\label{eq:d11-14}
\end{equation}
as well as the intertwining properties stated in the theorem; see
Fig. \ref{fig:d11-1} for a summary.

\uline{Proof of (\mbox{\ref{eq:d11-13}})} We will show instead
that $W_{\mu}^{*}\left(\mathbf{1}\right)=g_{+}$. From (\ref{eq:d11-9})
we note that if $f\in L^{2}\left(\mathbb{R},\mu\right)$ satisfies
$f=\mathbf{1}$, then $f\left(S_{v}\right)=I=$ the identity operator
in $\mathscr{H}$. Hence 
\[
W_{\mu}^{*}\left(\mathbf{1}\right)\underset{(\mbox{by }(\ref{eq:d11-8}))}{=}\mathbf{1}\left(S_{v}\right)g_{+}=g_{+},
\]
which is (\ref{eq:d11-13}).

\uline{Proof of (\mbox{\ref{eq:d11-14}})} For $\varphi\in C_{c}\left(0,1\right)$
we have $\widehat{\varphi}\in L^{2}\left(\mathbb{R},\mu\right)$,
and 
\begin{eqnarray*}
T_{\mu}^{*}T_{\mu}\left(F_{\varphi}\right) & = & T_{\mu}^{*}\left(\widehat{\varphi}\right)\\
 & \underset{(\mbox{by }(\ref{eq:d11-12}))}{=} & \left(\widehat{\varphi}d\mu\right)^{\vee}=F_{\varphi}.
\end{eqnarray*}
Taking now an approximation $\left(\varphi_{n}\right)\subset C_{c}\left(0,1\right)$
to the Dirac unit mass at $x=0$, we get (\ref{eq:d11-14}).\end{proof}
\begin{cor}
The deficiency indices of $D^{\left(F\right)}$ in $\mathscr{H}_{F}$
for $F\left(x\right)=e^{-\left|x\right|}$, $\left|x\right|<1$, are
$\left(1,1\right)$. \end{cor}
\begin{proof}
Take $\mathscr{H}=L^{2}\left(\mathbb{R}\right)=\left\{ f\mbox{ measurable on }\mathbb{R}\:\Big|\:\int_{\mathbb{R}}\left|f\left(x\right)\right|^{2}dx<\infty\right\} $,
where $dx=$ Lebesgue measure. 

Take $g_{+}:=\left(\frac{1}{\lambda+i}\right)^{\vee}\left(x\right)$,
$x\in\mathbb{R}$; then $g_{+}\in L^{2}\left(\mathbb{R}\right)=:\mathscr{H}$
since
\begin{eqnarray*}
\int_{\mathbb{R}}\left|g_{+}\left(x\right)\right|^{2}dx & \underset{(\mbox{Parseval})}{=} & \int_{\mathbb{R}}\left|\frac{1}{\lambda+i}\right|^{2}d\lambda\\
 & = & \int_{\mathbb{R}}\frac{1}{1+\lambda^{2}}d\lambda=\pi.
\end{eqnarray*}

Now for $S$ and $S_{v}$ in Theorem \ref{thm:d11-1}, we take 
\begin{equation}
S_{v}h=\frac{1}{i}\frac{d}{dx}h\mbox{ on }\left\{ h\in L^{2}\left(\mathbb{R}\right)\:\Big|\: h'\in L^{2}\left(\mathbb{R}\right)\right\} \mbox{ and}\label{eq:d11-16}
\end{equation}
\[
S=S_{v}\Big|_{\left\{ h\:\Big|\: h,h'\in L^{2}\left(\mathbb{R}\right),h\left(0\right)=0\right\} },
\]
then by \cite{Jor81}, we know that $S$ is an index $\left(1,1\right)$
operator, and that $g_{+}\in DEF_{+}\left(S\right)$. The corresponding
p.d. continuous function $F$ is the restriction to $\left|t\right|<1$
of the p.d. function:
\begin{eqnarray*}
\left\langle g_{+},U_{v}\left(t\right)g_{+}\right\rangle _{\mathscr{H}} & = & \int_{\mathbb{R}}\frac{1}{\lambda-i}\frac{e^{it\lambda}}{\lambda+i}d\lambda\\
 & = & \left(\frac{1}{1+\lambda^{2}}\right)^{\vee}\left(t\right)=\pi e^{-\left|t\right|}.
\end{eqnarray*}
\end{proof}
\begin{example}[Lévy-measures]
 Let $0<\alpha\leq2$, $-1<\beta<1$, $v>0$; then the Lévy-measures\index{measure!Lévy}
$\mu$ on $\mathbb{R}$ are indexed by $\left(\alpha,\beta,\nu\right)$,
so $\mu=\mu_{\left(\alpha,\beta,\nu\right)}$. They are absolutely
continuous\index{absolutely continuous} with respect to Lebesgue
measure $d\lambda$ on $\mathbb{R}$; and for $\alpha=1$, 
\begin{equation}
F_{\left(\alpha,\beta,\nu\right)}\left(x\right)=\widehat{\mu_{\left(\alpha,\beta,\nu\right)}}\left(x\right),\; x\in\mathbb{R},\label{eq:d11-17}
\end{equation}
satisfies 
\begin{equation}
F_{\left(\alpha,\beta,\nu\right)}\left(x\right)=\exp\left(-\nu\left|x\right|\cdot\left(1+\frac{2i\beta}{\pi}-\mbox{sgn}\left(x\right)\ln\left|x\right|\right)\right).\label{eq:d11-18}
\end{equation}
The case $\alpha=2$, $\beta=0$, reduces to the Gaussian distribution. 

The measures $\mu_{\left(1,\beta,\nu\right)}$ have infinite variance,
i.e., 
\[
\int_{\mathbb{R}}\lambda^{2}d\mu_{\left(1,\beta,\nu\right)}=\infty.
\]

As a Corollary of Theorem \ref{thm:d11-1}, we therefore conclude
that, for the restrictions, 
\[
F_{\left(1,\beta,\nu\right)}^{\left(res\right)}\left(x\right)=F_{\left(1,\beta,\nu\right)}\left(x\right),\;\mbox{in }x\in\left(-1,1\right),\;\mbox{(see }(\ref{eq:d11-17})-(\ref{eq:d11-18}))
\]
the associated Hermitian operator $D^{F^{\left(res\right)}}$ all
have deficiency indices $\left(1,1\right)$. 

In connection Levy measures\index{measure!Lévy}, see e.g., \cite{ST94}.
\end{example}

\chapter{\label{chap:question}Open Questions}

\section{From Restriction Operator to Restriction of p.d. Function}

\index{measure!tempered}

\index{spectrum}
\begin{question}
How to write a formula for the restriction of a suitable positive
definite (p.d.) function associated to a given $\left(1,1\right)$
operator?
\end{question}
In  \secref{index 11}, we create a model for $\left(1,1\right)$
operators $S$ with simple spectrum; and a choice of cyclic (defect)
vector $v_{0}$. Hence $S$ is the restriction of a selfadjoint operator
$T$.

From $T$ (simple spectrum) and the vector $v_{0}$, we get a Borel
probability measure $\mu$ on $\mathbb{R}$; and by Bochner, or Stone,
we get the corresponding continuous p.d. function $F$ on $\mathbb{R}$.
It would be nice to have a formula for the particular restriction
of $F$ which corresponds to the $\left(1,1\right)$ operator $S$
in the model.

The main difference between our measures on $\mathbb{R}$ and the
measures used in fractional Brownian motion and related processes
is that our measures are finite on $\mathbb{R}$, but the others aren\textquoteright t;
instead they are what is called tempered. See \cite{AL08}.

This means that they are unbounded but only with a polynomial degree.
For example, Lebesgue measure $dx$ on $\mathbb{R}$ is tempered.
Suppose $1/2<g<1$, then the positive measure $d\mu\left(x\right)=\left|x\right|^{g}dx$
is tempered, and this measure is what is used in accounting for the
p.d. functions discussed in the papers on Fractional Brownian motion.

But we could ask the same kind of questions for tempered measures
as we do for the finite positive measures. 

With this analogy, \cite{Jor81} was about $\left(1,1\right)$ operators
from Lebesgue measure $dx$ on $\mathbb{R}$. PJ also looked at $\left(m,m\right)$
operators, but only from Lebesgue measure.

\section{The Splitting $\mathscr{H}_{F}=\mathscr{H}_{F}^{\left(atom\right)}\oplus\mathscr{H}_{F}^{\left(ac\right)}\oplus\mathscr{H}_{F}^{\left(sing\right)}$ }

Let $\Omega$ be as usual, connected and open in $\mathbb{R}^{n}$;
and let $F$ be a p.d. continuous function on $\Omega^{-1}\Omega$.
We then pass to the RKHS $\mathscr{H}_{F}$. Suppose $Ext\left(F\right)$
is non-empty. We then get a unitary representation $U$ of $G=\mathbb{R}^{n}$,
acting on $\mathscr{H}_{F}$; and its associated PVM $P_{U}$. This
gives rise to an orthogonal splitting of $\mathscr{H}_{F}$ into three
parts, atomic, absolutely continuous, and continuous singular, defined
from $P_{U}$.

\index{atom}\index{absolutely continuous}\index{orthogonal}
\begin{question}
What are some interesting examples illustrating the triple splitting
of $\mathscr{H}_{F}$, as in  \subref{euclid}, eq. (\ref{eq:rn6})? 
\end{question}
It will be nice to see some examples where we have more than one non-zero
subspaces; so an example of $F$ be a p.d. continuous function on
$\Omega^{-1}\Omega$ where there is a mixture of the three possibilities,
atomic, absolutely continuous, and continuous singular. So we would
like an example of a $\left(F,\Omega\right)$, such that $F$ be a
p.d. continuous function on $\Omega^{-1}\Omega$, has its $\mathscr{H}_{F}$
splitting with more than one of the two non-zero.

So fix $\left(F,\Omega\right)$; and it would be interesting to study
cases where the same $\left(F,\Omega\right)$ has two of its three
subspaces non-zero? Or all three non-zero? The examples of $\left(F,\Omega\right)$
fall in one of the three boxes only. Can we get spectrum in all three
boxes, or at least not deposit everything in just one of the three? 

In \subref{euclid}, we constructed an interesting example for the
splitting of the RKHS $\mathscr{H}_{F}$: All three subspaces $\mathscr{H}_{F}^{\left(atom\right)}$,
$\mathscr{H}_{F}^{\left(ac\right)}$, and $\mathscr{H}_{F}^{\left(sing\right)}$
are non-zero; the first one is one-dimensional, and the other two
are infinite-dimensional. See Example \ref{ex:splitting} for details.

\section{The Case of $G=\mathbb{R}^{1}$}

We can move to the circle group $\mathbb{T}=\mathbb{R}/\mathbb{Z}$,
and we already started that. Or we can stay with $\mathbb{R}$, or
go to $\mathbb{R}^{n}$; or go to Lie groups (if Lie groups we must
also look at the universal simply connected covering groups.) All
very interesting questions. 

A FEW POINTS, about a single interval and asking for p.d. extensions
to $\mathbb{R}$:
\begin{enumerate}[leftmargin=*]
\item The case of $\mathbb{R}$, we are asking to extend a (small) p.d.
continuous function $F$ from an interval to all of $\mathbb{R}$,
so that the extension to $\mathbb{R}$ is also p.d. continuous; --
in this case we always have existence. This initial interval may even
be \textquotedblleft very small.\textquotedblright{} If we have several
intervals, existence may fail, but the problem is interesting. We
have calculated a few examples, but nothing close to a classification!
Here is what we know about extension from a single interval:
\item In all cases, we have existence, but the p.d continuous extension
to $\mathbb{R}$ may or may not be unique.
\item To decide the uniqueness, we turn to the RKHS $\mathscr{H}_{F}$ (remember
now we are on the interval), and we calculate deficiency indices there,
but we are in RKHS $\mathscr{H}_{F}$. They may be $\left(0,0\right)$
or $\left(1,1\right)$.

\begin{itemize}
\item In the first case, we have uniqueness; in the second case not.
\item Both cases are very interesting; also indices (0, 0).
\end{itemize}

By contrast, in our earlier papers, e.g., on momentum operators, we
never looked at RKHSs related to this other $\mathscr{H}_{F}$ we
now study.

\item Actually, computable ONBs in $\mathscr{H}_{F}$ would be very interesting.
So here is a fun project:\\
We can use what we know about selfadjoint extensions of operators
with deficiency indices $\left(1,1\right)$ to try to classify all
the p.d continuous extension to $\mathbb{R}$. But it is not trivial;
remember the interplay between operator extensions and p.d. extensions
is subtle. While both are \textquotedblleft extension\textquotedblright{}
problems, the interplay between them is subtle and non-intuitive.\\
\\
WARNING: This new classification question probably has nothing to
do with questions for the moment operator in its various incarnations,
see e.g., \cite{JPT11-1,JPT12} and \cite{PeTi13}.

\begin{itemize}
\item This is because there are lots of other operators with deficiency
indices $\left(1,1\right)$, having nothing to do at all with moment
operators. 
\item From examples we calculated, so far, for the relevant operators with
deficiency indices $\left(1,1\right)$ in RKHS $\mathscr{H}_{F}$
(remember now we are on the interval), are totally unrelated to moment
operators.
\item With $\mathscr{H}_{F}$, you are aiming to understand deficiency index-operators
in the case of $\left(1,1\right)$, but in the abstract; not just
an example.
\end{itemize}

There are two sources of earlier papers; one is M. Krein\index{Krein},
and the other L. deBranges. Both are covered beautifully in a book
by Dym and McKean; it is now in the Dover series.

\end{enumerate}

\section{The Extreme Points of $Ext\left(F\right)$ and $\Im\left\{ F\right\} $}

Given a locally defined p.d. function $F$, i.e., a p.d. function
$F$ defined on a proper subset in a group, then the real part $\Re\left\{ F\right\} $
is also p.d.. Can anything be said about the imaginary part, $\Im\left\{ F\right\} $?
See \secref{imgF}.

Assuming that $F$ is also continuous, then what are the extreme points
in the compact convex set $Ext\left(F\right)$, i.e., what is the
subset $ext\left(Ext\left(F\right)\right)$? How do properties of
$Ext\left(F\right)$ depend on the imaginary part, i.e., on the function
$\Im\left\{ F\right\} $? How do properties of the skew-Hermitian\index{operator!skew-Hermitian}
operator $D^{\left(F\right)}$ (in the RKHS $\mathscr{H}_{F}$) depend
on the imaginary part, i.e., on the function $\Im\left\{ F\right\} $?

\nomenclature{$\mathscr{H}_{F}$}{The reproducing kernel Hilbert space (RKHS) of $F$.}

\nomenclature{$Ext(F)$}{Set of unitary representations of $G$ on a Hilbert space $\mathscr{K}$, i.e., the triple $\left(U,\mathscr{K},k_{0}\right)$, that extends $F$.\\\\}

\nomenclature{$Ext_{1}(F)$}{The triples in $Ext(F)$, where the representation space is $\mathscr{H}_{F}$, and so the extension of $F$ is realized on $\mathscr{H}_{F}$.\\\\}

\nomenclature{$Ext_{2}(F)$}{$Ext(F) \setminus Ext_{1}(F)$, i.e., the extension is realized on an  enlargement Hilbert space.\\}

\nomenclature{$T_F$}{Mercer operator associated to $F$.\\}

\nomenclature{$F_{\varphi}$}{The convolution of $F$ and $\varphi$, where $\varphi \in C_{c}(\Omega)$.\\}

\nomenclature{$D^{(F)}$}{The derivative operator $F_{\varphi} \mapsto F_{\frac{d\varphi}{dx}}$}

\nomenclature{$DEF$}{The deficiency space of $D^{(F)}$ acting in the Hilbert space $\mathscr{H}_{F}$.}

\listoffigures

\listoftables

\printnomenclature[2cm]{}

\printindex{}

\bibliographystyle{alpha}
\bibliography{number5}

\end{document}